\documentclass[a4paper,11pt,twoside]{article}
\usepackage{amsmath,amssymb, mathtools} 
\usepackage{amscd} 
\usepackage{amsthm} 

\usepackage{dsfont} 
\usepackage{MnSymbol} 
\usepackage{stmaryrd} 

\usepackage{float} 

\usepackage[pdftex]{graphicx} 
\usepackage{pdfpages} 
\usepackage{xcolor} 
\usepackage[colorlinks,citecolor=blue,linkcolor=blue,urlcolor=blue]{hyperref} 
\usepackage{nameref} 
\usepackage{bm} 

\usepackage[
  top=3cm,
  bottom=3cm,
  left=3.5cm,
  right=2.6cm,
  headheight=16pt, 
  includehead,includefoot,
  heightrounded, 
]{geometry}

\usepackage{fancyhdr} 
\usepackage[utf8]{inputenc} 
\usepackage[english]{babel} 

\usepackage{tikz} 

\newtheorem{thm}{Theorem}
\newtheorem{conj}[thm]{Conjecture}
\newtheorem{lemma}[thm]{Lemma}
\newtheorem{coro}[thm]{Corollary}
\newtheorem{definition}[thm]{Definition}
\newtheorem{prop}[thm]{Proposition}
\newtheorem{example}[thm]{Example}
\newtheorem*{Remark}{Remark}
\newtheorem{oq}[thm]{Open question}
\newtheorem*{oq2}{Question ouverte}
\newtheorem*{oq3}{Open question}
\newtheorem*{fakethm}{Th\'eor\`eme}
\newtheorem*{fakethm2}{Theorem}
\newtheorem*{fakeprop}{Proposition}
\newtheorem*{fakeconj}{Conjecture}
\newtheorem*{fakecoro}{Corollary}
\newtheorem*{fakecoro2}{Corollaire}

\numberwithin{thm}{section} 
\numberwithin{equation}{section} 
\numberwithin{figure}{section} 

\newcommand*{\R}{\mathbb{R}}
\newcommand*{\C}{\mathbb{C}}
\renewcommand*{\H}{\mathbb{H}}
\renewcommand*{\S}{\Sigma}
\newcommand*{\g}{\mathfrak{g}}
\newcommand*{\h}{\mathfrak{h}}
\renewcommand*{\l}{\lambda}
\newcommand*{\mf}{\mathfrak}
\newcommand*{\mc}{\mathcal}

\newcommand*{\del}{\partial}
\newcommand*{\delbar}{\bar\partial}
\newcommand*{\cotang}{T^*\bm\hat{\mc{T}}^n}
\newcommand*{\T}{\bm\hat{\mc{T}}}

\DeclareMathOperator{\tr}{tr}
\DeclareMathOperator{\Sym}{Sym}
\DeclareMathOperator{\Symp}{Symp}
\DeclareMathOperator{\Hilb}{Hilb}

\DeclareMathOperator{\Diff}{Diff}
\DeclareMathOperator{\Vect}{Span}
\DeclareMathOperator{\Rep}{Rep}

\DeclareMathOperator{\diff}{Diff}
\DeclareMathOperator{\id}{id}

\DeclareMathOperator{\diag}{diag}

\DeclareMathOperator{\Comm}{Comm}
\DeclareMathOperator{\Hom}{Hom}

\DeclareMathOperator{\SU}{SU}

\DeclareMathOperator{\End}{End}
\DeclareMathOperator{\rk}{rk}
\DeclareMathOperator{\Stab}{Stab}
\DeclareMathOperator{\supp}{supp}
\DeclareMathOperator{\codim}{codim}
\DeclareMathOperator{\Span}{Span}
\DeclareMathOperator{\PSL}{PSL}
\DeclareMathOperator{\SL}{SL}
\DeclareMathOperator{\GL}{GL}
\DeclareMathOperator{\Res}{Res}
\DeclareMathOperator{\Lie}{Lie}

\DeclareMathOperator{\MCG}{MCG}
\DeclareMathOperator{\Vir}{Vir}

\begin{document}

\hypersetup{pageanchor=false}



\begin{titlepage}
\begin{center} 

\vspace*{3cm}

\hrulefill
 
\vspace*{0.3cm}
\Huge{\textbf{Higher Complex Structures and Higher Teichmüller Theory}}

\vspace*{-0.1cm}
\hrulefill

\vspace*{1.5cm}

\Large{Alexander Thomas} 

\vspace*{0.8cm}

\large{IRMA Strasbourg}

\vspace*{1.5cm}

\Large{Thèse de doctorat sous la direction de}

\vspace*{0.2cm}
\Large{Vladimir Fock}

\vspace*{1.2cm}
\large{10 Juin 2020}

\vspace*{0.6cm}
\large{Jury de thèse:}

\vspace*{0.6cm}

	\begin{tabular}{ll}
	Vladimir Fock & directeur de thèse \\
	Fran\c cois Labourie & rapporteur \\
	Richard Wentworth & rapporteur \\
	Nalini Anantharaman & examinatrice\\
	Steven Bradlow & examinateur \\
	Oscar García-Prada & examinateur \\
	Tamás Hausel & examinateur \\
	Nicolas Tholozan & examinateur
	\end{tabular}

\end{center}
\end{titlepage}

\begin{titlepage}

\hbox{}
	
\vfill

\noindent Alexander Thomas, \textit{Higher Complex Structures and Higher Teichmüller Theory}, PhD thesis, Université de Strasbourg, June 2020

\medskip
\noindent This work is licensed under the Creative Commons Attribution 4.0 International License. To view a copy of this license, please visit
\begin{center}
\url{https://creativecommons.org/licenses/by/4.0/}
\end{center}
or send a letter to Creative Commons, PO Box 1866, Mountain View, CA 94042, USA. 

\end{titlepage}

\begin{titlepage}
\vspace*{\fill}
\vspace*{-2cm}
\begin{flushright}
\Large{\textit{F\"ur Oma.}}
\end{flushright}
\vspace*{\fill}
\end{titlepage}

\begin{titlepage}
\hbox{}
\end{titlepage}

\hypersetup{pageanchor=true}
\pagenumbering{Roman}

\section*{Résumé court}

Dans cette thèse, on donne une nouvelle approche géométrique aux composantes des variétés de caractères. En particulier on construit une structure géométrique sur des surfaces, généralisant la structure complexe, et on explore son lien avec les composantes de Hitchin.

Cette structure, appelée structure complexe supérieure, est construite en utilisant le schéma de Hilbert ponctuel du plan.
Son espace des modules admet des propriétés similaires à la composante de Hitchin. On construit une courbe spectrale généralisée, une sous-variété (presque) Lagrangienne de l'espace cotangent complexifié de la surface.

Partant d'une structure complexe supérieure, on cherche à la déformer d'une fa\c con canonique en une connexion plate. 
L'espace de ces connexions plates, dites ``paraboliques'', s'obtient en imitant la réduction d'Atiyah--Bott. C'est un espace de paires d'opérateurs différentiels commutants. 
Sous une conjecture, on établit un difféomorphisme canonique entre l'espace des modules de notre structure géométrique et la composante de Hitchin.

Enfin, on généralise certaines constructions, comme le schéma de Hilbert ponctuel et la structure complexe supérieure, au cas d'une algèbre de Lie simple.

\vspace*{1cm}
\section*{Abstract}

In this PhD thesis, we give a new geometric approach to higher Teichmüller theory. In particular we construct a geometric structure on surfaces, generalizing the complex structure, and we explore its link to Hitchin components.

The construction of this structure, called higher complex structure, uses the punctual Hilbert scheme of the plane. Its moduli space admits similar properties to Hitchin's component. We construct a generalized spectral curve, an (almost) Lagrangian subvariety of the complexified cotangent space of the surface.

Given a higher complex structure, we try to canonically deform it to a flat connection. The space of such connections, called ``parabolic'', is obtained by imitating the Atiyah--Bott reduction. It is a space of pairs of commuting differential operators. Under some conjecture, we establish a canonical diffeomorphism between our moduli space and Hitchin's component.

Finally, we generalize certain constructions, like the punctual Hilbert scheme and the higher complex structure, to the case of a simple Lie algebra.

\cleardoublepage
\section*{Remerciements}
\vspace*{0.5cm}
\begin{flushright}
\textit{Voici mon secret. Il est très simple : on ne voit bien qu'avec le c\oe ur. L'essentiel est invisible pour les yeux.}

Antoine de Saint-Exupéry, Le Petit Prince
\end{flushright}
\vspace*{1cm}

\noindent Voilà le secret pour bien vivre les années d'une thèse. Car une thèse, c'est avant tout une aventure humaine. Malheureusement cet aspect humain est peu visible dans le rendu final d'une thèse, son manuscrit. Ici je voudrais remercier tous les aventuriers qui m'ont entouré, encouragé et aidé dans les contrées mathématiques et en dehors.


\bigskip \noindent
Tout d'abord je voudrais exprimer ma plus grande gratitude envers mon directeur de thèse, \textit{Vladimir Fock}. Tu as partagé avec moi ton riche univers mathématique, dans lequel tout est lié et tout s'explique souvent si facilement (même si cela peut prendre des détours surprenants et des années de réflexions). Même si tu n'es pas toujours joignable, surtout pour des problèmes administratifs, tu as été un encadrant formidable pour moi !

\medskip \noindent
Then I would like to thank all the members of the jury for all your comments and remarks!
Thanks to \textit{Fran\c cois Labourie} (also for your warm reception in Nice in January 2018), to \textit{Richard Wentworth} (also for the discussions we had in Moscow and Stony Brook), to \textit{Oscar Garcia-Prada} (also for your invitation to Madrid in November 2018, such a wonderful city!), to \textit{Steven Bradlow} (also for your invitation to Stony Brook in February 2019), to \textit{Tamás Hausel}, to \textit{Nalini Anantharaman} and to \textit{Nicolas Tholozan} (also for our discussion in Nancy in November 2019 and your remarks on the manuscript).

\medskip \noindent
A special thank you goes to \textit{Georgios Kydonakis, Johannes Horn} and \textit{Andrea Bianchi}, for all the discussions, exchanges and finally for your comments on the manuscript!

\noindent
Further I would like to thank the researchers and professors with whom I interacted through these years of PhD. First the members of IRMA, especially \textit{Olivier Guichard, Charles Frances, Florent Schaffhauser, Dragos Fratila, Ana Rechtman} and \textit{Tatiana Beliaeva}. Then the researchers I met in various conferences and occasions (and which are not cited below): \textit{Ga\"etan Borot, Brian Collier, Andy Sanders, Evgenii Rogozinnikov, Andreas Ott} and \textit{Fanny Kassel}.

\bigskip \noindent
La fin d'une thèse permet aussi de regarder en arrière. J'ai été porté à travers ces années par des formidables amis, des camarades et des enseignants passionnés. C'est une bonne occasion ici de vous exprimer toute ma gratitude !

\medskip\noindent
Mein Interesse an der Mathematik wurde schon früh geweckt, vor allem durch meine Lehrerin \textit{Frau Petra Wegert}. Ebenso danken möchte ich \textit{Herrn Stephan Hauschild} für das \glqq Drehtürmodell \grqq{} und \textit{Frank Göring} für das wunderbare Projekt rund um das Voderberg'sche Neuneck. 

\noindent
Danke an meine beiden \glqq Mitstreiter \grqq{}  \textit{Christoph und Florian}, und an meine Schulfreunde \textit{Nico, Max} und \textit{Iris}. 

\noindent
Durch die vielen Matheseminare und -olympiaden, unter den mysteriösen Zeichen MO, AIMO, MeMO, JuMa, BWM, JuFo, ITYM und LSGM abgekürzt, haben sich viele wertvolle Freundschaften ergeben. 
Ganz besonders denke ich an \textit{Kuno/Mani/Sebastian} (bitte eins auswählen) \textit{und Susi, Lukas, Xianghui} und \textit{Andrea}. Aber natürlich auch an \textit{Ilja, Jakob, Julius, Jean-Fran\c cois, Anne, Tim, Pham, Lisa, Fabian H, Achim, Michael}, ...

\medskip\noindent
La personne clé qui a fait germer en moi l'idée de faire mes études en France est \textit{Bodo Lass}. Un immense merci pour tes cours incroyables et passionnants sur la géométrie projective, ainsi que ton aide pour me retrouver en études supérieures en France !

\medskip \noindent
Après l'Abitur directement en classe préparatoire à Paris ! J'ai pu découvrir une nouvelle culture, si riche et fascinante, et approfondir toutes mes passions : les maths, les sciences, la musique et la langue fran\c caise.
Je suis particulièrement reconnaissant envers mes deux profs exceptionnels \textit{Yves Duval} et \textit{Ulrich Sauvage}.
Un grand merci à tous mes amis (HX1 - Torche !) : \textit{Paul, Andrei, Yuesong, Thomas, Théis, Hélicia, Sandrine, Lucas I, Jonathan, Leonard, Ulysse, Valère, Théodore, Lucas M, Diane}, ...

\medskip\noindent
Puis vient ``l'époque d'or'' à l'ENS de Lyon, un endroit idéal pour l'épanouissement personnel. Un grand merci à mes profs, surtout \textit{Jean-Claude Sikorav, Jean-Yves Welschinger} et \textit{Pascal Degiovanni}. Deux stages m'ont introduit à la recherche, un grand merci à mes encadrants : à \textit{Gwéna\"el Massuyeau} pour le stage autour des n\oe uds en 2014 à Strasbourg. Und an \textit{Prof. Carl-Friedrich Bödigheimer} für das Projekt über die Homologie der symmetrischen Gruppen in Bonn 2015.

\noindent
Les meilleurs moments à Lyon sont dus au fameux ``Bouquet'', le cercle secret des mélo\-manes. Impossible de formuler ces joies en mots ! Merci à tous les membres, surtout à \textit{Pierre} (le gourou), \textit{Louis, Emanuel, Célia, Flora, Laureena, Corentin} et \textit{Erwan}. A travers les départements scientifiques, pleins d'amitiés se sont nouées. Je pense avant tout à \textit{Vincent, Théo, Carine et Ruben, Laureline, Perrine, Garance, Keny, TPT, Lambert, Quentin, Leonard, Willy} et \textit{Edwin}.
Merci aussi à mes camarades \textit{Matthieu, Florent} et \textit{Benoît} pour les groupes de lecture.

\medskip\noindent
Vient enfin la période strasbourgeoise !

\noindent
Je remercie mes cobureaux \textit{Martin, Kien} et \textit{Amaury} pour l'ambiance et l'entraide. Puis mes ``frères académiques'' \textit{Nicolas} et \textit{Valdo} (combien de fois on s'est échangé pour savoir où était le ``chef'' ?) et mon futur frère Florian. 
J'ai été chaleureusement accueilli par les ``anciens doctorants''. Un grand merci à tous, en particulier à \textit{Audrey, Amandine} et \textit{le ``petit'}' et \textit{le ``grand'' Guillaume}.
Puis mes camarades doctorants formant la ``troupe du RU à midi''. Merci à \textit{Fred, Laura, Alix, Cuong, PA, Francisco, Leo} (le presque doctorant),  \textit{Lorenzo, Lukas, Djibril, l'autre Alexander, les deux Alexandre} (pas facile de nous distinguer !), \textit{Fran\c cois, Guillaume, Romane, Tsung-Hsuan} et \textit{Archia}.
Un point de rencontre important a été le séminaire doctorant et le comptoir (même si je n'ai pas été très assidu à ce dernier). Merci à tous les doctorants, en particulier (désolé si j'oublie quelques-uns !) \textit{Philippe, Thibault, Yohann, Xavier, les deux Maries, Manu, Gatien, Camille, Pierre, Luca, Greta, Antoine, Firat, Florian},...

\noindent Je remercie également l'équipe administrative (avant tout \textit{Madame Kostmann}), informatique et les bibliothécaires.

\noindent Une pensée va également aux étudiants auxquels j'ai donné des enseignements. Enseigner, c'est progresser soi-même dans la compréhension. Mon activité principale était le ``Cercle de maths de Strasbourg''. Merci à \textit{Tatiana, Alix, Laura} et \textit{Francisco} pour l'organisation du cercle et la participation au $\mathbb{T}\mathbb{F}\mathbb{J}\mathbb{M}^2$ !


\medskip\noindent
Ces années à Strasbourg sont marquées par mes activités en dehors de l'université. Je pense surtout au cercle des mélomanes (un prolongement du Bouquet). Merci à \textit{Deniz, Pierre J, Alix et Chloé, Thérèse, Lara, Corentin} et \textit{Eva} (qui nous a rejoints virtuellement pendant le confinement !). Je pense également à l'EVUS (merci aux super chefs \textit{Rémi} et \textit{Clotilde} et à tous les Evusiens !), au club de Go, aux Jongleurs du SUAPS et aux danseurs folk.
Danke an meine Freunde \textit{Jule und Hannah}, die ich durch einen so tollen Zufall kennenlernte!

\medskip\noindent
Par ailleurs, l'organisation d'une école d'été ``L'Agape'' m'a beaucoup inspiré. Merci à \textit{Pierre} (encore le gourou !), \textit{Fede, Jérémy} (tes pizzas sont inoubliables !), \textit{Vaclav, Titouan, Robin} et \textit{Johannes}. Cet événement s'inscrit dans un mouvement plus large, le \href{https://basic-research.org/}{BRCP} (Basic Research Community for Physics).

\medskip \noindent 
Merci à \textit{Claude, Fabienne} et \textit{Jeanne}, pour tous les accueils chaleureux que ce soit à Chaville, Brian\c con ou Houlgate !

\bigskip\noindent
Ein besonderer Dank geht an meine Familie, für die Unterstützung und die vielen gemeinsamen Momente! Einen tiefen Dank an \textit{Mutti und Andreas, Papsi und Petra, Oma Gisela und Opa Volker} (wie hätte ich je Klavier spielen gelernt?), \textit{Oma Dagi und Bodo, Opa Thomas und Katja, Antje, Uwe und Moritz, Nora und Pauline, Paula und Alma} und \textit{Onkel Ulli}. Auch \textit{Shira, Louis} und \textit{Louise} haben viel zu meinem Glück beigetragen.

\medskip\noindent
Ein ganz besonderer Gedanke geht an meine \textit{Uroma}, die leider von uns gegangen ist. Danke für alles, was wir geteilt haben! Dir widme ich diese Arbeit!

\bigskip\noindent
Enfin, un dernier mot pour toi, \textit{ma chère Eve}. Ton amour me porte à chaque instant ! Je ne saurais exprimer tous mes sentiments et ma gratitude. Un immense merci pour tes superbes montages des ``\href{https://www.youtube.com/channel/UCLyeqTvdc-Yr9jUZhmH_kKA}{Thomaths}'', et pour ta relecture entière de mon manuscrit !

\cleardoublepage
\tableofcontents

\clearpage

\section*{Index of notations}
\label{notations}

We gather here the main notations used in this thesis. We regroup them in general notations, notations on Lie groups and Lie algebras and then more specific notations for parts \ref{part1}, \ref{part2} and \ref{part3}. We also have a special notation for partitions.
\vspace*{0.5cm}

\paragraph{\underline{General notations:}}

\begin{center} 
\begin{tabular}{p{0.18\textwidth}p{0.78\textwidth}} 
$\Sigma$ & smooth closed connected surface, of genus $g\geq 2$\\
$z, \bar{z}$ & complex coordinates on $\Sigma$ \\
$\del, \delbar$ & differential with respect to $z$ and $\bar{z}$ resp.\\
$p, \bar{p}$ & linear coordinates on $T^{*\C}\Sigma$ \\
$x, y$ & real variables, or complex variables for polynomials in $\C[x,y]$ \\
$K$ & canonical line bundle of $\Sigma$, equals $T^{*(1,0)}\Sigma$ \\
$n$ & integer, mainly used in matrix groups like $\PSL_n(\R)$ \\
$\Gamma(.), H^0(.)$ & space of all (resp. holomorphic) sections of some bundle \\
$[a]$ & equivalence class of some element $a$ \\
$X \sslash G$ & hamiltonian reduction (symplectic quotient, or Marsden-Weinstein quotient) over the 0-coadjoint orbit \\
$\omega$ & symplectic form \\
$\Rep(\pi_1(\S),G)$ & character variety for $G$ \\
$\mc{T}(\Sigma), \mc{T}^2$ & Teichmüller space \\
$\mc{T}^n$ & Hitchin's component for $\PSL_n(\R)$ \\
$\bm\hat{\mc{T}}^n$ & moduli space of $n$-complex structures \\
$I$ & ideal, often in a punctual Hilbert scheme \\
$\Hilb^n(\C^2)$ & punctual Hilbert scheme of $n$ points of $\C^2$\\
$\Hilb^n_0(\C^2)$ & zero-fiber of the punctual Hilbert scheme \\
$\Hilb^n_{red}(\C^2)$ & reduced Hilbert scheme of $n$ points with barycenter 0 \\
$\Phi$ & connection from $n$-complex structure $\Phi = \Phi_1+\Phi_2$, or Higgs field \\
$\Phi_1, \Phi_2$ & $(1,0)$- and $(0,1)$-part of $\Phi$, or locally the associated matrices \\
$\mu_k$ & higher Beltrami differentials, define $n$-complex structure \\
$\mu$ & set of all Beltrami differentials $(\mu_2,...,\mu_n)$\\
$t_k$ & differentials, dual to $\mu_k$\\
$t$ & set of all differentials $(t_2,...,t_n)$ \\
${}_k\mu, {}_kt$ & conjugated higher complex structure and its deformation \\
$\Symp_0(T^*\S)$ & group of higher diffeomorphisms \\
$H$ & Hamiltonian, function on $T^*\S$\\
$M_x$ & multiplication operator by $x$ in some quotient $\C[x,y]/I$ \\
$\varepsilon$ & small parameter \\
$\bm\tilde{\S}$ & spectral curve inside $T^{*\C}\S$, or universal cover of $\S$ \\
$s$ & generic section of a bundle \\
$\langle e_i \mid e_j \rangle$ & Dirac's ``bra-ket'' notation for linear forms and vectors \\
 \end{tabular} 
\end{center}

\vspace*{0.5cm}

\pagebreak
\paragraph{\underline{Lie groups and Lie algebras:}}

\begin{center} 
\begin{tabular}{p{0.16\textwidth}p{0.8\textwidth}} 
$G$ & semisimple Lie group (complex or real) \\
$\g$ & semisimple Lie algebra (simple in part \ref{part3}) \\
$\mf{sl}_n, \mf{so}_n, \mf{sp}_{2n}$ & classical complex Lie algebras \\
$\h$ & Cartan subalgebra of $\g$ \\
$W$ & Weyl group of $\g$ \\
$\bm\hat{\g}$ & affine Lie algebra \\
$Z(A), Z(A, B)$ & centralizer of $A \in \g$, common centralizer of $A,B \in \g$
\end{tabular}
\end{center}
\vspace*{0.5cm}

\paragraph{\underline{Part 1:}}

\begin{center} 
\begin{tabular}{p{0.16\textwidth}p{0.8\textwidth}} 
$V$ & bundle induced by a higher complex structure\\
$P, Q$ & polynomials defining generators of an ideal $I\in \Hilb^n_{red}(\C^2)$\\
$v_k$ & coefficients of a Hamiltonian $H=v_2p+...+v_np^{n-1}$ \\
$\alpha_{i,j}$ & matrix entries of $M_{\bar{p}}$
\end{tabular}
\end{center}
\vspace*{0.5cm}

\paragraph{\underline{Part 2:}}

\begin{center} 
\begin{tabular}{p{0.16\textwidth}p{0.8\textwidth}} 
$\l$ & parameter in $\C P^1$ \\
$h$ & deformation parameter, equals $\l^{-1}$ \\
$\mc{O}(n)$ & holomorphic line bundle over $\C P^1$ with first Chern class $n$ \\
$\bm\hat{\mu}_k, \bm\hat{t}_k$ & coordinates of the space of parabolic connections \\
$\mc{A}$ & space of all connections \\
$\mc{G}$ & group of all gauge transformations \\
$\mc{P}$ & group of parabolic gauge transformations \\
$\xi_k$ & parabolic curvature\\
$\mc{M}_{H}(G)$ & moduli space of $G$-Higgs bundles on Riemann surface $\S$ \\
$\nabla, \bar{\nabla}$ & (1,0) and (0,1)-part of a covariant derivative\\
$A$ & a connection, often $A^*=-A$ \\
$A_1, A_2$ & (1,0) and (0,1)-part of $A$, or locally the associated matrices \\
$\mc{A}(\l)$ & affine connection $\mc{A}(\l) = \l\Phi + A+ \l^{-1}\Phi $\\
$\mc{A}_1(\l), \mc{A}_2(\l)$ & (1,0) and (0,1)-part of the connection $\mc{A}(\l)$\\
$\bm\hat{\alpha}_{i,j}$ & entries of matrix $\mc{A}_2(\l)$ \\
$\bm\hat{H}$ & quantum Hamiltonian \\
$\bm\hat{v}_k$ & coefficients of $\bm\hat{H}=\sum_k \bm\hat{v}_k\nabla^{k-1}$ \\
$\bm\hat{I}$ & left-ideal in the space of differential operators\\
$\bm\hat{P}, \bm\hat{Q}$ & polynomials defining $\bm\hat{I}$\\
$T$ & matrix defined by equation \eqref{matrixT}\\
$\varphi_i$ & parameters in $\Phi_1$, see Proposition \ref{phi1lower} \\
$\tau$ & Hitchin's split real form on $\mf{sl}_n$
\end{tabular}
\end{center}

\vspace*{0.5cm}

\paragraph{\underline{Part 3:}}

\begin{center} 
\begin{tabular}{p{0.16\textwidth}p{0.8\textwidth}} 
$\Hilb(\g)$ & $\g$-Hilbert scheme defined in \ref{def-g-hilb}\\
$\Hilb_0(\g)$ & zero-fiber of $\Hilb(\g)$ \\
$\Hilb^{reg}(\g)$ & regular part of $\Hilb(\g)$, see \ref{partshilb} \\
$\Hilb^{cycl}(\g)$ & cyclic part of $\Hilb(\g)$, see \ref{partshilb} \\
$\bm\hat{\mc{T}}_\g$ & moduli space of $\g$-complex structures \\
$I_\g$ & ideal of type $\g$\\
$\Symp(\g,\S)$ & group of higher diffeomorphisms of type $\g$\\
$e, f$ & nilpotent elements of a principal $\mf{sl}_2$-triple in $\g$\\
$ch$ & Chow map \\
$\rho$ & natural embedding of a classical $\g$ into some $\mf{gl}_n$\\
$S$ & matrix in $\mf{so}_{2n}$ defined in \eqref{matrixS}\\
$\sigma_n$ & extra Beltrami differential for $\mf{so}_{2n}$, coefficient of $S$\\
$\nu_{2n-2}$ & extra Beltrami differential for $\mf{so}_{2n}$, linked to $\sigma_n$\\
$\tau_n$ & extra holomorphic differential for $\mf{so}_{2n}$
\end{tabular}
\end{center}
\vspace*{0.5cm}

\paragraph{\underline{Partitions:}}
\label{notations2}

\begin{tabular}{p{0.16\textwidth}p{0.8\textwidth}} 
\end{tabular}
\vspace{0.5cm}

\noindent A partition $\pi$ of an integer $n$ is written $\pi \vdash n$. Further, decompose the partition into $\pi=(\pi_1,...,\pi_n)$ where $\pi_k$ is the number of $k$'s in the partition, i.e. $n=\sum_k k\pi_k$, and denote by $\lvert \pi \rvert =\pi_1+...+\pi_n$ the number of parts, also called the \textit{length} of $\pi$. In the Young diagram associated with $\pi$, the number $\pi_k$ counts the number of rows of length $k$ and $\lvert \pi \rvert$ is the total number of rows.


\cleardoublepage

\pagenumbering{arabic}

\section{Introduction et résumé (en fran\c cais)}
\vspace*{0.5cm}
\begin{flushright}
\textit{On ne peut rien enseigner à autrui, on ne peut que}

\textit{l'aider à le découvrir lui-même.}

Galileo Galilée
\end{flushright}
\vspace*{1cm}

\noindent L'objectif principal de cette thèse est de donner une interprétation géo\-métrique à la composante de Hitchin, l'objet d'étude de la théorie de Teichmüller supérieure. Pour cette raison, nous commen\c cons par un aper\c cu de la théorie de Teichmüller et sa généralisation, l'étude des variétés de caractères. Après nous donnons quelques bases de la théorie des fibrés de Higgs. Nous finissons par un résumé détaillé de la structure et des résultats de la thèse.

\subsection{Théorie de Teichm\"uller}

On considère $\Sigma$ une surface lisse compacte orientée de genre $g \geq 2$. On peut l'enrichir d'une structure géométrique. L'étude globale d'une telle structure mène à la notion d'un espace des modules, l'espace de toutes les structures géométriques modulo une relation d'équivalence. Il est surprenant que dans le cas d'une surface, des structures géométriques très différentes définissent un même espace de modules, \textbf{l'espace de Teichmüller}, qu'on dénote par $\mc{T}(\Sigma)$.

\medskip
\noindent \textit{Structure complexe.}
Une \textbf{structure complexe} sur une surface est un atlas maximal de cartes dont les images sont des ouverts de $\C$ tel que les fonctions de transitions sont holomorphes. Une surface munie d'une structure complexe est appelée \textbf{surface de Riemann}.
L'espace de Teichmüller est l'espace de toutes les structures complexes (compatibles avec l'orientation sur $\S$) modulo $\diff_0(\Sigma)$, le groupe des difféomorphismes isotopes à l'identité. Cela signifie qu'on identifie deux structures complexes lorsque l'une est le tiré-en-arrière de l'autre par une isotopie.

\medskip
\noindent \textit{Structure presque-complexe.}
Une \textbf{structure presque-complexe} sur une variété $M$ est la donnée d'un endomorphisme $J(m)$ sur l'espace tangent $T_mM$ pour tout point $m \in M$ tel que $J(m)$ dépend d'une fa\c con lisse de $m$ et vérifie $J(m)^2=-\id$. Un tel endomorphisme simule la multiplication par $i$.

\noindent Une structure complexe induit toujours une structure presque-complexe en utilisant la multiplication par $i$. Réciproquement, pour aller d'une structure presque-complexe vers un atlas complexe, il y a une obstruction donnée par une condition d'intégrabilité (c'est le théorème de Newlander-Nirenberg). Pour une surface, cette obstruction est toujours nulle : \textit{toute structure presque-complexe sur une surface est intégrable}. Ce fait est dû à Gauss dans le cas analytique réel (existence de coordonnées ``isothermes'') et à Korn et Lichtenstein dans un cadre plus général.

\noindent Par conséquent, on peut voir $\mathcal{T}(\Sigma)$ comme l'espace de toutes les structures presque-complexes (compatibles avec l'orientation) modulo isotopies.

\medskip
\noindent \textit{Structure Riemannienne.}
Considérons l'espace de toutes les métriques Riemanniennes sur $\Sigma$. On quotiente cet espace par les isotopies et par l'équivalence conforme. Deux métriques $g_1$ et $g_2$ sont dites \textit{conformément équivalentes} s'il existe une fonction lisse positive $\lambda: \Sigma\rightarrow \R^{+}$ telle que $g_2 = \lambda g_1$. Une classe conforme de métriques donne la notion d'orthogonalité. Ainsi pour une telle classe donnée, il y a exactement deux structures presque-complexes compatibles (rotation de $\pm 90^\circ$) et un choix d'orientation sur $\S$ permet d'en choisir une canoniquement. Par conséquent, l'espace des classes conformes de métriques est encore l'espace de Teichm\"uller.

\medskip
\noindent \textit{Structure hyperbolique.}
Une \textbf{structure hyperbolique} sur une surface est une métrique de courbure constante égale à $-1$. 
Le fameux \emph{théorème d'uniformisation de Poincaré} implique que toute métrique est conformément équivalente à une une métrique de courbure -1 (car le genre de $\S$ est $\geq 2$).
C'est pourquoi l'espace de Teichmüller est aussi l'espace de toutes les structures hyperboliques modulo isotopies.

\medskip
\noindent \textit{Variété des caractères.}
Il y a encore une autre fa\c con de voir $\mathcal{T}(\Sigma)$ : comme une composante connexe d'une \textbf{variété des caractères}. 

Étant donné un groupe $\Gamma$, on l'étudie à travers ses actions sur des espaces vectoriels. Pour une dimension $n$ donnée, une classe d'isomorphisme de représentations est donnée par un élément de $$\Hom(\Gamma, \GL_n(\C))/\GL_n(\C)$$ où l'action de $\GL_n(\C)$ est par conjugaison. Il s'est avéré fructueux de se restreindre à un sous-groupe de Lie $G$ de $\GL_n(\C)$ ou plus généralement à un groupe de Lie semi-simple $G$ quelconque. L'espace $\Hom(\Gamma,G)/G$, noté $\Rep(\Gamma,G)$, est appelé variété des caractères (de $\Gamma$ dans $G$).
Pour $\Gamma = \pi_1(\Sigma)$ le groupe fondamental d'une surface, les variétés de caractères sont liées aux structures géométriques sur $\Sigma$.

Dans notre cas, prenons une structure complexe sur $\Sigma$. On peut la relever sur son revêtement universel $\bm\tilde{\Sigma}$ qui est topologiquement un disque (car $g>1$). D'après le théorème d'uniformisation de Poincaré, $\bm\tilde{\Sigma}$ avec la structure complexe venant de $\Sigma$ est biholomorphe au plan hyperbolique $\mathbf{H}$ avec sa structure complexe standard. Pour retrouver $\Sigma$, il suffit de quotienter son revêtement universel par l'action du groupe fondamental $\pi_1(\Sigma)$ qui agit par isométries préservant l'orientation. Comme le groupe des isométries préservant l'orientation du plan hyperbolique $\mathbf{H}$ est $\PSL_2(\R)$, une telle action est donnée par un morphisme de $\pi_1(\Sigma)$ vers $\PSL_2(\R)$. Ce morphisme doit être fidèle et discret pour obtenir une action libre et proprement discontinue. Un tel morphisme est appelé \textbf{fuchsien}. Si deux représentations fuchsiennes ne diffèrent que d'une homographie, les structures complexes sur $\Sigma$ sont isotopes et réciproquement. Ainsi, l'espace de Teichmüller est la composante connexe de $\Rep(\pi_1(\Sigma), \PSL_2(\R))$ des morphismes discrets et fidèles.

\bigskip
\noindent Pour résumer, nous avons
\begin{align*}
\mathcal{T}(\Sigma) &= \{\text{structures complexes compatibles}\} / \Diff_0(\Sigma) \\
&= \{\text{structures presque-complexes compatibles}\} / \Diff_0(\Sigma) \\
&= \{\text{structures Riemanniennes}\} / (\Diff_0(\Sigma)+\text{équivalence conforme}) \\
&= \{\text{structures hyperboliques}\} / \Diff_0(\Sigma) \\
&= \text{ composante connexe de } \Rep(\pi_1(\Sigma),\PSL_2(\mathbb{R}))
\end{align*}

Ces multiples points de vue expliquent le grand intérêt porté sur l'espace de Teichmüller. Il est lui-même riche en structures :  il admet par exemple une structure complexe et une structure symplectique. On résume quelques propriétés importantes de l'espace de Teichmüller dans le théorème suivant qui est dû à Teichmüller, Ahlfors, Bers, Fricke, Goldman et d'autres. 

\begin{fakethm}[Teichmüller, Ahlfors, Bers, Fricke, Goldman,...]
L'espace de Teichmüller $\mathcal{T}(\Sigma)$ est une variété réelle contractile de dimension $6g-6$. Il a une structure complexe et une structure symplectique qui donnent ensemble une structure de variété de Kähler.

\noindent En tout point $\mu$ de $\mathcal{T}(\Sigma)$, son espace cotangent est donné par $$T^*_{\mu}\mathcal{T}(\Sigma) \cong H^0(K^2)$$ où $H^0(K^2)$ est l'espace des différentielles quadratiques holomorphes (par rapport à la structure complexe $\mu$).

\noindent Le groupe des difféotopies de $\Sigma$ agit proprement discontinûment sur $\mathcal{T}(\Sigma)$ en préservant la structure Kählerienne. L'espace quotient est l'espace des modules de Riemann $\mathcal{M}(\Sigma)$. 
\end{fakethm}

Le groupe des difféotopies d'une surface (en anglais \textit{mapping class group}) est le quotient du group de tous les difféomorphismes par le sous-groupe des isotopies. L'espace des modules de Riemann $\mathcal{M}(\Sigma)$ est l'espace des structures complexes modulo tous les difféomorphismes.
Pour une exposition excellente de la théorie de Teichmüller, nous recommandons fortement le livre de Hubbard \cite{Hub}. En particulier, il y prouve chaque partie du théorème précédent (voir théorèmes 6.5.1, 6.6.2, 6.7.2, 7.1.1 et 7.7.2). Pour un traitement historique via les applications quasi-conformes, voir le livre d'Ahlfors \cite{Ahlfors}.

\subsection{Théorie de Teichm\"uller supérieure}\label{higherteichth}
Dans son article pionnier \cite{Hit.1}, Nigel Hitchin décrit une composante connexe d'autres variétés de caractères avec des propriétés similaires à l'espace de Teichmüller. Ces composantes sont maintenant appelées \textbf{composantes de Hitchin} et leur étude \textbf{théorie de Teichmüller supérieure}. On recommande l'article \cite{Wienhard} pour un aper\c cu détaillé de cette théorie.

Le point de départ de la théorie est de remplacer dans la variété des caractères le groupe $\PSL_2(\R)$ par un autre groupe de Lie $G$, i.e. d'étudier l'espace $\Rep(\pi_1(\S), G)$. La construction de Hitchin fonctionne pour tout groupe de Lie $G$ sans centre associé à une forme réelle déployée d'une algèbre de Lie simple. Un exemple typique est $\PSL_n(\R)$. Nous nous restreignons à ce groupe pour la suite.

Pour $\PSL_n(\R)$, la composante de Hitchin peut être décrite simplement : d'après la théorie des représentations, on sait que $\SL_2(\C)$ admet une unique représentation irréductible de dimension $n$. Par restriction à $\SL_2(\R)$, on obtient une application $\SL_2(\R) \rightarrow \SL_n(\R)$ qui factorise à travers $\{\id, -\id\}$. En fin de compte, on obtient une application $\PSL_2(\R)\rightarrow \PSL_n(\R)$, appelée \textbf{application principale}.
En composant une application fuchsienne $\pi_1(\Sigma)\rightarrow \PSL_2(\mathbb{R})$ avec l'application principale on obtient un morphisme du groupe fondamental vers $\PSL_n(\R)$. Une telle composition est appelée $\mathbf{n}$\textbf{-fuchsienne}. Un morphisme $\pi_1(\S) \rightarrow \PSL_n(\R)$ est appelé $\mathbf{n}$\textbf{-Hitchin} s'il est possible de le déformer continûment en un morphisme $n$-fuchsien. 

L'espace des représentations $n$-Hitchin forme une composante connexe de la variété des caractères $\Rep(\pi_1(\S),\PSL_n(\R))$. On dénote la composante de Hitchin par  $\mathcal{T}^n$. Pour $n=2$ on retrouve l'espace de Teichmüller. 
On résume ce qui est connu sur la composante de Hitchin :

\begin{fakethm}[Hitchin, Goldman, Labourie]
La composante de Hitchin $\mathcal{T}^n$ est une variété réelle contractile de dimension $(n^2-1)(2g-2)$. Elle admet une structure symplectique et une action proprement discontinue par le groupe des difféotopies de $\Sigma$.
\end{fakethm}

La structure symplectique est donnée par une construction générale, due à Goldman, de structures symplectiques sur des variétés de caractères (cf. \cite{Goldman.2} et \cite{AtBott}). L'action proprement discontinue du groupe des difféotopies sur la composante de Hitchin a été décrite par Labourie, voir \cite{Lab.3}.

Une question naturelle est de savoir si la composante de Hitchin admet une description géométrique tel l'espace de Teichmüller. Plus précisément :
\begin{oq2}
Y a-t-il une structure géométrique sur la surface $\S$ dont l'espace des modules est la composante de Hitchin ?
\end{oq2}
Cette question est la motivation principale de cette thèse. 

\medskip \noindent
De nombreuses approches ont été proposées pour donner une interprétation géométrique à la composante de Hitchin. Goldman, Guichard--Wienhard, Labourie et d'autres ont décrit la composante de Hitchin via des structures géométriques sur des fibrés sur la surface. Pour $\PSL_3(\R)$ cette structure géométrique est la structure projective convexe introduite par Goldman dans \cite{Goldman}. Pour $n=4$, Guichard et Wienhard décrivent des structures convexes feuilletées (en anglais \textit{convex foliated structure}) sur l'espace tangent unitaire de $\S$ (voir \cite{Guichard}). Labourie introduit la notion fructueuse d'une \emph{représentation Anosov} dans \cite{Lab}. 
L'inconvénient de ces constructions (pour $n>3$) est que le fibré sur lequel sont définies les diverses structures géométriques n'est pas canoniquement associé à la surface. 

Toutes ces généralisations sont des structures géométriques \emph{rigides} (ce qui signifie que le groupe local d'automorphismes est de dimension finie). On peut les considérer comme une généralisation de la structure hyperbolique sur la surface. 
Notre approche généralise la structure complexe et aboutit à une structure \emph{flexible} (le groupe local des automorphismes contient les fonctions analytiques, donc est de dimension infinie).

\subsection{Connexions plates et fibrés de Higgs}

L'approche originelle de Hitchin dans \cite{Hit.1} pour détecter les composantes qui portent son nom est analytique via la théorie des fibrés de Higgs. On donne ici un aper\c cu très bref de connexions plates et de la théorie des fibrés de Higgs. Pour plus de détails, nous recommandons \cite{Wentworth}, \cite{Garcia}, et le cours d'Andrew Neitzke \cite{Neitzke}.

Le point de départ est le lien important entre la variété des caractères et les connexions plates. Cette relation porte le nom de \textbf{correspondance de Riemann-Hilbert}. Étant donné un $G$-fibré sur une surface $\S$ (où $G$ est un groupe de Lie) avec une connexion plate, alors sa monodromie donne une application $\pi_1(\S) \rightarrow G$. Réciproquement, une telle représentation provient toujours d'une connexion plate dans un fibré. Remarquons qu'un fibré qui admet une connexion plate a toutes ses composantes irréductibles de degré zéro. La variété des caractères $\Rep(\pi_1(\S),G)$ est ainsi identifiée à \textit{l'espace des classes de jauge de $G$-connexions plates dans des fibrés sur $\S$}.

L'action par conjugaison de $G$ sur $\Hom(\pi_1(\S), G)$, dont le quotient donne la variété des caractères, n'est pas libre. Il est connu que les points réguliers de $\Hom(\pi_1(\S),G)$ sont ceux pour lesquels la représentation de $\pi_1(\S)$ est irréductible. 
En général, la variété des caractères est ainsi singulière et non-séparée. 
Il est pourtant possible de définir le quotient $\Hom(\pi_1(\S),G) /G$ autrement pour donner une variété séparée, qui est lisse aux points réguliers. Sans entrer dans les détails, citons seulement deux possibilités : le quotient géométrique invariant (en anglais \textit{GIT quotient}) ou bien la réduction Hamiltonienne (ou réduction symplectique). Ces procédés font apparaître des conditions dites de stabilité.  

L'exemple classique est la variété des caractères pour le groupe unitaire $G=U(n)$. La condition de stabilité se formule le mieux dans le langage des fibrés : étant donné un fibré $E$ sur la surface $\S$, on appelle \textbf{pente} de $E$ (en anglais \textit{slope}) le quotient du degré de $E$ par son rang : $$\mu(E) = \frac{\deg E}{\rk E}.$$ Un fibré est appelé \textbf{stable} si pour tout sous-fibré strict $F$ on a $\mu(F)<\mu(E)$. On l'appelle \textbf{semi-stable} si $\mu(F) \leq\mu(E)$. Enfin un fibré est \textbf{polystable} s'il est la somme de fibrés stables ayant tous la même pente.

Le théorème célèbre de Narasimhan--Seshadri affirme (voir \cite{NS}) :
\begin{fakethm}[Narasimhan--Seshadri]
Toute représentation unitaire du groupe fondamental $\pi_1(\S)$ provient d'une connexion plate sur un fibré de degré zéro polystable. 

Réciproquement, sur un fibré holomorphe de degré zéro stable, il existe une unique (à jauge unitaire près) connexion unitaire plate compatible avec la structure holomorphe.
\end{fakethm}
Ce théorème se généralise à toute forme compacte d'un groupe de Lie complexe simple.

\bigskip
\noindent Pour obtenir des informations pour d'autres groupes de Lie, en particulier les formes déployées, la notion de fibré de Higgs s'est avérée fructueuse. On fixe une surface de Riemann $\S$. Un \textbf{fibré de Higgs} est un fibré holomorphe $V$ sur $\S$ muni d'une 1-forme holomorphe $\Phi$ à valeurs dans $\End(V)$. Dans une écriture plus sophistiquée : $\Phi \in H^0(\End(V)\otimes K)$ où $K=T^{*(1,0)}\S$ dénote le fibré canonique. L'élément $\Phi$ est appelé \textbf{champ de Higgs}. On peut y penser comme un \textit{vecteur tangent à l'espace des connexions}. C'est un élément qu'on peut ajouter à une connexion donnée $\nabla$ pour obtenir une nouvelle connexion $\nabla+\Phi$. Si on écrit localement $\nabla = d+A$ avec $A$ une 1-forme à valeur dans $\End(V)$, la différence entre $A$ et $\Phi$ est que l'action d'une jauge $g$ sur $A$ s'écrit $g.A = gAg^{-1}+gdg^{-1}$ tandis que sur $\Phi$, on n'a que la conjugaison : $g.\Phi = g\Phi g^{-1}$.

La condition de stabilité pour un fibré de Higgs $(V,\Phi)$ est la suivante : il faut que pour tout sous-fibré $\Phi$-invariant, la pente soit plus petite que celle de $V$. 
L'espace des modules des $G$-fibrés de Higgs, noté $\mc{M}_{H}(G)$, est \textit{l'espace des classes de jauges de fibrés de Higgs polystables}.

Le résultat principal de la théorie des fibrés de Higgs est la \textbf{correspondance de Hodge non-abelienne}, fruit de plusieurs articles de Corlette \cite{Corlette}, Donaldson \cite{Donaldson}, Hitchin \cite{Hit.2} et Simpson \cite{Simpson}.
On peut l'énoncer en deux parties : étant donné un fibré de Higgs $(V, \Phi)$ stable tel que $\deg V=0$, alors il existe une unique (à jauge unitaire près) connexion unitaire $A$ compatible avec la structure holomorphe telle que 
\begin{align*}
\delbar \Phi+[A^{(0,1)},\Phi] = 0 \\
F_A + [\Phi, \Phi^*]  = 0
\end{align*}
où $F_A$ désigne la courbure de la connexion $A$ et l'opération $*$ est l'adjoint par rapport à la structure hermitienne sur $V$ définie par $A$. 
Ces équations, appelées \textbf{équations de Hitchin}, sont équivalentes à la platitude de la connexion $\nabla  = d+\Phi+A+\Phi^*$.

Réciproquement, si un fibré $V$ admet une connexion plate $\nabla$ qui est complètement réductible (i.e. la monodromie donne une représentation complètement réductible) alors il existe une métrique sur $V$ telle que $\nabla$ se décompose en $d+\Phi+A+\Phi^*$ satisfaisant les équations de Hitchin.

Les deux parties ensemble donnent une équivalence entre l'espace des modules des fibrés de Higgs et les représentations complètement réductibles du groupe fondamental : $$\Rep^{c.r.}(\pi_1(\S),G) \cong \mc{M}_{H}(G) $$ où $G$ est un groupe de Lie complexe semi-simple.
Une fa\c con équivalente d'exprimer la correspondance de Hodge non-abelienne est de dire que l'espace $\mc{M}_{H}(G)$ est un espace hyperkählerien. Plus de détails sur ce point de vue seront exposés dans les sections \ref{hkhilbertscheme} et \ref{bigpicture}.

Dans son article \cite{Hit.1}, Hitchin introduit une application $$\mc{M}_{H}(\SL_n(\C))\rightarrow \bigoplus_{i=2}^n H^0(K^i)$$ appelée \textbf{fibration de Hitchin}. Elle associe à $(V,\Phi)$ les coefficients du polynôme caractéristique de $\Phi$. En utilisant la forme de Frobenius d'une matrice, Hitchin construit une section à la fibration dont l'image, à travers la correspondance de Hodge non-abelienne, est à monodromie dans $\PSL_n(\R)$. Il prouve que cette image est une composante connexe de $\Rep(\pi_1(\S), \PSL_n(\R))$. Ainsi la composante de Hitchin est paramétrée par des différentielles holomorphes.

\subsection{Résumé}

Cette thèse est composée de quatre parties et d'un appendice. Le grand thème est de construire une nouvelle approche géométrique à la théorie de Teichmüller supérieure.

\bigskip
\noindent La partie \ref{part1} traite de la construction et des propriétés d'une nouvelle structure géométrique sur des surfaces, la structure complexe supérieure.

Une structure complexe sur une surface $\S$ peut être encodée par la \textit{différentielle de Beltrami} $\mu_2$, qu'on peut voir comme une direction dans l'espace cotangent complexifié $T^{*\C}\S$. L'idée de la généralisation est de remplacer cette direction linéaire par un $n$-jet de courbe.

Pour formaliser cette idée, nous utilisons un outil algébrique : le \textit{schéma de Hilbert ponctuel du plan} $\Hilb^n(\C^2)$ qu'on introduit et explique en détail dans la section \ref{Hilbertscheme}. Un point dans $\Hilb^n(\C^2)$ est un idéal de $\C[x,y]$ de codimension $n$. On peut y penser comme l'espace des $n$-uplets de points du plan $\C^2$ conservant une information supplémentaire quand plusieurs points co\"incident.
La \textit{zéro-fibre} $\Hilb^n_0(\C^2)$ est l'espace des idéaux supportés en l'origine, càd. que les $n$ points co\"incident tous avec l'origine. Un élément générique de la zéro-fibre peut être considéré comme un $(n-1)$-jet de courbe, la courbe le long de laquelle les $n$ points sont entrés en collision.

Le schéma de Hilbert ponctuel admet une description matricielle comme la variété des classes de conjugaison de paires de matrices commutantes. La zéro-fibre correspond aux matrices commutantes qui sont nilpotentes. Voir la sous-section \ref{hilbmatrixviewpoint}.

Dans la section \ref{Highercomplexsection} on définit la \textit{structure complexe supérieure} (ou \textit{structure $n$-complexe}). On applique le schéma de Hilbert point par point aux espaces cotangents complexifiés $T_z^{*\C}\S$ pour obtenir un fibré en schémas de Hilbert noté $\Hilb^n(T^{*\C}\S)$. La structure complexe supérieure est définie comme une section de la zéro-fibre $\Hilb^n_0(T^{*\C}\S)$. Cette structure généralise la structure complexe qu'on retrouve pour $n=2$. Une telle structure est caractérisée par des \textit{différentielles de Beltrami supérieures} $\mu_2, ..., \mu_n$.
On peut penser à la structure $n$-complexe comme une sorte de pelouse sur la surface (polynomiale dans l'espace cotangent), ou bien comme un $n$-jet de surface à l'intérieur de $T^{*\C}\S$ le long de la section nulle, ou bien comme certaines 1-formes à valeurs dans les matrices qu'on peut écrire localement sous la forme $\Phi_1dz+\Phi_2d\bar{z}$ où $\Phi_1$ et $\Phi_2$ sont des matrices nilpotentes commutantes.

Pour obtenir un espace des modules de dimension finie, on considère les structures $n$-complexes modulo une relation d'équivalence. Comme la structure $n$-complexe est polynomiale dans l'espace cotangent, il est naturel de considérer des transformations polynomiales de $T^*\S$. De telles transformations peuvent être obtenues par des symplectomorphismes de $T^*\S$ engendrés par des Hamiltoniens polynomiaux. Ces transformations sont appelées \textit{difféomorphismes supérieurs}, et leur groupe est noté $\Symp_0$. 

Le premier résultat de la thèse est que la théorie locale de la structure $n$-complexe est triviale (voir Théorème \ref{loctrivial}) :
\begin{fakethm}[Théorie locale]
La structure $n$-complexe peut être localement trivialisée par un difféomorphisme supérieur. Autrement dit, deux structures complexes supérieures sont localement équivalentes.
\end{fakethm}

L'espace des modules $\T^n$ des structures $n$-complexes admet les propriétés suivantes (voir Théorème \ref{mainresultncomplex} et Proposition \ref{copyteich}) :

\begin{fakethm}[Théorie globale]
Pour une surface de genre au moins 2, l'espace des modules $\T^n$ est une variété contractile de dimension complexe $(n^2-1)(g-1)$. Son espace cotangent en un point $\mu$ est donné par (où $K$ désigne le fibré canonique)
$$T^*_{\mu}\bm\hat{\mathcal{T}}^n = \bigoplus_{m=2}^{n} H^0(K^m).$$
En outre, il y a une application d'oubli $\bm\hat{\mathcal{T}}^n \rightarrow \bm\hat{\mathcal{T}}^{n-1}$ et une copie de l'espace de Teichmüller $\mc{T}^2 \rightarrow \bm\hat{\mc{T}}^n$.
\end{fakethm}

Une propriété importante du schéma de Hilbert est qu'il admet une structure symplectique complexe et même \textit{hyperkählerienne}, qu'on étudie dans la section \ref{moreonhilbertschemes}. Ceci permet de décrire dans la section \ref{Spectralcurve1} l'espace cotangent total $\cotang$ qui va jouer un rôle important dans le lien avec la théorie de Teichmüller supérieure. Un point de $\cotang$ est caractérisé par des différentielles $t_k$ et des différentielles de Beltrami supérieures $\mu_k$ avec une condition sur $t_k$ qui généralise la condition d'holomorphicité. Dans le cas où $\mu_k=0$ pour tout $k$, on retrouve la condition $\delbar t_k=0$.

Un point de $\cotang$ est une section du fibré en schémas de Hilbert $\Hilb^n(T^{*\C}\S)$, ou bien un $n$-uplet de 1-formes ($n$ points dans chaque espace cotangent). Cette collection de 1-formes définit une \textit{courbe spectrale} $\bm\tilde{\S}$ dans $T^{*\C}\S$, qui est analysée dans \ref{sspectral}. La propriété principale de $\bm\tilde{\S}$ est qu'elle est Lagrangienne modulo $t^2$ (voir Théorème \ref{spectralcurveprop}).
Ainsi, les pé\-riodes de la 1-forme de Liouville de $T^{*\C}\S$ sont bien définies (modulo $t^2$). On conjecture que ces périodes donnent un système de coordonnées sur $\cotang$ et, avec une procédure limite adaptée, aussi sur $\T^n$.

Dans la section \ref{dualcomplexstructure}, on définit une structure conjuguée à la structure complexe supérieure ainsi qu'à un point de $\cotang$.
Pour $n=2$ l'espace $\cotang$ s'interprète comme l'ensemble des surfaces de demi-translation (\textit{half-translation surfaces} en anglais). Ainsi il admet une action naturelle de $\GL_2(\R)$. Dans la section \ref{gl2} nous montrons que $\GL_2(\R)$ agit également sur $\cotang$ pour tout $n$.

\bigskip
\noindent La partie \ref{part2} tisse un lien avec les composantes de Hitchin en exploitant la structure hyperkählerienne du schéma de Hilbert. Nous déformons $\cotang$ en un espace de connexions plates.
Un point de $\cotang$ étant décrit essentiellement par deux polynômes, leur déformation consiste en une paire d'opérateurs différentiels commutants à laquelle on peut associer une connexion plate. 
Pour faire le lien avec la théorie de Teichmüller supérieure, on cherche à associer \textit{canoniquement} une connexion plate à un point de $\cotang$. 
Nous présentons des résultats partiels dans cette direction.
Dans la section \ref{bigpicture}, on expose les grandes lignes de notre approche et on la compare à celle de Hitchin.

Du point de vue matriciel, une structure complexe supérieure est une 1-forme à valeurs dans $\mf{sl}_n(\C)$ de la forme $\Phi=\Phi_1+\Phi_2$. La déformation de cette structure consiste à l'inclure dans une famille de connexions à paramètre $\l$ de la forme $d+\l\Phi+A+\l^{-1}\Phi^*$. Dans la section \ref{holodiffs-flat-affine}, on extrait d'une connexion de ce type des différentielles holomorphes $t_k$.

Un théorème célèbre d'Atiyah--Bott établit une bijection entre l'espace des connexions plates et la réduction hamiltonienne de l'espace de toutes les connexions $\mc{A}$ par les jauges $\mc{G}$. La section \ref{parabolicreduction} étudie la réduction de $\mc{A}$ par un sous-groupe $\mc{P}$ des jauges, appelées \textit{jauges paraboliques}, qui fixent une direction donnée. Nous démontrons que la réduction $\mc{A}\sslash\mc{P}$ est un espace de paires d'opérateurs différentiels, paramétré par des variables $\bm\hat{t}_k$ et $\bm\hat{\mu}_k$ avec $2\leq k \leq n$.

Ces deux opérateurs différentiels commutent si les $\bm\hat{t}_k$ vérifient une certaine condition similaire à la condition des variables $t_k$ de $\cotang$. On peut obtenir cette condition sur $\bm\hat{t}_k$ comme l'application moment d'une deuxième réduction hamiltonienne. Dans \ref{symponconnections}, nous expliquons comment associer une jauge à un difféomorphisme supérieur. En effectuant une deuxième réduction par rapport à ces jauges, on obtient (voir Corollaire \ref{flatparaconnections}) :
\begin{fakethm}
La double réduction $\mathcal{A}\sslash\mathcal{P}\sslash\Symp_0$ est un espace de connexions plates.
\end{fakethm}

Dans la section \ref{parabolicwithlambda}, on généralise la réduction parabolique à l'espace des $h$-connexions. Ainsi on inclut $\mc{A}\sslash\mc{P}\sslash\Symp_0$ dans une famille à un paramètre d'espaces de même type. Appelons $\l \in \C^*$ ce paramètre ($\l=h^{-1}$). On obtient ainsi un espace de paires d'opérateurs différentiels dépendant d'un paramètre $\l$, paramétré par $\bm\hat{t}_k(\l)$ et $\bm\hat{\mu}_k(\l)$. Quand $\l$ tend vers l'infini on retrouve les deux polynômes décrivant l'espace $\cotang$. Plus précisément, dans un développement de Taylor, les plus hauts termes en $\l$ de $\bm\hat{t}_k(\l)$ et  $\bm\hat{\mu}_k(\l)$ donnent les coordonnées  $t_k$ et $\mu_k$ de $\cotang$. En particulier, on prouve dans le Théorème \ref{actionsymponlambdaconn} :
\begin{fakethm}
L'action infinitésimale de $\Symp_0$ sur les plus hauts termes $\mu_k$ des coordonnées $\bm\hat{\mu}_k$ de la réduction parabolique $\mc{A}(h)\sslash\mc{P}$
est identique à celle des difféomorphismes supérieurs sur la structure $n$-complexe.
\end{fakethm}
Par conséquent, l'espace $\mc{A}\sslash\mc{P}\sslash\Symp_0$ à paramètre $\l$ tend vers $\cotang$ quand $\l$ tend vers l'infini. Quand $\l$ tend vers 0, on retrouve la structure $n$-complexe conjuguée. 

Dans la section \ref{finalstep}, on essaie de démontrer qu'on peut associer canoniquement une connexion parabolique plate $\mc{A}(\l)$ à un point de $\cotang$. Ce serait un analogue à la correspondance de Hodge non-abélienne dans le contexte des fibrés de Higgs.
Nous obtenons des résultats partiels dans ce sens. Dans un cas particulier, nous retrouvons la correspondance de Hodge non-abelienne pour un champ de Higgs nilpotent.
Notre conjecture principale s'énonce ainsi (voir Conjecture \ref{nahc}) :
\begin{fakeconj}
Étant donné un élément $[(\mu_k, t_k)]\in T^*\bm\hat{\mc{T}}^n$ et une donnée finie (conditions initiales d'équations différentielles), il existe une unique (à jauge unitaire près) connexion plate $\mc{A}(\l)=\l\Phi+A+\l^{-1}\Phi^*$ vérifiant
\begin{enumerate}
\item Localement, $\Phi=\Phi_1 dz+\Phi_2 d\bar{z}$ avec $\Phi_1$ nilpotent principal et $\Phi_2=\mu_2\Phi_1+...+\mu_n\Phi_1^{n-1}$
\item $-\mc{A}(-1/\bar{\l})^*=\mc{A}(\l)$ (condition de réalité)
\item $t_k = \tr \Phi_1^{k-1}A_1$.
\end{enumerate}
De surcroît, si $t_k=0$ pour tout $k$, alors la monodromie de $\mc{A}(\l)$ est dans $\PSL_n(\R)$.
\end{fakeconj}
Admettant cette conjecture, nous en déduisons (voir Théorème \ref{mainthmm}) :
\begin{fakecoro2}
Admettant la conjecture précédente, il existe un isomorphisme canonique entre $\T^n$ et la composante de Hitchin $\mc{T}^n$.
\end{fakecoro2}

\bigskip
\noindent Dans la partie \ref{part3} du manuscrit, nous généralisons la partie \ref{part1} à un groupe de Lie simple $G$ quelconque. Nous définissons une généralisation du schéma de Hilbert ponctuel, associé à une algèbre de Lie simple $\g$. A l'aide de sa zéro-fibre nous construisons une nouvelle structure géométrique généralisant à la fois la structure complexe et la structure $n$-complexe (qui représente le cas $\g=\mf{sl}_n$), appelée structure $\g$-complexe.

Dans la section \ref{g-Hilbertscheme} nous introduisons le \textit{$\g$-schéma de Hilbert}, noté $\Hilb(\g)$, comme un sous-espace de la variété des éléments commutants $\{(A,B)\in \g^2 \mid [A,B]=0\}/G$. La zéro-fibre est constituée des éléments commutants nilpotents. Un premier résultat est (voir Corollaire \ref{hilbg0affine}) :
\begin{fakeprop}
La partie régulière de la zéro-fibre $\Hilb^{reg}_0(\g) = \Hilb^{reg}(\g)\cap \Hilb_0(\g)$ est une variété affine de dimension $\rk \g$.
\end{fakeprop}
Nous étudions en détail le cas d'une algèbre de Lie $\g$ classique.
Signalons que pour l'instant $\Hilb(\g)$ est un espace non-séparé (non-Hausdorff). On conjecture qu'on peut modifier sa définition pour obtenir un espace séparé. 

En utilisant $\Hilb_0(\g)$, nous construisons dans \ref{g-complexstructure} la \textit{structure $\g$-complexe}. On peut la voir comme une 1-forme à valeurs dans $\g$ avec certaines propriétés.
On démontre qu'une structure $\g$-complexe induit une structure complexe sur la surface. En plus, on paramètre une structure $\g$-complexe avec des différentielles de Beltrami supérieures.

Pour une algèbre de Lie $\g$ classique, on introduit des difféomorphismes supérieurs de type $\g$ qui agissent sur la structure $\g$-complexe. La théorie locale est donnée par le Théorème \ref{thm1} :
\begin{fakethm}
Pour $\g$ de type $A_n$, $B_n$ ou $C_n$, la structure $g$-complexe peut être trivialisée localement.

Pour $\g$ de type $D_n$, toutes les structures $\g$-complexes dont la différentielle supérieure $\sigma_n$ n'a pas de zéros sont localement équivalentes sous difféomorphismes supérieurs. Cependant, l'ensemble $\{z \in \C \mid\sigma_n(z)=0\}$ est un invariant.
\end{fakethm}

L'espace des modules $\bm\hat{\mc{T}}_\g$ de la structure $\g$-complexe admet les propriétés suivantes (voir Théorème \ref{thm2} et Proposition \ref{inclteich}) :
\begin{fakethm}
Pour $\g$ de type $A_n, B_n$ ou $C_n$, et une surface $\S$ de genre $g\geq 2$, l'espace $\bm\hat{\mathcal{T}}_{\g}$ est une variété contractile de dimension complexe $(g-1)\dim \g$. Son espace cotangent en un point $I$ est donné par 
$$T^*_{I}\bm\hat{\mathcal{T}}_{\g} = \bigoplus_{m=1}^{r} H^0(K^{m_i+1})$$ où $(m_1,...,m_r)$ désignent les exposants de $\g$ et $r=\rk \g$ le rang de $\g$.
De plus, l'application principale $\mf{sl}_2 \rightarrow \g$ induit une inclusion de l'espace de Teichmüller dans $\bm\hat{\mc{T}}_\g$.

\noindent Pour le type $D_n$, l'espace $\bm\hat{\mathcal{T}}_{\g}$ est un espace topologique contractile. Le lieu où les zéros de la différentielle de Beltrami supérieure $\sigma_n$ sont discrets dans $\S$ est une variété lisse avec les propriétés ci-dessus (dimension, espace cotangent et copie de l'espace de Teichmüller).
\end{fakethm}
Comme pour la structure $n$-complexe, on peut associer à un point de $T^*\bm\hat{\mc{T}}_\g$ une courbe spectrale $\bm\tilde{\S}\subset T^{*\C}\S$ qui est Lagrangienne.

\bigskip
\noindent Dans la dernière partie \ref{part4}, nous exposons les questions et conjectures qui restent pour l'instant ouvertes et nous discutons des liens possibles entre les structures complexes supérieures et d'autres sujets dans la théorie de Teich\-müller supérieure, comme par exemple les réseaux spectraux de Gaiotto-Moore-Neitzke, les opères, les variétés amassées et la symétrie miroir.

\bigskip
\noindent On inclut un appendice \ref{appendix:B} dans lequel on traite les éléments réguliers d'une algèbre de Lie semi-simple. Ce matériel est nécessaire uniquement pour la partie \ref{part3}.

Les sous-sections avec un astérisque ne sont pas essentielles pour la compréhension globale et peuvent être sautées dans une première lecture.


\newpage
\section{Introduction and summary (in English)}
\vspace*{0.5cm}
\begin{flushright}
\textit{Dass ich erkenne, was die Welt}

\textit{Im Innersten zusammenhält.}

\textit{[...]}

\textit{Wie alles sich zum Ganzen webt,}

\textit{Eins in dem andern wirkt und lebt!}\footnote{\textit{That I may understand whatever / Binds the world's innermost core together. / [...] /  How each to the Whole its selfhood gives, / One in another works and lives!} Translation by A.S. Kline}

Johann Wolfgang von Goethe, ``Faust''
\end{flushright}
\vspace*{0.3cm}

\subsection{Introduction}\label{introenglish}

The main goal of this PhD thesis is to construct a new geometric approach to higher Teichmüller theory. We give here a short introduction to this theory, selecting those issues which we need for our work.

\bigskip
\noindent Teichmüller and higher Teichmüller theory study global aspects of geometric structures on surfaces and their interaction with representations of the fundamental group of the surface into a Lie group.

Consider a smooth compact oriented surface $\S$ of genus at least 2. Equip $\S$ with a complex structure (compatible with the orientation of $\S$), so that $\S$ becomes a Riemann surface. The space of all complex structures is infinite-dimensional since we can locally change the complex structure by a holomorphic map. The group $\Diff_0(\S)$ of diffeomorphisms of $\S$ isotopic to the identity acts on complex structures. The quotient space, i.e. the space where we identify all complex structures which differ by an element of $\Diff_0(\S)$, is called \textbf{Teichmüller space}, denoted by $\mc{T}(\S)$.

Surprisingly, Teichmüller space is also the moduli space of other geometric structures on the surface. To a complex structure, you can associate in a unique way a hyperbolic structure, i.e. a Riemannian metric with constant curvature -1. It is the metric which in local complex coordinates $(z,\bar{z})$ can be written $e^{\varphi} dz d\bar{z}$ where $\varphi$ satisfies \textit{Liouville's equation} $\frac{1}{2}\del\delbar \varphi = e^{\varphi}$. Conversely you can associate a unique complex structure to a given hyperbolic metric. Therefore, Teichmüller space is also the space of all hyperbolic structures modulo $\Diff_0(\S)$.

There is also an important link to representations of the fundamental group $\pi_1(\S)$. A complex structure on $\S$ can be lifted to its universal cover $\bm\tilde{\Sigma}$ which is topologically the hyperbolic plane. By the uniformization theorem of Poincaré, $\bm\tilde{\Sigma}$ with the induced complex structure from $\S$ is biholomorphic to the hyperbolic plane $\mathbf{H}$. To get the Riemann surface $\S$ back, it is sufficient to quotient $\mathbf{H}$ by the action of the fundamental group $\pi_1(\S)$ which acts by isometries. The isometry group of $\mathbf{H}$ being $\PSL_2(\R)$, the action is given by a map $\pi_1(\S)\rightarrow \PSL_2(\R)$. This homomorphism has to be faithful and discrete in order to get a free and properly discontinuous action on $\mathbf{H}$. Such a morphism is called \textbf{fuchsian}. Two representations give the same complex structure on $\S$ iff they differ by a Möbius transformation. 
Therefore, Teichmüller space can be identified with the connected component of the \textbf{character variety} $$\Rep(\pi_1(\Sigma), \PSL_2(\R))=\Hom(\pi_1(\Sigma), \PSL_2(\R))/\PSL_2(\R)$$ whose morphisms are discrete and faithful.

Thanks to the different descriptions, Teichmüller space enjoys lots of interesting properties: it is a \textit{contractible differentiable manifold with a complex and a symplectic structure, which both together give a Kähler structure. Its cotangent space at a point $T^*_\mu\mc{T}(\S)$ is given by holomorphic quadratic differentials. The mapping class group $\MCG(\S)$ acts properly discontinuously on $\mc{T}(\S)$ preserving the Kähler structure}.
We warmly recommend Hubbard's book \cite{Hub} on Teichmüller theory. In particular he proves all the properties of Teichmüller space we just listed. For a historical account via quasiconformal mappings, we refer to Ahlfors' book \cite{Ahlfors}.

\bigskip
\noindent In his seminal paper \cite{Hit.1}, Nigel Hitchin describes a connected component of other character varieties with properties similar to Teichmüller space. These components are now called \textbf{Hitchin components} and their study \textbf{higher Teichmüller theory}. We refer to \cite{Wienhard} for an overview on this vast theory.

The starting point is to replace the group $\PSL_2(\R)$ by another Lie group $G$, i.e. to study the character variety $\Rep(\pi_1(\S), G)$. Hitchin's construction works for a Lie group associated to a split form of a complex simple Lie algebra. A good example is $\PSL_n(\R)$ to which we restrict in the sequel. We denote by $\mc{T}^n$ Hitchin's component for $\PSL_n(\R)$.

From the representation theory point of view, the Hitchin component consists of those morphisms $\pi_1(\S) \rightarrow \PSL_n(\R)$ which can be continuously deformed to a composition of the form $\pi_1(\S) \rightarrow \PSL_2(\R)\rightarrow \PSL_n(\R)$ where the first arrow is a fuchsian map and the second is the \textbf{principal map}. The principal map in our case is the unique irreducible representation of $\PSL_2(\R)$ of dimension $n$.

We summarize the main properties of Hitchin's component in the following theorem:
\begin{fakethm2}[Hitchin, Goldman, Labourie]
The Hitchin component $\mathcal{T}^n$ is a contractible real manifold of dimension $(n^2-1)(2g-2)$. It admits a symplectic structure and a properly discontinuous action of the mapping class group of $\S$.
\end{fakethm2}

The symplectic structure comes from a general construction due to Goldman (see \cite{Goldman.2} and \cite{AtBott}). The properly discontinuous action of the mapping class group was described by Labourie in \cite{Lab.3}.

\bigskip
\noindent A natural question to ask is whether Hitchin components allow a geometric description like Teichmüller space. More precisely:
\begin{oq3}
Is there a geometric structure on the surface $\S$ whose moduli space is Hitchin's component?
\end{oq3}
\noindent This question is the main motivation for this thesis.

The search for a geometric origin of Hitchin's component is not new. Goldman, Guichard--Wienhard, Labourie and others describe Hitchin's component via geometric structures on bundles over the surface. For $\PSL_3(\mathbb{R})$, this geometric structure is the convex projective structure described by Goldman in \cite{Goldman}. For $n=4$, Guichard and Wienhard describe convex foliated structures on the unit tangent bundle in \cite{Guichard}. Labourie introduces the fruitful concept of Anosov representations in \cite{Lab}. 
The drawback of these constructions (for $n>3$) is that the bundle on which the geometric structure is defined is not canonically associated to the surface. 

All these generalizations are rigid geometric structures (meaning that the local automorphism group is finite dimensional). Our generalization is not rigid in this sense but behaves as a generalization of complex structures (with local automorphism group holomorphic functions which are infinite dimensional).

\bigskip
\noindent Hitchin's original approach uses Higgs bundle theory, especially the hyperkähler structure of the moduli space of Higgs bundles giving the non-abelian Hodge correspondence. For more details on Higgs bundles, we recommend \cite{Wentworth}, \cite{Garcia}, and the lecture notes of Andrew Neitzke \cite{Neitzke}.

Given a Riemann surface $\S$, a \textbf{Higgs bundle} is a holomorphic bundle $V$ equipped with a holomorphic $\End(V)$-valued 1-form $\Phi$. In more sophisticated language, we have $\Phi \in H^0(\End(V)\otimes K)$ where $K=T^{*(1,0)}\S$ denotes the canonical bundle. The element $\Phi$ is called the \textbf{Higgs field}. You can think of it as a \textit{cotangent vector to the space of connections}. Given a connection $\nabla=d+A$, we can add the Higgs field $\Phi$ to $\nabla$. The only difference between $A$ and $\Phi$ is their transformation under a gauge $g$: whereas $g.A = gAg^{-1}+gdg^{-1}$ we have $g.\Phi=g\Phi g^{-1}$.

The moduli space of $G$-Higgs bundles $\mc{M}_H(G)$ is the space of all gauge classes of polystable Higgs bundles. One of its main properties is that $\mc{M}_H$ is a hyperkähler mani\-fold, i.e. it admits a 1-parameter family of Kähler structures. In one Kähler structure, we get the moduli space of Higgs bundles with its complex structure. In another complex structure, we get the character variety $\Rep(\pi_1(\S), G^{\C})$ for the complex Lie group $G^{\C}$. There is a way to describe all different Kähler structures in a hyperkähler manifold at once: the \textbf{twistor approach} from \cite{HKLR}.

For the moduli space of Higgs bundles, the twistor approach gives the \textbf{non-abelian Hodge correspondence}, fruit of several papers of Corlette 
\cite{Corlette}, Donaldson \cite{Donaldson}, Hitchin \cite{Hit.2} and Simpson \cite{Simpson}.
Given a stable Higgs bundle $(V, \Phi)$ with $\deg V=0$, there is a unique (up to unitary gauge) unitary connection $A$ compatible with the holomorphic structure such that 
\begin{align*}
\delbar \Phi+[A^{(0,1)},\Phi] = 0 \\
F_A + [\Phi, \Phi^*]  = 0
\end{align*}
where $F_A$ denotes the curvature of $A$ and the operation $*$ is the adjoint with respect to the Hermitian metric $V$ induced by $A$.
These equations, called \textbf{Hitchin equations}, are equivalent to the flatness of the connection $\nabla  = d+\Phi+A+\Phi^*$.

Conversely, if a bundle $V$ has a flat connection $\nabla$ which is completely reducible, there is a metric on $V$ such that $\nabla$ decomposes as 
$d+\Phi+A+\Phi^*$ satisfying Hitchin equations.

In his paper \cite{Hit3}, Hitchin introduces a map $$\mc{M}_{H}(\SL_n(\C))\rightarrow \bigoplus_{i=2}^n H^0(K^i)$$ called \textbf{Hitchin fibration}. It associates to $(V, \Phi)$ the coefficients of the characteristic polynomial of $\Phi$. Using a principal slice, Hitchin constructs in \cite{Hit.1} a section to the fibration whose image, through the non-abelian Hodge correspondence, has monodromy in $\PSL_n(\R)$.
He proves that this image is a connected component of $\Rep(\pi_1(\S), \PSL_n(\R))$. 

We will see lots of similarities but also differences between Hitchin's original approach and our approach.

\subsection{Summary}

As already pointed out, the main motivation of the thesis is to give a new geometric approach to Hitchin components. We construct a new geometric structure on a surface, called higher complex structure, which generalizes the complex structure. We define a moduli space of higher complex structures which shares numerous properties with Hitchin's component. 

To get a direct link to character varieties, we define a setting analogous to Higgs bundles, but on a smooth surface (without underlying complex structure). We replace the holomorphic Higgs field $\Phi$ by a matrix-valued 1-form $\Phi=\Phi_1+\Phi_2$ (decomposed into $(1,0)$- and $(0,1)$-part in some reference complex structure) where $\Phi\wedge \Phi=0$. So locally, $\Phi_1$ and $\Phi_2$ can be identified with two nilpotent commuting matrices. We analyze how to deform this 1-form to get a flat connection. We wish to establish a correspondence between flat connections and 1-forms of type $\Phi_1+\Phi_2$.
We present partial results in that direction. Assuming this correspondence, we get a canonical diffeomorphism between the moduli space of higher complex structures and Hitchin's component.

The manuscript is divided into four parts and one appendix.

\bigskip
\noindent Part \ref{part1} treats the construction and properties of the new geometric structure on surfaces, the higher complex structure.

A complex structure on a surface $\S$ can be encoded by the \textit{Beltrami differential} $\mu_2$, which determines a direction in each complexified cotangent space $T^{*\C}_z\S$. The idea of the generalization is to replace the linear direction by an $n$-jet of a curve.

In order to formalize this idea, we use the \textit{punctual Hilbert scheme of the plane} $\Hilb^n(\C^2)$ which we introduce and explain in detail in section \ref{Hilbertscheme}. The space $\Hilb^n(\C^2)$ is the space of all ideals of $\C[x,y]$ of codimension $n$. Roughly speaking it parameterizes $n$ points in the plane $\C^2$ and retains some extra information whenever several points coincide.
The \textit{zero-fiber} $\Hilb^n_0(\C^2)$ is the space of all ideals supported at the origin, i.e. when all $n$ points coincide with the origin. A generic element of the zero-fiber can be considered as an $(n-1)$-jet of a curve, the curve along which the $n$ points collapse into the origin.

The punctual Hilbert scheme admits another description, in purely linear algebraic terms. It is the variety of conjugacy classes of pairs of commuting matrices. The zero-fiber corresponds to nilpotent matrices. See subsection \ref{hilbmatrixviewpoint}.

In section \ref{Highercomplexsection} we define the \textit{higher complex structure} (also called \textit{$n$-complex structure}). We apply the Hilbert scheme pointwise to the complexified cotangent spaces $T_z^{*\C}\S$ to get a Hilbert scheme bundle $\Hilb^n(T^{*\C}\S)$. The higher complex structure is defined as a section of the zero-fiber $\Hilb^n_0(T^{*\C}\S)$. It is a generalization of complex structures which we recover for $n=2$. A higher complex structure is parameterized by \textit{higher Beltrami differentials} $\mu_2,...,\mu_n$.
We can think of the higher complex structure either as a kind of lawn on the surface (polynomial in the cotangent bundle), or as an $(n-1)$-jet of a surface inside $T^{*\C}\S$ along the zero-section, or in the matrix viewpoint as a subset of $\mf{sl}_n$-valued 1-forms locally of the form $\Phi_1dz+\Phi_2d\bar{z}$ where $\Phi_1$ and $\Phi_2$ are commuting nilpotent matrices.

To define a finite-dimensional moduli space, the space of $n$-complex structures has to be considered modulo some equivalence relation. Since the $n$-complex structure is polynomial in the cotangent space, it is natural to consider polynomial transformations of the cotangent bundle. These can be described by special symplectomorphisms of $T^*\S$ generated by polynomial Hamiltonians. These special symplectomorphisms are called \textit{higher diffeomorphisms}, and their group is denoted by $\Symp_0$.

The first result of this work is that the local theory of the $n$-complex structure is trivial (see Theorem \ref{loctrivial}):
\begin{fakethm2}[Local theory]
Any $n$-complex structure can be locally trivialized under a higher diffeomorphism.
In other words, any two higher complex structures are locally equivalent.
\end{fakethm2}

The moduli space $\T^n$ of $n$-complex structures has the following properties (see Theorem \ref{mainresultncomplex} and Proposition \ref{copyteich}):

\begin{fakethm2}[Global theory]
For a surface $\Sigma$ of genus $g\geq 2$ the moduli space $\bm\hat{\mathcal{T}}^n$ is a contractible manifold of complex dimension $(n^2-1)(g-1)$. In addition, its cotangent space at any point $\mu=(\mu_2,...,\mu_n)$ is given by (where $K$ denotes the canonical bundle)
$$T^*_{\mu}\bm\hat{\mathcal{T}}^n = \bigoplus_{m=2}^{n} H^0(K^m).$$
In addition, there is a forgetful map $\bm\hat{\mathcal{T}}^n \rightarrow \bm\hat{\mathcal{T}}^{n-1}$ and a copy of Teichmüller space $\mc{T}^2 \rightarrow \bm\hat{\mc{T}}^n$.
\end{fakethm2}

One important feature of the punctual Hilbert scheme is that it is a \textit{hyperkähler manifold}. In section \ref{moreonhilbertschemes} we explore the symplectic and hyperkähler structure of $\Hilb^n(\C^2)$. With this, in section \ref{Spectralcurve1}, we can describe the total cotangent space $\cotang$ which plays an important role for the link to higher Teichmüller theory. A point in $\cotang$ is characterized by differentials $t_k$ and higher Beltrami differentials $\mu_k$ with a condition on $t_k$ generalizing the holomorphicity condition. For $\mu_k=0$ for all $k$, we recover $\delbar t_k=0$.

A point in $\cotang$ can be seen as a section of the Hilbert scheme bundle $\Hilb^n(T^{*\C}\S)$, or equivalently as a collection of $n$ complex 1-forms ($n$ points in every cotangent space). All together, they give a spectral curve $\bm\tilde{\S}$ in $T^{*\C}\S$ which is described in \ref{sspectral}. The main property is that $\bm\tilde{\S}$ is Lagrangian modulo $t^2$ (see Theorem \ref{spectralcurveprop}).
Thus, the periods of the Liouville 1-form of $T^{*\C}\S$ are well-defined modulo $t^2$. We conjecture that these periods give a coordinate system on $\cotang$ and, by some limit procedure, also on $\T^n$.

In section \ref{dualcomplexstructure} we define a conjugated higher complex structure and a conjugated space to $\cotang$.
For $n=2$ the space $\cotang$ has an interpretation as space of half-translation surfaces. Thus there is a natural $\GL_2(\R)$-action. In section \ref{gl2} we show that $\GL_2(\R)$ acts on $\cotang$ for all $n$.

\bigskip
\noindent Part \ref{part2} ties a link to Hitchin components taking advantage of the hyperkähler structure of the Hilbert scheme. 
We deform $\cotang$ to a space of flat connections.
A point in $\cotang$ is essentially described by two polynomials, the deformation space consists of a pair of commuting differential operators to which we can associate a flat connection. 
To get a link to higher Teichmüller theory, we want to \textit{canonically} associate a flat connection to a point in $\cotang$. We have partial results in this direction. 
In section \ref{bigpicture}, we describe our approach and compare it to Hitchin's approach.

In the matrix viewpoint, a higher complex structure is a $\mf{sl}_n(\C)$-valued 1-form of the form $\Phi=\Phi_1+\Phi_2$. 
To deform this structure, we include it into a family of connections with parameter $\l$ of the form $d+\l\Phi+A+\l^{-1}\Phi^*$. In section \ref{holodiffs-flat-affine} we extract holomorphic differentials $t_k$ from connections of this type.

A famous theorem of Atiyah--Bott describes the space of flat connections as the hamiltonian reduction of the space of all connections $\mc{A}$ by the gauge group $\mc{G}$. Section \ref{parabolicreduction} studies the reduction of $\mc{A}$ by a subgroup $\mc{P}$ of all gauges, those fixing a given direction, which we call \textit{parabolic gauge transformations}.
We prove that the reduction $\mc{A}\sslash\mc{P}$ is a space of pairs of differential operators, parameterized by variables $\bm\hat{t}_k$ and $\bm\hat{\mu}_k$.

These differential operators commute under some condition on $\bm\hat{t}_k$ which is similar to the higher holomorphicity condition. We can obtain this condition on $\bm\hat{t}_k$ as the moment map of a second hamiltonian reduction. In \ref{symponconnections}, we explain how to associate a gauge to a higher diffeomorphism. We have to perform the reduction with respect to these gauges to obtain the following result (see Corollary \ref{flatparaconnections}):
\begin{fakethm2}
The double reduction $\mathcal{A}\sslash\mathcal{P}\sslash\Symp_0$ gives a space of flat connections.
\end{fakethm2}
In section \ref{parabolicwithlambda} we generalize the parabolic reduction to the space of $h$-connections. In this way we include $\mc{A}\sslash\mc{P}\sslash\Symp_0$ into a one-parameter family of spaces of the same type. Let us call $\l$ this parameter ($\l=h^{-1}$). Thus, we get a space of pairs of differential operators depending on $\l$, parameterized by $\bm\hat{t}_k(\l)$ and $\bm\hat{\mu}_k(\l)$. In the limit $\l \rightarrow \infty$ we get the two polynomials describing $\cotang$ back.
More precisely, in a Taylor development in $\l$, the highest terms of $\bm\hat{t}_k(\l)$ and $\bm\hat{\mu}_k(\l)$ are the coordinates $t_k$ and $\mu_k$ of $\cotang$. In particular, we prove in Theorem \ref{actionsymponlambdaconn}:
\begin{fakethm2}
The infinitesimal action of $\Symp_0(\Sigma)$ on the highest terms $\mu_k$ of the coordinates $\bm\hat{\mu}_k(\l)$ of the parabolic reduction $\mc{A}(h)\sslash\mc{P}$ is the same as the infinitesimal action of higher diffeomorphisms on the $n$-complex structure.
\end{fakethm2}
As a consequence, the space $\mc{A}\sslash\mc{P}\sslash\Symp_0$ with parameter $\l$ tends to $\cotang$ when $\l$ tends to infinity. When $\l$ tends to $0$, we recover the conjugated $n$-complex structure. 

In section \ref{finalstep}, we try to prove that one can canonically associate a flat parabolic connection $\mc{A}(\l)$ to a point in $\cotang$. This would be an analog to the non-abelian Hodge correspondence in the Higgs bundle setting.
We obtain partial results in this direction. In particular, we recover in a special case the non-abelian Hodge correspondence for a nilpotent Higgs field.
Our main conjecture can be formulated as follows (see Conjecture \ref{nahc}):
\begin{fakeconj}
Given an element $[(\mu_k, t_k)]\in T^*\bm\hat{\mc{T}}^n$ and some finite extra data (initial conditions to differential equations), there is a unique (up to unitary gauge) flat connection $\mc{A}(\l)=\l\Phi+A+\l^{-1}\Phi^*$ satisfying 
\begin{enumerate}
\item Locally, $\Phi=\Phi_1 dz+\Phi_2 d\bar{z}$ with $\Phi_1$ principal nilpotent and $\Phi_2=\mu_2\Phi_1+...+\mu_n\Phi_1^{n-1}$
\item $-\mc{A}(-1/\bar{\l})^*=\mc{A}(\l)$ (reality condition)
\item $t_k = \tr \Phi_1^{k-1}A_1$.
\end{enumerate}
In addition, if $t_k=0$ for all $k$ then the monodromy of $\mc{A}(\l)$ is in $\PSL_n(\R)$.
\end{fakeconj}

Admitting this conjecture, we deduce (see Theorem \ref{mainthmm}): 
\begin{fakecoro}
Admitting the previous conjecture, there is a canonical diffeomorphism between $\T^n$ and Hitchin's component $\mc{T}^n$.
\end{fakecoro}

\bigskip
\noindent In part \ref{part3} of the manuscript, we generalize part \ref{part1} to other complex simple Lie groups $G$. We define a generalization of the punctual Hilbert scheme, associated to a simple Lie algebra $\g$. With its zero-fiber we construct a geometric structure, called $\g$-complex structure, generalizing both complex and $n$-complex structures (which are special cases with $\g=\mf{sl}_n$).

In section \ref{g-Hilbertscheme} we define the \textit{$\g$-Hilbert scheme}, denoted by $\Hilb(\g)$, as a subspace of the commuting variety $\{(A,B)\in \g^2 \mid [A,B]=0\}/G$. The zero-fiber consists of commuting nilpotent elements. A first result is (see Corollary \ref{hilbg0affine}):
\begin{fakeprop}
The regular zero-fiber $\Hilb^{reg}_0(\g) = \Hilb^{reg}(\g)\cap \Hilb_0(\g)$ is an affine variety of dimension $\rk \g$.
\end{fakeprop}
We study the case of a classical Lie algebra in detail.
Note that for the moment $\Hilb(\g)$ is a topological space which is not Hausdorff. We conjecture that there exists a modified version of $\Hilb(\g)$ which gives a Hausdorff space.

By means of $\Hilb_0(\g)$, we construct in \ref{g-complexstructure} the \textit{$\g$-complex structure}. It can be seen as a special $\g$-valued 1-form on $\S$.
We prove that a $\g$-complex structure induces a complex structure on the surface. Furthermore, we parameterize a $\g$-complex structure by higher Beltrami differentials.

For a classical Lie algebra $\g$, we introduce higher diffeomorphisms of type $\g$ which act on $\g$-complex structures. The local theory is determined in Theorem \ref{thm1}:
\begin{fakethm2}
For $\g$ of type $A_n$, $B_n$ or $C_n$, the $\g$-complex structure can be locally trivialized, i.e. there is a higher diffeomorphism of type $\g$ which sends all higher Beltrami differentials to 0 for all small $z \in \C$.

For $\g$ of type $D_n$, all $\g$-complex structures with non-vanishing $\sigma_n$ are locally equivalent under higher diffeomorphisms. However, the zero locus of $\sigma_n$ is an invariant.
\end{fakethm2}

The moduli space $\bm\hat{\mc{T}}_\g$ of the $\g$-complex structure admits the following properties (see Theorem \ref{thm2} and Proposition \ref{inclteich}):
\begin{fakethm2}
For $\g$ of type $A_n, B_n$ or $C_n$, and a surface $\Sigma$ of genus $g\geq 2$, the moduli space $\bm\hat{\mathcal{T}}_{\g}$ is a contractible manifold of complex dimension $(g-1)\dim \g$. In addition, its cotangent space at any point $I$ is given by 
$$T^*_{I}\bm\hat{\mathcal{T}}_{\g} = \bigoplus_{m=1}^{r} H^0(K^{m_i+1})$$ where $(m_1,...,m_r)$ are the exponents of $\g$ and $r=\rk \g$ denotes the rank of $\g$.
In addition, the principal map $\mf{sl}_2 \rightarrow \g$ induces an inclusion of Teichmüller space into $\bm\hat{\mc{T}}_\g$.

\noindent For type $D_n$, the moduli space $\bm\hat{\mathcal{T}}_{\g}$ is a contractible topological space. The locus where the zero-set of the higher Beltrami differential $\sigma_n$ is a discrete set on $\S$ is a smooth manifold with the same properties as above (dimension, cotangent space and copy of Teichmüller space).
\end{fakethm2}
Like for the $n$-complex structure, we can associate to a point in $T^*\bm\hat{\mc{T}}_\g$ a spectral curve $\bm\tilde{\S}\subset T^{*\C}\S$ which is Lagrangian.

\bigskip
\noindent In the last part \ref{part4}, we expose questions and conjectures which remain open for the moment and we discuss connections between higher complex structures and other topics in higher Teichmüller theory, as for example spectral networks, opers, W-geometry, cluster varieties and SYZ-mirror symmetry.

\bigskip
\noindent We include an appendix \ref{appendix:B} on regular elements in semisimple Lie algebras. This material is only needed for part \ref{part3}.

The subsections marked with an asterisk are not essential for the global understanding and can be skipped in a first reading.

\cleardoublepage
\vspace*{3cm}
\part{Higher complex structures and punctual Hilbert schemes}\label{part1}
\vspace*{2cm}
\begin{flushright}
\textit{The problem of creating something which is new, but which is consistent}

\textit{with everything which has been seen before, is one of extreme difficulty.}

Richard Feynman, Lectures II, page 20-10
\end{flushright}
\vspace*{1cm}

\noindent In this first part, we construct a new geometric structure on a smooth surface generalizing the complex structure. To define this so-called higher complex structure we use the punctual Hilbert scheme of the plane. The moduli space of higher complex structures, a generalization of the classical Teichm\"uller space, is shown to have several properties in common with Hitchin's component.

\medskip
\noindent First, we review the complex structure on surfaces through the Beltrami differential. Then we introduce the main tool for the new geometric structure: the punctual Hilbert scheme of the plane. It admits several equivalent definitions: as space of ideals, as resolution of the configuration space and as space of pairs of commuting matrices.

Using the punctual Hilbert scheme, we define the higher complex structure. To define an equivalence relation on them, we introduce the notion of higher diffeomophisms which are special hamiltonian diffeomorphisms of the cotangent space $T^*\S$. We show that the local theory of higher complex structures is trivial, i.e. that locally any two higher complex structures are equivalent under higher diffeomorphisms.

We then define the moduli space of higher complex structures and show several theorems which indicate its similarity to Hitchin's component. In particular, it is a contractible manifold of the same dimension as Hitchin's component and admits a copy of Teichmüller space inside. Our moduli space has a natural complex structure, a proper discontinuous action of the mapping class group and we can describe its cotangent space.

We deepen the study of the punctual Hilbert scheme, especially its symplectic and hyperkähler structure. With this analysis, we describe the total cotangent bundle to our moduli space of higher complex structures and we construct a spectral curve.

Finally, we discuss a conjugated higher complex structure and an action of $\GL_2(\R)$ on the cotangent bundle to our moduli space, generalizing the action on half-translation surfaces.

\medskip
\noindent Some aspects from the construction of the higher complex structure were already done in my Master thesis. Since they form the starting point for the whole PhD thesis, we include them here. Most of the material in this part was published in \cite{FockThomas}. The Beltrami approach to Teichmüller space and the punctual Hilbert scheme are not new. Still, the parametrization of $\Hilb^n(\C^2)$ by coordinates $(t_k, \mu_k)_{1\leq k \leq n}$ has never been carried out to that extent.
\vspace*{\fill}

\cleardoublepage

\section{Complex structures on surfaces}
\label{Complexstructures}

In this section we review the approach to complex and almost complex structures on a surface via the Beltrami differential. This will prepare us for the generalization to higher complex structures which will follow in section \ref{Highercomplexsection}. 

Recall that a \textbf{complex structure} on a manifold is a complex atlas with holomorphic transition functions. Retaining only the information that every tangent space $T_z\Sigma$ has the structure of a complex vector space, i.e. is equipped with an endomorphism $J(z)$ whose square is $-\id$, we obtain an \textbf{almost complex structure}.
A theorem due to Gauss and Korn-Lichtenstein states that every almost complex structure on a surface comes from a complex structure (i.e. can be integrated to a complex structure). Thus, both notions are equivalent on a surface.

So the complex structure is encoded in the operators $J(z)$. These can be better understood by diagonalization. The characteristic polynomial of $J(z)$ being $X^2+1$, the eigenvalues are $\pm i$. So we need to complexify the tangent space to see the eigendirections: $$T^{\mathbb{C}}\Sigma = T^{1,0}\Sigma \oplus T^{0,1}\Sigma$$ where $T^{1,0}\Sigma$ is the eigenspace associated to eigenvalue $i$.
Furthermore, the eigendirections are conjugated to each other: $T^{1,0}\Sigma = \overline{T^{0,1}\Sigma}$.

Therefore the complex structure is entirely encoded by $T^{0,1}\Sigma$, i.e. a direction in the complexified tangent space. Thus, we can see a complex structure as a section of the projectivized  tangent bundle $\mathbb{P}(T^{\mathbb{C}}\Sigma)$.

Let us describe this viewpoint in coordinates. To do this, we fix a reference complex coordinate $z=x+iy$ on $\Sigma$. This gives a basis $(\partial, \bar{\partial})$ in $T^{\mathbb{C}}\Sigma$ where $\partial = \frac{1}{2}(\partial_x - i\partial_y) \text{ and } \bar{\partial} = \frac{1}{2}(\partial_x + i\partial_y)$. The generator of the linear subspace $T^{0,1}_z\Sigma$ can be normalized to be  $\bar{\partial}-\mu(z,\bar{z}) \partial$, (where $\mu$ is a coordinate on $\mathbb{C}P^1$, and thus can take infinite value). The coefficient $\mu$ is called the \textbf{Beltrami differential}. There is one condition on $\mu$, coming from the fact that the vector $\bar{\partial}-\mu \partial$ and its conjugate $\partial-\bar{\mu} \bar{\partial}$ have to be linearly independent since they are eigenvectors of $J(z)$ corresponding to different eigenvalues. A simple computation shows that the condition is equivalent to $\mu \bar{\mu} \neq 1$. If we restrict ourselves to complex structures compatible with the orientation of the surface (i.e. homotopy equivalent to the reference complex structure given by $\mu=0$). This means that $\mu \bar{\mu} < 1$.

We have only seen the Beltrami differential in a local chart. Changing coordinates $z \mapsto w(z)$ gives $\mu(z,\bar{z}) \mapsto \frac{d\bar{z}/d\bar{w}}{dz/dw}\mu(z,\bar{z})$, so $\mu$ is of type $(-1,1)$, i.e. a section of $K^{-1}\otimes \bar{K}$ where $K=T^{*(1,0)}\Sigma$ the canonical line bundle.

To sum up, we look at a complex structure on a surface as a given linear direction in every complexified tangent space, which is the same as a 1-jet of a curve at the origin. For the generalization, we need to consider rather the \textit{cotangent space} $T^*\Sigma$ (the operators $J(z)$ also act in $T^*_z\Sigma$). Higher complex structures will be given by an $n$-jet of a curve in the complexified cotangent space. To describe this idea precisely in geometric terms, we use the punctual Hilbert scheme of the plane.

\section{Punctual Hilbert scheme of the plane}\label{Hilbertscheme}

We present here the tool necessary for the higher complex structure: the punctual Hilbert scheme of the plane. The \textbf{Hilbert scheme} is the parameter space of all subschemes of an algebraic variety. In general this scheme can be quite complicated but here we are in a very specific case of zero-dimensional subschemes of $\mathbb{C}^2$. Nothing new is presented here, a classical reference is Nakajima's book \cite{Nakajima}. Details can also be found in Haiman's paper \cite{Haiman}.

\subsection{Definition}\label{hilbdef}
Consider $n$ points in the plane $\mathbb{C}^2$ as an algebraic variety, i.e. defined by some ideal $I$ in $\mathbb{C}[x,y]$. The function space $\mathbb{C}[x,y]/I$ is $n$-dimensional, since a function on $n$ points is defined by its $n$ values. So the ideal $I$ is of codimension $n$. 
This gives a simple example of a subscheme of dimension zero.
We define the \textbf{length} of a zero-dimensional subscheme to be the dimension of its function space. So the variety of $n$ distinct points is of length $n$. We will see that we get more interesting examples when two or several points collapse into one single point.
The moduli space of zero-dimensional subschemes of length $n$ is called the punctual Hilbert scheme:

\begin{definition}
The \textbf{punctual Hilbert scheme} $\Hilb^n(\mathbb{C}^2)$ of length $n$ of the plane is the set of ideals of $\mathbb{C}\left[x,y\right]$ of codimension $n$: 
$$\Hilb^n(\mathbb{C}^2)=\{I \text{ ideal of } \mathbb{C}\left[x,y\right] \mid \dim(\mathbb{C}\left[x,y\right]/I)=n \}.$$  
\noindent The subspace of $\Hilb^n(\mathbb{C}^2)$ consisting of all ideals supported at 0, i.e. whose associated algebraic variety is $(0,0)$, is called the \textbf{zero-fiber} of the punctual Hilbert scheme and is denoted by $\Hilb^n_0(\mathbb{C}^2)$.
\end{definition}

\begin{Remark}
In the literature the punctual Hilbert scheme of a space $X$ is often denoted by $X^{[n]}$. In this thesis we will deal with several notions of Hilbert schemes, so we prefer the notation $\Hilb^n(X)$.
\hfill $\triangle$
\end{Remark}

Let us try to get a feeling of the form of a generic ideal in the Hilbert scheme. Given $n$ generic distinct points $P_1, ..., P_n$ in $\mathbb{C}^2$, we consider the ideal $I$ whose algebraic variety is given by these points. There is a unique polynomial $Q$ of degree $n-1$ whose graph passes through all points $P_i$, the Lagrange interpolation polynomial (see figure \ref{Hilbertschemefig}). 

\begin{figure}[h]
\centering
\includegraphics[height=4cm]{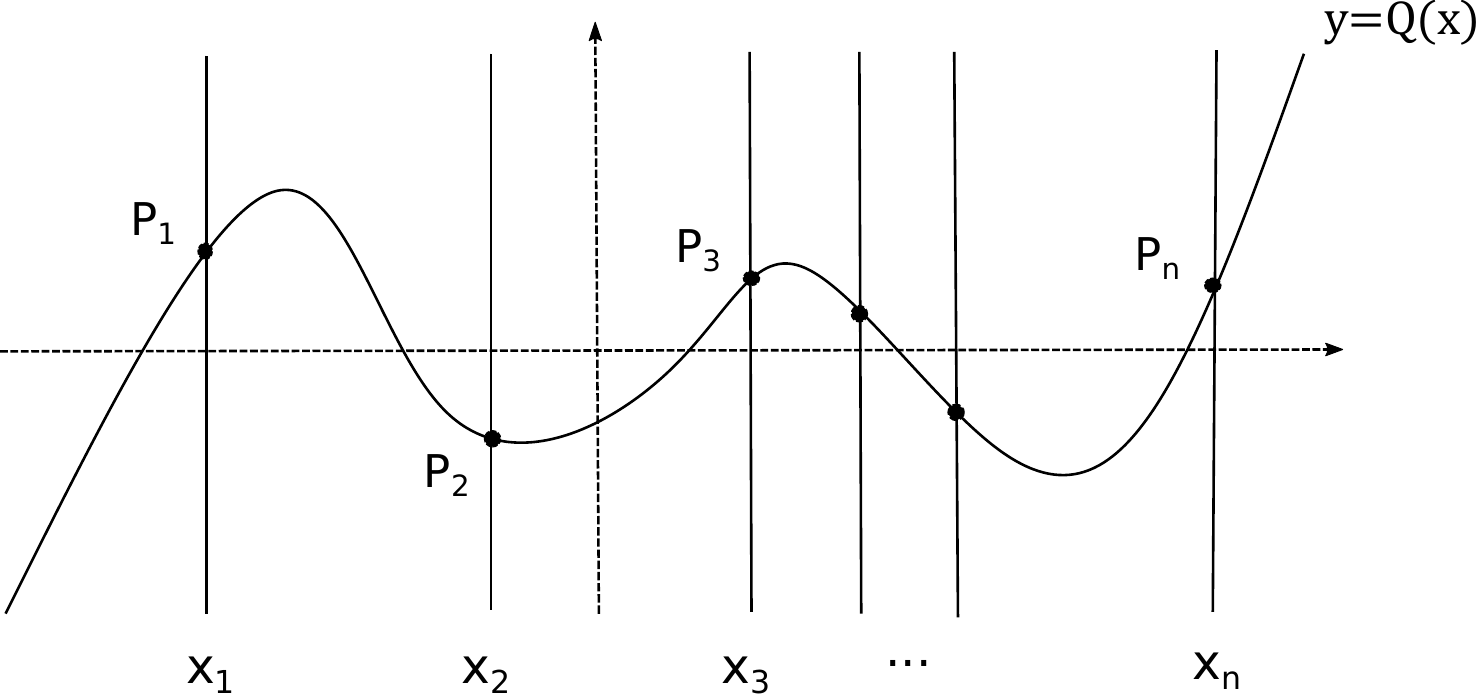}

\caption{Example of two generators}
\label{Hilbertschemefig}
\end{figure}

So we can choose $y=Q(x)$ to be in $I$. If we denote by $x_i$ the $x$-coordinate of the $i$-th point, we see that $\prod_i (x-x_i)$ is also in $I$. These two relations already determine the $n$ points. The ideal $I$ then has two generators and can be put into the form 
$$I=\left\langle x^n-t_1x^{n-1}-\cdots-t_n,-y+\mu_1+\mu_2x+...+\mu_nx^{n-1}\right\rangle.$$

\noindent A point in the zero-fiber of the Hilbert scheme is obtained by collapsing all $n$ points to the origin. In \cite{Iarrob}, it is shown that a generic point is obtained when the Lagrange interpolation polynomial $Q$ admits a limit (for example if all points glide along a given curve to the origin like in figure \ref{movingparticles} on the right). At the limit the constant term of $Q$, which is $\mu_1$, has to be zero. Since all $x_i$ become 0, all $t_i$ do as well. So we get an ideal of $\Hilb^n_0(\mathbb{C}^2)$ of the form
$$I=\left\langle x^n,-y+\mu_2x+...+\mu_nx^{n-1}\right\rangle.$$
There are other points in the zero-fiber if $n>2$. For instance in $\Hilb^3_0(\C^2)$ we find the ideal $\left\langle x^2,xy,y^2 \right\rangle$ which is not of the above form (it has three generators).

Notice that for $n=2$, the zero-fiber gives the projective line: $$\Hilb^2_0(\C^2) \cong \C P^1.$$

Roughly speaking, you can have the following picture in mind (see figure \ref{movingparticles}): moving around in the Hilbert scheme is the same as analyzing the movement of $n$ points in the plane. \textit{But whenever $k$ of them collide to a single point, the punctual Hilbert scheme retains an extra piece of information, which is the $(k-1)$-jet of the curve along which the $k$ points entered into collision.} For $k=2$, one can imagine one point fixed and the second hitting it. The Hilbert scheme retains the direction from which this second point came.

\begin{figure}[h]
\centering
\includegraphics[height=4cm]{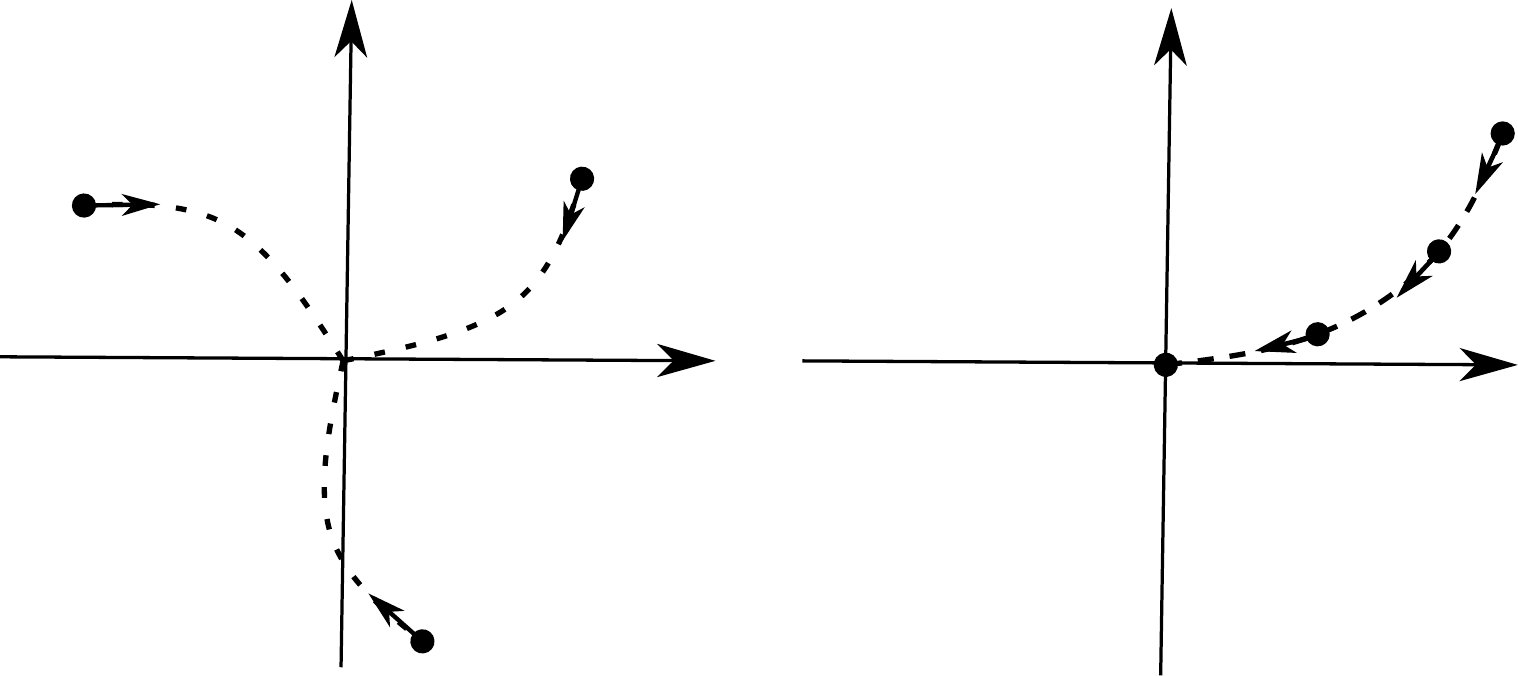}

\caption{Hilbert scheme as moving particles}
\label{movingparticles}
\end{figure}

\subsection{Structure}\label{structurehilbscheme}
We have the following theorem, due to Fogarty and Grothendieck (see \cite{Fogarty}, see also theorem 1.15 in Nakajima's book \cite{Nakajima}), giving the structure of the punctual Hilbert scheme:
\begin{thm}
The punctual Hilbert scheme $\Hilb^n(\mathbb{C}^2)$ is a smooth and irreducible variety of dimension $2n$. A generic point is given by an ideal defining $n$ distinct points in the plane.

The zero-fiber $\Hilb^n_0(\mathbb{C}^2)$ is an irreducible variety of dimension $n-1$, but it is in general not smooth. A generic element of the zero-fiber is of the form 
$$I=\left\langle x^n,-y+\mu_2x+...+\mu_nx^{n-1}\right\rangle.$$
\end{thm}

To get a better understanding about the structure of the punctual Hilbert scheme, we shortly describe how to get a set of charts. All the details are in Haiman's paper \cite{Haiman}.
An atlas of $\Hilb^n(\C^2)$ can be given in terms of Young diagrams. A \textbf{Young diagram} $D$ is a finite subset of $\mathbb{N}\times \mathbb{N}$ such that whenever $(i,j) \in D$ then the rectangle defined by $(i,j)$ and $(0,0)$ is entirely in $D$. We use matrix-like notation such that $(0,0)$ is in the upper left corner.
These diagrams play an important role for visualizing partitions. The set of all Young diagrams with $n$ squares is in bijection with the partitions of $n$. Indeed, given a Young diagram, you can read off the partition by adding the rows.

Now, to any Young diagram $D$, we can associate $$\mathcal{B}_D=\{x^iy^j \mid (i,j) \in D\}$$ a subset of the standard basis of the polynomial ring $\mathbb{C}\left[x,y \right]$. We then set 
$$U_D=\{I \in \Hilb^n(\mathbb{C}^2) \mid \mathcal{B}_D \text{ spans } \mathbb{C}\left[x,y \right]/I \}.$$
These $U_D$ are open affine subvarieties covering $\Hilb^n$.

Given an ideal $I \in \Hilb^n(\C^2)$, there is a simple algorithm to get a chart which covers $I$ (see also figure \ref{running-columns}): Since $1 \notin I$, we can take 1 as the first basis vector of $\mathbb{C}\left[x,y \right]/I$. Then, you run through $\left[0,n\right]\times \left[0,n\right]$ by columns, starting at $(0,0)$. Every time the vector $x^iy^j$ is linearly independent from those visited before, select it as an element for our basis. If it is not, then you get a relation and you can jump to the next column. 
The set of all relations is a generating set for $I$. See subsection \ref{hilbertschemesympl} for a description of coordinates due to Haiman.

\begin{figure}[h]
\centering

\begin{tikzpicture}[scale=0.4]

\begin{scope}[xshift=10cm, yshift=20cm]

\draw (0,0)--(8,0);
\draw (0,-2)--(8,-2);
\draw (0,-4)--(4,-4);
\draw (0,-6)--(4,-6);
\draw (0,-8)--(2,-8);

\draw (0,0)--(0,-8);
\draw (2,0)--(2,-8);
\draw (4,0)--(4,-6);
\draw (6,0)--(6,-2);
\draw (8,0)--(8,-2);

\draw [thick] (0,0)--(8,0)--(8,-2)--(4,-2)--(4,-6)--(2,-6)--(2,-8)--(0,-8)--cycle;

\draw [dotted] (8,0)--(10,0);
\draw [dotted] (8,-2)--(10,-2);
\draw [dotted] (4,-4)--(10,-4);
\draw [dotted] (4,-6)--(10,-6);
\draw [dotted] (2,-8)--(10,-8);
\draw [dotted] (0,-10)--(10,-10);
\draw [dotted] (0,-8)--(0,-10);
\draw [dotted] (2,-8)--(2,-10);
\draw [dotted] (4,-6)--(4,-10);
\draw [dotted] (6,-2)--(6,-10);
\draw [dotted] (8,-2)--(8,-10);
\draw [dotted] (10,0)--(10,-10);

\draw (1,-9) node {$\star$}; 
\draw (3,-7) node {$\star$}; 
\draw (5,-3) node {$\star$}; 
\draw (7,-3) node {$\star$}; 
\draw (9,-1) node {$\star$}; 

\draw (1,-1) node {$1$}; 
\draw (3,-1.12) node {$x$}; 
\draw (5,-1) node {$x^2$}; 
\draw (7,-1) node {$x^3$}; 
\draw (1,-3.12) node {$y$}; 
\draw (3,-3.12) node {$xy$}; 
\draw (1,-5) node {$y^2$}; 
\draw (3,-5) node {$xy^2$}; 
\draw (1,-7) node {$y^3$}; 
\end{scope}

\end{tikzpicture}
\caption{Chart for Hilbert scheme}
\label{running-columns}
\end{figure}

There are several equivalent ways to look on the punctual Hilbert scheme. We next discuss three of them: a blow-up of the configuration space, the set of pairs of commuting matrices and a dual viewpoint.

\subsection{Resolution of singularities*}\label{resofsing}

When we consider the set of $n$ points as an algebraic variety, we do not see their order. So we see them as a point in the configuration space of (not necessarily distinct) $n$ points, which is the quotient of $(\mathbb{C}^2)^n$ by the symmetric group $\mathcal{S}_n$. This quotient space, denoted by $\Sym^n(\mathbb{C}^2)$, is singular since the action of the symmetric group is not free.
There is a map, called the \textbf{Chow map}, from $\Hilb^n(\mathbb{C}^2)$ to $\Sym^n(\mathbb{C}^2)$ which associates to an ideal $I$ its support (the algebraic variety associated with $I$, i.e. a collection of points with multiplicities).
\begin{thm}
The punctual Hilbert scheme is a minimal resolution of the configuration space . In addition, you can get the punctual Hilbert scheme as a successive blow-up of the configuration space.
\end{thm}
A minimal resolution of a singular algebraic space $X$ is the ``smallest'' desingularization, i.e. a smooth space with a surjective map to $X$ satisfying a universal property. Fogarty proved in \cite{Fogarty} that $\Hilb^n(\C^2)$ is smooth and birational to $\Sym^n(\C^2)$. Haiman in \cite{Haiman} proved the blow-up, see also the account of Bertin \cite{Bertin}.

\begin{Remark}
In order to get a feeling of what happens in a general Lie algebra, notice that $n$ points of $\mathbb{C}^2$ is the same as two points in the Cartan $\mathfrak{h}$ of $\mathfrak{gl}_n$, and that the symmetric group is the Weyl group $W$ of $\mathfrak{gl}_n$. So the configuration space equals $\h^2/W$ for $\g=\mf{gl}_n$. This will be helpful in part \ref{part3}.
\hfill $\triangle$
\end{Remark}

\subsection{Matrix viewpoint}\label{hilbmatrixviewpoint}

To an ideal $I$ of codimension $n$, we can associate two matrices: the multiplication operators $M_x$ and $M_y$, acting on the quotient $\C[x,y]/I$ by multiplication by $x$ and $y$ respectively. To be more precise, we can associate a conjugacy class of the pair: $[(M_x,M_y)]$.

The two matrices $M_x$ and $M_y$ commute and they admit a cyclic vector, the image of $1 \in \C[x,y]$ in the quotient (i.e. 1 under the action of both $M_x$ and $M_y$ generate the whole quotient).

\begin{prop}[Matrix viewpoint]\label{bijhilbert}
There is a bijection between the Hilbert scheme and conjugacy classes of certain commuting matrices: $$\Hilb^n(\mathbb{C}^2) \cong \{(A,B) \in \mf{gl}_n^2 \mid [A,B]=0, (A,B) \text{ admits a cyclic vector}\} / GL_n.$$
\end{prop}

The inverse construction goes as follows: to a conjugacy class $[(A,B)]$, associate the ideal $I=\{P \in \C[x,y] \mid P(A,B)=0\}$, which is well-defined and of codimension $n$ (using the fact that $(A,B)$ admits a cyclic vector). For more details see \cite{Nakajima}.

Notice that the \textit{zero-fiber of the Hilbert scheme corresponds to nilpotent commuting matrices}.

Let us see how our coordinates for generic points $(t_i, \mu_i)_{1\leq i\leq n}$ are linked to the matrix point of view: generically the set $\mathcal{B}=(1,x,x^2,...,x^{n-1})$ forms a basis of the quotient $\C[x,y]/I$. In that generic case, the matrix of $M_x$ is a companion matrix:
$$M_x =  \begin{pmatrix} 
&  & & & t_n\\
1 &   & & & t_{n-1}\\
& 1 &  &  & t_{n-2} \\
 & & \ddots & & \vdots\\
 &  &  & 1 & t_1 
\end{pmatrix}.$$
It is uniquely determined by the equation $x\times x^{n-1} = x^n = \sum_{i=1}^n t_ix^{n-1}$.

The matrix of $M_y$ is now uniquely determined by its first column, i.e. by $y\times 1=\sum_{i=1}^{n-1} \mu_i x^{i-1}$. Indeed, we get the image of the basis $\mc{B}$ using the commutativity: the second column, the image of $x$, is given by $$xy = x\sum_{i=1}^{n-1} \mu_i x^{i-1} = \sum_{i=1}^{n-1} \mu_i x^{i} + \mu_n \sum_{i=1}^n t_ix^{n-1}.$$
The $k$\textsuperscript{th} column can be computed by $yx^{k-1}=x^{k-1}\sum_{i=1}^{n-1} \mu_i x^{i-1}$ where you have to express $x^m$ for $m\geq n$ in the basis $\mc{B}$. The first three columns are:
\begin{equation}\label{matrixtwo}
M_y =  \begin{pmatrix} 
\mu_1& \mu_n t_n & \mu_{n-1}t_n& \cdots \\
\mu_2 & \mu_1+\mu_nt_{n-1}  & \mu_n t_n+\mu_{n-1}t_{n-1} &\cdots \\
\mu_3& \mu_2+\mu_nt_{n-2} & \mu_1+\mu_n t_{n-1}+\mu_{n-1}t_{n-2}& \cdots \\
\vdots & \vdots & \vdots &\cdots \\
\mu_n & \mu_{n-1}+\mu_nt_1 & \mu_{n-2}+\mu_n t_2+\mu_{n-1}t_{1}& \cdots 
\end{pmatrix}.
\end{equation}
Another way to compute $M_y$ is to use the Lagrange interpolation polynomial $y=Q(x)$: $$M_y=M_{Q(x)}=Q(M_x)=\mu_1\id+\mu_2M_x+...+\mu_nM_x^{n-1}.$$

\subsection{Dual viewpoint*}\label{ddual}

In this subsection, we describe a dual viewpoint of the punctual Hilbert scheme using a pairing on $\C[x,y]$.

\medskip \noindent
We start with the definition of the pairing:
$$\left( P, Q \right) := P\left(\frac{\partial}{\partial x},\frac{\partial}{\partial y}\right).Q \Big|_{x=y=0}$$ where the little point means ``applied to''.

Computing the pairing in the standard basis $\{x^ny^m \mid n,m \in \mathbb{N}\}$, we get $\left( x^ny^m, x^{n'}y^{m'}\right) = n!m!\delta_{n,n'}\delta_{m,m'}.$
Thus, we see that the pairing is nothing else than the standard inner product of $\mathbb{R}\left[x,y \right]$ with weights $n!m!$ for $x^ny^m$ extended by $\mathbb{C}$-bilinearity. This shows in particular that $\left( ., . \right)$ is symmetric and non-degenerate.

Once we have a pairing, we can define the orthogonal complement $S^{\perp}$ of any subset $S$ of $\mathbb{C}\left[x,y \right]$. In the case where $S$ is an ideal, its orthogonal has special properties:

\begin{prop}
Let $I$ be an ideal of $\mathbb{C}\left[x,y \right]$. Then $I^{\perp}$ is a vector space stable under derivation and translation.
\end{prop}

\begin{proof}
For any subset $S$, it is easy to check that $S^{\perp}$ is a vector space, using the $\mathbb{C}$-bilinearity of the pairing. For the invariance, notice the following fundamental identity: 
\begin{equation}\label{pairing}
\left( PQ, R \right) = \left( P, \left(Q, R\right) \right).
\end{equation}
Thus, if $P$ is an element of $I$, $Q$ any polynomial and $R$ in $I^\perp$, we get that $\left(Q,R\right) = Q(\frac{\partial}{\partial x},\frac{\partial}{\partial y}).R $ also belongs to $I^\perp$. Therefore $I^{\perp}$ is stable under derivation. 
Finally, since $$P(x+a,y+b) = \exp\left(a\frac{\partial}{\partial x}+b\frac{\partial}{\partial y}\right).P(x,y) $$ we see that $I^\perp$ is also invariant under all translations.
\end{proof}

\begin{Remark}
The invariance of $I^\perp$ under translation shows that $I^\perp$ is a \textbf{subcoalgebra} of $\mathbb{C}\left[x,y \right]$ in the following sense: If $P \in I^{\perp}$, we have $\Delta P \in I^{\perp}\otimes I^{\perp}$ where $\Delta P (x_1,y_1,x_2,y_2)=P(x_1+x_2,y_1+y_2)$ is the dual operation to addition.
\hfill $\triangle$
\end{Remark}

\noindent Now, we can explicitly describe the orthogonal of the zero-fiber of the punctual Hilbert scheme, defined by taking the orthogonal to every ideal $I \in \Hilb^n_0$:
\begin{prop}
The orthogonal of $\Hilb^n_0(\mathbb{C}^2)$ is the space of all vector subspaces of $\mathbb{C}\left[x,y \right]$ of dimension $n$ which are invariant under translations. The same holds true when you replace ``translation'' by ``derivation''.
\end{prop}

\begin{proof}
The orthogonal complement sends vector spaces of codimension $n$ to vector spaces of dimension at most $n$. In fact, if we work in the ring of formal power series $\mathbb{C}\left[\left[x,y\right]\right]$ then the orthogonal is of dimension exactly $n$. But for the zero-fiber $\Hilb^n_0$, we cut at level $n$, that is $\left\langle x,y\right\rangle^n=0$. Thus, for $I \in \Hilb^n_0(\mathbb{C}^2)$, we see that $I^\perp$ is of dimension $n$ and by the previous proposition is invariant under translations and derivations. Conversely, if $J$ is an $n$-dimensional vector space invariant under all translations, it is in particular invariant under all derivations (=infinitesimal translations). Then formula \eqref{pairing} shows that $J^\perp$ is an ideal. Finally, since $J$ is finite-dimensional, there is an integer $m$ such that $\left\langle x,y\right\rangle^m \subset J^\perp$ showing that $J^\perp$ is supported on 0.
\end{proof}

Since $\left( ., . \right)$ is an inner product, we can identify the Hilbert scheme with the space of translation-invariant finite-dimensional subspaces.


The most important example is the dual of $I=\langle p^n, -\bar{p}+\mu_2p+...+\mu_np^{n-1}\rangle$, ideal of $\C[p,\bar{p}]$. The dual $I^{\perp}$ by the definition of the pairing is the space of all polynomials $P$ solving the system of differential equations
$$ \left \{\begin{array}{cl}
0 &= \; \del^n P \\
0 &= \; (-\delbar+\mu_2\del+...+\mu_n\del^{n-1})P.
\end{array}\right. $$

\section{Higher complex structures}\label{Highercomplexsection}

In this section, we define the higher complex structure using the punctual Hilbert scheme and explore its main properties. In order to define a moduli space of higher complex structures, we need to enlarge the group of diffeomorphisms of $\Sigma$ to the space of hamiltonian diffeomorphisms of the cotangent space $T^*\Sigma$ preserving the zero section. We then explore the local and global theory of that new structure.

\subsection{Definition and basic properties}\label{basichighercomplex}

In section \ref{Complexstructures}, we saw that a complex structure on a surface $\Sigma$ is uniquely given by a section $\sigma$ of $\mathbb{P}(T^{*\mathbb{C}}\Sigma)$, the (pointwise) projectivized complexified cotangent space, such that at any point $z \in \Sigma$, $\sigma(z)$ and $\bar{\sigma}(z)$ are linearly independent.
In the previous section, we saw that the projectivization is a special case of the zero-fiber of the punctual Hilbert scheme for $n=2$: $\mathbb{P}(T^{*\mathbb{C}}\Sigma) = \Hilb^2_0(T^{*\mathbb{C}}\Sigma)$. It is now easy to guess the generalization. 

\begin{definition}\label{highercomplexdef}
A \textbf{higher complex structure} of order $n$ on a surface $\Sigma$, in short $\mathbf{n}$\textbf{-complex structure}, is a section $I$ of $\Hilb^n_0(T^{*\mathbb{C}}\Sigma)$ such that at each point $z\in \S$ we have that the sum $I(z)+\overline{I}(z)$ is the maximal ideal supported at zero of $T_z^{*\mathbb{C}}\Sigma$.
\end{definition}

\begin{Remark} The space $\Hilb^n(T^{*\mathbb{C}}\Sigma)$ is the pointwise application of $\Hilb^n$ to $T^{*\mathbb{C}}_z\Sigma$ at every point $z$ of $\Sigma$. So it is a bundle of Hilbert schemes. We call it \textbf{Hilbert scheme bundle}. 
\end{Remark}

\noindent For $n=2$, the condition that $I + \overline{I}$ is maximal simply reads $\mu_2\bar{\mu}_2 \neq 1$ which is exactly the condition on the Beltrami differential (see section \ref{Complexstructures}). So we recover the complex structure.

In order to write our structure in coordinates, we fix a reference complex structure on $\S$ given by local coordinates $(z, \bar{z})$. We stress that the higher complex structure is \emph{independent} of this choice, only our coordinates depend on it. The linear coordinates on $T^{*\C}\S$ induced by $(z, \bar{z})$ are denoted by $(p, \bar{p})$. We can identify $p=\frac{\del}{\del z}$ and $\bar{p}=\frac{\del}{\del \bar{z}}$.

In the previous section, we saw that not all points in the zero-fiber can be written in the form of a Lagrange interpolation polynomial passing through the origin. Another important consequence of the extra condition is that it rules out non-generic ideals:

\begin{prop}\label{genericideal}
For an $n$-complex structure $I$, we can write at a point $z$ either $I(z, \bar{z})$ or its conjugate $\overline{I}(z, \bar{z})$ as $$\left\langle p^n, -\bar{p}+\mu_2(z,\bar{z})p+...+\mu_n(z,\bar{z})p^{n-1}  \right\rangle  \text{ with } \mu_2\bar{\mu}_2<1.$$ 
\end{prop}
We call the coefficients $\mu_k$ \textbf{higher Beltrami differentials}.
In the case where $I$ is of the form given in the proposition, we call the $n$-complex structure \textbf{compatible}. If $I$ is not compatible, then $\overline{I}$ is. This is analogous to the case of the Beltrami differential whose norm is either smaller than 1 or bigger than 1. 

\begin{proof}
The proposition concerns a cotangent fiber of one point $z$. So we can really work on $\mathbb{C}^2$ with coordinates $(p,\bar{p})$. 
Let $I_1$ be the set of all degree 1 polynomials which appear as elements of $I$. It is clear that $I_1$ is a vector subspace of $\mathbb{C}^2$ since $I$ is a vector space. We show that $I_1$ is of dimension 1. 

\noindent If $I_1=\{0\}$, then so is $\overline{I}_1=\{0\}$. But by $I\oplus \overline{I} = \left\langle p,\bar{p} \right\rangle$, we get $I_1\oplus\overline{I}_1 = \mathbb{C}^2$ which contradicts the hypothesis on $I$. 
If $I_1=\mathbb{C}^2$ then $I=\left\langle p,\bar{p}\right\rangle$ which contradicts the fact that it is of codimension $n\geq 2$. 

\noindent Therefore $I_1=\Vect(ap + b\bar{p})$ is of dimension 1. So $\overline{I}_1=\Vect(\bar{a} \bar{p} + \bar{b} p)$ and the condition $I\oplus \overline{I} = \left\langle p,\bar{p} \right\rangle$ is equivalent to 
$a\bar{a}\neq b\bar{b}$. Assume $a\bar{a}< b\bar{b}$ (the other case being similar and leads to $\overline{I}$ instead of $I$), then $I_1=\Vect(-\bar{p}+\mu_2p)$ with $\left|\mu_2\right|=\left|a/b\right|<1$.

\noindent Finally, since $-\bar{p}+\mu_2p \in I_1$, there is a relation of the form $\bar{p}=\mu_2p+\text{higher terms}$ in $I$. Iterating this equality by replacing it in any $\bar{p}$ appearing in the higher terms, we get an expression of $\bar{p}$ in terms of monomials in $p$. 
Since $p^n=0$ in $I$, we get $$\bar{p}=\mu_2p+\mu_3p^2+...+\mu_np^{n-1} \mod I.$$
To give an example, for $n=4$ and $\bar{p}=ap+bp\bar{p}$ we get $$\bar{p}=ap+bp(ap+b(ap+b\bar{p}))=ap+abp^2+ab^2p^3.$$ 
\end{proof}

Let us compute the global nature of the higher Beltrami differentials. We will see that $\mu_2$ is just the usual Beltrami differential, so of type $(-1,1)$.
Under a holomorphic coordinate transform $z \to z(w)$, we have $p = \frac{\partial}{\partial z} \mapsto \frac{dw}{dz}\frac{\partial}{\partial w}$ and similarly for $\bar{p}$. Hence, the first generator $p^n$ of $I$ just gets a linear factor of $(\frac{dw}{dz})^n$, so we can drop it in the ideal. The transformation of the second generator gives (where $\propto$ stands for ``proportional to''):
\begin{align*}
& -\bar{p}+\mu_2(z,\bar{z})p+...+\mu_n(z,\bar{z})p^{n-1}   \\
& \mapsto  -\frac{d\bar{w}}{d\bar{z}}\frac{\partial}{\partial \bar{w}}+\frac{dw}{dz}\mu_2(z,\bar{z})\frac{\partial}{\partial w}+...+\left(\frac{dw}{dz}\right)^{n-1}\mu_n(z,\bar{z})\left(\frac{\partial}{\partial w}\right)^{n-1}  \\
& \propto -\frac{\partial}{\partial \bar{w}}+\frac{d\bar{z}/d\bar{w}}{dz/dw}\mu_2(z,\bar{z})\frac{\partial}{\partial w}+...+\frac{d\bar{z}/d\bar{w}}{(dz/dw)^{n-1}}\mu_n(z,\bar{z})\left(\frac{\partial}{\partial w}\right)^{n-1} .
\end{align*}

\noindent Thus, we see that for $m=2,...,n$ we get $$\mu_{m}(w,\bar{w}) = \frac{d\bar{z}/d\bar{w}}{(dz/dw)^{m-1}}\mu_{m}(z,\bar{z}).$$
So $\mu_{m}$ is of type $(1-m,1)$, i.e. a section of $K^{1-m}\otimes \bar{K}$.

The various viewpoints of the punctual Hilbert scheme allow several interpretations of the higher complex structure:
\begin{itemize}
\item[$\bullet$] The previous Proposition \ref{genericideal} allows to think of a higher complex structure as a polynomial curve in the cotangent fiber attached to each point of the surface. Thus we get a ``hairy'' surface as in figure \ref{hairysurface} (with polynomial curved hair).

\item[$\bullet$] A section of $\Hilb^n(T^{*\C}\S)$ is generically an $n$-tuple of 1-forms, or equivalently an $n$-fold cover $\bm\tilde{\S}\subset T^{*\C}\S$. Going to the zero-fiber we can consider the $n$-complex structure as the collapse of this $n$-fold cover to the zero-section, or as the $(n-1)$-jet of a complex surface along the zero-section inside $T^{*\C}\S$.

\item[$\bullet$] In the matrix viewpoint of the punctual Hilbert scheme, you can see a higher complex structure as a conjugation class of a matrix-valued 1-form which can be locally written as $\Phi_1(z)dz+\Phi_2(z)d\bar{z}$ where $[(\Phi_1(z), \Phi_2(z))]$ is a point in $\Hilb^n_0(T_z^{*\C}\S)$ for all $z\in \S$, i.e. a pair of commuting nilpotent matrices with $\Phi_1(z)$ principal nilpotent. Globally, we have a matrix-valued 1-form $\Phi$ satisfying $\Phi\wedge \Phi=0$ and its $(1,0)$-component is principal nilpotent.

\item[$\bullet$] In the dual viewpoint (see \ref{ddual}), we can see the $n$-complex structure as the space of local functions $f$ on $\S$ satisfying $\del^nf=0$ and $(-\delbar+\mu_2\del+...+\mu_n\del^{n-1})f=0$. A complex structure on a surface is characterized by the notion of local holomorphic functions $f$. In a chart they are characterized by $(\delbar-\mu_2\del)f=0$ where $\mu_2$ is the usual Beltrami differential. The $n$-complex structure generalizes this functional approach. For $n=\infty$ and $\mu_k=0$ for all $k>n$, we get the space of functions satisfying only the second equation $(-\delbar+\mu_2\del+...+\mu_n\del^{n-1})f=0.$
\end{itemize}

\begin{figure}[h]
\centering
\includegraphics[height=4cm]{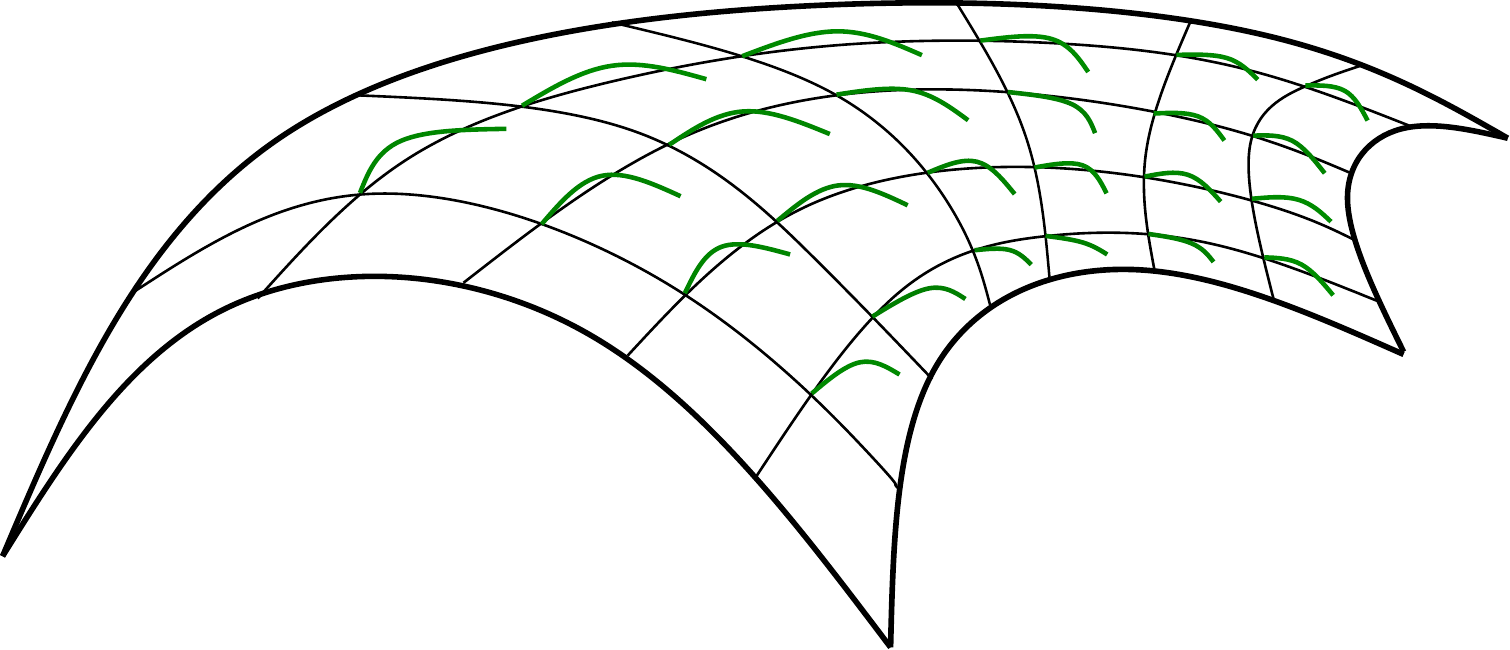}

\caption{Illustration of a higher complex structure}
\label{hairysurface}
\end{figure}

\subsection{Higher diffeomorphisms}\label{higherdiffs}

We wish to define a moduli space of higher complex structures which is finite-dimensional. Higher complex structures ``live'' in a neighborhood of the zero-section of $T^{*\mathbb{C}}\Sigma$. A diffeomorphism of $\Sigma$ extends linearly to $T^*\Sigma$ and its complexification. Thus the extension acts linearly on the polynomial curve (given by the higher complex structure). So the quotient of all $n$-complex structures by diffeomorphisms of $\Sigma$ is infinite-dimensional. Therefore, we have to enlarge the equivalence relation and quotient by a larger group. What we need is polynomial transformations in the cotangent bundle. These can be obtained by symplectomorphisms of $T^*\Sigma$ generated by a Hamiltonian (a function on $T^*\Sigma$):

\begin{definition}
A \textbf{higher diffeomorphism} of a surface $\Sigma$ is a hamiltonian diffeomorphism of $T^*\Sigma$ preserving the zero-section $\Sigma \subset T^*\Sigma$ setwise. The group of higher diffeomorphisms is denoted by $\Symp_0(T^*\Sigma)$.
A higher diffeomorphism is \textbf{of order} $\mathbf{n}$ if its Hamiltonian $H$ is a polynomial in $p$ and $\bar{p}$ of degree $n$.
\end{definition}

A higher diffeomorphism of order 1 is a usual diffeomorphism of $\Sigma$ linearly extended to the cotangent space.
Preserving the zero-section means that the Hamiltonian $H$ of a higher vector field can be chosen to vanish on the zero-section. In coordinates, this means that one can write the Hamiltonian as $H(z,\bar{z},p,\bar{p}) = \sum_{k,l} v_{k,l}(z,\bar{z})p^k\bar{p}^l$ with $v_{0,0}= 0$. Furthermore, since the Hamiltonian is a real function written in complex coordinates, we have the condition $v_{l,k} = \overline{v_{k,l}}$.

\medskip
\noindent Let us see how a higher diffeomorphism acts on the $n$-complex structure. Roughly speaking, a diffeomorphism of the cotangent bundle acts on the space of sections, and so also on the space of $n$-tuples of sections (corresponds to $n$ points in each fiber), so also on the zero-fiber of the Hilbert scheme (which can be seen as the limit when all $n$ points collapse to the origin). In this way, a higher diffeomorphism acts on an $n$-complex structure.

To be more precise, we need first to understand the variation of an ideal in the space of all ideals (the proof is direct and left to the reader):

\begin{prop}\label{idealvariation}
The space of infinitesimal variations of an ideal $I$ in a ring $A$ is the set of all $A$-module homomorphisms from $I$ to $A/I$.
\end{prop}


To compute the variation of an ideal, all we need is to compute the variation of its generators modulo $I$.
These generators are polynomial functions. A general fact of symplectic geometry asserts that the variation of a function $f$ under a flow generated by a Hamiltonian $H$ is given by the Poisson bracket $\{H,f\}$.
Therefore to compute the action of a higher diffeomorphism generated by a Hamiltonian $H$, we only have to compute the Poisson bracket of $H$ with the generators of $I$ and to take the result modulo $I$ (by Proposition \ref{idealvariation}).


Since we mod out by $I$, a Hamiltonian of order $n$ or higher has no effect on the $n$-complex structure. The Hamiltonians of degree $\geq n$ generate a normal subgroup in $\Symp_0(T^*\Sigma)$. The effective action is by the quotient group, consisting of higher diffeomorphisms of degree at most $n-1$. 

In the following subsection we compute the variation of the higher Beltrami differentials under a Hamiltonian and we deduce the local theory of higher complex structures.

\subsection{Local theory}

In this subsection, we restrict attention to an open neighborhood of the origin 0 in $\mathbb{C}$. We prove that any two higher complex structures are locally equivalent under higher diffeomorphisms. Before doing so, we have to compute the variation of the higher complex structure by a higher diffeomorphism.

As seen in the previous subsection, we have to compute Poisson brackets modulo the ideal $I$. A small argument simplifies the computations a lot. We call it the ``simplification lemma'':

\begin{lemma}[Simplification lemma]\label{simplificationlemma}
Let $I=\left\langle f_1, ..., f_r \right\rangle$ be an ideal of $\mathbb{C}[z,\bar{z},p,\bar{p}]$ such that $\{f_i, f_j\} = 0 \mod I$ for all $i$ and $j$. Then for all polynomials $H$ and all $k \in \{1,...,r\}$ we have $$\{H, f_k\} \mod I = \{H \mod I, f_k\} \mod I.$$
\end{lemma}
\begin{proof}
The only thing to show is that if we replace $H$ by $H+gf_l$ for some polynomial $g$ and some $l \in \{1,...,r\}$, the expression does not change. Indeed, 
$\{H+g f_l, f_k\}=\{H, f_k\}+g\{f_l,f_k\}+\{g,f_k\}f_l = \{H, f_k\}  \mod I$ using the assumption.
\end{proof}

For our ideal $I = \left\langle p^n,-\bar{p}+\mu_2p+...+\mu_np^{n-1} \right \rangle$, we have $\{p^n,-\bar{p}+\mu_2 p+...+\mu_n p^{n-1}\} = np^{n-1}(\partial\mu_2 p+...+\partial\mu_n p^{n-1})=0 \mod I$, so we can use the lemma. Therefore, we can reduce a Hamiltonian $H$ modulo $I$. So it can always be written as $H=\sum_k v_k p^{k-1}$.

Another small argument is that the first generator, $p^n$, does not change : $\{H,p^n\} \mod I = np^{n-1}\partial H \mod I = 0$ since there is no constant term in $H$. So we only have to settle the second generator.

Let us compute the variation of the second generator under the Hamiltonian $H=v_kp^{k-1}$. We get:
\begin{align*}
\{v_kp^{k-1}, -\bar{p}+\textstyle\sum_i \mu_ip^{i-1}\} =& \; (k-1)v_kp^{k-2}(\partial\mu_2p+...+\partial\mu_np^{n-1})+p^{k-1}\bar{\partial}v_k  \\
& -p^{k-1}(\mu_2+2\mu_3p+...+(n-1)\mu_np^{n-2})\partial v_k \\
=& \; p^{k-1}(\bar{\partial}v_k-\mu_2\partial v_k+(k-1)v_k\partial\mu_2) \\
&+\sum_{l=1}^{n-1-k} p^{k-1+l} ((k-1)v_k\partial\mu_{l+2}-(l+1)\mu_{l+2}\partial v_k) \mod I.
\end{align*}

\noindent Thus we obtain:
\begin{prop}[Variation of higher Beltrami differentials]\label{varmu}
The variation $\delta \mu_l$ under a Hamiltonian $H=v_kp^{k-1}$ is given by 
$$\delta \mu_l = \left \{ \begin{array}{cl}
(\bar{\partial}-\mu_2\partial+(k-1)\partial\mu_2)v_k & \text{ if  } l=k \\
((k-1)\partial\mu_{l-k+1}-(l-k)\mu_{l-k+1}\partial)v_k &\text{ if  } l>k \\
0 & \text{ if  } l<k.
\end{array} \right.$$
\end{prop}

Now, we are ready to state the local triviality of higher complex structures:

\begin{thm}[Local theory]\label{loctrivial}
Any two higher complex structures are locally equivalent under higher diffeomorphisms.
Using our coordinates this is equivalent to: any $n$-complex structure can be locally trivialized, i.e. there is a higher diffeomorphism which sends the structure to $(\mu_2(z,\bar{z}),...,\mu_n(z,\bar{z}))=(0,...,0)$ for all small $z\in \C$.
\end{thm}

The proof is in the spirit of the classical proof of Darboux theorem on local theory of symplectic structures.
\begin{proof}
The proof is by induction. For $n=2$, we already know the result which is Gauss and Korn-Lichtenstein's theorem on the existence of isothermal coordinates. So suppose that the statement is true for $n \geq 2$ and we show it for $n+1$.

By induction hypothesis, there is a higher diffeomorphism which makes $\mu_2(z)=...=\mu_n(z)=0$ for all $z$ near the origin. We construct a higher diffeomorphism generated by a Hamiltonian of degree $n$ giving $\mu_{n+1}(z)=0$ for all $z$ near 0. Since a Hamiltonian of degree $n$ does not affect the $\mu_k$ with $k\leq n$ (see previous proposition), we are done.

Let us try a Hamiltonian of the form $$H(z, \bar{z},p,\bar{p})=v_n(z,\bar{z},p,\bar{p})p^n$$ generating a flow $\phi_t$. 
We denote by $\mu_{n+1}^t(z,\bar{z})$ the image of $\mu_{n+1}(z,\bar{z})$ by $\phi_t$ (note that $\phi_t$ fixes the zero-section pointwise). The variation formula \ref{varmu}  for $\mu_2=0$ then reads
$$\frac{d}{dt}\mu_{n+1}^{t}(z,\bar{z}) = \bar{\partial}v_{n}(z,\bar{z},0,0)$$
Thus, the variation does not depend on time. We wish to have $\frac{d}{dt}\mu_{n+1}^{t}(z,\bar{z})=-\mu_{n+1}^{t=0}(z,\bar{z}).$ So we have to solve 
$$\bar{\partial}v_{n}(z,\bar{z},0,0)=-\mu_{n+1}^{0}(z,\bar{z}).$$
The inversion of the Cauchy-Riemann operator $\bar{\partial}$ is well-known. We denote its inverse by $T$. Explicitly, we have 
$$Tf(z)=\frac{1}{2\pi i}\int_{\mathbb{C}}\frac{f(\zeta)}{\zeta-z}d\zeta \wedge d\bar{\zeta}$$ for any square-integrable function $f$.

Therefore, on the zero-section we set $v_n(z,\bar{z},0,0)=-T\mu_{n+1}^{0}(z,\bar{z})$ (since $\mu_{n+1}$ is smooth, it is locally square-integrable). To define it everywhere, we choose a bump function $\beta$, in our case a function on $\mathbb{C}^2$ which is 1 in a neighborhood of the origin and 0 outside a bigger neighborhood of the origin, and we put $$v_n(z,\bar{z},p,\bar{p})=-\beta(p,\bar{p}) T\mu_{n+1}^{0}(z,\bar{z}).$$ So the Hamiltonian is defined everywhere and gives a compactly supported vector field which therefore can be integrated for all times. We then get $$\mu_{n+1}^t(z,\bar{z})=(1-t)\mu_{n+1}^{0}(z,\bar{z})$$
Therefore, at time $t=1$, $\mu_{n+1}$ vanishes everywhere.
\end{proof}

So as for symplectic structures, there is no local invariant for higher complex structures. The only interesting properties can appear in the global theory.

\subsection{Moduli space}

We are finally ready to define and study the moduli space of higher complex structures. We then show that it is a contractible ball of dimension $(n^2-1)(g-1)$ and describe its tangent and cotangent space.

\begin{definition}
The \textbf{moduli space of higher complex structures}, denoted by $\mathbf{\T^n}$, is the space of all compatible $n$-complex structures modulo higher diffeomorphisms. 
In formula: $$\bm\hat{\mathcal{T}}^n = \Gamma(\Hilb^n_0(T^{*\mathbb{C}}\Sigma)) / \Symp_0(T^*\Sigma).$$
\end{definition}

Recall that an $n$-complex structure is compatible if the Beltrami differential satisfies $\mu_2\bar{\mu}_2<1$. Using complex conjugation we get another copy of our moduli space for which $\mu_2\bar{\mu}_2>1$.

Since a higher diffeomorphism of order 1 is a usual diffeomorphism and only Hamiltonians of order at most $n-1$ act non-trivially on $n$-complex structures, we recover for $n=2$ the usual Teichm\"uller space: $$\bm\hat{\mathcal{T}}^2=\mathcal{T}(\Sigma).$$
The moduli space $\T^n$ has the following properties:

\begin{thm}[Global theory]\label{mainresultncomplex}
For a surface $\Sigma$ of genus $g\geq 2$ the moduli space $\bm\hat{\mathcal{T}}^n$ is a contractible manifold of complex dimension $(n^2-1)(g-1)$. Its cotangent space at any point $\mu=(\mu_2,...,\mu_n)$ is given by 
$$T^*_{\mu}\bm\hat{\mathcal{T}}^n = \bigoplus_{m=2}^{n} H^0(K^m).$$
In addition, there is a forgetful map $\bm\hat{\mathcal{T}}^n \rightarrow \bm\hat{\mathcal{T}}^{n-1}$.
\end{thm}

\noindent We essentially show the existence of a cotangent space at every point, so we have a manifold. The dimension of the cotangent space is computed by Riemann-Roch formula. We do not enter into details on issues about infinite-dimensional manifolds and quotients. 

\begin{proof}
To see that $\T^n$ is a manifold, we examine the infinitesimal variation around any point. This will also give a description of the tangent and cotangent space. By definition, we have $$\bm\hat{\mathcal{T}}^n = \{(\mu_2,...,\mu_n) \mid \mu_m \in  \Gamma(K^{1-m}\otimes \bar{K}) \; \forall \, m \text{ and } \left|\mu_2\right| < 1\} / \Symp_0(T^*\Sigma)$$
The infinitesimal variation around $\mu=(\mu_2,...,\mu_n)$ is then given by 
$$T_\mu\bm\hat{\mathcal{T}}^n = \{(\delta\mu_2,...,\delta\mu_n) \mid \delta\mu_m \in  \Gamma(K^{1-m}\otimes \bar{K}) \;\forall \, m \} / \Lie(\Symp_0(T^*\Sigma)).$$

\noindent In the previous subsection, we have seen that every $n$-complex structure is locally trivializable. So there is an atlas in which $\mu = (\mu_2,...,\mu_n)=0$. The Lie algebra of $\Symp_0(T^*\S)$ is given by Hamiltonians (functions on $T^*\S$). We can decompose the Hamiltonian into homogeneous parts of degrees 1 to $n-1$. All higher terms do not affect the $n$-complex structure. With the simplification lemma \ref{simplificationlemma} we can reduce $H$ to $\sum_{k=2}^n v_kp^{k-1}$. By Proposition \ref{varmu} (with $\mu_k=0$ for all $k$), we get 
$$T_\mu\bm\hat{\mathcal{T}}^n = \{(\delta\mu_2,...,\delta\mu_n) \} / (\bar{\partial}v_2,...,\bar{\partial}v_{n})$$ where $v_m$ is a section of $K^{m-1}$. Thus, the tangent space splits into parts: $$T_\mu\bm\hat{\mathcal{T}}^n = \{\delta\mu_2 \in \Gamma(\bar{K}\otimes K^{-1})\} / \bar{\partial}v_2 \oplus ... \oplus \{\delta\mu_n \in \Gamma(K^{-n+1}\otimes \bar{K})\} / \bar{\partial}v_{n}.$$
To compute the cotangent space, we use the pairing between differentials of type $(1-k,1)$ and of type $(k,0)$ given by integration over the surface. We get 
\begin{align*}
(\{\delta\mu_m\} / \bar{\partial}v_{m})^* &=  \{t_m \in \Gamma(K^m) \mid \textstyle\int t_m \bar{\partial}v_{m} = 0 \; \forall \, v_{m} \in \Gamma(K^{m-1}) \} \\
&= \{t_m \in \Gamma(K^m) \mid \textstyle\int \bar{\partial}t_m v_{m} = 0 \; \forall \, v_{m} \in \Gamma(K^{m-1}) \} \\
&=  \{t_m \in \Gamma(K^m) \mid \bar{\partial}t_m = 0 \} \\
&=  H^0(K^m).
\end{align*}

\noindent Therefore $$T^*_\mu\bm\hat{\mathcal{T}}^n = \bigoplus_{m=2}^{n}H^0(K^m).$$

\noindent Now, a standard computation using the Riemann-Roch formula shows that (for genus $g\geq 2$)
$$\dim H^0(K^m) = (2m-1)(g-1).$$
Therefore $$\dim{\bm\hat{\mathcal{T}}^n} = \dim{T^*_\mu\bm\hat{\mathcal{T}}^n } = \sum_{m=2}^{n} \dim{H^0(K^m)} = \sum_{m=2}^{n} (2m-1)(g-1) = (n^2-1)(g-1).$$

\noindent
The fact that $\T^n$ is contractible is direct: if two $n$-complex structures $\mu$ and $\mu'$ are equivalent, so are $t\mu$ and $t\mu'$ for $t\in \R_+$. Thus we can retract a given class $[\mu]$ to the class of the trivial $n$-complex structure by $[(1-t)\mu]=(1-t)[\mu]$, where $t$ varies from 0 to 1.

\noindent The forgetful map is simply given by the following: to the equivalence class of an $n$-complex structure given by an ideal $I$, we associate the equivalence class of $I + \left\langle p,\bar{p} \right\rangle^{n-1}$. This is independent of coordinates since $\left\langle p,\bar{p} \right\rangle$ is the maximal ideal supported at the origin. In coordinates, the map just forgets $\mu_n$.
\end{proof}

\noindent Composing the forgetful maps, we get a map from $\bm\hat{\mathcal{T}}^n$ to Teichm\"uller space. Therefore to any $n$-complex structure is associated a complex structure. The cotangent space of $\T^n$ at $I$ is the Hitchin base of holomorphic differentials where the holomorphicity is with respect to the associated complex structure of $I$.

Since the cotangent space is a complex vector space, we automatically have an almost complex structure on $\T^n$. Furthermore, since it is the moduli space of a geometric structure on the surface, the mapping class group acts naturally on it.

From the previous theorem, we see that our moduli space $\bm\hat{\mc{T}}^n$ shares a lot of properties with Hitchin's component, in particular the dimension and contractibility.
There is another common property to notice:
\begin{prop}\label{copyteich}
There is an injection from Teichm\"uller space into our moduli space.
\end{prop}
\begin{proof}
Conceptually this injection follows from the fact that there is a canonical injection of $\Hilb^2_0(\C^2)$ into $\Hilb^n_0(\C^2)$ which is preserved under hamiltonian symplectomorphisms of $\C^2$ fixing the origin. The image of the injection are those points for which the associated jet of a curve is linear.

In coordinates, this map is simply given by $$\psi([\mu_2]) = [(\mu_2,0,...,0)].$$
We have to check that $\psi$ is well-defined, independent of the reference complex structure and injective. Take $\mu_2$ and $\mu_2'$ two equivalent complex structures. This means that there is a diffeomorphism, generated by a linear Hamiltonian $H=pv+\bar{p}\bar{v}$, identifying $\mu_2$ as the pullback of $\mu_2'$.
The action of this Hamiltonian on the $n$-complex structure $(\mu_2,0,...,0)$ affects only the first term: 
$$\{H,-\bar{p}+\mu_2p\} = (\delbar -\mu_2\del+\del\mu_2)v \times p.$$ Thus, the higher diffeomorphism generated by $H$ gives an equivalence between $(\mu_2,0,...,0)$ and $(\mu_2',0,...,0)$. Hence, the map $\psi$ is well-defined and independent of the reference complex structure.

For injectivity, suppose $[(\mu_2,0,...,0)]$ is equivalent to $[(\mu_2',0,...,0)]$ via a higher diffeomorphism generated by $H$. Since terms of degree 2 or more do not affect $\mu_2$ (see Proposition \ref{varmu}), the equivalence is already obtained by the linear part of $H$, which is the extension of a diffeomorphism of $\S$. This diffeomorphism of $\S$ sends $\mu_2$ to $\mu_2'$, so they are equivalent.
\end{proof}

The same property holds for the Hitchin component which can be defined as the deformation space of representations of the form $\pi_1(\Sigma) \rightarrow \PSL_2(\R) \rightarrow \PSL_n(\R)$ where the first map is a fuchsian representation and the second one is the principal map (see introduction \ref{higherteichth}).

We conjecture the equivalence of Hitchin's component and the moduli space of higher complex structures:

\begin{conj}\label{mmaincon}
The moduli space of higher complex structures $\bm\hat{\mathcal{T}}^n$ is canonically diffeomorphic to Hitchin's component $\mathcal{T}^n$.
\end{conj}

The whole second part of this thesis is devoted to attack this conjecture.
The basic idea is that an $n$-complex structure naturally gives a bundle with a matrix-valued 1-form which can be deformed to a flat connection.

\subsection{Induced bundle}\label{indbundle}

A higher complex structure, and more generally a point in $\cotang$, naturally gives a bundle with some extra structure on $\S$. 

To any point in the Hilbert scheme bundle $\Hilb^n(T^{*\C}\S)$, we can canonically associate a vector bundle $V$ of rank $n$ whose fiber over a point $z$ is $\mathbb{C}[p,\bar{p}]/I(z, \bar{z})$.
We can glue these fibers together which gives an $n$-dimensional vector bundle over $\Sigma$. Locally, there is basis of the form $(s,ps,...,p^{n-1}s)$ for some generic section $s$ (the vector $s(z)$ has to be a cyclic vector for $(M_p, M_{\bar{p}})$ for all $z$). Under a coordinate change $z\mapsto w(z)$, this basis transforms in a diagonal way since $p^k \mapsto (\frac{dw}{dz})^kp^k$, so we get a bundle which is a direct sum of line bundles.

The matrix viewpoint of the punctual Hilbert scheme gives a $\mf{sl}_n$-valued 1-form locally of the form $M_pdz+M_{\bar{p}}d\bar{z}$ which acts on this bundle (locally $M_p$ by multiplication by $p$, $M_{\bar{p}}$ by multiplication by $\bar{p}$).
If $I$ is a higher complex structure, then we get the 1-form locally expressed as $$\Phi_1(z, \bar{z})dz+\Phi_2(z, \bar{z})d\bar{z}$$ where $(\Phi_1, \Phi_2)$ is a pair of commuting nilpotent matrices. In that case there is a preferred direction in each fiber: the common kernel of both $\Phi_1$ and $\Phi_2$. In our local basis it is generated by $p^{n-1}s$.

Since higher diffeomorphisms act on the ideals $I(z, \bar{z})$, they also act on $V$. This action does not preserve the direct sum decomposition into line bundles. If $V$ comes from a higher complex structure, we show that there is a flag structure on $V$ which is preserved. Put $I_k=I+\langle p,\bar{p} \rangle^k$, so we have $I_0=\C[p, \bar{p}] \supset I_1 \supset ...\supset I_n = I$. Define $$F_{n-k}=\ker (\mathbb{C}[p,\bar{p}]/I(z, \bar{z}) \rightarrow \C[p, \bar{p}]/I_k).$$ Then, the $F_k$ form an increasing complete flag with $\dim F_k=k$.
Locally take a basis $(s,ps,...,p^{n-1}s)$ and a symplectomorphism generated by $H$. An element $p^ks$ changes by $$\{H, p^ks\}=p^k\{H,s\}+kp^{k-1}(\del H)s.$$ Since $H$ has no constant terms, we get only terms with factor $p^l$ with $l\geq k$. So the action of $\Symp_0$ on $V$ is lower triangular, the flag structure is preserved.

\medskip
\noindent In order to exploit the full strength of the definition of a higher complex structure using the punctual Hilbert scheme, we analyze the latter in more detail.

\section{More on punctual Hilbert schemes}\label{moreonhilbertschemes}

In this section we describe the symplectic structure of $\Hilb^n(\C^2)$, study an interesting subspace, the reduced Hilbert scheme, and we discuss the hyperkähler structure of $\Hilb^n(\C^2)$.

\subsection{Symplectic structure}\label{hilbertschemesympl}

The punctual Hilbert scheme inherits a complex symplectic structure from the space of $n$ points $(\mathbb{C}^2)^n$. Denoting by $(x_i,y_i)$ the coordinates of the $i$-th point, the symplectic structure is simply $\omega=\sum_i dx_i\wedge dy_i$. Let us show that this expression extends to the Hilbert scheme. For this, we express $\omega$ in terms of the coordinates $t_k$ and $\mu_k$.
Denote by $M_x$ and $M_y$ the multiplication operators by $x$ and $y$ respectively in $\mathbb{C}[x,y]/I$ where $I$ is a generic element of $\Hilb^n(\mathbb{C}^2)$: $$I=\left\langle x^n-t_1x^{n-1}-...-t_n, -y+\mu_1+\mu_2x+...+\mu_nx^{n-1} \right\rangle.$$
Since $x^n-t_1x^{n-1}-...-t_n = \prod_i (x-x_i)$, we see that $M_x$ can be diagonalized to $\diag(x_1,...,x_n)$. Since $y=Q(x)$, we get $M_y=Q(M_x)$. So its diagonalized form is $\diag(Q(x_1),...,Q(x_n)) = \diag(y_1,...,y_n)$ since $Q$ is the interpolation polynomial. Since the trace is unchanged by conjugation, we get
$$\omega = \tr \diag(dx_1,...,dx_n)\wedge\diag(dy_1,...,dy_n) = \tr dM_x\wedge dM_y.$$
We have seen in the matrix viewpoint of the Hilbert scheme \ref{hilbmatrixviewpoint}, that in the basis $(1,x,x^2,...,x^{n-1})$, $M_x$ is a companion matrix, so $dM_x$ has only non-zero elements in the last column. We have also seen that $M_y$ is uniquely determined by its first column. Denote by $\alpha_{i,j}$ the matrix elements of $M_y$ in this basis. These can be expressed in terms of $\mu$ and $t$. Then we have 
\begin{equation}\label{symphilbb}
\omega = \tr dM_x\wedge dM_y = \sum_i dt_i\wedge d\alpha_{n,n+1-i}.
\end{equation}
This gives the complex symplectic structure on $\Hilb^n(\mathbb{C}^2)$. We see in particular that $t_i$ is not conjugated to $\mu_i=\alpha_{i,1}$ but to $\alpha_{n,n+1-i}$.

\begin{Remark}
An explicit formula for $\omega$ in coordinates $(t_i, \mu_i)$ can be given, using the notation for partitions (see page \pageref{notations2}): $$\omega = \sum_{i=1}^n \sum_{k=0}^{n-i} \left(\sum_{\nu \vdash k}(-1)^{\lvert \nu \rvert}\frac{\lvert \nu \rvert !}{\prod_j \nu_j !}t_1^{\nu_1}\cdots t_k^{\nu_k} \right) d\mu_{i+k}\wedge dt_i.$$
\hfill $\triangle$
\end{Remark}

\begin{Remark}
The symplectic structure on the Hilbert scheme can also be described in terms of the Poisson bracket between the coordinates. Computations give the surprisingly easy brackets: $$\{t_i,t_j\} = 0 = \{\mu_i,\mu_j\} \text{ and } \{\mu_i,t_j\} = t_{j-i}$$ with $t_0=1$ and $t_i=0$ for $i<0$.

To prove these expressions, one first computes the Poisson brackets of the two polynomials $P(z)=\sum_{i=1}^n t_iz^{n-i}$ and $Q(z)=\sum_{i=1}^{n}\mu_iz^{i-1}$. One gets $$\{P(z),P(w)\}=0=\{Q(z),Q(w)\} \text{ and } \{ P(z),Q(w)\} =-\frac{P(z)-P(w)}{z-w}.$$
\hfill $\triangle$
\end{Remark}

Haiman \cite{Haiman} described a way to find coordinates of $\Hilb^n(\C^2)$ in the chart associated with a Young diagram $D$ (see subsection \ref{structurehilbscheme}). For each box $B_x \in D$ consider the rightmost box $B_r \in D$ in the same row as $B_x$ and the bottommost box $B_b \in D$ in the same column as $B_x$ (see figure \ref{Haimancoo}). The box $B_{r+1}$ to the right of $B_r$ is not in $D$, so gives a linear combination of boxes in $D$. Denote by $b_{x,r}$ the coefficient of $B_b$ in this linear combination. Similarly, denote by $b_{x,b}$ the coefficient of $B_r$ in the linear combination associated to the box $B_{b+1}$ at the bottom of $B_b$. Haiman shows that the set $\{b_{x,r}, b_{x,b}\}_{x\in D}$ is a coordinate system. For example the Young diagram associated to the basis $(1,x,...,x^{n-1})$ has Haiman coordinates $(t_k, \alpha_{n, n+1-k})_{1\leq k \leq n}$.
\begin{figure}[h]
\centering
\includegraphics[height=2cm]{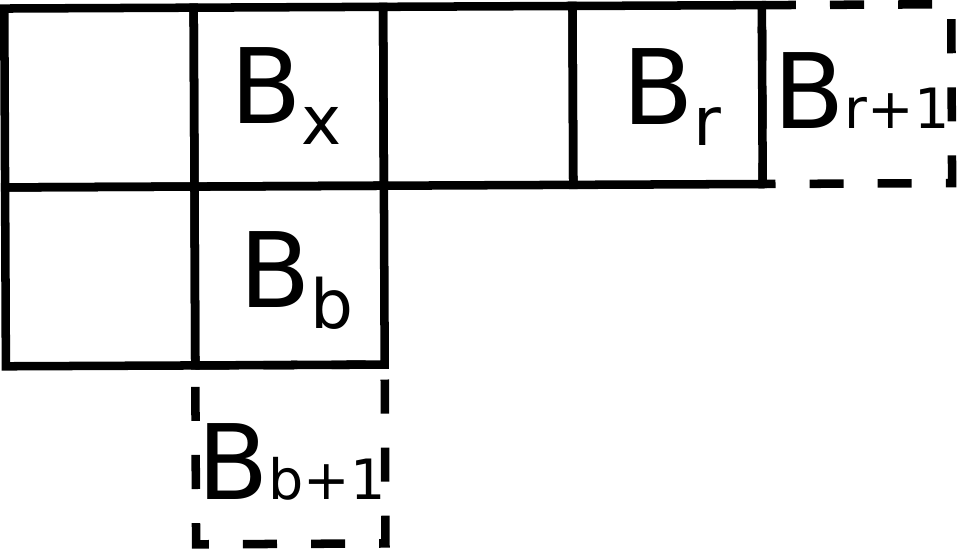}

\caption{Haiman's coordinates}
\label{Haimancoo}
\end{figure}
We have even more:
\begin{prop}
Haiman coordinates $\{b_{x,r}, b_{x,b}\}_{x\in D}$ are canonical coordinates with respect to the symplectic structure of the punctual Hilbert scheme.
\end{prop}
\begin{proof}
We have seen in equation \eqref{symphilbb} that the symplectic structure on the Hilbert scheme reads $\omega = \tr dM_x\wedge dM_y$.
Changing to the base adapted to the Young diagram $D$ (basis generated by monomials $x^iy^j$ where $(i,j) \in D$), the matrix $M_x$ becomes a matrix $N_x$ with entries 1 on the line under the main diagonal, apart from some columns where the linear combination associated to some $B_{r+1}$ is written. Similarly, the matrix $M_y$ becomes a matrix $N_y$ with entries 1 on the line under the main diagonal, apart from some columns where the linear combination associated to some $B_{b+1}$ is written. Finally, we compute
\begin{align*}
\omega &= \tr dM_x\wedge dM_y \\
&= \tr dN_x \wedge dN_y \\
&= \sum_{x\in D} db_{x,r}\wedge db_{x,b}.
\end{align*}
\end{proof}
\begin{Remark}
Using Haiman coordinates, the homology of $\Hilb^n(\C^2)$ can be computed easily.
\end{Remark}

\subsection{Reduced Hilbert scheme}\label{reduced-hilb-scheme}

It is easy to show that the zero-fiber $\Hilb^n_0(\C^2)$ is isotropic in $\Hilb^n(\C^2)$, but it cannot be Lagrangian because of the dimensions (the zero-fiber is of dimension $n-1$ whereas the Hilbert scheme is of dimension $2n$). We wish to define a symplectic subspace of $\Hilb^n(\C^2)$ such that $\Hilb^n_0(\C^2)$ becomes Lagrangian. This will be achieved with the reduced Hilbert scheme.

\begin{definition}
The \textbf{reduced Hilbert scheme} $\Hilb^n_{red}(\C^2)$ is the space of all elements of $\Hilb^n(\C^2)$ whose image under the Chow map ($n$ points counted with multiplicity modulo order) has barycenter the origin.
\end{definition}

Since the barycenter is independent of the order of the points, the punctual Hilbert scheme $\Hilb^n(\C^2)$ is topologically a direct product of $\mathbb{C}^2$ and the reduced Hilbert scheme. Thus the dimension of $\Hilb^n_{red}(\mathbb{C}^2)$ is $2n-2$. 

In the matrix viewpoint, the reduced Hilbert scheme gives matrices in $\mf{sl}_n$:

\begin{prop}[Matrix viewpoint]
$$\Hilb^n_{red}(\mathbb{C}^2) \cong \{(A,B) \in \mf{sl}_n^2 \mid [A,B]=0, (A,B) \text{ admits a cyclic vector}\} / \SL_n.$$
\end{prop}

In our coordinates, we get $t_1=0$ since $0=\tr M_p = t_1$. Further, since $M_{\bar{p}}=\mu_1 \id+\mu_2M_p+...+\mu_nM_p^{n-1}$ (see equation \eqref{matrixtwo}), we deduce
\begin{equation}\label{mu1value}
\mu_1 = -\frac{1}{n}\sum_{k=2}^{n-1}\mu_{k+1}\tr M_p^k = -\sum_{k=2}^{n-1} \frac{k}{n}t_k \mu_{k+1} \mod t^2.
\end{equation}

There is a third way to get the reduced Hilbert scheme: by a symplectic quotient. This construction has the advantage to directly provide that $\Hilb^n_{red}(\C^2)$ is symplectic and that the zero-fiber $\Hilb^n_0(\C^2)$ is a Lagrangian subspace.

Consider the action of $\mathbb{C}$ on $\mathbb{C}^2$ given by translation in the $y$-direction. This action induces an action on $(\mathbb{C}^2)^n$, which is hamiltonian with moment map $H(x,y)=x_1+\cdots+x_n$. Let us do the hamiltonian reduction: first we restrict to $H^{-1}(\{0\})$ which corresponds to $t_1=0$ since by Vieta's correspondence $t_1=x_1+\cdots x_n=H(x,y)$. Then we have to quotient out the action which means that we have to identify $y$ with the shifts $y+y_0$ for $y_0\in \C$.
Geometrically, we can use this shift to set the barycenter of the $n$ points to the origin. With this $\mu_1$ becomes a function of the other $\mu$'s and the $t$'s given by equation \eqref{mu1value}.

Therefore, we can define the reduced punctual Hilbert scheme of the plane to be the hamiltonian reduction $\Hilb^n(\mathbb{C}^2)\sslash \mathbb{C}$.
Hence, it inherits a symplectic structure. This symplectic structure is simply the one from $\Hilb^n(\mathbb{C}^2)$ with $t_1=0$ ($\mu_1$ only appears together with $t_1$ so also disappears). 

Now we can announce the main property:
\begin{prop}\label{lagr}
The zero-fiber $\Hilb^n_0(\mathbb{C}^2)$ is Lagrangian in $\Hilb^n_{red}(\mathbb{C}^2)$.
\end{prop}
\begin{proof}
Since on the zero-fiber, all $t_i$ vanish, we see that $\omega= \sum_i dt_i\wedge d\alpha_{n,n+1-i}$ vanishes as well. Since its dimension is $n-1$ which is half the dimension of the reduced Hilbert scheme, we are done.
\end{proof}

\noindent As a corollary, we get that the cotangent space to the zero-fiber $\Hilb^n_0(\mathbb{C}^2)$ is given by its normal bundle inside the reduced Hilbert scheme (that is a general property for Lagrangians using the symplectic form). At first order, i.e. modulo $t^2$ (terms which are at least quadratic in the $t_i$'s), the normal bundle coincides with the whole space $\Hilb^n_{red}(\mathbb{C}^2)$. Thus we can write:
\begin{equation}\label{hilbmodt2}
T^*\Hilb^n_0(\mathbb{C}^2) \cong \Hilb^n_{red}(\mathbb{C}^2) \mod t^2.
\end{equation}

\subsection{Hyperk\"ahler structure}\label{hkhilbertscheme}

The punctual Hilbert scheme inherits from $\C^2$ a hyperkähler structure, a very rich geometric structure. We give first a small overview of hyperkähler geometry using references \cite{Hit.4} and \cite{Nakajima}. Then we analyze the case of the punctual Hilbert scheme.

\subsubsection{Generalities on hyperk\"ahler structures}\label{generalitiesHK}

Roughly speaking, a hyperkähler manifold is modeled on the quaternions. In particular it has a 1-parameter family of Kähler structures. More precisely, a \textbf{hyperk\"ahler manifold}, or in short HK-manifold, is a Riemannian manifold $M$ equipped with three covariant constant (with respect to the Levi-Civita connection) orthogonal automorphisms $I, J$ and $K$ of the tangent bundle which satisfy the quaternionic identities $I^2 = J^2 = K^2 = IJK = -1.$

Any combination $I(\alpha, \beta, \gamma) = \alpha I +\beta J + \gamma K$ satisfies $I(\alpha, \beta, \gamma)^2=-\id$ iff $\alpha^2+\beta^2+\gamma^2 = 1$. So we have a 1-parameter family of complex structures on $M$, parameterized by a sphere. We write this parameter as $\l = (\alpha, \beta, \gamma) \in \C P^1$. Using the Riemannian metric, we get also a family of symplectic structures which combine with the complex structures to a family of Kähler structures.  The three symplectic structures associated with $I, J$ and $K$ are denoted by $\omega_1, \omega_2$ and $\omega_3$. 
In complex structure $I$, the complex form $\omega_\C = \omega_2+i\omega_3$ is holomorphic. In this complex structure, we write $\omega_\R = \omega_1$.

Hyperkähler structures are quite rigid. On a compact manifold, there can be at most a finite-dimensional space of hyperkähler structures. A theorem of Gray asserts that a HK-submanifold of a HK-manifold is totally geodesic (see \cite{Hit.4} and references therein). In particular no subvariety of $\H P^n$ is hyperkähler (neither $\H P^n$). It is also known that every compact complex manifold with holomorphic symplectic structure is HK.
The only known examples of compact HK-manifolds are complex tori and the punctual Hilbert scheme of a K3-surface.
A HK-manifold has \emph{vanishing Ricci curvature tensor}, so it is a solution to Einstein's vacuum equations which makes them interesting for theoretical physics.

We have seen that it is difficult to get examples of HK-manifolds. There are two general constructions: the hyperkähler quotient construction and the twistor approach, both developed in \cite{HKLR}. The first is a generalization of a symplectic quotient: starting with a group $G$ acting on a HK-manifold $M$, under some assumptions there is a moment map $\mu: M \rightarrow \g\otimes \R^3$ and the quotient $\mu^{-1}(\{0\})/G$ inherits a HK-structure from $M$.
We describe the second with more detail (see also \cite{HKLR}).

The twistor approach allows to put all complex structures together. For that, take a hyperk\"ahler manifold $M$ and consider $X_M = \C P^1\times M$ endowed with the complex structure at the point $(\l, m)$ given by $I_{\l, m}=(I_0, I_\l)$ where $I_0$ is the standard structure of $\C P^1$ and $I_\l$ is the complex structure of $M$ associated to $\l \in \C P^1$. The space $X_M$ is called the \textbf{twistor space}.
The projection $X_M \rightarrow \C P^1$ is holomorphic. A holomorphic section $s$ of this projection is called a \textbf{twistor line}. Adding the reality constraint that $s$ should take complex conjugate values on diametral opposite points gives the notion of a real twistor line.

The essential theorem in the twistor approach to hyperk\"ahler manifolds is that \textit{the space of all real twistor lines is isomorphic as HK-manifold to the initial space $M$}. See at the end of subsection \ref{bigpicture} for a precise statement.
Roughly speaking, any pair of points in different fibers determine a twistor line. So a single point determines a unique real twistor line passing through it and its conjugated point. So we can identify the space of real twistor lines with a fiber, i.e. $M$.

\subsubsection{Hyperk\"ahler structure of the punctual Hilbert scheme}
Let us turn to the punctual Hilbert scheme $\Hilb^n(\C^2)$ which inherits a HK-structure from $\C^2 \cong \H$. 
A general theorem of Beauville for two-dimensional manifolds $X$ asserts that if $X$ admits a holomorphic symplectic structure, so does $\Hilb^n(X)$ (see \cite{Nakajima} theorem 1.10). So we have a holomorphic symplectic structure on $\Hilb^n(\C^2)$.
If in addition $X$ is compact, then both $X$ and $\Hilb^n(X)$ are HK-manifolds. This does not directly apply to $\C^2$.

Nakajima's book \cite{Nakajima} (chapter 3) explains in detail the construction of $\Hilb^n(\C^2)$ as a hyperkähler quotient, using the matrix viewpoint.
We describe here shortly the twistor picture using references \cite{Etingof} and \cite{Boalch}. 

The punctual Hilbert scheme has a natural complex structure coming from $\C^2$. We denote this complex structure by $I$. Using the complex conjugation, we get another complex structure which is $-I$. In any other complex structure, we get the so-called \textbf{Calogero--Moser space} denoted here by $\Hilb^{n,q}(\C^2)$: $$\Hilb^{n,q}(\C^2) = \{(A,B)\in \GL_n(\C)^2\mid \rk ([A,B]-q\id)\leq 1\}/\GL_n(\C).$$ 
This definition is quite similar to the matrix description of the punctual Hilbert scheme in Proposition \ref{bijhilbert}. Indeed for $q=0$, it is an exercise that $\rk [A,B] \leq 1$ implies $[A,B]=0$ (see proposition 2.8 in \cite{Nakajima}, or exercise 2.2 in \cite{Boalch}). So we see that
$\Hilb^{n,0}(\C^2)$ is the punctual Hilbert scheme $\Hilb^n(\C^2)$.

There are multiple other ways to define the Calogero--Moser system. Is is the hamiltonian reduction of $T^*\mf{gl}_n$ by the $\GL_n$-conjugation action over the coadjoint orbit $\{q\id+L \mid \rk L \leq 1\}$. Still another way is to consider the non-commutative polynomial space $\C\langle x,y \rangle$ and quotient by the left-ideal generated by $yx=xy-q$. The multiplication operators $(M_x,M_y)$ are then in $\Hilb^{n,q}(\C^2)$.

There is a holomorphic $\C^*$-action on the Calogero-Moser space by scaling the matrices by a non-zero complex number $\l$. This changes $q$ to $q/\l^2$ so $\Hilb^{n,q}$ and $\Hilb^{n,q'}$ are biholomorphic for $q, q' \in \C^*$.  
Therefore, the twistor picture of the HK-structure on $\Hilb^n(\C^2)$ is as follows: at every complex structure apart from $\pm I$ we see the Calogero--Moser space, and at $\pm I$ we see the punctual Hilbert scheme (or its complex conjugate).

\subsubsection{Real symplectic structure of the punctual Hilbert scheme*}\label{realsympstructure}

In subsection \ref{hilbertschemesympl}, we computed the complex symplectic structure $\omega_\C$ of $\Hilb^n(\C^2)$ using $y_i=Q(x_i)$ (Lagrange interpolation) and the complex symplectic structure of $\C^2$.

Here we describe with more detail the real symplectic structure $\omega_\R$. It comes from the real symplectic structure of $\C^{2n}$. Using $(x_i, y_i)_{1\leq i \leq n}$ as complex coordinates of $\C^{2n}$ we have 
$$\omega_\R = \frac{i}{2}\sum_i (dx_i\wedge d\bar{x}_i + dy_i\wedge d\bar{y}_i).$$
This 2-forms is constant, so invariant under the action of the symmetric group $\mc{S}_n$. To see why it lifts to $\Hilb^n(\C^2)$, the blow-up of the configuration space, we express $\omega_\R$ in coordinates $(\mu_k, t_k)_{1\leq k \leq n}$ of $\Hilb^n(\C^2)$.

The same strategy as for $\omega_\C$ applies, but we have to push it further. We get a description for $\omega_\R$ modulo $t^2$. The reader might skip the description of $\omega_\R$ since we will not use it in the sequel.

\medskip
\noindent A generic point $I=\langle x^n+t_1x^{n-1}+...+t_n, -y+\mu_1+\mu_2x+...+\mu_nx^{n-1}\rangle$ of the Hilbert scheme admits $(1,x,x^2,...,x^{n-1})$ as basis for $\C[x,y]/I$. So we get a link between $x$ and $y$ via the Lagrange interpolation polynomial $Q$: $y=Q(x)=\sum_i \mu_ix^{i-1}$. The complex conjugates $\bar{x}$ and $\bar{y}$ can also be expressed in this basis. We define $m_i$ to be the coefficient of $\bar{x}$ for the basis element $x^{i-1}$, i.e.
$$\bar{x}=\sum_i m_ix^{i-1}.$$
We then get $\bar{y}=\sum_i \bar{\mu}_i \bar{x}^{i-1} = \sum_i \bar{\mu}_i (\sum_j m_j x^{j-1})^{i-1}$ which can be computed modulo $I$. In order to get $m_i$ as function of the coordinates $(\mu_k, t_k)_{1\leq k \leq n}$, we use the resultant.

Recall that the resultant of two polynomials $P_1$ and $P_2$ is a polynomial expression of their coefficients which is zero iff $P_1$ and $P_2$ have a common root. The resultant has an explicit expression given by the determinant of a matrix.
In our setting, consider the system
$$ \left \{\begin{array}{cl}
0 &= x^n-t_1x^{n-1}-...-t_n  \\
0 &= m_n x^{n-1}+...+m_2x+m_1-\bar{x}
\end{array}\right. $$
as a polynomial system in $x$ with coefficients in $\C [\bar{x}]$. So the first equation is of degree $n$ and the second of degree $n-1$. Since both equations have a common zero (any value for $x$), the resultant is 0:

$$0= \begin{vmatrix} 1 &  & & m_n & &  \\ -t_1 & \ddots & & \vdots & \ddots & \\ -t_2 & & 1 & m_2 & & \\ \vdots & \ddots & -t_1 & m_1-\bar{x} & & m_n \\ -t_n & &-t_2 & &\ddots & \vdots \\ &\ddots & \vdots & & & m_2 \\ & & -t_n & & & m_1-\bar{x} \end{vmatrix}.$$

The first column is repeated $n-1$ times (with a shift) and the $n$-th column is repeated $n$ times, so in total we get a matrix of size $2n-1$. It is clear that we get a polynomial equation in $\bar{x}$ of degree $n$. This equation has to be $$\bar{x}^n+\bar{t}_1\bar{x}^{n-1}+...+\bar{t}_n = 0$$ so we can compare coefficients to solve for  the $m_i$. An explicit computation gives $$\bar{t}_k = \sum_{l=k}^n t_l \left(\sum_{\substack{\pi \vdash l \\ \left| \pi \right| = k}}l\times\frac{(k-1)!}{\prod_j \pi_j!}m_2^{\pi_1}\cdots m_{l+1}^{\pi_l} \right) \mod t^2$$ where we use the notation for partitions (see page \pageref{notations2}).
Details for a completely analogous computation are given in section \ref{dualcomplexstructure} below.
We see in particular that $\bar{t}_n = t_nm_2^n$, so $$m_2 = \sqrt[n]{\bar{t}_n/t_n}.$$
The first appearance of $m_i$ (with $i>2$) is in the relation for $\bar{t}_{n+2-i}$ as a linear term, so we can solve for it. With this procedure, we can get an explicit expression for all $m_i \mod t^2$.

Then, we can compute $\bar{y}_i$. Since it appears via $d\bar{x}_i\wedge d\bar{y}_i$ in $\omega_\R$, we only need to compute $\bar{y}_i$ modulo $t$. We get 
\begin{align*}
\bar{y}_i =& \sum_{k=1}^n \bar{\mu}_k x^{k-1} = \sum_{k=1}^n \bar{\mu}_k (m_2x_i+...+m_nx_i^{n-1})^{k-1} \\
=& \sum_{j=1}^{n-1} x_i^j \left( \sum_{\pi \dashv j} \bar{\mu}_{\left| \pi \right|+1} \frac{\left| \pi \right|!}{\prod_m \pi_m!} m_2^{\pi_1}\cdots m_{j+1}^{\pi_j} \right) \; \mod t.
\end{align*}

To finish, put the expressions for $\bar{x}_i$ and $\bar{y}_i$ in $\omega_\R = \frac{i}{2}\sum_i dx_i\wedge d\bar{x}_i + dy_i\wedge d\bar{y}_i$ and write all symmetric expressions in $x_i$ with the coordinates $t_i$, the elementary symmetric polynomials in $x_i$.

For example take $n=2$. We get $m_2 = \sqrt{\bar{t}_2/t_2}$ and after some computation 
$$(-2i)\omega_\R = 2\left| t_2 \right| d\bar{\mu}_2\wedge d\mu_2 + \frac{1+\left| \mu_2 \right|^2}{2\left| t_2 \right|}dt_2\wedge d\bar{t}_2 + \frac{1}{\left| t_2 \right|}(\bar{\mu}_2t_2 \;d\bar{t}_2\wedge d\mu_2 + \mu_2\bar{t}_2 \;dt_2\wedge d\bar{\mu}_2).$$

\begin{Remark}
The knowledge of the coefficients $m_i \mod t^2$ allows us to compute any polynomial $P(x_i, \bar{x}_i, y_i, \bar{y}_i)$ which is invariant under the symmetric group $\mc{S}_n$ (acting separately on $x, \bar{x}, y$ and $\bar{y}$) in coordinates $(\mu_i, t_i)$ modulo $t^2$.

For example, take a polynomial $Q(x)=\sum_i t_i x^{n-i}$ of degree $n$ with roots $x_1, ..., x_n$. Of course there is no general formula for $x_i$ in terms of the coefficients $t_i$, but we can compute the sum of the distances to the origin of the roots modulo $t^2$ (where we use $\sum_i x_i^k = kt_k \mod t^2$):
$$\sum_i x_i\bar{x}_i = \sum_i x_i(m_1+m_2x_i+...+m_nx_i^{n-1}) = \sum_{j=2}^n j m_j t_j \mod t^2. $$
\hfill $\triangle$
\end{Remark}

\section{A spectral curve}\label{Spectralcurve1}

In this section, we exploit the symplectic structure of the Hilbert scheme in order to describe the total cotangent bundle $\cotang$, which is of paramount importance in part \ref{part2}, and to construct a spectral curve associated to a point in $\cotang$.

\subsection{Cotangent bundle of the moduli space of \texorpdfstring{$n$}{n}-complex structures}\label{cotangsection}
We already described the cotangent space at one point $I$ to $\T^n$ in Theorem \ref{mainresultncomplex}. Here we describe the total cotangent bundle $T^*\bm\hat{\mathcal{T}}^n$.

First we show that the Hilbert scheme bundle $\Hilb^n_{red}(T^{*\C}\S)$ inherits a symplectic structure from the punctual Hilbert scheme.
Recall that the symplectic structure of $\Hilb^n(\C^2)$ is given by $\sum_i dt_i\wedge d\alpha_{n,n+1-i}$ (see equation \eqref{symphilbb}).
Analyzing the global nature of the matrix elements $\alpha_{i,j}$ of $M_y$, one can easily show that $t_i(z)\alpha_{n,n+1-i}(z)$ is of type (1,1), so can readily be integrated over $\Sigma$. Thus, the symplectic structure of the Hilbert scheme naturally extends to a symplectic structure of the Hilbert scheme bundle $\Hilb^n(T^{*\mathbb{C}}\Sigma)$ (more precisely to the space of its sections) given by $$\omega = \int_{\Sigma} \sum_i dt_i(z)\wedge d\alpha_{n,n+1-i}(z).$$

\noindent By equation \eqref{hilbmodt2}, we see that near the zero-section, a point in $T^*\Hilb^n_{0}(T^{*\C}\S)$ can be identified with a section of the reduced Hilbert scheme bundle $\Hilb^n_{red}(T^{*\C}\S)$, so an ideal of the form $$\left\langle p^n+t_2(z)p^{n-2}+...+t_n(z),-\bar{p}+\mu_1(z)+\mu_2(z)p+...+\mu_n(z)p^{n-1} \right\rangle.$$
To go to $\cotang$ we have to quotient out the action by higher diffeomorphisms.

We already computed the variation of $\mu_k$ under a Hamiltonian in Proposition \ref{varmu}. Imitating the same computation as for the cotangent space at one point, but now around an arbitrary $n$-complex structure $(\mu_2,...,\mu_n)$, we get the following:

\begin{thm}[Condition $(\mc{C})$]\label{conditionC}
The cotangent bundle of the moduli space $\T^n$ is given by 
\begin{align*}
T^*\bm\hat{\mathcal{T}}^n= \Big\{& [(\mu_2, ..., \mu_n, t_2,...,t_n)] \mid  \mu_k \in \Gamma(K^{1-k}\otimes \bar{K}), t_k \in \Gamma(K^k) \text{ and } \; \forall k\\
& (-\bar{\partial}\!+\!\mu_2\partial\!+\!k\partial\mu_2)t_{k}+\sum_{l=1}^{n-k}((l\!+\!k)\partial\mu_{l+2}+(l\!+\!1)\mu_{l+2}\partial)t_{k+l}=0 \Big\}.
\end{align*}
\end{thm}
Notice that for $\mu=0$, we get $\bar{\partial}t_k=0$, i.e. $t_k \in H^0(K^k)$ a holomorphic differential as previously computed. The condition on $t_k$, to which we refer to as ``\textbf{condition} $\mathbf{(\mc{C})}$'', can be seen as a \textit{generalized holomorphicity condition}.
Notice further that we have to take the equivalence class $[(\mu_2, ..., \mu_n, t_2,...,t_n)]$ since we have to quotient out the action of higher diffeomorphisms.
\begin{proof}
We proceed as in the proof of Theorem \ref{mainresultncomplex}. We have $$T^*_\mu\bm\hat{\mathcal{T}}=\{t_2,...,t_n \mid t_i\in \Gamma(K^i) \text{ and } \int_{\Sigma}\sum_l t_l \delta \mu_l =0 \}.$$
To single out the condition on $t_k$, we choose a Hamiltonian $H=wp^{k-1}$. Using the formula for $\delta\mu_l$ given in Proposition \ref{varmu} and integration by parts, we compute:
\begin{align*}
\int_\S \sum_l t_l \delta \mu_l =& \int_\S t_{k}(\bar{\partial}\!-\!\mu_2 \partial\!+\!(k\! -\!1)\partial\mu_2)w+\sum_{l=1}^{n-k}t_{k+l}(-(l\!+\!1)\mu_{l+2}\partial+(k\!-\!1)\partial \mu_{l+2})w \\
=& \int_\S w \left((-\bar{\partial}\!+\!\mu_2\partial\!+\!k\partial\mu_2)t_{k}+\sum_{l=1}^{n-k}((l\!+\!k)\partial\mu_{l+2}+(l\!+\!1)\mu_{l+2}\partial)t_{k+l}  \right)  
\end{align*}

\noindent This has to be zero for all $w$, therefore we precisely get condition $(\mc{C})$.
\end{proof}

There is an alternative computation of the total cotangent space $T^*\bm\hat{\mathcal{T}}^n$ as hamiltonian reduction. 
We use the property that $T^*(X/G)=  (T^*X)\sslash G$ for a hamiltonian action of $G$ on some symplectic space $X$. Recall that $\Hilb^n_0(\C^2)$ is Lagrangian in $\Hilb^n_{red}(\C^2)$, so the cotangent space to the zero-fiber is the whole space modulo $t^2$ (see equation \eqref{hilbmodt2}).
\begin{align*}
T^*\bm\hat{\mathcal{T}}^n &= T^*\left(\Gamma(\Hilb^n_0(T^{*\mathbb{C}}\Sigma))/\Symp_0(T^*\Sigma)\right) & \\
&= \Gamma(T^*\Hilb^n_0(T^{*\mathbb{C}}\Sigma)) \sslash \Symp_0(T^*\Sigma)  \\
&= \Gamma(T^{normal}\Hilb^n_0(T^{*\mathbb{C}}\Sigma))\sslash \Symp_0(T^*\Sigma)   \\
&= \Gamma(\Hilb^n_{red}(T^{*\mathbb{C}}\Sigma)) \sslash \Symp_0(T^*\Sigma) & \mod t^2.
\end{align*}

Therefore, we see that $T^*\bm\hat{\mathcal{T}}^n$ is the hamiltonian reduction of the reduced Hilbert scheme bundle by higher diffeomorphisms:
\begin{equation}\label{hkquotientofmodulispace}
\cotang \cong \Gamma\left(\Hilb^n_{red}(T^{*\mathbb{C}}\Sigma)\right) \sslash \Symp_0(T^*\Sigma)  \mod t^2.
\end{equation}
In the sequel, whenever we write an element of $\cotang$ as a $\Symp_0$-equivalence class of an ideal $I \in \Hilb^n_{red}(T^{*\mathbb{C}}\Sigma)$, we have to compute at first order in $t$. So we can consider \textit{the $t$'s to be infinitesimal small compared to the $\mu$'s}.

The space $\cotang$ is the main actor in part \ref{part2} because of the following conjecture:
\begin{conj}\label{hkcotang}
The cotangent bundle $\cotang$ admits a hyperkähler structure near the zero-section.
\end{conj}
There are three good reasons to believe in the conjecture:
\begin{itemize}
\item[$\bullet$] \textit{Construction by hyperkähler quotient}: Equation \ref{hkquotientofmodulispace} points towards a possible hyperkähler reduction. 
Indeed under some mild conditions a complex symplectic reduction $X\sslash G^\C$ is isomorphic to a hyperkähler quotient $X/// G^\R$. In our case $X=\Gamma(\Hilb^n_{red}(T^{*\mathbb{C}}\Sigma))$ is hyperkähler, since $\Hilb^n_{red}(\C^2)$ is. So it is plausible that $\cotang$ can be obtained as hyperkähler quotient of $\Gamma(\Hilb^n_{red}(T^{*\mathbb{C}}\Sigma))$ by the real group $\Symp_0(T^*\S)$. Notice that the complexified Lie algebra of $\Symp_0(T^*\S)$, i.e. the space of smooth complex-valued functions on $T^*\S$, has the same action on $X$ as the real Lie algebra since a Hamiltonian $H$ acts the same as $H\mod I$ (simplification lemma \ref{simplificationlemma}).

\item[$\bullet$] \textit{Feix--Kaledin structure}: If Hitchin's component and our moduli space $\T^n$ are diffeomorphic, Hitchin's component gets a complex structure. With its Goldman symplectic structure, there is good hope to get a Kähler structure. A general result of Feix and Kaledin (see \cite{Feix} and \cite{Kaledin}) asserts that for a Kähler manifold $X$, there is a neighborhood of the zero-section in $T^*X$ which admits a hyperkähler structure. 

\item[$\bullet$] \textit{Construction by twistor approach}: In part \ref{part2} of this thesis we describe a 1-parameter deformation of $\cotang$ which is a good candidate to be the twistor space of $\cotang$.
\end{itemize}

\subsection{Spectral curve}\label{sspectral}
In this subsection, we construct a ramified cover $\bm\tilde{\Sigma}$ over the surface $\Sigma$ inside its complexi\-fied cotangent bundle $T^{*\mathbb{C}}\Sigma$ associated to a point of $\cotang$.

Define polynomials $P(p)=t_2p^{n-2}+...+t_n$ and $Q(p)=\mu_1+\mu_2 p+...+\mu_n p^{n-1}$ where $\mu_1$ is known from equation \eqref{mu1value}: $\mu_1=-\sum_{k=2}^{n-1} \frac{k}{n}t_k \mu_{k+1} \mod t^2$. Put $I= \left\langle -p^n+P(p), -\bar{p}+Q(p) \right\rangle$. Define $\bm\tilde{\Sigma} \subset T^{*\mathbb{C}}\Sigma$ by the equations $p^n=P$ and $\bar{p}=Q$. This curve $\bm\tilde{\Sigma}$ is called \textbf{spectral curve}. This is a ramified covering space with $n$ sheets.

\begin{thm}\label{spectralcurveprop}
We have $\{ -p^n+P,-\bar{p}+Q\} = 0 \mod I \mod t^2$ iff $I \in T^*\bm\hat{\mathcal{T}}^n$.
\end{thm}
The theorem states that the spectral curve is Lagrangian ``near the zero-section'' iff the ideal comes from the cotangent bundle of the moduli space of the $n$-complex structure. Concretely this means that the $t_k$ satisfy condition $(\mc{C})$ appearing in Theorem \ref{conditionC}.

The proof is a direct computation. Expanding the result in $p$ gives the conditions of the previous theorem as coefficients. The coefficient of $p^{n-1}$ vanishes because of the special value of $\mu_1$. 

\begin{proof}
We decompose the Poisson bracket $\{-p^n+P,-\bar{p}+Q\}$ into three parts: $$\{-p^n+P,-\bar{p}+Q\} = \{P,-\bar{p}+Q-\mu_1 \}+\{P, \mu_1\}-\{p^n,-\bar{p}+Q\}.$$
To compute the first term, we can interpret $P$ as a Hamiltonian and use the variation formula for the Beltrami differentials. To be precise, set $I_0=\langle p^n, -\bar{p}+Q-\mu_1 \rangle$, the typical ideal from the zero-fiber of the Hilbert scheme. 
Since $P=0 \mod t$ and since we compute modulo $t^2$ we can reduce $I$ modulo $t$ which gives $I_0$. Hence from Proposition \ref{varmu} with $H=P$, we see that the first terms equals
$$\sum_{m=0}^{n-1}p^m\left((\bar{\partial}\!-\!\mu_2\partial\!+\!m\partial\mu_2)t_{n-m} +\sum_{l=0}^{m-1}(l\partial\mu_{m+2-l}-(m\!+\!1\!-\!l)\mu_{m+2-l}\partial)t_{n-l}  \right).$$

\noindent The second term $\{P,\mu_1\}$ simply vanishes modulo $t^2$ since $\mu_1 = 0 \mod t$. 

\noindent The third and last term $\{p^n, -\bar{p}+Q\}$ gives
\begin{align*}
\{p^n, -\bar{p}+Q\}  &=  np^{n-1}(\partial\mu_1+\partial\mu_2p+...+\partial\mu_np^{n-1})  \\
&= np^{n-1}\partial\mu_1+n\sum_{m=0}^{n-1}p^m\sum_{l=0}^m \partial\mu_{m+2-l}t_{n-l} \mod I \mod t^2
\end{align*}
with $t_1=0$ and $\mu_{n+1}=0$. Adding up with the term above concludes: The coefficient of $p^{m}$ (with $0\leq m <n-1$) is precisely condition $(\mc{C})$ for $t_{n-m}$ and the coefficient of $p^{n-1}$ vanishes with the special value of $\mu_1$.
\end{proof}

The theorem allows to associate a spectral curve to a point in $\cotang$. Notice that the spectral curve here is independent of a complex structure on the surface $\Sigma$, but it lies in the complexified cotangent bundle. Hitchin's spectral curve depends on the complex structure and lies in the holomorphic cotangent bundle (so it is trivially Lagrangian as a one-dimensional subspace of a two-dimensional symplectic space). For $\mu_k=0$ for all $k$, our spectral curve $\bm\tilde{\Sigma}$ lies in the holomorphic cotangent bundle and can be identified with Hitchin's spectral curve (with complex structure $\mu_2=0$ on $\Sigma$). So $\bm\tilde{\Sigma}$ can be seen as Hitchin's spectral curve deformed into the $\bar{p}$-direction of $T^{*\mathbb{C}}\Sigma$ by the $n$-complex structure.

On the spectral curve $\bm\tilde{\Sigma}$, there is a line bundle $L$ with fiber the eigenspace of $M_p$, the multiplication operator by $p$ in $\mathbb{C}[p,\bar{p}]/I$, since the characteristic polynomial of $M_p$ is given by $P$. The pushforward of $L$ to $\Sigma$ by the covering map gives the rank $n$ bundle with fiber $\mathbb{C}[p,\bar{p}]/I$. So we get a similar spectral data as for Hitchin's spectral curve.

Since $\bm\tilde{\Sigma}$ is Lagrangian to order 1, we can compute the periods of the Liouville form $\alpha=pdz+\bar{p}d\bar{z}$ restricted to $\bm\tilde{\S}$ modulo $t^2$.
The map $\cotang \rightarrow H^1(\bm\tilde{\S}, \C)/H^1(\S, \C)$ which sends a point $[(\mu_k, t_k)]$ to $[\alpha \mid_{\bm\tilde{\S}}]$ should be a local diffeomorphism, i.e. the periods of $\alpha$ should give coordinates on $\cotang$. 
Let us check that both spaces have the same dimension: with the Riemann-Hurwitz theorem one computes the genus of $\bm\tilde{\S}$ to be $$g(\bm\tilde{\S})=n^2(g-1)+1$$ where $g$ is the genus of $\S$. Hence we get $$\dim H^1(\bm\tilde{\S})-\dim H^1(\S)= 2g(\bm\tilde{\S})-2g = (2g-2)(n^2-1) = \dim \cotang.$$
The ratios of periods at the limit when all $t$'s go to 0 (so $\bm\tilde{\Sigma}$ collapses to $\Sigma$) should give coordinates on the moduli space $\bm\hat{\mathcal{T}}^n$. The main problem is that the periods are only defined up to modulo $t^2$.

\section{Conjugated higher complex structures}\label{dualcomplexstructure}
A higher complex structure is a section of $\Hilb^n_0(T^{*\C}\Sigma)$. So it is given at any point $z\in \Sigma$ by an ideal 
$$I=\langle p^n, -\bar{p}+\mu_2p+...+\mu_np^{n-1} \rangle .$$

There is a natural notion of conjugated structure. Of course we cannot take $\bar{\mu}_k$ since we do not have any complex structure on $\Sigma$. But we can use the natural complex conjugation on the complexified cotangent bundle $T^{*\C}\Sigma$ to get 
$$\overline{I}=\langle \bar{p}^n, -p+\bar{\mu}_2\bar{p}+...+\bar{\mu}_n\bar{p}^{n-1} \rangle .$$
To get the conjugated structure, we have to express $\overline{I}$ in the same form as $I$, i.e. as 
$$\overline{I}=\langle p^n, -\bar{p}+{}_2\mu p+...+{}_n\mu p^{n-1} \rangle$$
where $({}_2\mu,...,{}_n\mu)$ are the parameters of the conjugated $n$-complex structure.

For instance for $n=2$, we get $\overline{I}=\langle \bar{p}^2, -p+\bar{\mu}_2\bar{p}\rangle = \langle p^2, -\bar{p}+\frac{1}{\bar{\mu}_2}p\rangle$, so we get ${}_2\mu=\frac{1}{\bar{\mu}_2}$. 
With similar explicit computations, we get
\begin{align*}
{}_3\mu &= -\frac{\bar{\mu}_3}{\bar{\mu}_2^3} \\
{}_4\mu &= \frac{-\bar{\mu}_2\bar{\mu}_4+2\bar{\mu}_3^2}{\bar{\mu}_2^5}.
\end{align*}

To get a general formula for ${}_k\mu$ we use formal series reversion.
This means that we consider $\bar{p}=\mu_2p+...+\mu_np^{n-1}$ as a function in $p$ which we want to revert, i.e. we have $\bar{p}=f(p)$ and we wish to develop $p=f^{-1}(\bar{p})$ in a power series. Since we work modulo $I$, so in particular modulo $p^n$ and $\bar{p}^n$, we get a finite polynomial. This reversion is known as \textit{Lagrange's inversion}. To write down explicit formulas, we use the notation for partitions, see page \pageref{notations2}.

The Lagrange inversion formula gives:
\begin{prop}\label{dualhighercomplex}
$${}_k\mu = \frac{1}{\bar{\mu}_2^{2k-3}}\sum_{\pi \vdash k-2} \frac{(-1)^{\lvert \pi \rvert}(\lvert \pi \rvert +k-2)!}{n! \pi_1!\hdots \pi_{k-2}!}\bar{\mu}_2^{k-2-\lvert \pi \rvert}\bar{\mu}_3^{\pi_1}\hdots\bar{\mu}_n^{\pi_{n-2}}.$$
\end{prop}

\bigskip
\noindent In the sequel, we also need the conjugated notion to a point in $\cotang$ near the zero-section. By equation \eqref{hkquotientofmodulispace}, this is given by a $\Symp_0$-equivalence class of a section of $\Hilb^n_{red}(T^{*\C}\Sigma)$. Without the action of higher diffeomorphisms, we have an ideal 
$$I(z, \bar{z})=\langle p^n-t_2p^{n-2}-...-t_n, -\bar{p}+\mu_1+\mu_2p+...+\mu_np^{n-1} \rangle $$
where all $t_k$ are small. This means that we compute to order 1 in $t$'s (i.e. modulo $t^2$).

The conjugated structure comes from $\overline{I}$ which is expressed in the above form:
\begin{align*}
\overline{I} &=\langle \bar{p}^n-\bar{t}_2\bar{p}^{n-2}-...-\bar{t}_n, -p+\bar{\mu}_1+\bar{\mu}_2\bar{p}+...+\bar{\mu}_n\bar{p}^{n-1} \rangle \\
&= \langle p^n-{}_2tp^{n-2}-...-{}_nt, -\bar{p}+{}_1\mu+{}_2\mu p+...+{}_n\mu p^{n-1} \rangle
\end{align*}
where $({}_kt, {}_k\mu)$ are the parameters of the conjugated structure.

Since the $t_k$ are infinitesimal small compared to the $\mu_k$, the conjugated structure ${}_k\mu$ is still given by Proposition \ref{dualhighercomplex} above. 
For the formula for ${}_kt \mod t^2$, we get the following result:

\begin{prop}\label{dualdeformedstructure}
The conjugated structure $({}_kt, {}_k\mu)$ is given by ${}_k\mu$ as in  the previous Proposition \ref{dualhighercomplex} and ${}_kt$ is given by $${}_kt = \sum_{l=k}^{n} \bar{t}_l \left( \sum_{\substack{\pi \vdash l \\ \left| \pi \right| = k}}l\times\frac{(k-1)!}{\prod_j \pi_j!}\bar{\mu}_2^{\pi_1}\cdots\bar{\mu}_{l+1}^{\pi_l} \right).$$
\end{prop}
Some special values can be easily computed:
\begin{align*}
{}_nt &= \bar{\mu}_2^n\bar{t}_n \\
{}_2t &= \sum_{l=2}^{n}\bar{t}_l\sum_{m=1}^{l-1}\frac{l}{2}\bar{\mu}_{m+1}\bar{\mu}_{l-m+1}.
\end{align*}
The hurried reader may skip the proof.

\begin{proof}
Our approach is to use the resultant. Recall that the resultant of two polynomials $P$ and $Q$, denoted here by $\Res(P,Q)$, is a polynomial expression of their coefficients which is zero iff $P$ and $Q$ have a common root. The resultant has an explicit expression given by the determinant of a matrix.
In our setting, consider the system
$$ \left \{\begin{array}{cl}
0 &= \bar{p}^n-\bar{t}_2\bar{p}^{n-2}-...-\bar{t}_n  \\
0 &= \bar{\mu}_n\bar{p}^{n-1}+...+\bar{\mu}_2 \bar{p}+\bar{\mu}_1-p
\end{array}\right. $$
as a polynomial system in $\bar{p}$ with coefficients in $\C [p]$. So the first equation is of degree $n$ and the second of degree $n-1$. Since both equations have a common zero (any value for $\bar{p}$), the resultant is 0:

\begin{equation}\label{resultant}
0= \begin{vmatrix} 1 &  & & \bar{\mu}_n & &  \\ 0 & \ddots & & \vdots & \ddots & \\ -\bar{t}_2 & & 1 & \bar{\mu}_2 & & \\ \vdots & \ddots & 0 & \bar{\mu}_1-p & & \bar{\mu}_n \\ -\bar{t}_n & &-\bar{t}_2 & &\ddots & \vdots \\ &\ddots & \vdots & & & \bar{\mu}_2 \\ & & -\bar{t}_n & & & \bar{\mu}_1-p \end{vmatrix}
\end{equation}

The first column is repeated $n-1$ times (with a shift) and the $n$-th column is repeated $n$ times, so in total we get a matrix of size $2n-1$. It is clear that we get a polynomial equation in $p$ of degree $n$. The coefficient of this polynomial give the exact expression for the conjugated deformed structure. Computing modulo $t^2$ gives the same result for ${}_kt$ as above in Proposition \ref{dualdeformedstructure}.

We give some details since this involves some interesting combinatorics.
To compute ${}_kt$, we need the coefficient of $p^{n-k}$ in the determinant. For this, we take all $p$'s on the diagonal, apart from those with positions $n-1+i_1, n-1+i_2, ..., n-1+i_k$ where $i_1<i_2<...<i_k$. Further, we see that apart from the term $p^n$, every other term has at least one $t$, since the matrix modulo $t$ is upper triangular. Once we have chosen some $\bar{t}_l$ in line some line $L$, we have to choose all $p$'s in the lines bigger than $L$ (if not we get another $t$-contribution), so $L=n-1+i_k$ and $l\geq k$. We also get $i_k\leq l$, so we gave an increasing sequence $(i_m)_{1\leq m \leq k}$ in the interval $\{1,...,l\}$. 

To such a sequence, we can associate a partition $\pi$ of $l$ in $k$ parts as follows: take the differences $d_m= i_{m+1}-i_m$ for $m=1,...,k$ where we write $i_{k+1}=i_1+l$. We then have $d_m>0$ and $\sum_m d_m = l$. Thus, the differences define a partition.

The other way around, given a partition $\pi \vdash l$ with $k$ parts, the number of sequences which gives $\pi$ via their differences is 
\begin{equation}\label{partitiondiff}
l\times \frac{k!}{\prod_j \pi_j!}\times \frac{1}{k}.\end{equation}
This goes as follows: to get the sequence modulo $l$, choose a starting point ($l$ choices), choose a sum representation of the partition $\pi$, i.e. a representation of $l$ as sum of $k$ positive integers ($\frac{k!}{\prod_j \pi_j!}$ choices), then you get a sequence by adding the terms to the starting point. You have to divide by $k$ since by cyclicity, every sequence is counted $k$ times.

Coming back to the computation of the determinant, to compute the coefficient of $p^{n-k}$, we choose some $l\geq k$ and a sequence $0<i_1<...<i_k<l+1$, we take all $p$'s apart those in position $n-1+i_1, n-1+i_2, ..., n-1+i_k$ and we choose the $\bar{t}_l$ in line $n-1+i_k$. Eliminating the corresponding lines and columns, we are left with a diagonal matrix with $\pi_m$ times $\bar{\mu}_{m+1}$ on the diagonal for all $m$ (and some 1's). Using the formula \eqref{partitiondiff} for the number of sequences giving the same partition, we directly get the formula of the proposition. 
\end{proof}

\begin{Remark}
As a consequence of the relationship between $k$-subsets of $\{1,...,l\}$ and partitions of $l$ with $k$ parts, we get 
$$\binom{l}{k} = \frac{l}{k}\sum_{\substack{\pi \vdash l \\ \left| \pi \right| = k}}\frac{k!}{\prod_j \pi_j!}.$$
\hfill $\triangle$
\end{Remark}

\section{\texorpdfstring{$\GL_2(\R)$}{GL(2,R)}-action on \texorpdfstring{$T^*\bm\hat{\mc{T}}^n$}{cotangent moduli space}}\label{gl2}

There is a well-known $\GL_2(\R)$-action on the cotangent bundle of Teichmüller space, coming from its interpretation as half-translation surfaces. We describe here a generalization to the cotangent bundle $T^*\bm\hat{\mc{T}}^n$.


\subsection{Half-translation surfaces}\label{half-translation}

A point in the cotangent bundle to Teichmüller space $T^*\mc{T}^2$ is given by a couple $(\S, t)$ where $\S$ is a Riemann surface and $t$ a holomorphic quadratic differential $t\in H^0(K^2)$. We assume that $t\neq 0$. Around a point $z_0 \in \S$, the map $z\in \S \mapsto \int_{z_0}^{z} \sqrt{t}$ gives a well-defined local chart and the transition functions between two such charts are translations or $-\id$. 
Indeed the map is well-defined since $t$ is holomorphic, so $\sqrt{t}$ is closed. The transition between a chart centered at $z_0$ and a chart centered at $z_1$ is given by a translation by the complex number $\int_{z_0}^{z_1}\sqrt{t}$. Finally, there is a sign ambiguity in choosing the square root of $t$, so we can get the transition function $-\id$.

A couple $(\S, t)$ is called a \textbf{half-translation surface}. The ``half'' in the terminology comes from the sign ambiguity. A half-translation surface can always be represented by a polygon in $\R^2$ where every edge is identified to another edge by a translation or a composition of translation and $-\id$.

The group $\GL_2(\R)$ acts naturally on the space of polygons, so on the space of half-translation surfaces. This action plays an important role in the theory of billiards, which is equivalent to the study of the geodesic flow on half-translation surfaces. We recommend \cite{Wright} and the chapter \cite{Masur} of Masur and Tabachnikov for an introduction to billiards and half-translation surfaces.

\subsection{Description of the generalization}

To generalize the action to higher complex structures, recall from \ref{sspectral} that a point in $T^*\bm\hat{\mc{T}}^n$ defines a spectral curve $\bm\tilde{\S} \subset T^{*\C}\Sigma$. We have seen that $\bm\tilde{\S}$ is Lagrangian modulo $t^2$ (Theorem \ref{spectralcurveprop}). Thus the Liouville form $\alpha = pdz + \bar{p} d\bar{z}$ is closed on $\bm\tilde{\S}$ modulo $t^2$. In this section, it is important that $p$ and $\bar{p}$ are independent variables (since a fiber of the complexified cotangent bundle $T^{*\C}\S$ is of complex dimension 2). We write $x=p$ and $y=\bar{p}$ to mark this independence. This notation is necessary since the complex conjugates of $x$ and $y$ will appear.

Heuristically, the periods of the spectral curve, i.e. the integrals of $\alpha$ along closed paths, should give coordinates on $T^*\bm\hat{\mc{T}}^n$. We can change $\alpha$ to $\alpha' = a\alpha + \bar{b}\bar{\alpha}$ with $a,b \in \C$. The form $\alpha'$ is still closed so its periods should give another point of $\cotang$. 
Changing $\alpha$ to $\alpha'$ gives an action on $(x,y)$ which reads
\begin{align*}
x &\mapsto x' = ax+\bar{b}\bar{y} \\
y &\mapsto y' = \bar{b}\bar{x}+a y.
\end{align*}
In coordinates $(x,\bar{y})$ this transformation if given by a matrix $\left(\begin{smallmatrix} a & \bar{b} \\ b & \bar{a} \end{smallmatrix}\right)$.
The group $$\left\{\begin{pmatrix} a & \bar{b} \\ b & \bar{a} \end{pmatrix} \mid a\bar{a}-b\bar{b}\neq 0\right\}$$ is isomorphic to $GL_2(\R)$. Indeed, the base change from real coordinates $(X,Y)$ to complex conjugated coordinates  $(X+iY, X-iY)$ (or light-cone coordinates) transforms a matrix $\left(\begin{smallmatrix} A & B \\ C & D \end{smallmatrix}\right)$ in $\GL_2(\R)$ to a matrix of the above form with $2a=A+D-i(B-C)$ and $2b=A-D+i(B+C)$.

\begin{prop}\label{gl2action}
There is a $\GL_2(\R)$-action on $\cotang$.
\end{prop}
To proof the proposition, we describe the action in our coordinates. Recall that $T^*\bm\hat{\mc{T}}^n$ is parameterized by a $\Symp_0$-equivalence class of coordinates $[(\mu_k, t_k)]_{2\leq k\leq n}$. Recall that we obtain these coordinates by the relations 
\begin{align*}
x^n &= t_2 x^{n-2}+...+t_{n-1}x+t_n \\
y &=\mu_1+\mu_2x+...+\mu_nx^{n-1} 
\end{align*}
which generate an ideal $I$ in the reduced Hilbert scheme.
We show that the $\GL_2(\R)$-action on $x$ and $y$ gives an action on $\mu_k$ and $t_k$, equivariant with respect to higher diffeomorphisms. 
We first compute this action for $n=2$.

\begin{example}
For $n=2$, we start with $x^2=t_2$ and $y=\mu_2x$. We can express $\bar{x}$ as $\bar{x}=m_2x$ where $m_2=\sqrt{\bar{t}_2/t_2}$. We see that we have to assume $t_2\neq 0$ like for half-translation surfaces.

We then get $x'=ax+\bar{b}\bar{y} =(a+\bar{b}\bar{\mu}_2m_2)x$ and thus $$t_2' = (x')^2 = (a+\bar{b}\bar{\mu}_2m_2)^2t_2.$$
Further, we have $y'=ay+\bar{b}\bar{x}=(a\mu_2+\bar{b}m_2)x$ and we wish to have $y'=\mu_2'x'=\mu_2'(a+\bar{b}\bar{\mu}_2m_2)x$. Hence we get
$$\mu_2' = \frac{a\mu_2+\bar{b}m_2}{a+\bar{b}\bar{\mu}_2m_2}$$ which is nearly a Möbius transformation.
\hfill $\triangle$
\end{example}

\begin{proof}
First, we show that the action lifts to the Hilbert scheme bundle $\Hilb^n_{red}(T^{*\C}\S)$ modulo $t^2$, i.e. to $T^*\Hilb^n_{0}(T^{*\C}\S)$.
The philosophy is to express every variable in the basis $(1,x,...,x^{n-1})$ of $\C[x,y]/I$. The coefficients of $x^n$ are the $t_i$, and the one of $y$ are the $\mu_i$, which are our coordinates. The coefficients for $\bar{x}$ are the $m_i$ which can be computed by an algorithm described in \ref{realsympstructure}. We have to assume $t_n\neq 0$.
From this we can compute the coefficients of $\bar{y}$ and then also for $x'=ax+\bar{b}\bar{y}$. We write $x'=u_1+u_2x+...+u_nx^{n-1}$.
Then consider the system
$$ \left \{\begin{array}{cl}
0 &= x^n-t_2x^{n-2}-...-t_n \\
0 &= u_nx^{n-1}+u_{n-1}x^{n-2}+...+u_1-x'
\end{array}\right. $$
as a system in $x$ over $\C[x']$. The resultant is a polynomial in $x'$ of degree $n$ which is 0 since there is a common solution to the system:
$$0= \begin{vmatrix} 1 &  & & u_n & &  \\ 0 & \ddots & & \vdots & \ddots & \\ -t_2 & & 1 & u_2 & & \\ \vdots & \ddots & 0 & u_1-x' & & u_n \\ -t_n & &-t_2 & &\ddots & \vdots \\ &\ddots & \vdots & & & u_2 \\ & & -t_n & & & u_1-x' \end{vmatrix}.$$
This has the same form as equation \eqref{resultant}, so we know the formula for the coefficients of the polynomial in $x'$ modulo $t^2$. These coefficients have to be $t_i'$, so we get an expression for them. 
For example we get for the highest term $$t_n'=u_2^nt_n=(a+\bar{b}\bar{\mu}_2m_2)^nt_n$$ where $m_2=\sqrt[n]{\bar{t}_n/t_n}$.

Finally we get two expressions for the coefficients of $y'$ in our basis. The first comes from $y'=ay+\bar{b}\bar{x}$, the second from $y'=\mu_1'+\mu_2'x'+...+\mu_n'(x')^{n-1}$. Comparing coefficients, we get a linear system in $\mu_i'$ which is already in Gauss triangular form, so we can successively solve for $\mu_2', \mu_3'$ and so on up to $\mu_n'$.
We get for example $$\mu_2'=\frac{a\mu_2+\bar{b}m_2}{a+\bar{b}\bar{\mu}_2m_2}$$ which is the same expression as for the case $n=2$.
Notice that $t_1'=0$ and that $\mu_1'$ is an explicit expression in the other variables, given by equation \eqref{mu1value}. 

Finally we have to prove the equivariance with respect to higher diffeomorphisms. We have to prove that the $\GL_2(\R)$-action commutes with the action of $\Symp_0(T^*\S)$. For this, it is sufficient to show that the $\GL_2(\R)$-action commutes with the infinitesimal action of higher diffeomorphisms which is described by Poisson brackets. The fact that both actions commute on variables $(x,\bar{y})$ is easy:
the $\GL_2(\R)$-action is linear and the infinitesimal action of a Hamiltonian $H$ is given by $x\mapsto x+\varepsilon \{H,x\}$ and $\bar{y}\mapsto \bar{y}+\varepsilon\{H,\bar{y}\}$. Since the Poisson bracket is bilinear, the two actions commute.

Since a generic point of $\Hilb^n_{red}(T^{*\C}\S)$ can be described by variables $(x, \bar{y})$ and since the commutativity is a closed condition, the actions commute on the Hilbert scheme bundle. Thus, the action of $\GL_2(\R)$ descends on $\cotang$.
\end{proof}


\cleardoublepage
\vspace*{3cm}
\part{Link to higher Teichmüller theory}\label{part2}
\vspace*{2cm}
\begin{flushright}
\textit{Quels que soient les jeux de mots et les acrobaties de la logique,}

\textit{comprendre c'est avant tout unifier. [...] }

\textit{Comprendre le monde pour un homme, c'est le réduire à l'humain,}

\textit{le marquer de son sceau. [...]}

\textit{Cette nostalgie d'unité, cet appétit d'absolu illustre}

\textit{le mouvement essentiel du drame humain.}


Albert Camus, Le mythe de Sisyphe
\end{flushright}
\vspace*{1cm}

\noindent In this part we tie a link to Hitchin components taking advantage of the hyperkähler structure of the Hilbert scheme. We deform the space $\cotang$, including it into a one-parameter family of Kähler manifolds.
The generic fiber of the deformation is a space of pairs of commuting differential operators to which we can associate a flat connection. We construct this space of flat connections by a double hamiltonian reduction.
To get a link to higher Teichmüller theory, we try to canonically associate a flat connection to a point of $\cotang$. Adding a reality constraint gives flat connections with real monodromy associated with the zero-section $\T^n\subset \cotang$.

To read this part, you need to know some basics of the punctual Hilbert scheme (subsections \ref{hilbdef} and \ref{hilbmatrixviewpoint}), the construction of the higher complex structure and its moduli space (section \ref{Highercomplexsection}) and the more detailed description of the cotangent bundle $\cotang$ (subsection \ref{cotangsection}).

\medskip
\noindent We first specify our strategy to get a link from higher complex structures to character varieties using a hyperkähler picture analogous to the moduli space of Higgs bundles. We compare our approach to Hitchin's approach.

In the matrix viewpoint of $\cotang$, we start with a $\mf{sl}_n$-valued 1-form $\Phi=\Phi_1+\Phi_2$ (decomposed into its $(1,0)$ and $(0,1)$-part) where $\Phi \wedge \Phi = 0$. Locally this allows to consider $\Phi_1$ and $\Phi_2$ as a pair of nilpotent commuting matrices. We deform this 1-form into some $h$-connection of the form $\Phi+hd+hA+h^2\Phi^*$. Dividing by $h$, we get a connection of the form $d+\l \Phi+A+\l^{-1}\Phi^*$ (where $\l=h^{-1}$).
We recover holomorphic differentials in flat connections of this form.

In the Hilbert scheme viewpoint, a point of $\cotang$ is essentially described by two polynomials. So the deformation space consists of a pair of commuting differential operators to which we can associate a flat connection. 
We describe the space of commuting differentials operators by a two-step hamiltonian reduction. Starting with the space of all connections $\mc{A}$, we consider a parabolic subgroup $\mc{P}$ of all gauge transformations, those which fix a given direction. The reduction $\mc{A}\sslash\mc{P}$ consists of pairs of differential operators which commute only under some extra condition. This condition can be obtained as the moment map of a second reduction. The group needed for this is the group of higher diffeomorphisms which act by gauge on the reduction $\mc{A}\sslash\mc{P}$.

We then perform the parabolic reduction on $h$-connections. This corresponds to including a parameter $\l$ into the differential operators. We show that when $\l$ tends to infinity, we get our space $\cotang$. For $\l \rightarrow 0$, we get the conjugated higher complex structure.

Finally, we try to prove an analog to the non-abelian Hodge correspondence: given a point in $\cotang$, show existence and uniqueness of a real twistor line passing through this point.
We are able to give some partial results and conjectures in this direction. Assuming the existence of real twistor lines we show that our moduli space $\T^n$ is diffeomorphic to Hitchin's component $\mc{T}^n$.

\medskip \noindent 
\textbf{Warning.} Whereas all propositions and theorems in this part do not rely on conjectures (whenever it is not explicitly stated otherwise), the interpretation of the results is conjectural, especially the interpretation in the twistor approach. We use this interpretation in analogy with the Higgs bundle moduli space. It helps to motivate and understand our methods and gives a big picture which makes it clear what has still to be done. Our interpretation relies on the conjecture that $\cotang$ is hyperkähler near the zero-section and that the space of flat parabolic $h$-connections is its twistor space.

\medskip \noindent
Most of the material is published in the preprint \cite{FockThomas2}. The idea of the parabolic reduction was introduced in \cite{BFK}.
\vspace*{\fill}

\cleardoublepage

\section{Introduction: two hyperk\"ahler pictures}\label{bigpicture}

In this section, we describe our approach and compare it to Hitchin's approach since they both use a hyperkähler manifold.

\medskip
\noindent Recall from the introduction \ref{introenglish} that in \cite{Hit.1}, Hitchin's approach to construct components in the character variety is to use the hyperkähler structure of the moduli space of Higgs bundles $\mc{M}_H$. The non-abelian Hodge correspondence allows to canonically associate a flat connection to a Higgs bundle $(V,\Phi)$. 

In subsection \ref{generalitiesHK} we have seen the twistor description of a hyperkähler manifold $M$. The twistor space $X_M$ allows to describe the whole 1-parameter family of complex structures at once.

To compare Hitchin's approach to ours, we give a picture of the twistor space of $\mc{M}_H$ in figure \ref{HK}. Denote by $\l$ the coordinate on $\C P^1$. On the left hand side, the twistor space is in one complex structure, say at $\l=\infty$, the moduli space of Higgs bundles $\mc{M}_{H}$. For $\l=0$, we see the conjugated complex structure. In all other $\l$, we see the complex structure of the character variety $\Rep(\pi_1(\S), G^{\C})$, which can be seen as hamiltonian reduction of the space of all connections $\mc{A}$ by all the gauge transformations $\mc{G}$ (Atiyah--Bott reduction for unitary gauge). 
Going from $\l=0$ to $\l=1$ is the non-abelian Hodge correspondence. Finally, there is the Hitchin fibration going from $\mc{M}_{H}$ to a space of differentials. This fibration admits a section whose monodromy, via the non-abelian Hodge correspondence, is in the split real form. For $G=\SL_n(\C)$, we get flat $\PSL_n(\R)$-connections.




\begin{figure}[h]
\centering
\begin{tikzpicture}[scale=1.5]
	\draw (0,0) circle (1cm);
\draw [domain=180:360] plot ({cos(\x)},{sin(\x)/3});
	\draw [domain=0:180, dotted] plot ({cos(\x)},{sin(\x)/3});
	\draw [fill=white] (0,1) circle (0.04);
	\draw [fill=white] (0,-1) circle (0.04);
	\draw (0,1.2) node {$\overline{\mc{M}}_H$};
	\draw[below] (0,1) node {$0$};
	\draw (0,-1.2) node {$\mc{M}_H$};
	\draw[above] (0,-1) node {$\infty$};
	\draw (0,-1.5) node {$\downarrow$};
	\draw (0,-1.9) node {$\bigoplus_{i=2}^n H^0(K^i)$};
	\draw [domain=-30:32, dashed, ->] plot ({0.33*cos(\x)}, {0.33*sin(\x)-1.5});
	
	\draw (1,-1.4) node {Hitchin};
	\draw (1,-1.62) node {section};
	\draw (-1.25,-1.4) node {Hitchin};
	\draw (-1.25,-1.65) node {fibration};
	\draw (1.85,0.15) node {$\Rep(\pi_1\S, G^{\C})$};
	\draw (1.65,-0.18) node {$\cong \mc{A}\sslash \mc{G}$};
	
	\begin{scope}[xshift=4.6cm]
	\draw (0,0) circle (1cm);
	\draw [domain=180:360] plot ({cos(\x)},{sin(\x)/3});
	\draw [domain=0:180, dotted] plot ({cos(\x)},{sin(\x)/3});
	\draw [fill=white] (0,1) circle (0.04);
	\draw [fill=white] (0,-1) circle (0.04);
	\draw (0,1.2) node {$\overline{T^*\mc{T}^n}$};
	\draw[below] (0,1) node {$0$};
	\draw (0,-1.2) node {$\cotang$};
	\draw[above] (0,-1) node {$\infty$};
	\draw (0,-1.5) node {$\downarrow$};
	\draw (0,-1.8) node {$\T^n$};
	\draw [domain=-40:24, dashed, ->] plot ({0.33*cos(\x)}, {0.33*sin(\x)-1.5});
	
	\draw (1,-1.4) node {zero-};
	\draw (1,-1.6) node {section};
	\draw (-1.25,-1.4) node {canonical};
	\draw (-1.25,-1.65) node {projection};
	\draw (2,0.15) node {$\Rep(\pi_1\S, G^{\C}) \cong$};
	\draw (2,-0.18) node {$(\mc{A}\sslash\mc{P})\sslash\Symp_0$};
	\end{scope}
	\end{tikzpicture}
	\caption{Twistor space for Higgs bundles and $\cotang$}\label{HK}
\end{figure}

In our approach, the role of $\mc{M}_H$ is played by the cotangent space $\cotang$ which is conjecturally hyperkähler near the zero-section (see discussion around Conjecture \ref{hkcotang}).
An element of $\cotang$ is characterized by an equivalence class of higher Beltrami differentials $\mu_k$ and holomorphic differentials $t_k$.
We have seen in \ref{indbundle} that there is a bundle associated with a point in $\cotang$. The higher complex structure given by the $\mu_k$ gives a 1-form locally of the form $\Phi=\Phi_1dz+\Phi_2d\bar{z}$ where $(\Phi_1, \Phi_2)$ is a pair of commuting nilpotent matrices. 
Hitchin chooses for the Higgs field a point in a principal slice of the Lie algebra, which is a variety of dimension $\rk \g$. In our setting, we choose a principal nilpotent $\Phi_1$, and $\Phi_2$ in the centralizer of $\Phi_1$ which is also of dimension $\rk \g$.
The Hitchin fibration simply becomes the projection map $\cotang \rightarrow \T^n$ and the Hitchin section is the zero-section $\T^n\subset \cotang$.

In the twistor space, in complex structure at $\l=\infty$, we see the cotangent bundle $\cotang$. At the opposite point $\l=0$ we see the conjugated structure $\overline{T^*\mc{T}^n}$ characterized by $[({}_k\mu,{}_kt)]$ (see section \ref{dualcomplexstructure}). We will show that in all other complex structures, we see the character variety $\Rep(\pi_1(\S), G^{\C})$, this time seen as the double reduction of the space of connections  $\mc{A}$ by parabolic gauge $\mc{P}$ and then by higher diffeomorphisms $\Symp_0$.
An analog of the non-abelian Hodge correspondence would allow to canonically associate a flat connection to a point in $\cotang$.

Finally, there are holomorphic differentials $t_k$ in both approaches, but they appear differently. In Hitchin's approach, we start from a fixed Riemann surface $\S$ and then consider holomorphic differentials $t_k$ which determine the Higgs field. In our approach, we start from a smooth surface $\S$ and then equip it with a higher complex structure (which varies). The differentials appear in the connection $A$ of the deformation $\l \Phi+A+\l^{-1}\Phi^*$ and are holomorphic with respect to the higher complex structure. Only for trivial higher complex structures we recover the usual notion of holomorphicity. 
In the next section \ref{holodiffs-flat-affine}, the appearance of holomorphic differentials is explained with more detail.

\medskip \noindent
We stress that the twistor picture for $\cotang$ is still conjectural. To prove its validity, we have to apply theorem 3.3 from \cite{HKLR}, giving the twistor approach to HK manifolds. We need a complex manifold $Z$ of dimension $2m+1$ (where $m=\dim \T^n$) such that	
\begin{enumerate}
\item $Z$ is a holomorphic fiber bundle over $\C P^1$,
\item The bundle admits a family of holomorphic sections with normal bundle isomorphic to $\C^{2m}\otimes \mc{O}(1)$,
\item There is a holomorphic section $\omega$ of $\Lambda^2T^*F\otimes \mc{O}(2)$ which gives a symplectic form on each fiber $F$,
\item $Z$ admits a real structure compatible with the other structures and inducing the antipodal map on $\C P^1$.
\end{enumerate}

Our candidate for $Z$ is the space of flat parabolic $h$-connections. The map to $\C P^1$ is given by $h$. The quadratic symplectic structure (quadratic since twisted by $\mc{O}(2)$) is described in \ref{sympstrconnections}. The real structure is described at the beginning of section \ref{finalstep}, but is still problematic since we only know it for trivial $n$-complex structure. The existence of twistor lines is also discussed in \ref{finalstep}, but only partial results are obtained.

\section{Holomorphic differentials in flat affine connections}\label{holodiffs-flat-affine}

Given a higher complex structure in the matrix viewpoint, which can be locally written as $\Phi(z)=\Phi_1(z)dz+\Phi_2(z) d\bar{z}$ with $\Phi_2(z) =\mu_2\Phi_1+...+\mu_n\Phi_1^{n-1}$, suppose we have a flat connection of the form 
$$\mc{A}(\l)=\l \Phi + A+ \l^{-1}\Phi^*$$
where $A^*=-A$ and the $*$-operator is the hermitian conjugate $A^{\dagger}$. Decompose $A$ into its $(1,0)$- and $(0,1)$-part: $A=A_1 + A_2$.

One can consider $\mc{A}(\l)$ as a connection $\mc{A}$ with values in the affine Lie algebra $\widehat{\mf{sl}}_n$.
We want to know which data uniquely determines such a flat connection.  In this section, we first describe a standard form for $\mc{A}(\l)$. Then we show that in a quite general setting, we can explicitly extract holomorphic differentials out of a flat connection of this kind. These differentials should belong to the data we need to determine $\mc{A}(\l)$.


\subsection{Standard form}\label{standard-form}

We start with $\mc{A}(\l)=\l\Phi + A + \l^{-1}\Phi^*$ as above and reduce it to a standard form.

\begin{lemma}\label{phi1lower}
There is a unitary gauge such that $\Phi_1$ becomes lower triangular with entries of coordinates $(i+1,i)$ given by positive real numbers of the form $e^{\varphi_i}$ for all $i=1,...,n-1$.
\end{lemma}
\begin{proof}
The gauge acts by conjugation on $\Phi_1(z)$. Since $\Phi_1(z)$ is nilpotent, for every $z\in \S$, there is an invertible matrix $G(z)\in \GL_n(\C)$ such that $G\Phi_1G^{-1}$ is strictly lower triangular. Since $\Phi(z)$ varies smoothly with $z$, so does $G(z)$. We omit the dependence in $z$ in the sequel of the proof.

We decompose $G$ as $G=TU$ where $T$ is lower triangular (not strict) and $U$ is unitary (Gram-Schmidt). Then the matrix $U\Phi_1U^{-1} = T^{-1}(G\Phi_1G^{-1})T$ is already lower triangular. So we have conjugated $\Phi_1$ to a lower triangular matrix via a unitary gauge.

Finally, we use a diagonal unitary gauge to change the arguments of the matrix elements with coordinates $(i+1,i)$ to zero. Since $\Phi_1$ is principal nilpotent, all these elements are non-zero, so strictly positive real numbers which can be written as $e^{\varphi_i}$ with $\varphi_i \in \R$.
\end{proof}

Notice that the unitary gauge preserves the operation $*$, so the form $\l\Phi+A+\l^{-1}\Phi^*$ is preserved. 
Now, we show that for $\mu=0$, the matrix $A_1$ is upper triangular. Notice the importance of $\Phi_1$ being principal nilpotent.

\begin{lemma}\label{a1upper}
For $\Phi_2=0$ (trivial higher complex structure) and $\Phi_1$ lower triangular, the flatness of $\mc{A}(\l)$ implies that $A_1$ is upper triangular.
\end{lemma}
\begin{proof}
We write $A_1=A_l + A_u$ where $A_l$ and $A_u$ are respectively the strictly lower and the (not strictly) upper part of $A_1$. Thus we have $A_2 = -A_l^{*}-A_u^{*}$.

The flatness condition at the term $\l$ gives $$0=\bar{\partial}\Phi_1 + [\Phi_1,A_u^{*}]+[\Phi_1,A_l^{*}].$$
Since the first two terms are lower triangular (the operation $*$ exchanges upper and lower triangular matrices), so is the third term $[\Phi_1,A_l^{*}]$.

A simple computation shows that a commutator between a principal nilpotent lower triangular matrix and a non-zero strictly upper triangular matrix can never be strictly lower triangular. Thus, $A_l=0$.
\end{proof}

\subsection{Holomorphic differentials}
We show that from a flat connection of the form $\mc{A}(\l)$ above we can extract holomorphic differentials.

\begin{example}\label{n3example}
Consider the case with $\mf{sl}_3$. Suppose further $\Phi_2 = 0$, so we can put the connection into the standard form:
$$\mc{A}(\l)=\mc{A}_1(\l)+\mc{A}_2(\l)=\begin{pmatrix} a_0 & b_0 & c_0 \\ \l c_1 & a_1 & b_1 \\ \l b_2 & \l c_2 & a_2 \end{pmatrix}dz +\begin{pmatrix} -\bar{a}_0 & \l^{-1}\bar{c}_1 & \l^{-1}\bar{b}_2 \\ -\bar{b}_0 & -\bar{a}_1 & \l^{-1}\bar{c}_1 \\ -\bar{c}_0 & -\bar{b}_1 & -\bar{a}_2\end{pmatrix}d\bar{z}$$
where $\mc{A}_1(\l)$ and $\mc{A}_2(\l)$ are the $(1,0)$ and $(0,1)$-part respectively.

Put $T=\left(\begin{smallmatrix}  & 1 & \\  & &1 \\ \l &  &\end{smallmatrix}\right)$, then we can write $\mc{A}_1=a+bT+cT^2$ where $a, b$ and $c$ are diagonal matrices with entries $a_0, a_1, a_2$ for $a$ and the same for $b$ and $c$. The second matrix can be written\footnote{There is a sign problem, but this does not matter. In fact we can treat $\bar{a}, \bar{b}$ and $\bar{c}$ as independent variables, only the form of $\mc{A}(\l)$ is important.} $\mc{A}_2 = \bar{a}+T^{-1}\bar{b}+T^{-2}\bar{c}$. We denote by $a'$ the diagonal matrix $a$ shifted by 1, i.e. $(a')_i = a_{i+1}$. Notice the relation $aT = Ta''.$

The flatness condition gives
\begin{align*}
0=\del \mc{A}_2-\delbar\mc{A}_1+[\mc{A}_1,\mc{A}_2] &= T^{-2}(\del \bar{c}+\bar{c}(a''-a))+T^{-1}(\del \bar{b}+\bar{b}(a'-a)+b'\bar{c}-b''\bar{c}'') \\
&+ (\del \bar{a} - \delbar a+b\bar{b}-b''\bar{b}''+c\bar{c}-c'\bar{c}') \\
&+ (-\delbar b +b(\bar{a}'-\bar{a})+c\bar{b}'-c''\bar{b}'')T+(-\delbar c+c(\bar{a}''-\bar{a}))T^2. 
\end{align*}
Define $t_3=cc'c''$ and $t_2=bc'+b'c''+b''c$. We directly get from the flatness equations that
$$\delbar t_3 = cc'c''(\bar{a}''-\bar{a})+cc'c''(\bar{a}-\bar{a}')+cc'c''(\bar{a}'-\bar{a}'') = 0$$
and since $\delbar(bc') = cc'\bar{b}'-c'c''\bar{b}''$, we also get $$\delbar t_2 = 0.$$
We notice that $t_3=\tr \Phi_1^2 A_1$ and $t_2=\tr \Phi_1A_1$.
\hfill $\triangle$
\end{example}

This phenomenon generalizes to $\mf{sl}_n$. Suppose $\Phi_2=0$ (i.e. trivial higher complex structure). Then by the standard form from Lemma \ref{phi1lower} and \ref{a1upper} we can assume $\Phi_1$ strictly lower diagonal and $A_1$ upper diagonal. 

\begin{prop}\label{holodifff}
Let $\mc{A}(\l) = \l\Phi_1+A+\l^{-1}\Phi_1^*$ be a flat affine connection with $\Phi_2=0$ and $\Phi_1$ strictly lower triangular. Then $t_k := \tr \Phi_1^{k-1}A_1$ is a holomorphic $k$-differential.
\end{prop}
\begin{proof}
The flatness of $\mc{A}(\l)$ gives $\delbar\Phi_1 = [\Phi_1, A_2]$ and $\del A_2-\delbar A_1 + [A_1,A_2]+[\Phi_1,\Phi_1^*]=0$. We deduce $\delbar (\Phi_1^{k-1})=[\Phi_1^{k-1}, A_2]$. We then compute:
\begin{align*}
\delbar t_k &= \delbar \tr\Phi_1^{k-1}A_1 \\
&= \tr ([\Phi_1^{k-1}, A_2]A_1+ \Phi_1^{k-1}\delbar A_1) \\
&= \tr \Phi_1^{k-1}([A_2,A_1]+\delbar A_1) \\
&= \tr \Phi_1^{k-1}(\del A_2+[\Phi_1, \Phi_1^*]) \\
&= \tr [\Phi_1^{k-1}, \Phi_1]\Phi_1^* + \tr \Phi_1^{k-1}\del A_2 \\
&= 0.
\end{align*}
The last line comes from the fact that $\Phi_1^{k-1}$ is strictly lower triangular and $A_2=-A_1^*$ is lower triangular, so their pairing is zero.
\end{proof}
\begin{Remark}
In the proof, we see that the proposition stays true for any operation $*$ which is linear, involutive and exchanges upper and lower triangular matrices.
\hfill $\triangle$
\end{Remark}
\begin{Remark}
The previous proposition and its proof are quite similar to Hitchin's setting: there the holomorphic differentials are in the Higgs field $\Phi$ and we can get them via $$t_k=\tr \Phi^k.$$ We can then directly check their holomorphicity using the flatness:
$$\delbar t_k = \delbar \tr\Phi^k = \tr [\Phi^{k-1},A_2] = 0.$$
\hfill $\triangle$
\end{Remark}

For both, our setting and Hitchin's setting, the holomorphic differentials allow another description in terms of a characteristic polynomial. In Hitchin's case, the coefficients of the characteristic polynomial of $\Phi$ are holomorphic differentials (not the same as $\tr \Phi^k$ but closely linked).

In our setting, we will now see that $\tr \Phi_1^{k-1}A_1$ is the highest $\l$-term of the coefficient of $X^{n-k}$ in the characteristic polynomial $\chi_{\mc{A}_1(\l)}(X)=\det (\l\Phi_1+A_1-X\id).$

Consider for example $k=2$. The coefficient of $X^{n-2}$ is given by $$\tr \Lambda^2(\l\Phi_1+A_1) = \l^2 \tr \Lambda^2\Phi_1 + \l \tr T + \tr \Lambda^2 A_1$$ where $\Lambda^2$ is the second exterior product and $T$ is a term mixing both $A_1$ and $\Phi_1$.
 Since $\Phi_1$ is nilpotent, the $\l^2$-term vanishes. Using Dirac's ``bra-ket'' notation (see page \pageref{notations}), the mixed term $T$ is given by
$$\sum_{i<j} \langle e_i\wedge e_j \mid \Phi_1e_i \wedge A_1e_j \rangle$$ where $(e_i)$ is the standard basis of $\C^n$.
Since $\Phi_1$ is strictly lower triangular it can only strictly increase the index of $e_i$, and since $A_1$ is upper triangular, it can only lower the indices. Thus, $\langle e_i\wedge e_j \mid \Phi_1e_i \wedge A_1e_j \rangle = -\Phi_{i,j}A_{j,i}$ (where $\Phi_{i,j}$ and $A_{i,j}$ denote the entries of $\Phi_1$ and $A_1$) and therefore the highest $\l$-term of the coefficient of $X^{n-2}$ is given by $-\l\tr \Phi_1 A_1$.

The same reasoning holds for any $k$: the coefficient of $X^{n-k}$ is given by $$\tr \Lambda^k(\l\Phi_1+A_1) = \l^k \tr \Lambda^k\Phi_1 + \l^{k-1} \tr \text{ mixed term} + \text{ lower terms}.$$ The $\l^k$-term vanishes since $\Phi_1$ is nilpotent. The mixed term $T$ is given by 
$$T=\sum_{i_1<\cdots <i_k} \langle e_{i_1}\wedge ...\wedge e_{i_k} \mid \Phi_1e_{i_1}\wedge ...\wedge \Phi_1e_{i_{k-1}}\wedge A_1e_{i_k}\rangle.$$
Since $\Phi_1$ strictly increases the index of $e_i$ and $A_1$ lowers it, we get 
$$T= \sum_{i_1<\cdots <i_k} (-1)^{k-1}\Phi_{i_1,i_2}\cdots\Phi_{i_{k-1},i_{k}}A_{i_k,i_1} = (-1)^{k-1} \tr \Phi_1^{k-1}A_1.$$
Using Proposition \ref{holodifff} we have proven:
\begin{prop}
The highest $\l$-terms in the coefficients of the characteristic polynomial of $\mc{A}_1(\l)=\l\Phi_1+A_1$ are holomorphic differentials.
\end{prop}
\begin{Remark}
The description of the holomorphic differentials with the characteristic polynomial is more general than the direct formula $\tr\Phi_1^{k-1}A_1$ since it can be generalized to Lie algebras other than $\mf{sl}_n$.
\hfill $\triangle$
\end{Remark}

Finally, we will see in Proposition \ref{parametrisationlambda}, that the formula $t_k=\tr \Phi_1^{k-1}A_1$ stays true even if $\Phi_2 \neq 0$, but the holomorphicity condition for $t_k$ becomes the condition $(\mc{C})$ from the cotangent bundle $T^*\bm\hat{\mc{T}}^n$ (see Theorem \ref{conditioncinconnection}).

\section{Parabolic connections and reduction}\label{parabolicreduction}

In this section, we describe the generic fiber of the twistor space of $\cotang$ from figure \ref{HK} which is a space of flat connections. 
The deformation of the Higgs moduli space turns out to be $\Rep(\pi_1(\S), G^{\C})$, which can also be written as a hamiltonian reduction of all connections by all gauge (the famous Atiyah--Bott reduction). In our setting, there is another reduction, in two steps, which also gives flat connections. 

The idea about the deformation of $\cotang$ is to replace the polynomial functions on $T^*\S$ by differential operators. The higher complex structure is given by two polynomials (the generators of $I$), so in the deformation one gets a pair of differential operators.


The space of pairs of differential operators can be obtained by a reduction of all connections by some specific parabolic gauge. This procedure was first introduced by Bilal, Fock and Kogan in \cite{BFK}. In that paper, the authors also describe some ideas for generalized complex and projective structures. Our higher complex structures are the mathematically rigorous version of their ideas. Our treatment of the parabolic reduction is independent of their paper and follows some other notation. The question about how to impose the commutativity condition on the differential operators remained open in their paper. We show that the answer is given by a second reduction with respect to the group of higher diffeomorphisms.

\subsection{Atiyah--Bott reduction}
Before going to the parabolic reduction, we recall the classical reduction of connections by gauge transforms, developed by Atiyah and Bott in their famous paper \cite{AtBott}.

Let $\Sigma$ be a surface and $G$ be a semisimple Lie group with Lie algebra $\g$. Let $E$ be a trivial $G$-bundle over $\Sigma$. Denote by $\mathcal{A}$ the space of all $\g$-connections on $E$. It is an affine space modeled over the vector space of $\g$-valued 1-forms $\Omega^1(\Sigma, \g)$. Further, denote by $\mathcal{G}$ the space of all gauge transforms, i.e. bundle automorphisms. We can identify the gauge group with $G$-valued functions: $\mathcal{G}=\Omega^0(\Sigma,G)$.

On the space of all connections $\mathcal{A}$, there is a natural symplectic structure given by $$\bm\hat{\omega} = \int_{\Sigma} \tr \;\delta A \wedge \delta A$$ where tr
denotes the Killing form on $\g$ (the trace for matrix Lie algebras).
Since $\mathcal{A}$ is an affine space, its tangent space at every point is canonically isomorphic to $\Omega^1(\Sigma, \g)$. So given $A \in \mathcal{A}$ and $A_1, A_2 \in T_A\mathcal{A} \cong \Omega^1(\Sigma, \g)$, we have $\bm\hat{\omega}_A(A_1, A_2) = \int_{\Sigma} \tr \; A_1\wedge A_2$. Note that $\bm\hat{\omega}$ is constant (independent of $A$) so $d\bm\hat{\omega} = 0$. Further, the 2-form $\bm\hat{\omega}$ is clearly antisymmetric and non-degenerate (since the Killing form is). Remark finally that this construction only works on a surface.

Now we can state the famous and surprising theorem due to Atiyah and Bott (see end of chapter 9 in \cite{AtBott} for unitary case, see section 1.8 in Goldman's paper \cite{Goldman.2} for the general case):
\begin{thm}[Atiyah, Bott 1983]
The action of gauge transforms on the space of connections is a Hamiltonian action and the moment map is the curvature. Thus, the Hamiltonian reduction $\mathcal{A}\sslash\mathcal{G}$ gives the moduli space of flat connections.
\end{thm}

Let us explain the moment map with more detail: the moment map $m$ is a map from $\mathcal{A}$ to $\Lie(\mathcal{G})^*$. The Lie algebra $\Lie(\mathcal{G})$ is equal to $\Omega^0(\Sigma, \g)$, so its dual is isomorphic to $\Omega^2(\Sigma, \g)$ via the pairing $\int_{\Sigma} \tr$. 
On the other hand, given a connection $A$, its curvature $F(A)$ is a $\g$-valued 2-form, i.e. an element of $\Omega^2(\Sigma,\g)$. Hence, the map $m$ is well-defined.

\begin{proof}[Idea of proof]
The action of a gauge transform $g$ on a connection $A$ is given by $g.A = gAg^{-1}+gdg^{-1}$ coming from expanding $g(d+A)g^{-1}$.
So the action on a tangent vector $\delta A$ is given by $g.\delta A = g\delta A g^{-1}$. So we have $g^*\bm\hat{\omega} = \int \tr g\delta A\wedge \delta A g^{-1} = \bm\hat{\omega}$. Thus, we see that $\bm\hat{\omega}$ is invariant under the action, i.e. the action is symplectic.

Let us compute the infinitesimal action by an element $g=1+\varepsilon$. We get $(1+\varepsilon).A = A +[\varepsilon,A]-d\varepsilon$. So we have a vector field $A_{\varepsilon}=[\varepsilon,A]-d\varepsilon$ on $\mathcal{A}$. Denote by $\iota$ the interior product on differential forms. Now we compute
\begin{align*}
\iota_{A_{\varepsilon}}\bm\hat{\omega} &= \int_{\Sigma} \tr \; A_{\varepsilon} \wedge \delta A = \int_{\Sigma} \tr \;([\varepsilon,A]-d\varepsilon) \wedge \delta A \\
&= \int_{\Sigma}\tr [\delta A,A]\varepsilon +\int_{\Sigma} \tr \; \varepsilon \; d\delta A \\
&= \int_{\Sigma} \tr \left(\varepsilon \; \delta(dA+A\wedge A)\right)
\end{align*}
where we used an integration by parts and the cyclicity property of the Killing form. The moment map appears in the last line: $m(A) = dA+A\wedge A = F(A)$ the curvature.
\end{proof}

\subsection{Parabolic reduction}

\subsubsection{Setting and coordinates}\label{settingparab}

In subsection \ref{indbundle} we have seen how to associate a rank $n$-bundle $V$ over $\S$ to a higher complex structure. Moreover we have seen that there is a privileged direction in each fiber, the common kernel of $\Phi_1$ and $\Phi_2$. This gives a line-subbundle $L$ in $V$. We want to mimic the Atiyah--Bott reduction with the extra constraint of fixing $L$. That is why we consider the subspace of gauge transformations fixing the subbundle $L$ (more precisely its dual). 

Let us take the same setting as for the Atiyah--Bott reduction with $G=\SL_n(\mathbb{C})$. But instead of all gauge transforms $\mathcal{G}$, we consider the subgroup $\mathcal{P}\subset \mathcal{G}$ consisting of matrices of the form 
$$\begin{pmatrix}
* & \cdots & * & * \\
\vdots & & \vdots & \vdots \\
* & \cdots & *& * \\
0 & \cdots& 0 & * \\
\end{pmatrix}$$
i.e. preserving the last direction in the dual space. We want to compute and analyze the hamiltonian reduction $\mathcal{A}\sslash\mathcal{P}$, which we call \textbf{space of parabolic connections}.

\begin{Remark}
The reason to consider a fixed direction in the dual bundle and not in the bundle itself is purely of technical advantage.
\hfill $\triangle$
\end{Remark}

Since $\mathcal{P}\subset \mathcal{G}$, we know by the Atiyah--Bott theorem that the action of $\mathcal{P}$ on the space of connections $\mathcal{A}$ is hamiltonian with moment map $m: A\mapsto i^*F(A)$ where $i: \mathcal{P} \hookrightarrow \mathcal{G}$ is the inclusion and $i^*: \Lie(\mathcal{G})^* \twoheadrightarrow \Lie(\mathcal{P})^*$ the induced surjection on the dual Lie algebras. Since $G=\SL_n(\mathbb{C})$, the map $i^*$ is explicitly given by forgetting the first $n-1$ entries in the last column. This means that $m^{-1}(\{0\})$ is the space of all $A \in \mathcal{A}$ such that the curvature $F(A)$ is concentrated in the last column: 
$$\begin{pmatrix}
 0& \cdots & 0 & \xi_n\\
 \vdots &  & \vdots &\vdots\\
 \vdots &  & \vdots & \xi_2\\
 0& \cdots &0 & 0 \\
\end{pmatrix}.$$

In order to give a description in coordinates of the hamiltonian reduction $\mathcal{A}\sslash\mathcal{P}$, we fix a reference complex structure on the surface $\Sigma$. We take a connection $A\in \mathcal{A}$ and decompose it into its holomorphic and anti-holomorphic parts: $A=A_1 + A_2$. As a covariant derivative, we set $\nabla = \partial + A_1$ and $\bar{\nabla} = \bar{\partial}+ A_2$.
Using the parabolic gauge, it is possible to reduce $A_1$ locally to a companion matrix:
\begin{equation}\label{firstmatrix}
A_1 \sim
\begin{pmatrix} 
&  &  & \bm\hat{t}_n\\
1 &   &  &\vdots \\
& \ddots & &\bm\hat{t}_2 \\
 &  & 1 & 0 \\
\end{pmatrix}dz.
\end{equation}
The existence of such a gauge is proven in subsection \ref{existence-para-gauge}. Reducing $A_1$ to the above form means that we choose a basis of the form $B= (s, \nabla s, \nabla^2 s,..., \nabla^{n-1}s)$. This takes all the gauge freedom. A connection in $\mathcal{A}\sslash\mathcal{P}$ verifies $[\nabla, \bar{\nabla}]\nabla^i s = 0$ for $i=0,1,...,n-2$ since $[\nabla, \bar{\nabla}] = F(A)$ is the curvature which is concentrated on the last column. It follows that $\bar{\nabla}\nabla^i s = \nabla^i \bar{\nabla} s$ for all $i=1,...,n-1$. Thus, the connection is fully described by $\nabla^n s$ and $\bar{\nabla}s$. We can write these expressions in the basis $B$:
\begin{equation}\label{eqqq1}
\nabla^n s = \bm\hat{t}_{n}s + \bm\hat{t}_{n-1} \nabla s+...+\bm\hat{t}_{2}\nabla^{n-2}s 
\end{equation}
\begin{equation}\label{eqqq2}
\bar{\nabla}s = \bm\hat{\mu}_1 s + \bm\hat{\mu}_2 \nabla s + ... + \bm\hat{\mu}_n \nabla^{n-1}s.
\end{equation}
Notice that $\bm\hat{t}_1=0$ since $\tr A_1=0$.

The second part $A_2$ is uniquely determined by its first column given by equation \eqref{eqqq2}. Since $\bar{\nabla}\nabla^i s = \nabla^i \bar{\nabla} s$ for $i=1,...,n-1$, the $i$-th column of $A_2$ is given by applying $(i-1)$ times $\nabla$ to the first column.
We get a 1-form of the following type:
\begin{equation}\label{secondmatrix}
A_2 \sim
\begin{pmatrix} 
\bm\hat{\mu}_1& \partial \bm\hat{\mu}_1+\bm\hat{\mu}_n \bm\hat{t}_n & \cdots\\
\bm\hat{\mu}_2 &  \bm\hat{\mu}_1+\partial \bm\hat{\mu}_2+\bm\hat{\mu}_n \bm\hat{t}_{n-1} & \cdots \\
\vdots & \vdots & \vdots \\
\bm\hat{\mu}_{n-1} & \bm\hat{\mu}_{n-2}+\partial \bm\hat{\mu}_{n-1}+\bm\hat{\mu}_n \bm\hat{t}_2 & \cdots \\
\bm\hat{\mu}_n & \bm\hat{\mu}_{n-1} +\partial \bm\hat{\mu}_n & \cdots 
\end{pmatrix}d\bar{z}.
\end{equation}

\begin{Remark} Notice that modulo $\partial$ (meaning that you drop all terms with a partial derivative), equations \eqref{eqqq1} and \eqref{eqqq2} become the relations of $p^n$ and $\bar{p}$ in a generic ideal of $\Hilb^n_{red}(\C^2)$. So $A_1$ and $A_2$ become the multiplication operators by $p$ and $\bar{p}$ respectively.
\hfill $\triangle$
\end{Remark}

The functions $(\bm\hat{\mu}_2, ..., \bm\hat{\mu}_n, \bm\hat{t}_{2}, ..., \bm\hat{t}_{n})$ completely parameterize $\mathcal{A}\sslash\mathcal{P}$ since it is possible to express $\bm\hat{\mu}_{1}$ in terms of these using that the second matrix is traceless. 
We call an element of $\mathcal{A}\sslash\mathcal{P}$ a \textbf{parabolic connection}. 
We consider $\mc{A}\sslash\mc{P}$ as a subspace of $\mc{A}$ by using the representative $A_1 + A_2$ with $A_1$ of the local form \ref{firstmatrix} and $A_2$ like in \ref{secondmatrix}.
Its parabolic curvature is concentrated on the last column: $[\nabla, \bar{\nabla}]\nabla^{n-1}s = \xi_n s + \xi_{n-1} \nabla s + ...+ \xi_2 \nabla^{n-2}s$. The following proposition allows to compute the parabolic curvature easily.

\begin{prop}[Parabolic curvature]\label{thmcourbure}
$[\nabla^n,\bar{\nabla}]s = \sum_{k=2}^n \xi_k\nabla^{n-k}s.$
\end{prop}
\begin{proof}
Since the first $n-1$ columns of the curvature $F(A)$ are 0, we have $[\nabla, \bar{\nabla}]\nabla^is = 0$ for $i=0,1,...,n-2$. Using Leibniz's rule and induction on $k$, we can prove that $[\nabla^k, \bar{\nabla}]s = 0$ for $k=1,...,n-1$. Indeed, it is true for $k=1$ and we have $[\nabla^{k+1}, \bar{\nabla}]s = \nabla [\nabla^k,\bar{\nabla}]s + [\nabla,\bar{\nabla}]\nabla^k s=0$ whenever $k\leq n-2$.

\noindent Therefore, we get $$[\nabla^n, \bar{\nabla}]s = \nabla[\nabla^{n-1}, \bar{\nabla}]s + [\nabla, \bar{\nabla}]\nabla^{n-1}s = [\nabla, \bar{\nabla}]\nabla^{n-1}s = \sum_{k=2}^n \xi_k\nabla^{n-k}s$$ 
by the last column of the curvature.
\end{proof}

Inside the non-commutative ring of differential operators, we define the left-ideal $\bm\hat{I}=\langle \nabla^n-\bm\hat{P}, -\bar{\nabla}+\bm\hat{Q} \rangle$ where $\bm\hat{P}$ and $\bm\hat{Q}$ are defined in equations \eqref{eqqq1} and \eqref{eqqq2} respectively. We can express the previous proposition as $$[\nabla^n,\bar{\nabla}]=\sum_{k=2}^n \xi_k\nabla^{n-k} \mod \bm\hat{I}.$$

Notice finally that our coordinates $(\bm\hat{\mu}_2, ..., \bm\hat{\mu}_n, \bm\hat{t}_{2}, ..., \bm\hat{t}_{n})$ do not behave like tensors under coordinate change $z\mapsto w(z)$.
We will see in the following section \ref{parabolicwithlambda} that if we introduce a parameter $\l$ we get at the semiclassical limit tensors out of our coordinates.

\subsubsection{Example \texorpdfstring{$n=2$}{n=2}}\label{casen2}

Consider a parabolic $\SL(2, \C)$-connection locally written as $A=A_1dz+A_2d\bar{z}$. 
The first matrix $A_1$ is a companion matrix of the form $\left( \begin{smallmatrix} 0 & \bm\hat{t}_2 \\ 1 & 0 \end{smallmatrix} \right)$.

Let us compute the transformed matrix $A_2$. It is the image of  the operator $\bar{\nabla}$ in a basis $(s,\nabla s)$. Put $\bar{\nabla}s = \bm\hat{\mu}_1 s+ \bm\hat{\mu}_2 \nabla s.$
The second column can be computed using $\bar{\nabla}\nabla s = \nabla \bar{\nabla}s - [\nabla,\bar{\nabla}]s = \nabla \bar{\nabla}s$ and $\nabla^2 s = \bm\hat{t}_2s$.
Since the trace of the matrix is zero, we get $\bm\hat{\mu}_1 = -\frac{1}{2}\partial \bm\hat{\mu}_2$. Hence
$$A_2 = \left( \begin{array}{cc}
	-\frac{1}{2}\partial \bm\hat{\mu}_2 & -\frac{1}{2}\partial^2 \bm\hat{\mu}_2+\bm\hat{t}_2\bm\hat{\mu}_2 \\
	\bm\hat{\mu}_2 & \frac{1}{2}\partial \bm\hat{\mu}_2 
\end{array} \right).$$
The curvature is of the form $\left( \begin{smallmatrix} 0 & \xi_2 \\ 0 & 0 \end{smallmatrix} \right)$
where
$$\xi_2 = (\bar{\partial}-\bm\hat{\mu}_2\partial-2\partial\bm\hat{\mu}_2)\bm\hat{t}_2+\frac{1}{2}\partial^3\bm\hat{\mu}_2.$$
Suppose that the curvature $\xi_2$ is 0. We can then look for flat sections $\Psi = (\psi_1,\psi_2)$. The first condition $(\partial+A_1)\Psi=0$ gives $\psi_1=-\partial \psi_2$ and $$(\partial^2-\bm\hat{t}_2)\psi_2 = 0.$$
The second condition $(\bar{\partial}+A_2)\Psi= 0$ only gives one extra condition: $$(\bar{\partial}-\bm\hat{\mu}_2\partial+\frac{1}{2}\partial\bm\hat{\mu}_2)\psi_2 = 0.$$
For $\bm\hat{\mu}_2=0$ this just means that $\psi_2$ is holomorphic and we get an ordinary differential equation $(\partial^2-\bm\hat{t}_2)\psi_2 = 0$. For $\bm\hat{\mu}_2\neq 0$, the second condition is still a holomorphicity condition, but with respect to another complex structure. 

For general $n$, a flat section $\Psi=(\psi_k)_{1\leq k \leq n}$ is of the form $\psi_{n-k}=\del^k\psi_n$ and there are two equations on $\psi_n$. The first equation comes from the last column in $A_1$, so directly generalizes to $(\partial^n-\bm\hat{t}_1\partial^{n-1}-...-\bm\hat{t}_n)\psi_n = 0$. The generalized holomorphicity condition comes from the last row in $A_2$:
\begin{equation}\label{diffops}
(-\bar{\partial}+\bm\hat{\alpha}_{nn}+\bm\hat{\alpha}_{n,n-1}\partial+...+\bm\hat{\alpha}_{n,1}\partial^{n-1})\psi_n=0
\end{equation}
where $\bm\hat{\alpha}_{ij}$ denote the entries of $A_2$ which have an explicit but complicated expression in terms of the $\bm\hat{\mu}_k$ and $\bm\hat{t}_k$.

\subsubsection{Existence of parabolic gauge*}\label{existence-para-gauge}

In this rather technical subsection (the hurried reader could skip it), we prove the existence of a parabolic gauge:
\begin{prop}[Existence of parabolic gauge]
For a generic connection $A=A_1+A_2$, there is a gauge $P\in \mathcal{C}^{\infty}(\Sigma,\SL(n,\mathbb{C}))$ with last row zero except for the last entry (parabolic gauge) such that $A_1$ is locally transformed into a companion matrix.
\end{prop}
\begin{proof}
We begin by setting up notations. The matrix $P$ we look for, is of the following form:

$$P= \left( \begin{array}{cccc}
	p_{11} & ... & p_{1,n-1} &p_{1n} \\
	\vdots & & \vdots & \vdots \\
	p_{n-1,1} & ...& p_{n-1,n-1}& p_{n-1,n} \\
	0 & 0 & 0 & p_{nn}
\end{array} \right) = \left( \begin{array}{cc}
	L_1 & p_{1n} \\
	\vdots & \vdots \\
	L_{n-1} & p_{n-1,n} \\
	0 & p_{nn}
\end{array} \right)$$
where $L_k$ denotes the $k$\textsuperscript{th} line in the matrix without the last entry (i.e. a row vector of length $n-1$).
Denote the entries of $A_1$ by $a_{ij}$. We adopt Einstein's summation convention in this section (automatic summation over repeated indices).

We want that $P$ transforms $A_1$ into a companion matrix under gauge transform $PA_1P^{-1}+P\partial(P^{-1})$.
Since $P$ is of determinant 1, $A_1$ stays traceless.
Using $P\partial{P^{-1}} = -\partial P P^{-1}$, we get
\begin{equation}\label{ppara}
PA_1-\partial P = \left( \begin{array}{cccc}
	 & &  & *\\
	1 & & & \vdots \\
	& \ddots & & * \\
  & & 1& 0
\end{array} \right) P = \left( \begin{array}{cc}
	0 & *\\
	L_1 & * \\
	\vdots & \vdots \\
  L_{n-1}& *
\end{array} \right).
\end{equation}
This gives $n^2-n$ equations by the first $n-1$ columns.

Our strategy is the following: we express $p_{ij}$ for $1\leq i \leq n-1$ and $1\leq j \leq n-1$ in terms of $a_{ij}$ (the ``constants'') and the $(p_{kn})_{k=1,...,n}$ (and their derivatives). Then we get an expression of $p_{nn}$ in terms of $a_{ij}$. Finally we compute $p_{kn}$ for $k=1,...,n-1$.

The matrix equation \eqref{ppara} above gives 
\begin{equation}\sum_{k=1}^n p_{ik}a_{kj}-\partial p_{ij} = p_{i-1,j} \label{auxmat} \end{equation} 
$\forall 1\leq i \leq n$ and $\forall 1\leq j\leq n-1$ where we have put $p_{0j}= 0$.

Setting $i=n$ (and $j<n$), we get $p_{n-1,j} = a_{nj}p_{nn}$, i.e. $L_{n-1}$ is $p_{nn}$ times the last row of $A_1$.
Setting $i=n-1$ (and $j<n$), we get $p_{n-2,j}=p_{n-1,k}a_{kj}-\partial p_{n-1,j} = p_{nn}(a_{nk}a_{kj}-\partial a_{nj})-a_{nj}\partial p_{nn} +p_{n-1,n}a_{nj}$.

By continuing, we see that for $2\leq i \leq n$, we get our first goal: the equations \eqref{auxmat} express the $p_{ij}$ for $1\leq i \leq n-1$ and $1\leq j \leq n-1$ in terms of $a_{ij}$ and the $(p_{kn})_{k=1,...,n}$.

To achieve our second goal, we prove the following:
\begin{lemma}
Denote by $P_0$ the square-submatrix of the $p_{ij}$ for $1\leq i \leq n-1$ and $1\leq j \leq n-1$. We then have $$\det(P_0) = A p_{nn}^{n-1}$$ where $A$ is some constant only depending on the $a_{ij}$.
\end{lemma}
\begin{proof}
We interpret the equation for $PA_1-\partial P$ as a condition on the covariant derivative $\nabla = (-\partial+A_1)$ on $P$, acting from the right. The factor $p_{nn}$ is interpreted as a scalar denoted by $f$.
Put $a=(a_{n1},a_{n2},...,a_{n,n-1})$ the last row of $A_1$ which we consider as a row vector. We already noticed that $L_{n-1}=fa$. In what follows we write $\nabla a$ for $-\del a + aA_1$.
The other equations of \eqref{auxmat} now successively give 
\begin{align}
L_{n-2} &= \nabla(fa)+p_{n-1,n}a \nonumber \\
L_{n-3} &= \nabla L_{n-2} + p_{n-2,n}a = \nabla(\nabla(fa)+p_{n-1,n}a ) + p_{n-2,n}a \nonumber\\
L_{n-k} &= \nabla L_{n-k+1}+p_{n-k+1,n}a  \label{parabdetail}
\end{align}
Thus, we can write $L_{n-k} = \sum_{l=0}^{k-1} \alpha_{l,k}\nabla^l a$ with $\alpha_{k-1,k} = f$. The other $\alpha_{l,k}$ are functions of $a_{ij}$ and the $(p_{kn})_{k=1,...,n-1}$.

With this expression for $L_{n-k}$, we see that $P_0$ in the basis $(\nabla^{n-2}a,..., \nabla a, a)$ is upper-triangular with $f$ on the diagonal. Hence, its determinant is 
$$\det(P_0) = f^{n-1}\det(\nabla^{n-2} a, ..., \nabla a,a) = p_{nn}^{n-1} A.$$
\end{proof}

\begin{coro}
Since $1=\det P = p_{nn} \det P_0 $, we get $$p_{nn} = A^{-\frac{1}{n}}.$$
\end{coro}

We see in particular the condition under which the parabolic gauge exists: we need that $A=\det(\nabla^{n-2} a,...,\nabla a, a) \neq 0$.

To finish, take equations \eqref{parabdetail} for $i=1$ which give \begin{equation} 0=\nabla L_1 + p_{1,n}a = f\nabla^{n-1}a + \sum_{l=0}^{n-2}\alpha_{l,n}\nabla^l a \label{auxmat2} \end{equation}
with $\alpha_{l,n} = p_{l+1,n} +$ terms with $p_{k,n}$ with $k>l+1$.
Then, express $\nabla^{n-1}a$ in the basis $(a,\nabla a, ..., \nabla^{n-2}a)$: $\nabla^{n-1}a = \beta_0 a+\beta_1\nabla a +...+\beta_{n-2}\nabla^{n-2}a$ (thus the $\beta_k$ depend only on the $a_{ij}$). By the freedom of $(a,\nabla a, ..., \nabla^{n-2}a)$, we get out of \eqref{auxmat2} $$\alpha_{l,n}+f\beta_l = 0.$$
Therefore, we can successively express $p_{1n}, p_{2n}$, ... up to $p_{n-1,n}$ in terms of $a_{ij}$ and $p_{nn}$ (which we already expressed in terms of the $a_{ij}$).

This proves the existence of the parabolic gauge.
\end{proof}

\subsection{Higher diffeomorphisms and flat connections}\label{symponconnections}
To get from parabolic connections to flat connections, we define an action of higher diffeomorphisms on the space of parabolic connections $\mathcal{A}\sslash\mathcal{P}$. We prove that this action is hamiltonian and show that the double reduction $\mathcal{A}\sslash\mathcal{P}\sslash\Symp_0(T^*\Sigma)$ is a space of flat connections.

\subsubsection{Action of higher diffeomorphisms}
Recall that the description in coordinates of the space of parabolic connections relies on a basis $B$ of the form $(s, \nabla s, ..., \nabla^{n-1}s)$. A variation $\delta s$ of the section $s$ can be expressed in this basis: $$\delta s = v_1 s+v_2\nabla s +...+v_{n}\nabla^{n-1}s = \bm\hat{H}s$$ where $\bm\hat{H}=v_1+v_2\nabla+...+v_{n}\nabla^{n-1}$ is a differential operator of degree $n-1$.

The Lie algebra of higher diffeomorphisms $\Lie(\Symp_0(T^*\Sigma))$ is the space of functions on $T^*\S$ which can be deformed to differential operators on $\S$. The infinitesimal action of higher diffeomorphisms on parabolic connections is given by a base change induced by $s \mapsto s+\varepsilon \delta s$ such that the basis $B$ preserves its form.
More specifically, to a higher diffeomorphism generated by $H=v_2p+v_3p^2+...+v_{n}p^{n-1}$ we associate the variation $\bm\hat{H}=v_1+v_2\nabla+v_3\nabla^2+...+v_{n}\nabla^{n-1}$ where $v_1$ is uniquely determined by the other $v_i$ by the condition that the infinitesimal gauge transform is of trace zero.

\begin{Remark}
This only defines the infinitesimal action. The question about how to integrate the action to the whole group $\Symp_0$, or maybe to a deformation of it, has to be worked out.
\end{Remark}

Let us describe how to compute the matrix $X$ describing the infinitesimal base change induced by a higher diffeomorphism. 
Write the base change as $$(s,\nabla s,...,\nabla^{n-1}s) \mapsto (s,\nabla s,...,\nabla^{n-1}s)+\varepsilon (\delta s,\nabla \delta s,...,\nabla^{n-1}\delta s).$$ So the first column of $X$ is just given by $Xs = \delta s = v_1s+v_2\nabla s+...+v_n\nabla^{n-1}s$. The second is given by $X\nabla s = \nabla \delta s = \nabla(v_1s+v_2\nabla s+...+v_n\nabla^{n-1}s)$. 
We notice that the construction of this matrix $X$ is exactly the same as for the matrix $A_2$ with the only difference that the variables in $A_2$ are called $\bm\hat{\mu}_k$ instead of $v_k$. Since both matrices are traceless, even the terms $v_1$ and $\bm\hat\mu_1$ coincide. 
Notice that if $X$ is a parabolic gauge, i.e. the $(n-1)$ first entries of the last column are zero, then $v_k=0 \;\forall k$, so $X=0$.
\begin{prop}\label{computeX}
The matrix $X$ of the gauge coming from a higher diffeomorphism is given by $$X=A_2 \mid_{\bm\hat{\mu}_k\mapsto v_k}.$$
\end{prop}

Let us indicate how to compute the action of a higher diffeomorphism on our coordinates $(\bm\hat{t}_k, \bm\hat{\mu}_k)$. The coordinates $\bm\hat{t}_k$ are given by the relation $\nabla^n s = \bm\hat{P}s$ where $\bm\hat{P}=\bm\hat{t}_2\nabla^{n-2}+...+\bm\hat{t}_n$. The variation $\delta \bm\hat{P}$ satisfies 
$$\nabla^n (s+\varepsilon \bm\hat{H}s) = (\bm\hat{P}+\varepsilon \delta \bm\hat{P})(s+\varepsilon \bm\hat{H}s)$$
which gives 
\begin{equation}\label{variation-hat-t}
\delta \bm\hat{P} = [\bm\hat{H}, -\nabla^n+\bm\hat{P}] \mod \bm\hat{I}
\end{equation}
where $\bm\hat{I} = \langle \nabla^n-\bm\hat{P}, -\bar{\nabla}+\bm\hat{Q} \rangle$ is a left ideal of differential operators.

Similarly, the coordinates $\bm\hat{\mu}_k$ are given by $\bar{\nabla}s = \bm\hat{Q}s$ where $\bm\hat{Q} = \bm\hat{\mu}_1+\bm\hat{\mu}_2\nabla...+\bm\hat{\mu}_n\nabla^{n-1}$. We can easily compute the variation of $\bm\hat{Q}$ to be 
\begin{equation}\label{variation-hat-mu}
\delta \bm\hat{Q} = [\bm\hat{H}, -\bar{\nabla}+\bm\hat{Q}] \mod \bm\hat{I}.
\end{equation}

\begin{Remark}
In the case of the Hilbert scheme, we used the variation formula from symplectic geometry $\delta f = \frac{df}{dt} = \{H,f\}$ (see equation \eqref{idealvariation}). Here we find the deformed version of this: $\delta \bm\hat{P} = [\bm\hat{H}, \bm\hat{P}]$ where $\bm\hat{H}$ is the quantum Hamiltonian and $\bm\hat{P}$ some operator. In the next section, we introduce a deformation parameter $h$ and we get $\delta \bm\hat{P}(h) = \frac{1}{h}[\bm\hat{H}(h), \bm\hat{P}(h)]$.
\hfill $\triangle$
\end{Remark}

\subsubsection{Double reduction to flat connections}

We have just seen that higher diffeomorphisms act on the space of parabolic connections by gauge transforms. Since we see $\mc{A}\sslash\mc{P}$ as a subset of $\mc{A}$ and since the gauge action on $\mc{A}$ is hamiltonian, we see that the action of higher diffeomorphisms on $\mathcal{A}\sslash\mathcal{P}$ is also hamiltonian. It is not surprising that the moment map is nothing else than the parabolic curvature:

\begin{thm}
The infinitesimal action of higher diffeomorphisms $\Symp_0(T^*\Sigma)$ on the space of parabolic connections $\mathcal{A}\sslash\mathcal{P}$ is hamiltonian with moment map
$$m(\bm\hat{t}_i,\bm\hat{\mu}_j).(v_2,...,v_n) = \int_{\Sigma} \sum_{i=1}^n x_{n,n+1-i}\xi_i$$ where $x_{i,j}$ are the matrix elements of the gauge $X$ and $\xi_i$ is the parabolic curvature of the parabolic connection described by $(\bm\hat{t}_i,\bm\hat{\mu}_i)_{2\leq i \leq n}$.
\end{thm}
Some explanation for the moment map is necessary: $m$ goes from the space $\mathcal{A}\sslash\mathcal{P}$, which is described by coordinates $(\bm\hat{t}_i,\bm\hat{\mu}_j)$, into $\Lie(\Symp_0)^*$, the dual to the Lie algebra of higher diffeomorphisms. The Lie algebra of higher diffeomorphisms is described by Hamiltonians of the form $v_2p+...+v_np^{n-1}$. To such a function, we compute the associated matrix $X$ (see Proposition \ref{computeX}) from which we take the last row for computing $m$. All elements of $X$ are functions depending on the $v_k$ and the $\bm\hat{t}_k$. The parabolic curvature described by the $\xi_i$ is a function of $(\bm\hat{t}_i,\bm\hat{\mu}_j)$.

\begin{proof}
Our computation is analogous to the Atiyah--Bott reduction. 
An infinitesimal gauge transform given by $X$ affects $A_1$ and $A_2$ by
\begin{align*}
\chi(A_1) = [X,A_1]-\partial X \\
\chi(A_2) = [X,A_2] -\bar{\partial}X.
\end{align*}

The symplectic form on $\mathcal{A}\sslash\mathcal{P}$ is the restriction of the one on $\mathcal{A}$, so we can compute 
\begin{align*}
\iota_{\chi}\omega_{\mathcal{A}\sslash\mathcal{P}} &= \int \tr \left(\chi(A_1)\delta A_2-\chi(A_2)\delta A_1 \right)\\
&= \int \tr ([X,A_1]-\partial X )\delta A_2-([X,A_2] -\bar{\partial}X)\delta A_1 \\
&= \int \tr ([A_1,\delta A_2]+\delta\partial A_2-[A_2,\delta A_1]-\delta\bar{\partial}A_1)X \\
&= \int \tr \delta (\partial A_2-\bar{\partial}A_1+[A_1,A_2])X \\
&= \delta \int \tr F(A)X \\
&= \delta \int \sum_{i=1}^n x_{n,n+1-i}\xi_i.
\end{align*}

Therefore $$m= \int_{\Sigma} \sum_{i=1}^n x_{n,n+1-i}\xi_i.$$
\end{proof}
\begin{coro}\label{flatparaconnections}
The double reduction $\mathcal{A}\sslash\mathcal{P}\sslash\Symp_0$ gives a space of flat connections:
$$\mathcal{A}\sslash\mathcal{P}\sslash\Symp_0 \cong \{A_1+A_2\in \mathcal{A}\sslash\mathcal{P} \mid \xi_i=0 \; \forall i\} / \Symp_0$$
with $A_1$ locally of the form \ref{firstmatrix} and $A_2$ like in \ref{secondmatrix}.
\end{coro}
The corollary directly follows from the previous theorem since $m(\bm\hat{t}_i, \bm\hat{\mu}_j)(v_2,...,v_n)=0$ for all $v_2,...,v_n$ implies $\xi_i=0$ for all $i$. 

We call the double reduction space $\mathcal{A}\sslash\mathcal{P}\sslash\Symp_0$ the space of \textbf{flat parabolic connections}. 
For $n=2$ we get those flat $\SL_2(\C)$-connections whose monodromy is the developing map of a complex projective structure on $\S$. For general $n$, we probably get a complicated subset of the space of all flat connections.


\section{Parabolic reduction of \texorpdfstring{$h$}{h}-connections}\label{parabolicwithlambda}

In this section, we study the parabolic reduction on $h$-connections to get the twistor space description from figure \ref{HK}.
The main idea is the following: a point in $\cotang$ is a $\Symp_0$-equivalence class of ideals of the form $$I=\langle -p^n+t_2p^{n-2}+...+t_n, -\bar{p}+\mu_1+\mu_2 p+...+\mu_n p^{n-1} \rangle.$$ Replace the polynomials by $h$-connections using the rule $p \mapsto \nabla=h\del + A_1(h)$ and $\bar{p} \mapsto \bar{\nabla}=h\delbar+A_2(h)$ where $h$ is a formal parameter. This corresponds to the deformation of a higher complex structure $\Phi$ to $\Phi+hd+hA+h^2\Phi^*=h(d+\l\Phi+A+\l^{-1}\Phi^*)$ where $\l=h^{-1}$.

For $h\neq 0$ we divide the connection by $h$ to get a usual connection with parameter $\l$.
For all $\l \in \C^*$ fixed, we get the same space as described in the previous section \ref{parabolicreduction}, i.e. the space of flat parabolic connections. For $\l \rightarrow \infty$ we get the cotangent space $\cotang$. For $\l \rightarrow 0$ we get the space of conjugated structures $[({}_kt, {}_k\mu)]$ (see subsection \ref{dualcomplexstructure}).


\subsection{Parametrization}

Take $\mc{A}(\l)=\l \Phi + A + \l^{-1}\Phi^*$ where $\Phi$ is in the Hilbert scheme bundle $\Hilb^n_0(T^{*\C}\Sigma)$ and $\Phi^*$ in the conjugated bundle. This means that locally $\Phi(z,\bar{z})=\Phi_1(z,\bar{z})dz+\Phi_2(z,\bar{z})d\bar{z}$ with $\Phi_1 \in \mf{sl}_n$ is a principal nilpotent element and $\Phi_2$ is in the centralizer of $\Phi_1$, i.e. $[\Phi_1,\Phi_2]=0$.
We define $\mc{A}_1(\l)=\l \Phi_1 + A_1 + \l^{-1}\Phi_2^*$ and $\mc{A}_2(\l)=\l \Phi_2 + A_2 + \l^{-1}\Phi_1^*$, i.e. the $(1,0)$-part and $(0,1)$-part of $\mc{A}(\l)$. We also define $\nabla=\partial + \mc{A}_1(\l)$ and $\bar{\nabla}=\bar{\partial}+\mc{A}_2(\l)$.

\begin{Remark}
For the moment $\Phi^*$ is an arbitrary conjugated higher complex structure. We will later impose a reality condition (in section \ref{finalstep}) such that the $*$-operator becomes the hermitian conjugate. In the twistor language, $\mc{A}(\l)$ is a general twistor line and later on we will focus on \emph{real} twistor lines.
\hfill $\triangle$
\end{Remark}

As for the case without parameter, there is a parabolic gauge which transforms $\mc{A}(\l)$ locally to 
\begin{equation}\label{paragauge}
\left(\begin{array}{cccc}
	 & & & \bm\hat{t}_n(\l) \\
	1& & & \vdots \\
	 & \ddots & & \bm\hat{t}_2(\l) \\
	 & & 1& 0 
\end{array}\right)
dz + \left(\begin{array}{cc}
	 \bm\hat{\mu}_1(\l) &  \\
	 \bm\hat{\mu}_2(\l)&    \bm\hat{\alpha}_{ij}(\l)\\
	 \vdots &  \\
	 \bm\hat{\mu}_n(\l) & 
\end{array}\right) d\bar{z}
\end{equation}
where $\bm\hat{\alpha}_{ij}(\l)$ and $\bm\hat{\mu}_1(\l)$ are explicit functions of the other variables. So we can take $(\bm\hat{t}_i(\l), \bm\hat{\mu}_i(\l))_{i=2,...,n}$ as coordinates.

This local representative comes from a basis of the form $(s, \nabla s, ..., \nabla^{n-1}s)$ for some section $s$. We then get our coordinates by
\begin{equation}\label{paraboliccoord1}
\nabla^n s = \bm\hat{t}_{n}(\l)s + \bm\hat{t}_{n-1}(\l) \nabla s+...+\bm\hat{t}_{2}(\l)\nabla^{n-2}s 
\end{equation}
\begin{equation}\label{paraboliccoord2}
\bar{\nabla}s = \bm\hat{\mu}_1(\l) s + \bm\hat{\mu}_2(\l) \nabla s + ... + \bm\hat{\mu}_n(\l) \nabla^{n-1}s.
\end{equation}

You can compute the $\bm\hat{\alpha}_{ij}(\l)$ using $\bar{\nabla}\nabla^ks = \nabla^k\bar{\nabla}s$ for $k\leq n-1$ which holds since the curvature $[\nabla, \bar{\nabla}]$ is concentrated in the last column.

\begin{example}\label{examplen2}
Take $n=2$ and consider $\Phi_1 = \left(\begin{smallmatrix} 0 & 0 \\ b_1 & 0\end{smallmatrix}\right)$, $A_1 = \left(\begin{smallmatrix} a_0 & a_1 \\ a_2 & -a_0\end{smallmatrix}\right)$, $\Phi_2=\mu_2\Phi_1$ and $A_2 = -A_1^{\dagger}$. So we have 
$$\mc{A}_1(\l)=\begin{pmatrix} a_0 & a_1+\l^{-1}\bar{\mu}_2\bar{b}_1 \\ a_2+\l b_1 & -a_0\end{pmatrix}\; \text{ and } \;\mc{A}_2(\l)=\begin{pmatrix} -\bar{a}_0 & -\bar{a}_2+\l^{-1}\bar{b}_1 \\ -\bar{a}_1+\l\mu_2b_1 & \bar{a}_0\end{pmatrix}.$$
We look for $P=\left(\begin{smallmatrix} p_1 & p_2 \\ 0 & 1/p_1\end{smallmatrix}\right)$ such that 
$$P\mc{A}_1(\l)P^{-1}+ P\partial P^{-1} = \begin{pmatrix} 0 & \bm\hat{t}_2(\l) \\ 1 & 0\end{pmatrix}.$$
Multiplying by $P$ from the right, one can solve the system. One finds $p_1=(\l b_1+a_2)^{1/2}$ and $p_2=-\frac{a_0}{p_1}+\frac{\partial p_1}{p_1^2}$. Hence 
$$\bm\hat{t}_2(\l)=\l a_1b_1+ \text{ constant term }+\l^{-1}\bar{\mu}_2 a_2\bar{b}_1.$$ Transforming $\mc{A}_2(\l)$ with $P$ we get 
$$\bm\hat{\mu}_2(\l) = \frac{-\bar{a}_1+\l\mu_2 b_1}{\l b_1+a_2} = \frac{-\bar{a}_1+\l\mu_2 b_1}{\l b_1-\bar{\mu}_2\bar{a}_1}$$ where we used $a_2=-\bar{\mu}_2\bar{a}_1$ coming from the flatness of $\mc{A}(\l)$.

\noindent For $\l\rightarrow \infty$, we can develop the rational expression of $\bm\hat{\mu}_2(\l)$ to get 
$$\bm\hat{\mu}_2(\l)=\mu_2+(\mu_2\bar{\mu}_2-1)\sum_{k=1}^\infty \frac{\bar{\mu}_2^{k-1}\bar{a}_1^k}{b_1^k}\l^{-k}.$$
For $\l \rightarrow 0$, we get 
$$\bm\hat{\mu}_2(\l)=\frac{1}{\bar{\mu}_2}+(1-\mu_2\bar{\mu}_2)\sum_{k=1}^\infty \frac{b_1^k}{\bar{\mu}_2^{k+1}\bar{a}_1^k}\l^{k}.$$
Notice that we get ${}_2\mu=1/\bar{\mu}_2$ as leading term (see section \ref{dualcomplexstructure}).
\hfill $\triangle$
\end{example}

The example shows several phenomena which are true in general: 

\begin{prop}\label{parametrisationlambda}
The $\bm\hat{\mu}_k(\l)$ are rational functions in $\l$. The highest term in $\l$ when $\l \rightarrow \infty$ is $\l^{2-k}\mu_k$ where $\mu_k$ is the higher Beltrami differential from the $n$-complex structure. For $\l\rightarrow 0$ we get as lowest term $\l^{k-2}{}_k\mu$ where  ${}_k\mu$ is the conjugated $n$-complex structure. 

The $\bm\hat{t}_k(\l)$ are also rational functions in $\l$. For $\l \rightarrow \infty$, the highest term is given by $\l^{k-1}t_k$, and the lowest term for $\l \rightarrow 0$ is given by $\l^{1-k}{}_kt$ where 
$$t_k=\tr A_1\Phi_1^{k-1} \text{ and } {}_kt = \tr A_1 (\Phi_2^*)^{k-1}.$$
\end{prop}

Notice the appearance of the holomorphic differentials $t_k = \tr A_1\Phi_1^{k-1}$ which we extracted from an affine connection in Theorem \ref{holodifff} (for $\Phi_2=0$). 
We will see later that $(\mu_k, t_k)$ is a point of the cotangent bundle $\cotang$, and that $({}_k\mu, {}_kt)$ is the conjugated structure, which justifies the notation.
\begin{proof}
The whole point is to analyze equations \eqref{paraboliccoord1} and \eqref{paraboliccoord2} in detail.
Let us start with $$\bar{\nabla}s = \bm\hat{\mu}_1(\l) s + \bm\hat{\mu}_2(\l) \nabla s + ... + \bm\hat{\mu}_n(\l) \nabla^{n-1}s.$$
Since $\bar{\nabla}s=(\bar{\partial}+\l\Phi_2+A_2+\l^{-1}\Phi_1^*)s$ the highest $\l$-term is $\l\Phi_2s=\l\mu_2\Phi_1s+...+\l\mu_n\Phi_1^{n-1}s$. On the other side, the highest term of $\nabla^ks$ is $\l^k\Phi_1^ks$ for $0\leq k \leq n-1$. For generic $s$ the set $(s, \Phi_1s, ..., \Phi_1^{n-1}s)$ is a basis. Hence, we can compare the highest terms and deduce that for $\l\rightarrow \infty$: $$\bm\hat{\mu}_k(\l) = \l^{2-k}\mu_k+\text{ lower terms}.$$ 

Similarly, the set $(s, \Phi_2^*s, ..., \Phi_2^{*(n-1)}s)$ is generically a basis. Comparing highest terms and using $\Phi_1^*={}_2\mu \Phi_2^*+...+{}_n\mu\Phi_2^{*(n-1)}$, we get for $\l\rightarrow 0$: $$\bm\hat{\mu}_k(\l) = \l^{k-2}{}_k\mu+\text{ higher terms}.$$

In any case, we can decompose $\nabla^ks$ and $\bar{\nabla}s$ in the basis $(s,\Phi_1s,...,\Phi_1^{n-1}s)$ and notice that the defining equations for $\bm\hat{\mu}_k$ is a quotient of two polynomials in $\l$, i.e. $\bm\hat{\mu}_k$ is a rational function in $\l$.
The same decomposition gives that $\bm\hat{t}_k$ is a rational function in $\l$.

The last thing is to study the asymptotic behavior of $\bm\hat{t}_k$. For that, we have to study
$$\nabla^n s = \bm\hat{t}_{n}(\l)s + \bm\hat{t}_{n-1}(\l) \nabla s+...+\bm\hat{t}_{2}(\l)\nabla^{n-2}s.$$ 
The highest term of $\nabla^ns$ is not $\l^n \Phi_1^n$ since $\Phi_1^n=0$. The next term is given by $$\l^{n-1}\sum_{l=0}^{n-1}\Phi_1^l \circ(\partial+A_1)\circ\Phi_1^{n-1-l}s$$ where $\circ$ denotes the composition of differential operators.
On the other side, the highest terms are given by $\bm\hat{t}_k\l^{n-k}\Phi_1^{n-k}s$. When $\l$ goes to infinity, we compare coefficients in the basis $(s, \Phi_1s, ..., \Phi_1^{n-1}s)$ as before. Using Dirac's ``bra-ket'' notation (see page \pageref{notations}), we get
\begin{align*}
\l^{n-k}\bm\hat{t}_k &= \l^{n-1} \langle \Phi_1^{n-k}s \mid \sum_{l=0}^{n-1}\Phi_1^l \circ(\partial+A_1)\circ\Phi_1^{n-1-l} \mid s \rangle \\
&= \l^{n-1}\sum_{l=0}^{n-k} \langle \Phi_1^{n-k-l}s\mid (\partial +A_1)\circ\Phi_1^{n-1-l}\mid s\rangle \\
&= \l^{n-1}\sum_{l=0}^{n-k} \langle \Phi_1^{n-k-l}s\mid (\partial +A_1)\circ\Phi_1^{k-1}\mid \Phi_1^{n-k-l}s\rangle \\
&= \l^{n-1}\tr(\partial+A_1)\circ\Phi_1^{k-1}\\
&= \l^{n-1}\tr A_1 \Phi_1^{k-1}.
\end{align*}
In the last line, we used that $\tr \del \circ \Phi_1^{k-1} = 0$ since $\Phi_1$ is strictly lower triangular which is preserved under derivation.
This precisely gives the expression for $t_k$ as stated in the proposition. 
The same analysis goes through for $\l \rightarrow 0$.
\end{proof}

\begin{Remark}
We will see in Theorem \ref{conditioncinconnection} that the parameters $(\mu_k, t_k)$ can be identified with those from $\cotang$. From equation \eqref{hkquotientofmodulispace} we know that the $t$'s can be considered as infinitesimal small compared to $\mu$. That is why the $t_k$ appear in $A$ and not in $\Phi$.
Using a diagonal gauge of the form $\diag(1,\l, \l^2, ..., \l^{n-1})$, one can get for $\l\rightarrow \infty$ the asymptotic behavior $\bm\hat{t}_k(\l)=t_k+\text{ lower terms}$ and $\bm\hat{\mu}_k(\l)=\l\mu_k+\text{ lower terms}$. 
This illustrates that $t_k$ is infinitesimal smaller than $\mu_k$.
\hfill $\triangle$
\end{Remark}

At the end of subsection \ref{settingparab} we have noticed that $\bm\hat{t}_k$ and $\bm\hat{\mu}_k$ do not transform as tensors. We now show that the highest terms, $t_k$ and $\mu_k$, are tensors. Recall that $K=T^{*(1,0)}\Sigma$ is the canonical bundle and that $\Gamma(.)$ denotes the space of sections.
\begin{prop}\label{highesttermtensor}
We have $t_i \in \Gamma(K^i)$ and $\mu_i \in \Gamma(K^{1-i}\otimes \bar{K})$.
\end{prop}
\begin{proof}
Consider a holomorphic coordinate change $z\mapsto w(z)$. We compute how $\mu_i(z)$ and $t_i(z)$ change.

For $\mu_i$, notice that $\Phi_1dz \mapsto \Phi_1 \frac{dz}{dw}dw$, so using $$\Phi_2d\bar{z} = \mu_2(z)\Phi_1dz+...+\mu_n\Phi_1^{n-1}dz^{n-1}$$ we easily get $\mu_i(z)=\frac{d\bar{z}/d\bar{w}}{(dz/dw)^{i-1}}\mu_i(w)$.

For $t_i$, we use $t_i=\tr(\Phi_1^{i-1}A_1)$ where $\Phi_1$ and $A_1$ are both $(1,0)$-forms, thus $t_i$ is a $(i,0)$-form, i.e. a section of $K^i$.
\end{proof}

\subsection{Symplectic structure}\label{sympstrconnections}
In this subsection, we explicitly describe the symplectic structure of the space of parabolic $h$-connections $\mathcal{A}(h)\sslash\mathcal{P}$.

We know that on the space of all connections $\mathcal{A}$, the symplectic structure is simply given by $$\omega_{\mathcal{A}} = \frac{1}{2}\int_{\Sigma} \tr \delta A \wedge \delta A.$$

Let us describe $\omega_{\mathcal{A}(h)\sslash\mathcal{P}}$ in our coordinates. We have seen that the reduction $\mathcal{A}(h)\sslash\mathcal{P}$ is a subspace of $\mathcal{A}$, consisting of those connections locally of the form 
$$
\left(\begin{array}{cccc}
	 & & & \bm\hat{t}_n(\l) \\
	1& & & \vdots \\
	 & \ddots & & \bm\hat{t}_2(\l) \\
	 & & 1& 0 
\end{array}\right)
dz + \left(\begin{array}{cc}
	 \bm\hat{\mu}_1(\l) &  \\
	 \bm\hat{\mu}_2(\l)&    \bm\hat{\alpha}_{ij}(\l)\\
	 \vdots &  \\
	 \bm\hat{\mu}_n(\l) & 
\end{array}\right) d\bar{z}
$$
where $\bm\hat{\alpha}_{ij}(\l)$ and $\bm\hat{\mu}_1(\l)$ are explicit functions of the other variables.
The symplectic form $\omega_{\mathcal{A}(h)\sslash\mathcal{P}}$ is thus the restriction of $\omega_{\mathcal{A}}$ to this subspace. 

In the space of affine connections we consider, we have $\delta A = \l \delta \Phi + \delta A +\l^{-1}\delta \Phi^*$. Since $\delta \Phi\wedge \delta \Phi = 0 = \delta \Phi^*\wedge \delta \Phi^*$, we see that the symplectic structure $\l\omega_{\mathcal{A}(h)\sslash\mathcal{P}}$ is \emph{quadratic} in $\l$.

Now, we know that $\mc{A}_1(\l)$ is a companion matrix, so the entries of the first $(n-1)$ columns of $\delta \mc{A}_1(\l)$ are zero and its last column is given by $\delta \bm\hat{t}_k(\l)$. In addition, $\tr(AB)$ is the usual pairing on matrices, so we get 
$$\omega_{\mathcal{A}(h)\sslash\mathcal{P}} = \int_{\Sigma} \sum_{k=2}^n \delta \bm\hat{t}_k(\l)\wedge\delta \bm\hat{\alpha}_{n,n-k+1}(\l) \; dz\wedge d\bar{z} .$$
For instance for $n=2$ ($G=\SL(2, \mathbb{C})$), we simply get $\omega=\int \delta \bm\hat{t}_2(\l)\wedge \delta \bm\hat{\mu}_2(\l) \; dz\wedge d\bar{z}$. 

\subsection{Action of higher diffeomorphisms}

In \ref{symponconnections} we have described an infinitesimal action of $\Symp_0(T^*\Sigma)$ on the space of parabolic connections $\mc{A}\sslash\mc{P}$. 
Recall that to write a representative of an element of $\mc{A}\sslash\mc{P}$, we use a basis of the form $(s,\nabla s,..., \nabla^{n-1}s)$. A higher diffeomorphism changes the section $s$ and thus the whole basis. 
The same action holds for the parabolic $h$-connections. In particular, Corollary \ref{flatparaconnections} about flat parabolic connections stays true.

Here we analyze the infinitesimal action of $\Symp_0$ on $\mc{A}(h)\sslash\mc{P}$, in particular what it does on the highest terms $\mu_k$ and $t_k$. There are two steps: a local analysis and a global analysis.

\subsubsection{Local analysis}
We prove that the action of higher diffeomorphisms on the highest terms $\mu_k$ of the parabolic reduction is precisely the action on the $n$-complex structure. So we can trivialize it locally.

Take a change of section $\delta s = \bm\hat{v}_1s+\bm\hat{v}_2\nabla s+...+\bm\hat{v}_{n}\nabla^{n-1}s=\bm\hat{H}s$. We have previously seen in equation \eqref{variation-hat-mu} that the change of coordinates $\delta \bm\hat{\mu}_k$ can be computed by $$\delta \bm\hat{Q}=[\bm\hat{H}, \bm\hat{Q}] \mod \bm\hat{I}$$
where $\bm\hat{I}=\langle -\nabla^n+\bm\hat{t}_2\nabla^{n-2}+...+\bm\hat{t}_n, -\bar{\nabla}+\bm\hat{\mu}_1+\bm\hat{\mu}_2\nabla+...+\bm\hat{\mu}_n\nabla^{n-1} \rangle$ is a left-ideal in the space of differential operators.

Since we have a parameter $\l$ in our setting, the variations $\bm\hat{v}_k$ also depend on $\l$. More precisely, for $k\geq	 2$ we have that $\bm\hat{v}_k(\l)$  is a rational function in $\l$ with highest term $\l^{2-k}v_k$ when $\l \rightarrow \infty$. 
Notice that $\bm\hat{v}_1$ is not a free parameter, but depends on the others. It assures that the trace of the gauge transform is zero. One can compute that $\bm\hat{v}_1$ has highest term of degree 0.

\begin{Remark}
It is not clear for the moment how to determine the precise expression for $\bm\hat{v}_k(\l)$ from a higher diffeomorphism generated by some Hamiltonian $H$. The highest $\l$-terms in $\bm\hat{v}_k(\l)$ are given by the coefficients of $H=v_2p+...+v_np^{n-1}$.
\hfill $\triangle$
\end{Remark}


We can now state:
\begin{thm}\label{actionsymponlambdaconn}
The infinitesimal action of $\Symp_0(T^*\Sigma)$ on the highest terms $\mu_k$ of the coordinates $\bm\hat{\mu}_k(\l)$ of the space of parabolic connections with parameter is the same as the infinitesimal action of higher diffeomorphisms on the $n$-complex structure.
\end{thm}
The reason for the theorem to be true is roughly speaking that the Poisson bracket is the semi-classical limit of commutators of differential operators. 
The strategy of the proof is the following: we prove the theorem first for $\mu_2$, and then for $\mu_k$ ($k>2$) supposing $\mu_2=...=\mu_{k-1}=0$ which simplifies the computations.
From Proposition \ref{varmu}, we know that the infinitesimal action of a Hamiltonian $H=v_2p+...+v_np^{n-1}$ on the higher Beltrami differentials is given by 
$$\delta \mu_2 = (\bar{\partial}-\mu_2\partial+\partial\mu_2)v_2$$
for $\mu_2$ and for $\mu_k$, supposing $\mu_2=...=\mu_{k-1}=0$,  we simply have $$\delta \mu_k = \bar{\partial}v_{k}.$$

\begin{proof}
First, we compute the variation of $\mu_2$ using equation \eqref{variation-hat-mu}:
$$\delta \bm\hat{\mu}_1+\delta \bm\hat{\mu}_2\nabla+...+\delta \bm\hat{\mu}_n\nabla^{n-1}=[\bm\hat{v}_1+\bm\hat{v}_2\nabla +...+\bm\hat{v}_{n}\nabla^{n-1}, -\bar{\nabla}\!+\!\bm\hat{\mu}_1\!+\!\bm\hat{\mu}_2\nabla+...+\bm\hat{\mu}_n\nabla^{n-1}] \mod \bm\hat{I}.$$
 Since the highest $\l$-term of $\bm\hat{\mu}_2$ is of degree 0, we are interested in the part of degree $0$ of the coefficient of $\nabla$ in $[\bm\hat{v}_1+\bm\hat{v}_2\nabla +...+\bm\hat{v}_{n}\nabla^{n-1}, -\bar{\nabla}+\bm\hat{\mu}_1+\bm\hat{\mu}_2\nabla+...+\bm\hat{\mu}_n\nabla^{n-1}] \mod \bm\hat{I}$.

We first look on contributions coming from $[\bm\hat{v}_k\nabla^{k-1}, \bm\hat{\mu}_l\nabla^{l-1}]$ for $k, l\geq 2$:
If $k+l-3<n$ then we do not reduce modulo $\bm\hat{I}$, so the highest term in $\l$ is of degree $4-(k+l)$. Since we have $k, l\geq 2$, the highest term comes from $k=l=2$, which gives $v_2\del\mu_2-\mu_2\del v_2$.
If $k+l-3\geq n$, we can have terms with $\nabla^m$ with $n\leq m \leq k+l-3$. So we have to use $\bm\hat{I}$ to reduce it. This reduction gives $\nabla^m = c(\l)\nabla+\text{ other terms}$, and the highest term of $c(\l)$ is of degree $m-2 \leq k+l-5$. Hence, the highest term for $[\bm\hat{v}_k\nabla^{k-1}, \bm\hat{\mu}_l\nabla^{l-1}]$ is $4-(k+l)+k+l-5=-1$.

The contributions from $\bm\hat{\mu}_1$ and $\bm\hat{v}_1$ also have degree at most -1. There is one more contribution in degree 0 coming from $[\bm\hat{v}_2\nabla, -\bar{\nabla}]$, which gives $\delbar v_2$. Therefore, we have $$\delta \mu_2 = (\bar{\partial}-\mu_2\partial+\partial\mu_2)v_2.$$

Now, suppose $\mu_2=...=\mu_{k-1}=0$ and compute the variation $\delta \mu_k$ under an action generated by $\bm\hat{v}_k\nabla^{k-1}+...+\bm\hat{v}_n\nabla^{n-1}$. From $$\delta \bm\hat{\mu}_k\nabla^{k-1}+...+\delta \bm\hat{\mu}_n\nabla^{n-1}=[\bm\hat{v}_k\nabla^{k-1} +...+\bm\hat{v}_{n}\nabla^{n-1}, -\bar{\nabla}+\bm\hat{\mu}_1+\bm\hat{\mu}_2\nabla+...+\bm\hat{\mu}_n\nabla^{n-1}] \mod \bm\hat{I}$$ we can analyze as above the contribution to the term of degree $2-k$ of the coefficient of $\nabla^{k-1}$. Since $\bm\hat{v}_l$ is of degree at most $2-l$ and $\bm\hat{\mu}_l$ of degree at most $1-l$ for $l<k$ (since we suppose that $\mu_l=0$), we can see that $[\bm\hat{v}_l\nabla^{l-1}, \bm\hat{\mu}_m\nabla^{m-1}]$ cannot contribute to the highest degree. The only contribution comes from the term with $-\bar{\nabla}$. Thus, $$\delta \mu_k = \bar{\partial}v_k.$$
This concludes the proof since the action of higher diffeomorphisms on the $n$-complex structure has the same expression.
\end{proof}
\begin{coro}
With the action of higher diffeomorphisms, we can locally render $\Phi_2=0$.
\end{coro}
The corollary directly follows from the previous theorem and the fact that the higher complex structure can be locally trivialized (Theorem \ref{loctrivial}), i.e. we can render $\mu_2=...=\mu_n=0$ locally and since $\Phi_2=\mu_2\Phi_1+...+\mu_n\Phi_1^{n-1}$ this implies $\Phi_2=0$.

\begin{Remark}
We see that a term $\bm\hat{v}_k\nabla^{k-1}$ can influence $\bm\hat{\mu}_i$ with $i<k$ (unlike the case higher complex structures where the simplification lemma \ref{simplificationlemma} holds), but it does not influence the highest term $\mu_i$. In the same vein, a term $\bm\hat{v}_k\nabla^{k-1}$ with $k>n$ acts on parabolic connections, but not on the highest terms.
\hfill $\triangle$
\end{Remark}
\begin{Remark}
We can also compute the action of $\Symp_0(T^*\Sigma)$ on the coordinates $\bm\hat{t}_k$ and their highest terms $t_k$. The computation gives that $t_k$ transforms as a cotangent vector to the $n$-complex structure, i.e. like $I=\langle -p^n+t_2p^{n-2}+...+t_n, -\bar{p}+\mu_2p+...+\mu_np^{n-1} \rangle$ under the action of higher diffeomorphisms, modulo $t^2$.
\hfill $\triangle$
\end{Remark}

\subsubsection{Global analysis}
We show that the highest term in $\l$ in the zero-curvature condition relates $(\mu_k, t_k)$ to the cotangent bundle $T^*\bm\hat{\mc{T}}^n$.

We know that the moment map of the hamiltonian action of $\Symp_0(T^*\Sigma)$ on $\mc{A}\sslash\mc{P}$ is given by $\xi_k=0$, i.e. the remaining curvature of a parabolic connection has to vanish. For connections with parameter $\l$, this gives $\xi_k(\l)=0$.

\begin{thm}\label{conditioncinconnection}
The highest term in $\l$ of $\xi_k(\l)=0$ gives the condition $(\mc{C})$ of the cotangent bundle $T^*\bm\hat{\mc{T}}^n$ (see Theorem \ref{conditionC}).
\end{thm}
The proof strategy is to reduce the analysis of the highest term in the parabolic curvature to the expression $\xi_k \mod \bm\hat{t}^2 \mod \partial^2$. The following lemma shows that we then get condition $(\mc{C})$. The hurried reader may skip the proof of this technical lemma.


\begin{lemma}\label{curvaturemodmod}
The parabolic curvature modulo $\bm\hat{t}^2$ and $\del^2$ gives condition $(\mathcal{C})$ on $T^*\bm\hat{\mathcal{T}}^n$:
$$\xi_k = (\bar{\partial}\!-\!\bm\hat{\mu}_2\partial\!-\!k\partial\bm\hat{\mu}_k)\bm\hat{t}_k-\sum_{l=1}^{n-k}\left((l\!+\!k)\partial\bm\hat{\mu}_{l+2}+(l\!+\!1)\bm\hat{\mu}_{l+2}\del\right)\bm\hat{t}_{k+l} \mod \bm\hat{t}^2 \mod \partial^2.$$
\end{lemma}
\begin{proof}
The proof is a combination of several formulas:

\begin{enumerate}
\item Proposition \ref{thmcourbure} together with the expression of the differential operators (see \eqref{diffops}) give 
$$[\partial^n-\bm\hat{t}_2\partial^{n-2}-...-\bm\hat{t}_n,-\bar{\partial}+\bm\hat{\alpha}_{nn}+\bm\hat{\alpha}_{n,n-1}\partial+...+\bm\hat{\alpha}_{n,1}\partial^{n-1}] = \sum_{k=2}^n \xi_k\partial^{n-k} \mod \bm\hat{I}$$
where $\bm\hat{\alpha}_{ij}$ are the entries of the matrix $A_2$.

\item Link between Poisson bracket and commutator: 
\begin{align*}
& \{p^n-\bm\hat{t}_2p^{n-2}-...-\bm\hat{t}_n,-\bar{p}+\bm\hat{\mu}_1+\bm\hat{\mu}_2p+...+\bm\hat{\mu}_np^{n-1}\} \\
 &= \lim_{h \rightarrow 0}\frac{1}{h}[h^n\partial^n-\bm\hat{t}_2h^{n-2}\partial^{n-2}-...-\bm\hat{t}_n,-h\bar{\partial}+\bm\hat{\mu}_1+\bm\hat{\mu}_2h\partial+...+\bm\hat{\mu}_nh^{n-1}\partial^{n-1}]\Big|_{\substack{h\partial\mapsto p \\ h\bar{\partial}\mapsto \bar{p}}}.
\end{align*}

\item Link between $\mod \partial^2$ and brackets: for $h$-connections $D_1(h)$ and $D_2(h)$, we have
$$\lim_{h \rightarrow 0} \frac{1}{h}[D_1(h),D_2(h)] = \frac{1}{h}[D_1(h),D_2(h)] \mod \partial^2.$$

\item The formula from Proposition \ref{spectralcurveprop} linking the Poisson bracket to condition $(\mc{C})$ (asserting that the spectral curve is Lagrangian):
\begin{align*}
& \{p^n-\bm\hat{t}_2p^{n-2}-...-\bm\hat{t}_n,-\bar{p}+\bm\hat{\mu}_1+\bm\hat{\mu}_2p+...+\bm\hat{\mu}_np^{n-1}\} \\
= &\sum_{k=2}^n\left( (\bar{\partial}\!-\!\bm\hat{\mu}_2\partial\!-\!k\partial\bm\hat{\mu}_k)\bm\hat{t}_k-\sum_{l=1}^{n-k}\left((l\!+\!k)\partial\bm\hat{\mu}_{l+2}+(l\!+\!1)\bm\hat{\mu}_{l+2}\del\right)\bm\hat{t}_{k+l} \right)p^{n-k} \mod \bm\hat{t}^2, I.
\end{align*}
\end{enumerate}

Now, we are ready to conclude.
By a direct computation, we can see that modulo $\bm\hat{t}^2, \partial^2$ we can replace $\bm\hat{\alpha}_{n,n+1-l}$ by $\bm\hat{\mu}_l$ in point 2. Define $D_1(h) = h^n\partial^n-\bm\hat{t}_2h^{n-2}\partial^{n-2}-...-\bm\hat{t}_n$ and $D_2(h)=-h\bar{\partial}+\bm\hat{\mu}_1+\bm\hat{\mu}_2h\partial+...+\bm\hat{\mu}_nh^{n-1}\partial^{n-1}$. Then using 1. to 4. and computing modulo $\bm\hat{t}^2$ and $\del^2$, we get:
\begin{align*}
\sum_{k=2}^n \xi_k p^{n-k} &= \sum_{k=2}^n \xi_k (h\partial)^{n-k} \Big|_{h\partial\mapsto p} \\
&= \left(\frac{1}{h}[D_1(h),D_2(h)] \right)\Big|_{\substack{h\partial\mapsto p \\h\bar{\partial}\mapsto\bar{\partial}}} \mod \bm\hat{I}\\
&= \lim_{h \rightarrow 0}\frac{1}{h}[D_1(h),D_2(h)]\Big|_{\substack{h\partial\mapsto p \\h\bar{\partial}\mapsto\bar{\partial}}} \mod \bm\hat{I}\\
&=  \{p^n-\bm\hat{t}_2p^{n-2}-...-\bm\hat{t}_n,-\bar{p}+\bm\hat{\mu}_1+\bm\hat{\mu}_2p+...+\bm\hat{\mu}_np^{n-1}\} \mod I\\
&= \sum_{k=2}^n \left((\bar{\partial}\!-\!\bm\hat{\mu}_2\partial\!-\!k\partial\bm\hat{\mu}_k)\bm\hat{t}_k-\sum_{l=1}^{n-k}\left((l\!+\!k)\partial\bm\hat{\mu}_{l+2}+(l\!+\!1)\bm\hat{\mu}_{l+2}\del\right)\bm\hat{t}_{k+l} \right)p^{n-k}
\end{align*}
Comparing coefficients, we get the lemma.
\end{proof}

We can now give the proof of Theorem \ref{conditioncinconnection}:
\begin{proof}
From the explicit expression of $\xi_k(\l)$, we know that only derivatives, $\bm\hat{t}_k$'s and $\bm\hat{\mu}_k$'s appear. Since we are only interested in the highest term, we can replace $\bm\hat{t}_k$ by $\l^{k-1}t_k$ and $\bm\hat{\mu}_k$ by $\l^{2-k}\mu_k$. Hence, we get an expression which is a tensor, since both $t_k$ and $\mu_k$ are tensors (by Proposition \ref{highesttermtensor}). Since one term is $\bar{\partial}t_k$, we know that the highest term of $\xi_k(\l)$ is a section of $K^k\otimes \bar{K}$ and is of degree $k-1$ in $\l$.

In addition, we know that every term in $\xi_k$, apart from $\bar{\partial}t_k$, has at least one partial derivative $\partial$, which adds a $K$-factor to the tensor. The rest is thus at most of type $K^{k-1}\otimes \bar{K}$. The $\bar{K}$-factor comes from a unique $\mu_m$ in each term. Once this $\mu_m$ fixed, only partial derivatives $\partial$ and $t_k$'s contribute to the $K$-factor.

Since $t_k$ comes with a factor $\l^{k-1}$, we see that whenever there is a term with a factor $t_it_j$, the contribution in $\l$ is $\l^{i+k-2}$ which is not optimal, since $t_{i+j}$ would contribute with $\l^{i+j-1}$. In the same vein, whenever there is a term with at least two $\partial$, so that the rest is a tensor of type at most $K^{k-2}\otimes \bar{K}$, this term does not have an optimal contribution in $\l$.

Therefore, the highest term in $\xi_k(\l)$ is the same as in $\xi_k(\l) \mod \bm\hat{t}^2 \mod \partial^2$. Finally, the statement of the previous Lemma \ref{curvaturemodmod} concludes the proof of Theorem \ref{conditioncinconnection}.
\end{proof}

With the previous theorem, we now understand the global meaning of the highest terms $(\mu_k, t_k)$: the $\mu_k$ are the higher Beltrami differentials coming from the higher complex structure, whereas the $t_k$ are a cotangent vector to that higher complex structure. 
We can say that the \textit{semi-classical limit of $\mc{A}\sslash\mc{P}\sslash\Symp_0$ is $\cotang$}, which confirms the twistor space picture \ref{HK}.

The question remains how to determine the coefficients of lower degree in $\bm\hat{\mu}_k$ and $\bm\hat{t}_k$. This will be discussed in the next section.



\section{Conjectural geometric approach to Hitchin components}\label{finalstep}

In this section, we try to construct an analog to the non-abelian Hodge correspondence in our setting: the existence and uniqueness of real twistor lines. We give partial results and conjectures. Assuming the existence of real twistor lines, we prove a canonical diffeomorphism between higher complex structures and Hitchin components. 

Consider $\mc{A}(\l)=\l\Phi + A + \l^{-1}\Phi^*$ where $\Phi=\Phi_1+\Phi_2$ is given by an $n$-complex structure. Now, we look at $\mc{A}(\l)$ as a twistor line, i.e. a section of the twistor space.
We impose the reality condition $$-\mc{A}(-1/\bar{\l})^*=\mc{A}(\l).$$
Notice that $-1/\bar{\l}$ is the diametrically opposed point of $\l$ in $\C P^1$. For trivial $n$-complex structure the $*$-operator is the hermitian conjugate $M^*=M^\dagger=\bar{M}^\top$. Intrinsically, the operation $A \mapsto -A^*$ is an antiholomorphic involution which corresponds to the compact real form of $\mf{sl}_n$.
\begin{Remark}
For general higher complex structure, the real structure $*$ has to be defined in such a way that $\Symp_0$ preserves it. It might be necessary to have a hermitian structure on the bundle.
\hfill $\triangle$
\end{Remark}

\subsection{Case \texorpdfstring{$n=2$}{n=2} and \texorpdfstring{$n=3$}{n=3}}\label{n2n3}
Let us study the examples of smallest rank, those with $n=2$ and $n=3$. We work locally, so we can suppose that the $n$-complex structure is trivial, i.e. $\mu_k=0$ for $k=2, 3$. We use the standard form from subsection \ref{standard-form}.

For $n=2$, write $\Phi_1 = \left(\begin{smallmatrix} 0 & 0 \\ e^{\varphi} & 0\end{smallmatrix}\right)$, $A_1 = \left(\begin{smallmatrix} a_0 & a_1 \\ a_2 & -a_0\end{smallmatrix}\right)$ and $A_2 = -A_1^{\dagger}$. So we have 
$$\mc{A}(\l)=\begin{pmatrix} a_0 & a_1 \\ a_2+\l e^{\varphi} & -a_0\end{pmatrix}dz+ \begin{pmatrix} -\bar{a}_0 & -\bar{a}_2+\l^{-1}e^{\varphi} \\ -\bar{a}_1 & \bar{a}_0\end{pmatrix}d\bar{z}.$$
Notice that this is example \ref{examplen2} with $\mu_2=0$ and $b_1=e^{\varphi}$. The flatness equation gives 
$$ \left \{\begin{array}{cl}
a_2 e^{\varphi} &= \; 0 \\
\bar{\partial}\varphi &= \; -2\bar{a}_0 \\
\bar{\partial}a_1 &=\;  2\bar{a}_0a_1 \\
\partial \bar{a}_0+\bar{\partial}a_0 &= \; -a_1\bar{a}_1-e^{2\varphi}.
\end{array}\right. $$
The first equation gives $a_2=0$, the second $a_0=-\frac{\del \varphi}{2}$, the third is automatic once we write $a_1=t_2e^{-\varphi}$, where $t_2$ is the holomorphic quadratic differential from the affine connection (see section \ref{holodiffs-flat-affine}). Finally, the last equation gives $$\partial\bar{\partial} \varphi = e^{2\varphi} + t_2\bar{t}_2e^{-2\varphi}$$ which is the so-called $\mathbf{\cosh}$\textbf{-Gordon equation}, which is elliptic for small $t_2$. So we see that the flat connection is uniquely determined by $\mu_2=0, t_2$ and a solution to the $\cosh$-Gordon equation. More details for this case can be found in \cite{Fock}, in particular a link to minimal surface sections in $\S\times \R$.

\medskip
\noindent For $n=3$, take $\Phi_1 = \left(\begin{smallmatrix}  & & \\ c_1 & & \\ b_2 & c_2 &\end{smallmatrix}\right)$. As for $n=2$ the matrix $A_1$ is upper triangular. Thus, we get 
$$\mc{A}(\l)=\begin{pmatrix} a_0 & b_0 & c_0 \\ \l c_1 & a_1 & b_1 \\ \l b_2 & \l c_2 & a_2 \end{pmatrix}dz+ \begin{pmatrix} -\bar{a}_0 & \l^{-1}\bar{c}_1 & \l^{-1}\bar{b}_2 \\ -\bar{b}_0 & -\bar{a}_1 & \l^{-1}\bar{c}_2 \\ -\bar{c}_0 & -\bar{b}_1 & -\bar{a}_2\end{pmatrix}d\bar{z}.$$
Notice that we get the same expression as in example \ref{n3example}.
With a diagonal gauge, we can suppose $c_1=e^{\varphi_1}, c_2=e^{\varphi_2} \in \R_+$. Further, we know the expressions for the holomorphic differentials to be $t_3=c_0c_1c_2$ and $t_2=b_0c_1+b_1c_2+b_2c_0$, hence $c_0=t_3e^{-\varphi_1-\varphi_2}$ and $b_1=-e^{\varphi_1-\varphi_2}b_0-b_2t_3e^{-2\varphi_2-\varphi_1}$.

The flatness condition and the zero trace condition then give $a_0=-\frac{2}{3}\partial\varphi_1-\frac{1}{3}\partial\varphi_2$, $a_1=\frac{1}{3}\partial\varphi_1-\frac{1}{3}\partial\varphi_2$ and $a_2=-a_0-a_1$.

Let us consider the case where $t_2=t_3=0$. Then $c_0=0$ and $b_1=-e^{\varphi_1-\varphi_2}b_0$. The remaining equations of the flatness are
$$ \left \{\begin{array}{cl}
\bar{\partial}b_2 &= \; b_2(\bar{\partial}\varphi_1+\bar{\partial}\varphi_2)-\bar{b}_0(e^{\varphi_2}+e^{2\varphi_1-\varphi_2}) \\
-\bar{\partial}b_0 &= \;  b_0\bar{\partial}\varphi_1+\bar{b}_2e^{\varphi_2} \\
2\partial\bar{\partial}\varphi_1 &= \;  2e^{2\varphi_1}-e^{2\varphi_2}+b_2\bar{b}_2+b_0\bar{b}_0(2-e^{2\varphi_1-2\varphi_2}) \\
2\partial\bar{\partial}\varphi_2 &= \;  2e^{2\varphi_2}-e^{2\varphi_1}+b_2\bar{b}_2+b_0\bar{b}_0(-1+2e^{2\varphi_1-2\varphi_2}). 
\end{array}\right. $$

For $b_0=b_2=0$ we get the \textbf{Toda integrable system} for $\mf{sl}_3$. This is the same solution as the one obtained from the non-abelian Hodge correspondence applied to the principal nilpotent Higgs field.
We see that we need some extra data in order to impose $b_0=b_2=0$. The two variables $b_0$ and $b_2$ are solutions to a system of differential equations. Thus, we only need some initial conditions. 

For $t_2=0$ and $t_3\neq 0$, if we impose $b_0=b_1=b_2=0$ and $\varphi_1=\varphi_2=\varphi$, the flatness becomes \textbf{\c{T}i\c{t}eica's equation} 
\begin{equation}\label{Titeica}
2\del\delbar \varphi = e^{2\varphi}+t_3\bar{t}_3e^{-4\varphi}.
\end{equation}
From \cite{Loftin}, we know that \c{T}i\c{t}eica's equation is linked to affine spheres, minimal embeddings and Hitchin representations.

Before going to the general case, we push the similarity to Higgs bundles further by choosing a special gauge.

\subsection{Higgs gauge}\label{Higgsgauge}

Up to now, we have seen the flat connection $\mc{A}(\l)$ in two gauges. The first, which we call \textit{symmetric gauge}, is the form $\mc{A}(\l)=\l\Phi+A+\l^{-1}\Phi^*$ where $A_2=-A_1^*$ and the $*$-operator is the hermitian conjugate. The second, which we call \textit{parabolic gauge} and which in the literature is sometimes called \textit{$W$-gauge} or \textit{Drinfeld--Sokolov gauge}, is the form described in equation \eqref{paragauge} where our parameters $\tilde{t}_k(\l)$ and $\tilde{\mu}_k(\l)$ appear. The existence of parabolic gauge (see subsection \ref{existence-para-gauge}) assures that one can go from the symmetric to the parabolic gauge. In Higgs theory, there is a third gauge used, which we call \textit{Higgs gauge}, characterized by $A_2=0$ and by the fact that $\Phi_1$ is a companion matrix. 
Here we show that for trivial higher complex structure, there exists the Higgs gauge in our setting. This allows to compare even closer the Higgs bundle setting and our setting.

\medskip\noindent
We start with the existence of the Higgs gauge for trivial higher complex structure. We denote by $\mc{E}_-$ the sum of the negative simple roots, i.e. $\mc{E}_-=\left(\begin{smallmatrix} 0&&& \\ 1 &0&&\\ &\ddots &\ddots& \\ &&1& 0\end{smallmatrix}\right)$.
\begin{prop}
For $\mu=0$ and a flat connection $\l\Phi+A+\l^{-1}\Phi^*$ in symmetric gauge, there is a gauge $P$ which is lower triangular transforming $\Phi_1$ to $\mc{E}_-$ and $A_2$ to 0.
\end{prop}
\begin{proof}
The statement is equivalent to the following two equations: $$P\Phi_1=\mc{E}_-P \;\text{ and }\; PA_2-\delbar P = 0.$$
The first matrix equation allows to express all entries $p_{i,j}$ of $P$ in terms of the last row $(p_{n,k})_{1\leq k \leq n}$.

We then put $\Phi_1 = P^{-1}\mc{E}_-P$ into the flatness equation $0=\delbar\Phi_1+[A_2, \Phi_1]$. After some manipulation, we get 
$$0=[\mc{E}_-, (\delbar P)P^{-1}-PA_2P^{-1}].$$
We know that the centralizer of $\mc{E}_-$ are polynomials in $\mc{E}_-$. Hence we get $$\delbar P-PA_2=\begin{pmatrix} 0&&&\\ w_2&0&&\\ \vdots&\ddots&\ddots&\\ w_n &\cdots &w_2 &0 \end{pmatrix}P.$$
Looking at the $n$ equations given by the last row, we can choose $(p_{n,k})_{1\leq k \leq n}$ such that $w_2=...=w_n=0$. Therefore $\delbar P =PA_2$, i.e. $A_2$ is transformed to 0.
\end{proof}

In the Higgs gauge, our flat connection takes the following form:
\begin{prop}
We suppose $\mu=0$. The flat connection $\mc{A}(\l)$ in Higgs gauge is locally given by $$(\l\mc{E}_-+A)dz+\l^{-1}\mc{E}_-^*d\bar{z}$$ where the $*$-operation is given by $M^*=HM^{\dagger}H^{-1}$ for some hermitian matrix $H$. Further, we have $\tr \mc{E}_-^kA = t_{k+1}$ and $A=-(\del H) H^{-1}$.
\end{prop}
\begin{proof}
From the existence of Higgs gauge, we know that $\Phi_1=\mc{E}_-$ and $A_2=0$. Since $\mu=0$, we also have $\Phi_2=0$. A direct computation shows that if $P$ denotes the matrix from the Higgs gauge, the matrix $\Phi_1^*$ transforms to $PP^{\dagger}\mc{E}_-^{\dagger}(PP^{\dagger})^{-1}$. So $H=PP^{\dagger}$ which is indeed a hermitian matrix. 

\noindent Since $P$ is lower triangular, $t_{k+1}=\tr \Phi_1^kA_1$ transforms to $t_{k+1}=\tr \mc{E}_-^kA$.

\noindent Finally, since $A_2=P^{-1}\delbar P$ and $A_2=-A_1^{\dagger}$, we get $A_1=-(\del P^{\dagger})P^{\dagger \;-1}$ which transforms under $P$ to $A=-\del(PP^{\dagger})(PP^{\dagger})^{-1}=-(\del H)H^{-1}$.
\end{proof}

We see that $\mc{A}(\l)$ in the Higgs gauge becomes close to a Higgs bundle. But in our setting the holomorphic differentials are in $A$, and not in the Higgs field.
We illustrate the similarity for $n=2$.

\begin{example}
For $n=2$ and $\mu=0$, we have seen in the previous subsection \ref{n2n3} that in symmetric gauge, our connection reads
$$\mc{A}(\l)=\begin{pmatrix} -\frac{\del \varphi}{2} & t_2e^{-\varphi} \\ \l e^\varphi & \frac{\del \varphi}{2} \end{pmatrix}dz+\begin{pmatrix} \frac{\delbar \varphi}{2} & \l^{-1} e^{\varphi} \\ -\bar{t}_2 e^{-\varphi} & -\frac{\delbar \varphi}{2} \end{pmatrix}d\bar{z}.$$
The flatness condition is equivalent to the $\cosh$-Gordon equation $\del\delbar\varphi = e^{2\varphi}+t_2\bar{t}_2e^{-2\varphi}$ and $\delbar t_2 = 0$.

A direct computation gives the form in parabolic gauge:
$$\mc{A}(\l)=\begin{pmatrix} 0 & \bm\hat{t}_2(\l) \\ 1 & 0 \end{pmatrix}dz+\begin{pmatrix} -\frac{1}{2}\del \bm\hat{\mu}_2 & -\frac{1}{2}\del^2 \bm\hat{\mu}_2+\bm\hat{t}_2\bm\hat{\mu}_2 \\ \bm\hat{\mu}_2(\l) & \frac{1}{2}\del \bm\hat{\mu}_2 \end{pmatrix}d\bar{z}$$
where $\bm\hat{t}_2(\l)=\l t_2 + (\del \varphi)^2-\del^2\varphi$ and $\bm\hat{\mu}_2(\l)=-\l^{-1}\bar{t}_2e^{-2\varphi}$.

In Higgs gauge, we get 
$$\mc{A}(\l)=\begin{pmatrix} -\del\varphi-t_2p_2e^{-\varphi/2} & t_2 \\ \l-a_1 &  \del\varphi+t_2p_2e^{-\varphi/2}\end{pmatrix}dz+\begin{pmatrix} -\l^{-1}p_2e^{3\varphi/2} & \l^{-1} e^{2\varphi} \\ -\l^{-1} p_2^2 e^{\varphi} & \l^{-1}p_2e^{3\varphi/2} \end{pmatrix}d\bar{z}$$
where $a_1=(\del p_2+\frac{3}{2}p_2\del\varphi+t_2p_2^2e^{-\varphi/2}) e^{-\varphi/2}$ and $p_2$ comes from the matrix of the Higgs gauge and satisfies $\delbar p_2=-\bar{t}_2e^{-3\varphi/2}+p_2\frac{\delbar \varphi}{2}$. 

Finally, we can compare to the non-abelian Hodge correspondence which gives
\begin{equation}\label{nahc-connection}
\mc{A}(\l)=\begin{pmatrix} -\del\varphi & \l t_2 \\ \l &  \del\varphi\end{pmatrix}dz+\begin{pmatrix} 0 & \l^{-1} e^{2\varphi} \\ \l^{-1}\bar{t}_2e^{-2\varphi}\end{pmatrix}d\bar{z}.
\end{equation}
The flatness condition is equivalent to the $\sinh$-Gordon equation $\del\delbar\varphi=e^{2\varphi}-t_2\bar{t}_2e^{-2\varphi}$, which is elliptic, and $\delbar t_2 = 0$.
Notice that for $t_2=0$ (in both approaches), the connection \eqref{nahc-connection} is the same as our connection in the Higgs gauge (since $p_2=0$ for $t_2=0$).
\hfill $\triangle$
\end{example}

For non-trivial $n$-complex structure $\mu\neq 0$, there is no Higgs gauge. Even for $n=2$, one can check that there is no $P$ satisfying $P\Phi_1=\mc{E}_-P$ and $PA_2-\delbar P = 0$.

\subsection{General case}\label{flatconnectionlambda}
Set $t=(t_2,...,t_n)$ and $\mu=(\mu_2,...,\mu_n)$. To examine the existence of an analog to the non-abelian Hodge correspondence, we discuss the cases when $t=0$ or $\mu=0$.

\bigskip
\noindent \textbf{\textit{Case $t=0$ and $\mu=0$.}}
For the trivial structure we find the following result, generalizing the observations for $n=2$ and $n=3$ from the previous subsection \ref{n2n3}.
\begin{prop}\label{linktohiggs}
For $\Phi_2=0$ and $t=0$, the flat connection $\mc{A}(\l)$ is uniquely determined up to some finite initial data. There is a choice of initial data such that the flatness equations are equivalent to the Toda integrable system. In particular $\mc{A}(\l)$ is the same as the connection given by the non-abelian Hodge correspondence applied to a principal nilpotent Higgs field.
\end{prop}
\begin{proof}
Using Lemmas \ref{phi1lower} and \ref{a1upper}, we can write $\mc{A}_1(\l)$ in the following form: $$\mc{A}_1(\l)=a_0+a_1T+...+a_nT^n$$ where $a_i$ are diagonal matrices and $T$ is given by 
\begin{equation}\label{matrixT}
T=\begin{pmatrix} & 1 & & \\ &&\ddots & \\ &&& 1 \\ \l &&& \end{pmatrix}.
\end{equation}
We denote by $a_{i,j}$ the $j$-th entry of the diagonal matrix $a_i$ and $a_i'$ the shifted matrix with $a'_{i,j} = a_{i,j+1}$. We write $a^{(k)}$ for the shift applied $k$ times. Notice that $aT=Ta^{(n-1)}$. We can then write 
$$\mc{A}_2(\l)= a_0^*+T^{-1}a_1^*+...+T^{-n}a_n^*$$ where $a^*_{i,j}=\pm \bar{a}_{i,j}$, the sign depends on whether the coefficient comes with a $\l$ or not in $\mc{A}_2(\l)$.

By the standard form (Lemma \ref{phi1lower}) we can further impose $a_{n,i}=e^{\varphi_i}$ for $i=1,...,n-1$ and $a_{n,0}=0$ since $0=t_n=\prod_i a_{n,i}$. 
One of the flatness equations gives $\delbar a_n = a_n (a_0^{(n-1)}-a_0)$. Together with the condition that the trace is 0, we can compute $a_0$. We get
\begin{equation}\label{a0i}
a_{0,i}= \sum_{k=1}^{i-1}\frac{k}{n}\del\varphi_k-\sum_{k=i}^{n-1}\frac{n-k}{k}\del\varphi_k.
\end{equation}
The other equations give a system of differential equations in $a_1, ..., a_{n-1}$ which is quadratic. It allows the solution $a_i=0$ for all $i=1,...,n-1$. 
In that case, using a diagonal gauge $\diag (1,\lambda, ..., \lambda^{n-1})$ the connection $\mc{A}(\l)$ becomes 
\begin{equation}\label{mu0t0}
\mc{A}(\l)=\begin{pmatrix} *&&& \\ e^{\varphi_1} &* && \\ & \ddots &* & \\ && e^{\varphi_{n-1}} & * \end{pmatrix}dz+\begin{pmatrix}* & e^{\varphi_1} && \\ &* & \ddots & \\ &&* & e^{\varphi_{n-1}} \\ &&&* \end{pmatrix}d\bar{z}
\end{equation}
where on the diagonals are the $a_{0,i}$ and $-\bar{a}_{0,i}$ given by equation \eqref{a0i}.
This is precisely the form of the Toda system. It is known that the Hitchin equations for a principal nilpotent Higgs field are the Toda equations for $\mf{sl}_n$ (see \cite{AF}, proposition 3.1).
\end{proof}
Notice that in particular the gauge class of the connection $\mc{A}(\l)$ is independent of $\l\in \C^*$ (i.e. we have a variation of Hodge structure). This is an intrinsic property which might be used to fix the initial data.

Putting \eqref{mu0t0} in parabolic gauge, we get the following explicit formula for our coordinates $\tilde{t}(\l)$ and $\tilde{\mu}(\l)$ (see also proposition 3.1 and 4.4 in \cite{AF}):
\begin{prop}
For $\mu=0$ and $t=0$, one can choose initial conditions such that $\tilde{\mu}_k(\l)=0$ and $\tilde{t}_k(\l)=w_k$ for all $k$, where the $w_k$ are given by $\det (\del-A_1)=\prod_i(\del-a_{0,i}) = \del^n+w_2\del^{n-2}+...+w_n$ (a ``Miura transform''). Furthermore, $A_1$ is diagonal given by equation \eqref{a0i} and the parabolic gauge is upper triangular.
\end{prop}

\bigskip
\noindent \textbf{\textit{Case $t=0$.}}
We get the following result:
\begin{prop}\label{monodromyreal}
For $t=0$, the connection $\mc{A}(\l)$ is determined by the flatness condition and by some initial conditions. Its monodromy is in $\PSL_n(\R)$.
\end{prop}
The idea of the proof is the following: locally, one can trivialize the higher complex structure, so we are led to $\mu=0$ and $t=0$. Thus $\mc{A}(\l)$ is given by the non-abelian Hodge correspondence and we can apply Hitchin's strategy to prove real monodromy, which is a local argument.

\begin{proof}
By Theorem \ref{actionsymponlambdaconn}, we know that we can locally render $\Phi_2=0$ by trivializing the $n$-complex structure. Thus we can choose the initial conditions such that $\mc{A}(\l)$ is given by the non-abelian Hodge correspondence applied to the nilpotent Higgs field $\Phi_1$ (see Proposition \ref{linktohiggs}).

In \cite{Hit.1}, Hitchin constructs a real form $\tau$, associated with the split real form, which for $\mf{sl}_n$ is given by a rotation of the matrix by $180$ degrees composed with complex conjugation. He shows that $\tau^*\Phi_1=\Phi_1^*$ and that $\tau^*\mc{A}(\l)=\mc{A}(\l)$. This is a local statement, therefore the monodromy of $\mc{A}(\l)$ has to be in the fixed point set of $\tau$, so in $\PSL_n(\R)$.
\end{proof}

\bigskip
\noindent \textbf{\textit{Case $\mu=0$.}}
For trivial $n$-complex structure, the standard form from Lemmas \ref{phi1lower} and \ref{a1upper} allow to consider $\mc{A}(\l)$ as an affine connection with special properties. We denote by $\mc{L}(\mf{sl}_n)$ the loop algebra of $\mf{sl}_n$. It is defined by $\mc{L}(\mf{sl}_n) = \mf{sl}_n \otimes \C[\lambda,\lambda^{-1}]$, the space of Laurent polynomials with matrix coefficients.
There is another way to think of elements of $\mc{L}(\mf{sl}_n)$: as an infinite periodic matrix $(M_{i,j})_{i,j \in \mathbb{Z}}$ with $M_{i,j}=M_{i+n,j+n}$ and finite width (i.e. $M_{i,j}=0$ for all $\left| i+j \right|$ big enough).
The isomorphism is given as follows: to $\sum_{i=-N}^N N_i\l^i$ we associate $M_{i,j}=(N_{k_j-k_i})_{r_i,r_j}$ where $i=k_in+r_i$ and $j=k_jn+r_j$ are the Euclidean divisions of $i$ and $j$ by $n$ (so $0\leq r_i,r_j <n$), see also figure \ref{affine-matrix}.
\begin{figure}[h]
\centering
\begin{tikzpicture}[scale=1]

\draw (0,0)--(3,0);
\draw (0,1)--(3,1);
\draw (0,2)--(3,2);
\draw (0,3)--(3,3);
\draw (0,0)--(0,3);
\draw (1,0)--(1,3);
\draw (2,0)--(2,3);
\draw (3,0)--(3,3);

\draw (-0.3,1.5) node {$\hdots$};
\draw (3.3,1.5) node {$\hdots$};
\draw (1.5,-0.3) node {$\vdots$};
\draw (1.5,3.3) node {$\vdots$};
\draw (-0.3,3.3) node {$\ddots$};
\draw (3.3,-0.3) node {$\ddots$};

\draw [dashed, gray] (-0.3,2.3)--(2.3,-0.3);
\draw [dashed, gray] (0.7,3.3)--(3.3,0.7);

\draw [white, fill=white] (0.5,1.5) circle (0.3);
\draw [white, fill=white] (1.5,0.5) circle (0.3);
\draw [white, fill=white] (1.5,2.5) circle (0.3);
\draw [white, fill=white] (2.5,1.5) circle (0.3);

\draw (1.5,1.5) node {$N_0$};
\draw (2.5,0.5) node {$N_0$};
\draw (0.5,2.5) node {$N_0$};
\draw (1.5,2.5) node {$N_1$};
\draw (2.5,1.5) node {$N_1$};
\draw (2.5,2.5) node {$N_2$};
\draw (0.5,1.5) node {$N_{-1}$};
\draw (1.5,0.5) node {$N_{-1}$};
\draw (0.5,0.5) node {$N_{-2}$};

\draw [domain=-20:20] plot ({0.2+3.5*cos(\x)},{1.4+5.5*sin(\x)});
\draw [domain=159:199] plot ({2.7+3.5*cos(\x)},{1.4+5.5*sin(\x)});

\end{tikzpicture}

\caption{Affine matrix as infinite periodic matrix}
\label{affine-matrix}
\end{figure}
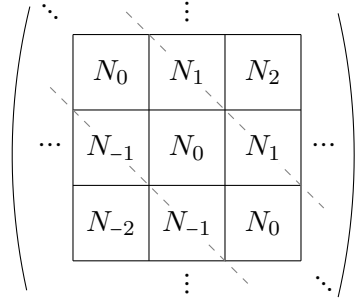

In the second viewpoint, a connection $\l\Phi+A+\l^{-1}\Phi^*$ with $\Phi_1$ lower triangular, $\Phi_2=0$ and thus $A_1$ upper triangular, is precisely an infinite matrix with period $n$ and width $n$ (shown in figure \ref{affine-matrix} by dashed lines). The $(1,0)$-part $\mc{A}_1(\l)$ is upper triangular ($\Phi_1$ is lower triangular but $\l\Phi_1$ is upper triangular in the infinite matrix) and the $(0,1)$-part $\mc{A}_2(\l)$ is lower triangular.

Thus, the flatness of $\mc{A}(\l)$ is a \textit{generalized Toda system}, replacing the tridiagonal property by ``width equal to periodicity''. For $t_i=0$ for $i=2,...,n-1$ but $t_n\neq 0$, we should get the usual affine Toda system for $\mc{L}(\mf{sl}_n)$ as described in \cite{Baraglia}. 

\begin{Remark}
In order to describe $h$-connections, we can include parameters into $\mc{L}(\mf{sl}_n)$ by considering its central extension $\widehat{\mf{sl}}_n$ or central coextension.
\hfill $\triangle$
\end{Remark}

Since for $t=0$ we get an elliptic system, the system stays elliptic for at least small $t\neq 0$, since ellipticity is an open condition (Cauchy-Kowalewskaya theorem). So the generalized Toda system can be solved for small $t$.

The study of this generalized Toda system is subject of future research.


\bigskip
\noindent \textbf{\textit{General case.}}
For $\mu\neq 0$ and $t\neq 0$, the system is still elliptic at least for small $t$, since it is for $t=0$.
We should get a generalized Toda system with differentials $t_k$ satisfying the higher holomorphicity condition $(\mc{C})$.

We conjecture that the connection $\mc{A}(\l)=\l\Phi+A+\l^{-1}\Phi^*$ is uniquely determined by $\mu$ and $t$.
To be more precise:

\begin{conj}\label{nahc}
Given an element $[(\mu_k, t_k)]\in T^*\bm\hat{\mc{T}}^n$ and some finite extra data (initial conditions to differential equations), there is a unique (up to unitary gauge) flat connection $\mc{A}(\l)=\l\Phi+A+\l^{-1}\Phi^*$ satisfying
\begin{enumerate}
\item Locally, $\Phi=\Phi_1 dz+\Phi_2 d\bar{z}$ with $\Phi_1$ principal nilpotent and $\Phi_2=\mu_2\Phi_1+...+\mu_n\Phi_1^{n-1}$
\item $-\mc{A}(-1/\bar{\l})^*=\mc{A}(\l)$ (reality condition)
\item $t_k = \tr \Phi_1^{k-1}A_1$.
\end{enumerate}
In addition, if $t_k=0$ for all $k$, then the monodromy of $\mc{A}(\l)$ is in $\PSL_n(\R)$.
\end{conj}

Another formulation for the conjecture, which would generalize lemma 3 in \cite{Fock}, is the following:
\begin{conj}
Let $\mc{A}(\l)$ be a family of $\PSL_n(\C)$-connections on $\Sigma$ satisfying
\begin{enumerate}
\item Flatness: $d\mc{A}(\l)+\mc{A}(\l)\wedge \mc{A}(\l)=0$
\item Polynomiality: $\mc{A}(\l)=\l\Phi+A+\l^{-1}\Phi^*$
\item Non-degeneracy: $\Phi_1$ is principal nilpotent
\item Reality: $-\mc{A}(-1/\bar{\l})^*=\mc{A}(\l)$.
\end{enumerate}
Then there is an open dense subset of such connections, which are, modulo a choice of unitary gauge and a choice of a higher complex structure on $\Sigma$, parameterized by holomorphic differentials $t_k\in H^0(K^k)$ for $k=2,...,n$ and some finite data.
\end{conj}


The open subset consists of those connections where $\Phi_1$ is principal nilpotent and such that $\mu_2\bar{\mu}_2 \neq 1$.

\subsection{Link to Hitchin's component}

Assuming the existence and uniqueness of real twistor lines, we get the desired link to Hitchin's component:
\begin{thm}\label{mainthmm}
If Conjecture \ref{nahc} holds true, there is a canonical diffeomorphism between our moduli space $\T^n$ and Hitchin's component $\mc{T}^n$.
\end{thm}
\begin{proof}
With Conjecture \ref{nahc} we get a canonical way to associate a flat connection $\mc{A}(\l=1)$ to a point in $\cotang$. By Proposition \ref{monodromyreal} the monodromy of $\mc{A}(\l)$ for $t=0$ is in $\PSL_n(\R)$. Following Hitchin's argument from theorem 7.5  in \cite{Hit.1}, we prove that the zero-section in $\cotang$ where $t=0$ describes a connected component of $\Rep(\pi_1(\S), \PSL_n(\R))$. 

Since $\T^n$ is closed in $\cotang$, the image of the map $s:\T^n \rightarrow \Rep(\pi_1(\S), \PSL_n(\R))$ is a closed submanifold.
Furthermore both spaces have the same dimension by Theorem \ref{mainresultncomplex}.
Therefore the image of $s$ is an open and closed submanifold, i.e. a connected component.

Finally, for $\mu=0$ we get the same connection $\mc{A}(\l)$ as by the non-abelian Hodge correspondence of the principal nilpotent Higgs field. So the component described by $\T^n$ and Hitchin's component $\mc{T}^n$ coincide.
\end{proof}

Notice that the map between $\mc{T}^n$ and $\T^n$ is something like an exponential map. For $n=2$, Hitchin's description of Teichmüller space is exactly via the exponential map identifying a fiber of the cotangent bundle $T^*_\mu\mc{T}^2$ to $\mc{T}^2$. 

For $n=2$ we can explicitly give the correspondence between Hitchin's description of Teichmüller space and the description with the Beltrami differential. In \cite{Hit.2} Hitchin parameterizes $\mc{T}^2$ with holomorphic quadratic differentials as follows: given a hyperbolic metrix $\rho_0$ and $t\in H^0(K^2)$, one associates the metric $\rho:=\rho_0 + t\bar{t}/\rho_0+t+\bar{t}$. Locally write $\rho_0=e^\varphi dz d\bar{z}$. Then $\rho$ is of constant curvature -1 if $\varphi$ satisfies the $\sinh$-Gordon equation 
$$\del\delbar \varphi = e^\varphi-t\bar{t}e^{-\varphi}$$ which is elliptic for all $t$.
To find $\mu$, we have to write $\rho$ as $fdwd\bar{w}$ with $f$ some positive real function and $dw=dz+\mu d\bar{z}$. After some computation, we get a quadratic relation: 
$$t\mu^2-(e^\varphi+t\bar{t}e^{-\varphi})\mu+\bar{t}=0.$$

For $t=0$ and $\mu_k=0$ for $k=3,...,n$ (but $\mu_2\neq 0$), we conjecture that the monodromy of $\mc{A}(\l)$ is $n$-fuchsian, i.e. the composition of a fuchsian map with the principal map $\pi_1(\S)\rightarrow \PSL_2(\R)\rightarrow \PSL_n(\R)$. It is known that in the Hitchin parametrization, the fuchsian locus of $\mc{T}^n$ is described by $t_k=0$ for $k=3,...,n$.

\begin{coro}
Hitchin's component has a natural complex structure. Further, there is a natural action of the mapping class group on it, preserving the complex structure.
\end{coro}
The first statement follows from Theorem \ref{mainresultncomplex} since we explicitly know the cotangent space at a point. The second simply follows by the description of Hitchin's component as moduli space of some geometric structure on the surface. Labourie describes this action in \cite{Lab.3} and shows that it is properly discontinuous using cross ratios.



\cleardoublepage
\vspace*{3cm}
\part{Generalization to other Lie groups}\label{part3}
\vspace*{2cm}
\begin{flushright}
\textit{En un mot, pour tirer la loi de l'expérience, il faut généraliser ;}

\textit{c'est une nécessité qui s'impose à l'observateur le plus circonspect.}

Henri Poincaré, La valeur de la science
\end{flushright}
\vspace*{1cm}

\noindent Nigel Hitchin's construction for components in character varieties works for any adjoint group $G$ associated with a split form of a simple complex Lie algebra. The construction of $n$-complex structures gives an approach to $\PSL_n(\R)$-Hitchin components.
In this part, we describe a generalization of the concepts from part \ref{part1} which might give a geometric approach to $G$-Hitchin components. We define a so-called $\g$-complex structure using a generalization of the punctual Hilbert scheme associated with a simple complex Lie algebra $\g$. At the beginning we treat an arbitrary simple Lie algebra $\g$, but at some point we restrict attention to classical Lie algebras.


To read this part, you only need to know basic facts about the punctual Hilbert scheme (subsections \ref{hilbdef} and \ref{hilbmatrixviewpoint}). In order to see the analogy between the $\g$-complex structure and the $n$-complex structure (for which $\g=\mf{sl}_n$), you should know the construction of the higher complex structure, its moduli space and the spectral curve (section \ref{Highercomplexsection} and subsection \ref{sspectral}).

\medskip \noindent
Our inspiration how to define these new objects is twofold: on the one hand we use the various descriptions of the punctual Hilbert scheme, especially the matrix viewpoint, in order to generalize to an arbitrary $\g$. On the other, we got inspiration from Hitchin's original paper \cite{Hit.1} (section 5) where he starts with a principal nilpotent element and deforms it into an element of a principal slice, a generalized companion matrix. 

We signal to the reader that our definition of the $\g$-Hilbert scheme might get changed in the future, since as it is defined now, it is a non-Hausdorff space. There should be a way to get a nice topological space, using some adapted GIT quotient. This possible modification will not affect the $\g$-complex structure since only the regular part of the $\g$-Hilbert scheme plays a role in its construction.

\medskip
\noindent We first introduce the $\g$-Hilbert scheme and some interesting subspaces and analyze their properties. We then proceed to the construction of the $\g$-complex structure and define higher diffeomorphisms of type $\g$. The local theory is trivial, like for the $n$-complex structure (with a subtlety for $\g$ of type $D_n$). The moduli space of $\g$-complex structures $\bm\hat{\mc{T}}_{\g}$ enjoys similar properties as in the case $\g=\mf{sl}_n$, and so it shares several properties with Hitchin's component. In particular, both are contractible and there is a copy of Teichmüller space inside them. We then define a spectral curve which is equipped with some extra structure depending on $\g$ giving the spectral data described in \cite{Hit3}. 

\medskip
\noindent Most of the material of this part is published in the preprint \cite{Thomas}. A collection of facts on regular elements in semisimple Lie algebras which we need in this part can be found in appendix \ref{appendix:B}.

\vspace*{\fill}

\cleardoublepage

\section{Generalized punctual Hilbert scheme}\label{g-Hilbertscheme}

In this section, we generalize the punctual Hilbert scheme to a $\g$-Hilbert scheme and explore the properties of the new object. We analyze in detail its regular part in the case of a classical Lie algebra.
We use lots of analogies to the punctual Hilbert scheme and its properties from sections \ref{Hilbertscheme} and \ref{moreonhilbertschemes}.

\subsection{Definitions and first properties}

The punctual Hilbert scheme $\Hilb^n(\C^2)$ has several descriptions:
\begin{itemize}
\item as a space of ideals (the \textit{idealic viewpoint})
\item as a desingularization of the configuration space $\h^2/W$ for $\g=\mf{gl}_n$
\item as a space of commuting matrices (the \textit{matrix viewpoint}). 
\end{itemize}
It is the matrix viewpoint which will be generalized. So let us quickly recall it here:
$$\Hilb^n(\C^2) \cong \{(A,B) \in \mf{gl}_n^2 \mid [A,B]=0, (A,B) \text{ admits a cyclic vector}\} / GL_n.$$

The main difficulty is to find an intrinsic condition which generalizes the existence of a cyclic vector. Here is our proposal:

\begin{definition}\label{def-g-hilb}
The generalized punctual Hilbert scheme, or $\mathbf {\g}$\textbf{-Hilbert scheme}, denoted by $\Hilb(\g)$, is defined by
$$\Hilb(\g) = \{(A,B) \in \g^2 \mid [A,B]=0, \dim Z(A,B) = \rk \g\} /G$$
where $Z(A,B)$ denotes the common centralizer of $A$ and $B$, i.e. the set of elements $C \in \g$ which commute with both $A$ and $B$.
\end{definition}

The condition on the dimension of the common centralizer does not come from nowhere: Proposition \ref{doublecomm} of Appendix \ref{appendix:B} shows that $\rk \g$ is the minimal possible dimension for the centralizer of a commuting pair.
Define the \textit{commuting variety} by $$\Comm(\g)=\{(A,B)\in \g^2 \mid [A,B]=0\}.$$ The $\g$-Hilbert scheme is the set of all regular points of $\Comm(\g)$ modulo $G$.

\begin{Remark}
Ginzburg has defined the notion of a principal nilpotent pair in \cite{Ginzburg}, which is more restrictive than ours. He calls ``nil-pairs'' elements of our $\g$-Hilbert scheme, but it seems that he does not investigate them.
\hfill $\triangle$
\end{Remark}

Let us give two examples of elements in the $\g$-Hilbert scheme:
\begin{example}
Let $A \in \g$ be a regular element. Then by a theorem of Kostant (see \ref{thmKost}), its centralizer $Z(A)$ is abelian. So for any $B\in Z(A)$, we have $Z(A) \subset Z(B)$, thus $Z(A,B) = Z(A) \cap Z(B) = Z(A)$ is of dimension $\rk \g$. Therefore $[(A,B)] \in \Hilb(\g)$.

If $A$ is principal nilpotent, then $B\in Z(A)$ is also nilpotent. So $[(A,B)]\in \Hilb_0(\g)$, the zero-fiber defined below.

If $B=0$ then $[(A,0)]$ is in $\Hilb(\g)$ iff $A$ is regular.
\hfill $\triangle$
\end{example}

\begin{example}\label{Young}
Let $(A,B)$ be a commuting pair of matrices in $\mf{sl}_n$ admitting a cyclic vector, i.e. an element of the reduced Hilbert scheme. 
One way to get such a pair is the following construction: take a Young diagram (our convention is to put the origin in the upper left corner as for matrices) with $n$ boxes (see figure \ref{Youngdiag}). Associate to each box a vector of a basis of $\C^n$. Define $A$ to be the matrix which translates to the right, i.e. sends a vector to the vector in the box to the right or to 0 if there is none. Let $B$ be the matrix which translates to the bottom. Then $A$ and $B$ clearly commute and are nilpotent.
In Proposition \ref{cycliccentralizer} below, we show that $Z(A,B)$ is of minimal dimension in that case.
\hfill $\triangle$
\end{example}

\begin{figure}[h]
\centering
\includegraphics[height=2cm]{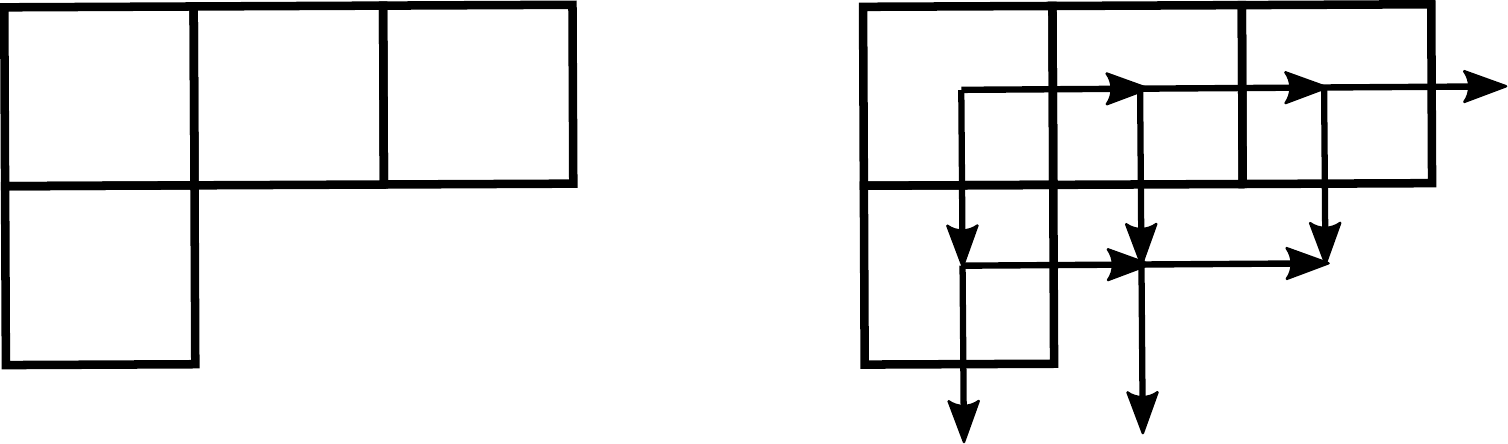}

\caption{Young diagram and commuting nilpotent matrices}
\label{Youngdiag}
\end{figure}

Guided by these examples, we define several subsets of the $\g$-Hilbert scheme and explore their relations.
First, we define the zero-fiber and the regular part which both play a mayor role in the definition of a $\g$-complex structure. We also define the cyclic part, which is not intrinsically defined since it uses a representation of $\g$. The cyclic part will be used to define a map to a space of ideals, getting a generalization of the original description of the punctual Hilbert scheme by ideals.

\begin{definition}\label{partshilb}
The \textbf{zero-fiber} of the $\g$-Hilbert scheme is defined by $$\Hilb_0(\g)=\{[(A,B)] \in \Hilb(\g) \mid A \text{ and } B \text{ nilpotent}\}.$$

We define the \textbf{regular part} of the $\g$-Hilbert scheme, denoted by $\Hilb^{reg}(\g)$, to be those conjugacy classes $[(A,B)]$ in which $A$ or $B$ is a regular element of $\g$.

Finally for classical $\g$, let $\rho$ denote the natural representation of $\g$ (i.e. $\mf{sl}_n \subset \mf{gl}_n, \mf{so}_n \subset \mf{gl}_n$ and $\mf{sp}_{2n} \subset \mf{gl}_{2n}$). Define the \textbf{cyclic part} of the $\g$-Hilbert scheme by 
$$\Hilb^{cycl}(\g) = \{(A,B) \in \g^2 \mid [A,B]=0, (\rho(A),\rho(B)) \text{ admits a cyclic vector}\}/G.$$
\end{definition}

\begin{Remark}
In the definition of the cyclic part, it would be more natural to consider the adjoint representation, but even in the case of $\mf{sl}_2$, this would give a map to a space of ideals, which is not the one of $\Hilb^2_{red}(\C^2)$.

Instead of the standard representation, one could also use a non-trivial representation of minimal dimension, which for classical $\g$ is always the standard representation, apart from types $D_3$ and $D_4$. For $D_3$, the two spin representations are of minimal dimension, and they give the exceptional isomorphism between $\mf{so}_6$ and $\mf{sl}_4$. For type $D_4$, there are three representations of minimal dimension, the standard one and the two spin representations. All of them are linked by outer automorphisms coming from the symmetry of the Dynkin diagram. Thus the cyclic part of the $\mf{so}_8$-Hilbert scheme is the same for all three representations. 

Taking the representation of minimal dimension has the additional advantage to be well-defined for all $\g$. This possibility has to be explored further.
\hfill $\triangle$
\end{Remark}

The first relation between the various Hilbert schemes is the inclusion of the cyclic part in the $\g$-Hilbert scheme, which justifies the name ``cyclic part'':
\begin{prop}\label{cycliccentralizer}
For $\g$ of classical type, we have $\Hilb^{cycl}(\g) \subset \Hilb(\g)$.
\end{prop}
\begin{proof}
Recall $\rho$ the natural representation of $\g$ on $\C^m$. For simplicity, we write $A$ instead of $\rho(A)$ here.

Let $(A,B) \in \g^2$ admitting a cyclic vector $v$. Let $C \in Z(A,B)$. Then $C$ is a polynomial in $A$ and $B$. Indeed, there is $P\in \C[x,y]$ such that $Cv = P(A,B)v$. Since $C$ commutes with $A$ and $B$, we get for any polynomial $Q$ that $CQ(A,B)v = Q(A,B)Cv = Q(A,B)P(A,B)v= P(A,B)Q(A,B)v$. Thus $C=P(A,B)$.

Therefore the common centralizer of $(A,B)$ in $\mf{gl}_m$ is $\C[A,B]/I$ where $I=\{P\in \C[x,y] \mid P(A,B) = 0\}$. We know from section \ref{Hilbertscheme} that $I$ is of codimension $m$ since $(A,B)$ admits a cyclic vector. Further, we have $Z(A,B) = Z_{\mf{gl}_m}(A,B) \cap \g$. One can easily check that for $\g$ of type $A_n$, a polynomial $P(A,B)$ is in $\g$ iff its constant term has a specific form, given by the other coefficients (to ensure trace zero). For type $B_n, C_n$ and $D_n$, $P(A,B)$ is in $\g$ iff $P$ is odd. One checks in each case that the dimension of $Z(A,B)$ equals the rank of $\g$.
\end{proof}

In general, the inclusion of the cyclic Hilbert scheme is strict as shows the following example:
\begin{example}
Consider $A=\left(\begin{smallmatrix} 0 & 1 & 0 \\ 0 & 0 & 0 \\ 0 & 0 & 0 \end{smallmatrix} \right)$ and $B=\left(\begin{smallmatrix} 0 & 0 & 1 \\ 0 & 0 & 0 \\ 0 & 0 & 0 \end{smallmatrix}\right)$ in $\mf{sl}_3$. One easily checks that the pair $(A,B)$ does not admit any cyclic vector, but that their common centralizer is of dimension 2. So $[(A,B)] \in \Hilb(\mf{sl}_3)\backslash \Hilb^{cycl}(\mf{sl}_3)$.
\hfill $\triangle$
\end{example}
This example will be used in subsection \ref{topology} to show that $\Hilb(\g)$ is not Hausdorff.

In general, there is no link between regular and cyclic part. Example \ref{Young} shows that cyclic elements are not always regular and the following example shows that regular element are not always cyclic:
\begin{example}\label{regnotcyclic}
For $\g$ of type $D_n$, let $f$ be a principal nilpotent element. Then one checks that $[(f,0)] \in \Hilb(\mf{so}_{2n})$ is regular but not cyclic (see also subsection \ref{Dn}).
\hfill $\triangle$
\end{example}

Let us turn to the regular part. It turns out that if one fixes a \textit{principal slice} $f+Z(e)$ in $\g$ (see Appendix \ref{appendix:B}), there is a preferred representative for regular classes:
\begin{prop}\label{paramit}
Any class $[(A,B)] \in \Hilb^{reg}(\g)$ where $A$ is regular can uniquely be conjugated to $(A \in f+Z(e), B \in Z(A))$.
\end{prop}
\begin{proof}
By the property of the principal slice, there is a unique conjugate of $A$ which is in the principal slice $f+Z(e)$. Denote still by $A$ and $B$ these conjugates. The only thing to show is that $B$ is unique which is done in the next lemma. 
\end{proof}

\begin{lemma}
If $A\in \g$ is regular, $g\in G$ such that $Ad_g(A)=A$ and $B\in Z(A)$, then $Ad_g(B)=B$.
\end{lemma}
\begin{proof}
By Kostant's theorem \ref{thmKost}, we know that $Z(A)$ is abelian. So the infinitesimal version of the lemma is true. We conclude by the connectedness of the stabilizer of $A$, given by the next lemma.
\end{proof}
\begin{lemma}
For a regular element $A \in \g$, its stabilizer $\Stab(A) = \{g \in G \mid Ad_g(A)=A\}$ in the adjoint group $G$ is connected.
\end{lemma}
\begin{proof}
Decompose $A$ into Jordan form: $A=A_s+A_n$ with $A_s$ semisimple, $A_n$ nilpotent and $[A_s,A_n]=0$. So $A_n\in Z(A_s)$. The structure of the centralizer $Z(A_s)$ is well-known: it is a direct sum of a Cartan $\h$ containing $A_s$ with all root spaces $\g_{\alpha}$ where $\alpha$ is a root such that $\alpha(A_s)=0$. It is also known that $Z(A_s)$ is reductive, so a direct sum $Z(A_s)=\mf{c}\oplus \g_s$ where $\mf{c}$ is the center and $\g_s$ is the semisimple part of $Z(A_s)$. In particular the center $\mf{c}$ is included in $\h$. So $A_n \in \g_s$ since $A_n$ is nilpotent. Denote by $G_s$ the Lie group with trivial center with Lie algebra $\g_s$.

We know that $A$ is regular iff $A_n$ is principal nilpotent in $\g_s$ (see \cite{Kost2}, proposition 0.4). 
We also know that the $G$-equivariant fundamental group of the orbit of $A$ (which is the space of connected components of $\Stab(A)$) is the same as the $\Stab(A_s)$-equivariant fundamental group of the $\Stab(A_s)$-orbit of $A_n$ (see Proposition 6.1.8. of \cite{Coll} adapted to the adjoint group). In other words, the connected components of $\Stab_G(A)$ are the same as the connected components of $\Stab_{G_s}(A_n)$ since the $\Stab(A_s)$-orbit of $A_n$ is equal to the $G_s$-orbit of $A_n$.

So we are reduced to the principal nilpotent case. Using the classification of simple Lie algebras, one can explicitly check in Collingwood-McGovern's book \cite{Coll} the tables 6.1.6. for classical $\g$ and the tables at the end of chapter 8 for exceptional $\g$ that the stabilizer of a principal nilpotent element is always connected.
\end{proof}
\begin{Remark}
It is surprising that the last lemma has never been stated (at least not to our knowledge). It would be interesting to find a direct argument, without using the classification of simple Lie algebras.
\hfill $\triangle$
\end{Remark}

\begin{coro}\label{hilbg0affine}
The regular zero-fiber $\Hilb^{reg}_0(\g) = \Hilb^{reg}(\g)\cap \Hilb_0(\g)$ is an affine variety of dimension $\rk \g$.
\end{coro}
\begin{proof}
This directly follows from the previous proposition \ref{paramit} using the fact that $A \in f+Z(e)$ is nilpotent iff $A=f$. So $\Hilb^{reg}_0(\g)$ is described by $Z(f)$ which is a vector space of dimension $\rk \g$.
\end{proof}

We know that both the regular and the cyclic part are in general strictly included in the $\g$-Hilbert scheme. But they are dense subspaces:
\begin{prop}\label{density}
The regular part $\Hilb^{reg}(\g)$ is dense in $\Hilb(\g)$. For classical $\g$, the cyclic part is also dense in $\Hilb(\g)$.
\end{prop}
\begin{proof}
By a theorem of Richardson (see \ref{Richardson}), the set of semisimple commuting pairs is dense in the commuting variety $\Comm(\g)$. So the set of semisimple regular elements is also dense in $\Comm(\g)$. Passing to the quotient by $G$, we get that the classes of semisimple regular pairs are dense in $\Hilb(\g)$ since $\Hilb(\g)\subset \Comm(\g)/G$ and all semisimple regular pairs are in $\Hilb(\g)$. Since the semisimple regular pairs are in the regular part, we get the density of $\Hilb^{reg}(\g)$ in $\Hilb(\g)$.

For classical $\g$, we have the same argument for the cyclic part since semisimple regular pairs are cyclic.
\end{proof}

To end the section, we state an analog of Kostant's theorem about abelian subalgebras of centralizers:
\begin{prop}
For any commuting pair $(A,B) \in \Comm(\g)$, there is an abelian subspace of dimension $\rk \g$ in the common centralizer $Z(A,B)$.
\end{prop}
\begin{proof}
The proof is completely analogous to Kostant's proof for Theorem \ref{thmKost}: we use a limit argument. Let $(A_n,B_n)$ be a sequence of regular semisimple pairs converging to $(A,B)$ (exists since regular semisimple pairs are dense). We know that $Z(A_n,B_n)$ is a $\rk \g$-dimensional abelian subspace of $\g$. Since the Grassmannian $Gr(\rk \g, \dim \g)$ is compact, there is a subsequence of $Z(A_n,B_n)$ which converges. It is easy to prove that the limit is included in $Z(A,B)$ and is commutative.
\end{proof}
\begin{coro}
For $[(A,B)] \in \Hilb(\g)$, the common centralizer $Z(A,B)$ is abelian.
\end{coro}

In the following sections, we generalize as far as possible the other viewpoints of the usual Hilbert scheme (resolution of configuration space and idealic viewpoint) to our setting.

\subsection{Chow map}

We want to generalize the Chow map, which goes from $\Hilb^n(\C^2)$ to the configuration space (see subsection \ref{resofsing}).

Fix a Cartan subalgebra $\h$ in $\g$. Recall the Jordan decomposition in a semisimple Lie algebra: for $x \in \g$, there is a unique pair $(x_s, x_n)$ with $x=x_s+x_n$, $x_s$ semisimple, $x_n$ nilpotent and $[x_s,x_n]=0$. For a semisimple element $x$, denote by $x^*$ a conjugate of $x$ in the Cartan $\h$ (unique up to $W$-action).

The \textbf{Chow map}  $ch: \Hilb(\g) \rightarrow \h^2/W$ is defined by $$ch([(A,B)]) = [(A_s^*,B_s^*)]$$ where the brackets $[.]$ denotes the equivalence class. For semisimple regular pairs, this map corresponds to a simultaneous diagonalization.

\begin{prop}
The Chow map $ch$ is well-defined and continuous.
\end{prop}
\begin{proof}
Since $[A,B]=0$, we also have $[A_s,B_s]=0$ by a simultaneous Jordan decomposition in a faithful representation. Hence there is a conjugate of the pair $(A_s,B_s)$ which lies in $\h^2$. Since the adjoint action of $G$ on $\g$ restricts to the $W$-action on $\h$, the map $ch$ is well-defined.

The map $x\mapsto x_s^*$ is continuous which simply follows from the continuity of eigenvalues. Hence the Chow map is continuous as well.
\end{proof}
\begin{Remark}
The Jordan decomposition $x\mapsto (x_s,x_n)$ is not continuous at all, since semisimple elements are dense in $\g$ for which we have $x_n=0$ and for all non-semisimple elements we have $x_n\neq 0$. But the map $x\mapsto x_s$ is continuous.
\hfill $\triangle$
\end{Remark}

This map permits to think of a generic element of $\Hilb(\g)$ as a point in $\h^2/W$, or via a representation of $\g$ on $\C^m$, as a set of $m$ points in $\C^2$ with a certain symmetry. For $\g=\mf{sl}_n$ for example, these are $n$ points with barycenter 0.

Since $\Hilb(\g)$ as a topological space is not Hausdorff (see subsection \ref{topology}), it cannot be a non-singular variety. Nevertheless we conjecture the following:

\begin{conj}\label{conj1}
There is a modified version of $\Hilb(\g)$, identifying some points, which is a smooth projective variety such that the Chow morphism is a resolution of singularities.
\end{conj}

\subsection{Idealic map}\label{idealic}

For $\g$ of classical type, we can associate to any regular element of the $\g$-Hilbert scheme an ideal, which we call \textit{idealic map}. In this subsection, $\g$ is a classical Lie algebra. Recall the natural representation $\rho$ of $\g$ on $\C^m$ (see Definition \ref{partshilb}). We write $A$ instead of $\rho(A)$.

We wish to define a map like in Proposition \ref{bijhilbert} between commuting matrices and ideals: \begin{equation}\label{ideal}[(A,B)] \mapsto I(A,B)=\{P\in \C[x,y] \mid P(A,B) = 0\}.\end{equation} 
If $[(A,B)] \in \Hilb^{cycl}(\g)$ is cyclic, this ideal is of codimension $m$. But if the pair is not cyclic, there is no reason why the codimension should be $m$. In fact, there are examples for $\g$ of type $D_n$ where the codimension is smaller. 

We wish the idealic map to be continuous, so $I$ has to be of constant codimension. A strategy would be to define the idealic map $I$ on the cyclic part $\Hilb^{cycl}(\g)$ (which is dense by Proposition \ref{density}) and to extend it by continuity. Unfortunately, the map can not be extended in a continuous way as the following example shows:
\begin{example}\label{idealnotcont}
Take $\g$ of type $D_n$. Denote by $f$ a principal nilpotent element. The pair $[(f,0)] \in \Hilb(\mf{so}_{2n})$ is not cyclic (see example \ref{regnotcyclic}). Using the matrix $S$ defined in equation \eqref{matrixS}, we can approach $(f,0)$ by $(f,tS)$ or by $(f+tS^\top,0)$ for $t\in\C^{\times}$ going to 0. These pairs are all cyclic. In the first case, the ideal is $I=\langle x^{2n-1}, xy, y^2=t^2x^{2n-2}\rangle$ which converges as $t$ goes to 0 to $\langle x^{2n-1}, xy, y^2\rangle$. In the second case, the ideal is $I=\langle x^{2n}+t^2, y \rangle$ converging to $\langle x^{2n},y \rangle$.
\hfill $\triangle$
\end{example}

Because of this difficulty, our strategy is to define a space of ideals $I_{\g}(\C^2)$, then a map $\Hilb^{cycl}(\g)\rightarrow I_{\g}(\C^2)$ and to extend it over the regular part $\Hilb^{reg}(\g)$ (in a non-continuous way). The last step is only necessary for $\g$ of type $D_n$ since for the other classical types the regular part is included in the cyclic part as we will see in the sequel. The extension for $D_n$ will be defined \textit{ad hoc} in subsection \ref{Dn}.

The previous section taught us to think of a generic element of $\Hilb(\g)$ as a $m$-tuple of points in $\C^2$ invariant under the Weyl group $W$. For type $A_n$ this means that the barycenter of the points is the origin. For the other classical types, this means that the set of points is symmetric with respect to the origin. Thus the defining ideal of these points is also invariant under the action of $W$. Hence the following definition.

\begin{definition}
We define the \textbf{space of ideals} of type $\g$, denoted by $I_{\g}(\C^2)$, to be the set of ideals in $\C[x,y]$ which are of codimension $m$ and $W$-invariant. For type $B_n, C_n$ and $D_n$ this means that $I$ is invariant under $(x,y)\mapsto (-x,-y)$.
\end{definition}

The map $I: \Hilb^{cycl}(\g)\rightarrow I_{\g}(\C^2)$ given by equation \eqref{ideal} above is well-defined. Indeed, the codimension is $m$ by cyclicity and the ideal is $W$-invariant since this is a closed condition and it is true on the dense subset of regular semisimple pairs.

Notice that $I_{\g}(\C^2)$ is the same for $\g$ of type $C_n$ or $D_n$. But we will see that the idealic map $I$ has not the same image in the two cases.
We will also see that for $\g$ of type $A_n, B_n$ or $C_n$ the idealic map is injective. But for type $D_n$ it is not (it is generically 2 to 1). This comes from the fact that the Weyl group acting on the generic $2n$ points, coming in $n$ pairs $(P_i, P_{i+1}=-P_i)$, cannot exchange $P_1$ and $P_2$ while leaving all other points fixed.

As for the usual Hilbert scheme, there is a direct link between the idealic map and the Chow morphism:
\begin{prop}
The Chow map $ch$ is the composition of the idealic map with the map which associates to an ideal its support, seen as an element of $\h^2/W$: 
$$ch([(A,B)]) = \supp I(A,B).$$
\end{prop}
\begin{proof}
The statement is true on regular semisimple pairs which is a dense subset. For $\g$ of type $A_n$, $B_n$ and $C_n$, it follows by continuity of both the Chow map and the idealic map. For $D_n$, our definition of the idealic map is to pick one of the various possible limits. In particular, the support of the ideal is still given by the Chow map.
\end{proof}

\subsection{Morphisms}\label{mu2}
In this subsection, we analyze the functorial behavior of the $\g$-Hilbert scheme. In particular we construct two maps linked to the zero-fiber of the Hilbert scheme of $\mf{sl}_2$ which will lead in the construction of the moduli space $\bm\hat{\mathcal{T}}_{\g}$ of $\g$-complex structures to maps from and to Teichm\"uller space.

Let $\psi: \g_1 \rightarrow \g_2$ be a morphism of Lie algebras. For $[(A,B)] \in \Hilb(\g_1)$, we can associate $[(\psi(A), \psi(B))]$ which is a well-defined map to $\Comm(\g_2)/G_2$. But there is no reason why $\dim Z(\psi(A), \psi(B))$ should be minimal. 

If we accept Conjecture \ref{conj1}, that there is a modified version of the $\g$-Hilbert scheme which is a resolution of $\h^2/W$, we have a functorial behavior:
\begin{prop}
Assuming Conjecture \ref{conj1}, there is an induced map $\Hilb(\g_1) \rightarrow \Hilb(\g_2)$.
\end{prop}
\begin{proof}
Choose Cartan subalgebras $\h_1$ and $\h_2$ such that $\psi(\h_1)=\h_2$. Consider the composition $\h_1^2 \rightarrow \h_2^2 \rightarrow \h_2^2/W_2$ using $\psi$ for the first arrow. Since $\psi$ induces a homomorphism between the Weyl groups, we can factor the composition to get a map $\h_1^2/W_1 \rightarrow \h_2^2/W_2$. Finally, consider the composition $\Hilb(g_1) \rightarrow \h_1^2/W_1 \rightarrow \h_2^2/W_2$ where the first arrow comes from the minimal resolution. This is a continuous map and by the universal property of a minimal resolution, the map lifts to $\Hilb(\g_1) \rightarrow \Hilb(\g_2)$.
\end{proof}

Let us study this induced map in the case of the reduced Hilbert scheme $\Hilb^n_{red}(\C^2)$, which is a minimal resolution (see subsection \ref{resofsing}). Take $\psi:\mf{sl}_m\rightarrow \mf{sl}_n$  inducing a map $\Hilb^m_{red}(\C^2)\rightarrow \Hilb^n_{red}(\C^2)$. In the matrix viewpoint, this map is not given by $[(\psi(A),\psi(B))]$. Consider for example the map $\psi: \mf{sl}_2\rightarrow \mf{sl}_4$ given on the standard generators $(e,f,h)$ of $\mf{sl}_2$ by 
$$\psi(e)=\left(\begin{smallmatrix} 0&&&1 \\&0&&\\&&0&\\&&&0\end{smallmatrix}\right), \psi(f)=\left(\begin{smallmatrix} 0&&& \\&0&&\\&&0&\\1&&&0\end{smallmatrix}\right) \text{ and } \psi(h)=\left(\begin{smallmatrix} 1&&& \\&0&&\\&&0&\\&&&-1\end{smallmatrix}\right).$$
 The element $[(h,0)] \in \Hilb^2_{red}(\C^2)$ corresponds to the ideal $I=\langle x^2-1,y\rangle$ which through $\psi$ goes to $\langle x^4-x^2,y\rangle$ which in turn gives the matrices $[(M,0)]$ where $M=\left(\begin{smallmatrix} 1&&& \\&0&1&\\&0&0&\\&&&-1\end{smallmatrix}\right)$.
This is not $[(\psi(h), \psi(0))]$.
It would be interesting to describe the induced map in the matrix viewpoint.

Despite this complication, there are two cases where a map between $\g$-Hilbert schemes exists naturally.

The first one is linked to the principal map $\psi: \mf{sl}_2 \rightarrow \g$ which induces a map 
\begin{equation}\label{teichcopy} \Hilb(\mf{sl}_2) \rightarrow \Hilb^{reg}(\g). \end{equation} 
Indeed, any non-zero element of $\mf{sl_2}$ is regular and cyclic. So if $[(A,B)] \in \Hilb(\mf{sl}_2)$ such that $A$ is non-zero, there is by Proposition \ref{paramit} a unique representative $(f+te,B\in Z(e+tf))$ where $(e,f,h)$ denotes the standard generators of $\mf{sl}_2$ and $t\in \C$. So the image is $[(\psi(f)+t\psi(e),\psi(B))]$. Since $(\psi(e), \psi(f), \psi(h))$ is a principal $\mf{sl}_2$-triple (property of the principal map), we know that $\psi(f)+t\psi(e)$ is in the principal slice, thus it is regular, so we land in $\Hilb^{reg}(\g)$.

The second one is a sort of inverse map to the first one, but only on the level of the zero-fiber. Given $[(A,B)]\in \Hilb^{reg}_0(\g)$ where $A$ is regular, there is a principal $\mf{sl}_2$-subalgebra $\mc{S}$ with $A$ as nilpotent element. There is no reason why $B$ should be in $\mc{S}$ but there is a ``best approximation'' in the following sense:

\begin{prop}\label{mu2prop}
Let $A$ be a principal nilpotent element and $B \in Z(A)$. Then there is a unique $\mu_2 \in \C$ such that $B-\mu_2 A$ is not regular.
\end{prop}
\begin{proof}
The strategy of the proof is to use Proposition \ref{prinnilp} of the appendix which characterizes principal nilpotent elements $x$ as those nilpotent elements whose values $\alpha(x)$ for all simple roots $\alpha$ are non-zero. So the proposition is equivalent to the statement that $\alpha_1(B) = \alpha_2(B)$ for all simple roots $\alpha_1$ and $\alpha_2$.

Let $R$ be a root system in $\h^*$ and denote by $R_+$ and $R_s$ the positive and respectively the simple roots.
We can conjugate $A$ to the element given by $\alpha(A)=1$ if $\alpha \in R_s$ and $\alpha(A)=0$ otherwise.

For two simple roots $\alpha_1$ and $\alpha_2$ such that $\alpha_1+\alpha_2 \in R$, using $[A,B]=0$ we get:
$$0=(\alpha_1+\alpha_2)([A,B])=\alpha_1(A)\alpha_2(B)-\alpha_2(A)\alpha_1(B) = (\alpha_2-\alpha_1)(B).$$

Since $\g$ is simple, its Dynkin diagram is connected, so $\alpha_1(B) = \alpha_2(B)$ for all simple roots. The common value $\mu_2$ gives the unique complex number such that $B-\mu_2 A$ is not regular.
\end{proof}

With this proposition, we can now define a map \begin{equation}\label{mu} \mu: \Hilb^{reg}_0(\g) \rightarrow \Hilb_0(\mf{sl}_2)\end{equation} given by $\mu([(A,B)])=[(e,\mu_2 e)]$ or $[(\mu_2 e, e)]$ depending whether $A$ or $B$ is regular. 

An equivalent way to define the map $\mu$ is the following: we can use the previous proposition \ref{mu2prop} to show that the centralizer $Z(A)$ of a principal nilpotent element is a direct product $$Z(A) = \Span(A) \times Z(A)^{irreg}$$ where $Z(A)^{irreg}$ denotes the irregular elements of $Z(A)$. The map $\mu$ is nothing but the projection to the first factor.

\begin{Remark}
We can describe the regular part of the $\g$-Hilbert scheme $\Hilb^{reg}(\g)$ as those classes $[(A,B)]$ such that $\Span(A,B)$ intersects the regular part $\g^{reg}$ non-trivially. This description is more symmetric since it does not prefer $A$ or $B$. From Proposition \ref{mu2prop} we see that the intersection of $\Span(A,B)$ with $\g^{reg}$ is the whole two-dimensional $\Span(A,B)$ from which we have to take out a line. Hence, the intersection has two components. 
\hfill $\triangle$
\end{Remark}

In the following subsections, we study the regular part $\Hilb^{reg}(\g)$ and its zero-fiber case by case for classical $\g$.

\subsection{Case \texorpdfstring{$A_n$}{An}}
Consider $\g=\mf{sl}_n$ (of type $A_{n-1}$). We describe first $\Hilb^{reg}_0(\mf{sl}_n)$, its idealic map and then $\Hilb^{reg}(\mf{sl}_n)$ using Proposition \ref{paramit}.

Fix the following principal nilpotent element (with 1 on the diagonal line just under the main diagonal):
$$f=\begin{pmatrix} & & & \\ 1& & & \\ &\ddots & & \\ & &1 & \end{pmatrix}.$$

This element $f$ is cyclic, so we know from \ref{cycliccentralizer} that the centralizer is given by polynomials: 
$Z(f)=\{\mu_2 f+\mu_3 f^2+...+\mu_n f^{n-1}\}$.
So an element of $\Hilb^{reg}_0(\mf{sl}_n)$ can be represented by $(f,Q(f))$ where $Q$ is a polynomial without constant term of degree at most $n-1$. The coefficients $\mu_i$ are called \textit{higher Beltrami coefficients}.

Since here we have $\Hilb^{reg}_0(\mf{sl}_n) \subset \Hilb^{cycl}(\mf{sl}_n)$ (already $f$ is cyclic), the idealic map is given by $$I(f,Q(f))=\{P\in \C[x,y] \mid P(f,Q(f))=0\} = \langle x^n, -y+Q(x) \rangle.$$ We recognize the big cell of the zero-fiber of the punctual Hilbert scheme.

To describe the whole regular part $\Hilb^{reg}(\mf{sl}_n)$, we take the following principal slice given by companion matrices:
$$\begin{pmatrix} & & & t_n\\ 1& & & \vdots\\ &\ddots & &t_2\\ & &1 & \end{pmatrix}.$$

Let $A$ be a matrix of companion type. Notice that the characteristic polynomial of a companion matrix is given by $x^n+t_2x^{n-2}+...+t_n$. 
Since $A$ is still cyclic, its centralizer consists of polynomials in $A$ with constant term determined by the other coefficients (in order to ensure trace zero). Thus, a representative of $\Hilb^{reg}(\mf{sl}_n)$ is given by $(A,B=Q(A))$.

The idealic map is thus given by 
$$I(A,B)=\langle x^n+t_2x^{n-2}+...+t_n, -y+\mu_1+\mu_2x+...+\mu_nx^{n-1} \rangle$$ where $\mu_1$ is given by $\mu_1=\sum_{k=2}^{n-1}\frac{k}{n}t_k\mu_{k+1} \mod t^2$ (see equation \eqref{mu1value}).
One recognizes the big cell of the reduced punctual Hilbert scheme. Notice that the idealic map is injective here.


\subsection{Case \texorpdfstring{$B_n$}{Bn}}\label{Bn}

Consider $\g=\mf{so}_{2n+1}$. Represent $\g$ on $\C^{2n+1}$ using the metric given by $g(e_i,e_j)=\delta_{i,n-j}$ (where $e_i$ are standard vectors), i.e. $g=\left(\begin{smallmatrix} & & 1\\ & \udots & \\ 1& & \end{smallmatrix}\right)$.
A matrix $A$ is in $\g$ iff $\sigma(A)=-A$ where $\sigma$ is the involution given by a reflection along the anti-diagonal. In other words $A \in \g$ iff $A_{i,j}=A_{n+1-j,n+1-i}$ for all $i,j$.

We fix the following principal nilpotent element:
$$f=\begin{pmatrix} &&&&&& \\ 1&&&&&& \\ &\ddots &&&&&\\ &&1&&&&\\ &&&-1&&& \\ &&&& \ddots &&\\ &&&&&-1&\end{pmatrix}.$$

This element is cyclic, so its centralizer by \ref{cycliccentralizer} consists of all odd polynomials: $Z(f)=\{\mu_2f+\mu_4f^3+...+\mu_{2n}f^{2n-1}\}$. A representative of $\Hilb^{reg}_0(\g)$ is thus given by $(f,Q(f))$ where $Q$ is an odd polynomial of degree at most $2n-1$. The coefficients $\mu_{2i}$ are called the higher Beltrami coefficients for $B_n$.

A principal slice is given by $$\begin{pmatrix} &&&&&t_{2n}& \\ 1&&&&\udots && -t_{2n}\\ &\ddots &&t_2&&\udots& \\  &&1&&-t_2&&\\ &&&-1&&& \\ &&&& \ddots &&\\ &&&&&-1&\end{pmatrix}.$$
Let $A$ be a matrix of this type. Its characteristic polynomial is given by $x^{2n+1}-2t_2x^{2n-1}+2t_4x^{2n-3}\pm ... +(-1)^n\times 2t_{2n}x$. So we can really think of the principal slice as a generalized companion matrix. Changing slightly $t_{2i}$ we can get rid of signs and the factor 2 in the characteristic polynomial, which we will do in the sequel.

The matrix $A$ is still cyclic, so we have $\Hilb^{reg}(\g) \subset \Hilb^{cycl}(\g)$. A representative of $\Hilb^{reg}(\g)$ is given by $(A,B=Q(A))$ where $Q$ is still an odd polynomial of degree at most $2n-1$. The idealic map is then given by 
$$I(A,B)=\langle x^{2n+1}+t_2x^{2n-1}+t_4x^{2n-3}+...+t_{2n}x, -y+\mu_2x+\mu_4x^3+...+\mu_{2n}x^{2n-1}\rangle.$$
This ideal is invariant under the map $(x,y)\mapsto (-x,-y)$. This is not surprising since a generic element of the $\g$-Hilbert scheme is a pair of two diagonal matrices which for $\mf{so}_{2n+1}$ are of the form $\diag(x_1,...,x_n,0,-x_n,...,-x_1)$ and $\diag(y_1,...,y_n,0,-y_n,...,-y_1)$. So they can be thought of as $2n+1$ points in $\C^2$ with one point being the origin and the other points being symmetric with respect to the origin. This set is invariant under the map $-\id$, so is its defining ideal.


The next type, $C_n$, is quite similar to $B_n$.

\subsection{Case \texorpdfstring{$C_n$}{Cn}}
 
Let $\g=\mf{sp}_{2n}$. We use the symplectic structure $\omega=\sum_i e_i\wedge e_{n+i}$ of $\C^{2n}$ to represent $\g$. So a matrix 
$$\left(\begin{array}{c|c}
  A  & B \\
\hline
  C & D
\end{array}\right)$$
is in $\g$ iff $D=-A^\top$ and $B$ and $C$ are symmetric matrices.

Fix the principal nilpotent by

$$f=\left(\begin{array}{cccc|cccc}
  &&&&&&& \\
	1&&&&&&&\\
	&\ddots&&&&&&\\
	&&1&&&&&\\ \hline
	&&&&&-1&&\\
	&&&&&&\ddots&\\
	&&&&&&&-1\\
	&&&1&&&&
\end{array}\right).$$

This element is cyclic, so its centralizer is given by odd polynomials: $Z(f)=\{\mu_2f+\mu_4f^3+...+\mu_{2n}f^{2n-1}\}$. As for $B_n$ we call the $\mu_{2i}$ higher Beltrami coefficients.

A principal slice is given by 

$$\left(\begin{array}{cccc|cccc}
  &&&& t_{2n}&&& \\
	1&&&&&t_{2n-2}&&\\
	&\ddots&&&&&\ddots&\\
	&&1&&&&&t_2\\ \hline
	&&&&&-1&&\\
	&&&&&&\ddots&\\
	&&&&&&&-1\\
	&&&1&&&&
\end{array}\right).$$

Let $A$ be an element of this form. Its characteristic polynomial is given by $x^{2n}-t_2x^{2n-2}+t_4x^{2n-4}\pm...+(-1)^nt_{2n}$. By changing signs in the $t_{2i}$ we can omit the minus signs in the characteristic polynomial.

The matrix $A$ is still cyclic so a representative of $\Hilb^{reg}(\mf{sp}_{2n})$ is given by $(A,B=Q(A))$ where $Q$ is an odd polynomial of degree at most $2n-1$.

The idealic map reads $$I(A,B)= \langle x^{2n}+t_2x^{2n-2}+t_4x^{2n-4}+...+t_{2n}, -y+\mu_2x+\mu_4x^3+...+\mu_{2n}x^{2n-1} \rangle.$$

As for $B_n$, this ideal is invariant under $-\id$ which comes from the fact that two diagonal matrices in $\mf{sp}_{2n}$ give $2n$ points in $\C^2$ which are symmetric with respect to the origin.


The last classical type, $D_n$, has some surprises.

\subsection{Case \texorpdfstring{$D_n$}{Dn}}\label{Dn}

Let $\g=\mf{so}_{2n}$. We use the same representation as for $B_n$. 

Fix the following principal nilpotent element:
$$f=\begin{pmatrix}
&&&&&&& \\
1&&&&&&& \\
&\ddots&&&&&& \\
&&1&&&&& \\
&&1&0&&&& \\
&&&-1&-1&&& \\
&&&&&\ddots && \\
&&&&&&-1&\\
\end{pmatrix}.$$

This elements is not cyclic, since $f^{2n-1}=0$. A direct computation shows that $Z(f)=\{\mu_2f+\mu_4f^3+...+\mu_{2n-2}f^{2n-3}\} \cup \{\sigma_n S\}$ where $S$ is the matrix 

\begin{equation}\label{matrixS}
S=\left(\begin{array}{@{}ccc|ccc@{}}
  &&&&& \\
	&&&&&\\
	1&&&&& \\ \hline
	-1&&&&&\\
	&&&&&\\
	&&1&-1&&
\end{array}\right).
\end{equation}

We can give an intrinsic definition of the matrix $S$: let $R$ be a root system and $v_{\alpha}$ be a root vector in $\g$ for the root $\alpha \in R$. Choose a basis $\alpha_1, ..., \alpha_n$ of $R$ (the simple roots) such that $\alpha_{n-1}$ and $\alpha_n$ correspond to the two non-adjacent vertices in the Dynkin diagram of $D_n$ (see figure \ref{Dynkin}). We can choose $f$ to be $\sum_i v_{\alpha_i}$. The matrix $S$ is then given by 
$$S=v_{\alpha_1+...+\alpha_{n-1}} \pm v_{\alpha_1+...+\alpha_{n-2}+\alpha_n}$$ where the sign depends on the choice of the root vectors.
\begin{figure}[h]
\centering
\includegraphics[height=2cm]{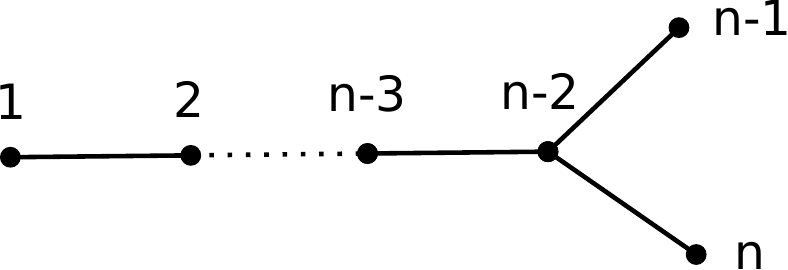}

\caption{Dynkin diagram for $D_n$}
\label{Dynkin}
\end{figure}
\vspace*{0.3cm}

A representative of $\Hilb^{reg}_0(\mf{so}_{2n})$ is given by $(A=f,B=Q(f)+\sigma_nS)$ where $Q$ is an odd polynomial of degree at most $2n-3$. Such a pair is cyclic iff $\sigma_n \neq 0$. 

Let us compute the ideal in the cyclic case. One easily checks that $fS=Sf$ and that $S^2=2f^{2n-2}$. Hence for $B=\mu_2f+...+\mu_{2n-2}f^{2n-2}+\sigma_nS$, we get $AB=fB=\mu_2f^2+...+\mu_{2n-2}f^{2n-2}$ and $B^2=(\mu_2f+...+\mu_{2n-2}f^{2n-3})^2+2\sigma_n^2f^{2n-2}$. Hence, the idealic map is given by
$$I(A,B)=\langle x^{2n-1}, xy=\mu_2x^2+\mu_4x^4+...+\mu_{2n-2}x^{2n-2}, y^2=\nu_2x^2+\nu_4x^4+...+\nu_{2n-2}x^{2n-2} \rangle$$
where $\nu_{2k}=\sum_{i=1}^{k}\mu_{2i}\mu_{2k+2-2i}$ for $k=1,...,n-2$ and $\nu_{2n-2}=2\sigma_n^2+\sum_{i=1}^{n-1}\mu_{2i}\mu_{2n-2i}$. So we see that $(\mu_2, \mu_4, ..., \mu_{2n-2}, \nu_{2n-2})$ is a set of independent variables which we call higher Beltrami differentials for $D_n$. We call $\sigma_n$ a higher Beltrami differential as well.
If $\sigma_n=0$, we define the idealic map to be the continuous extension of the above ideal which is still of the same form.
\begin{Remark}
We have seen in example \ref{idealnotcont} that inside $\Hilb^{cycl}(\mf{so}_{2n})$ there is no well-defined continuous extension of the idealic map. But inside the zero-fiber, the limit is unique.
\hfill $\triangle$
\end{Remark}

The Hilbert scheme is covered with charts indexed by partitions (see subsection \ref{structurehilbscheme}).
The chart in which $I$ is written corresponds to the partition $2n=(2n-1)+1$ which we write also $[2n-1,1]$. In fact, this is the highest partition of $2n$ of type $D_n$ (see \cite{Coll}, chapter 5 for special types of partitions).

A principal slice is given by

$$\left(\begin{array}{@{}cccc|cccc@{}}
  &&&\tau_n&-\tau_n&&t_{2n-2}& \\
	1&&&&&\udots&&-t_{2n-2}\\
	&\ddots&&t_2&t_2& &\udots&\\ 
	&&1&&&-t_2&&\tau_n\\ \hline
	&&1&0&&-t_2&&-\tau_n\\
	&&&-1&-1&&&\\
	&&&&&\ddots&&\\
	&&&&&&-1&
\end{array}\right).$$

Notice that the matrix for $\tau_n$ is $S^\top$.
Let $A$ be a matrix of this type. Its characteristic polynomial is given by $$\chi(A)=x^{2n}-4t_2x^{2n-2}+4t_4x^{2n-4}\pm... +(-1)^{n-1}\times 4t_{2n-2}x^2+(-1)^n\tau_n^2.$$ By changing signs and factors in $t_{2i}$ and $\tau_n$, we can omit signs and the factor 4 in the characteristic polynomial.

One can compute that the minimal polynomial of $A$ is equal to the characteristic polynomial iff $\tau_n\neq 0$. So $A$ is cyclic iff $\tau_n\neq 0$ (by Proposition \ref{regularsln}). In that case, the centralizer consists of all odd polynomials in $A$ of degree at most $2n-1$. If $\tau_n=0$, the centralizer is given by $$Z(A)=\{\mu_2A+\mu_4A^3+...+\mu_{2n-2}A^{2n-3}\} \cup \{\sigma_n S_t\}$$ where the matrix $S_t$ is given by $S_t= S+t_{2n-2}S^\top$. 
The minimal polynomial is given by $\chi(x)/x$ (which is a polynomial since $\tau_n=0$).

The pair $(A,B)$ is cyclic iff either $\tau_n \neq 0$ or $\tau_n=0$ and $\sigma_n \neq 0$. 
In the first case, the idealic map is given by
$$I=\langle x^{2n}+t_2x^{2n-2}+t_4x^{2n-4}+...+t_{2n-2}x^2+\tau_n^2, -y+\mu_2x+\mu_4x^3+...+\mu_{2n}x^{2n-1}\rangle.$$
In the second case, we need three generators for the ideal, like for the zero-fiber. We can compute that
\begin{align*}
I(A,B)= \; \langle x^{2n-1} &= u_2x+u_4x^3+...+u_{2n-2}x^{2n-3}+uy,\\
  xy &=  v_0+v_2x^2+...+v_{2n-2}x^{2n-2}, \\
 y^2 &=  w_0+w_2x^2+...+w_{2n-2}x^{2n-2}\rangle
\end{align*}
where the coordinates can be chosen to be $$(u_2,u_4,...,u_{2n-2},u,v_2,...,v_{2n-2},w_{2n-2})$$ i.e. all the other variables are functions of these. These are Haiman coordinates as explained in figure \ref{Haimancoo}. In particular, we can see that the  coordinates $u$ and $\nu_{2n-2}$ are canonically conjugated.

The second ideal is in the chart corresponding to the partition $[2n-1,1]$ whereas the first corresponds to the trivial partition $[2n]$. If $u \neq 0$ we can write the second ideal in the first chart, i.e. perform a coordinate change in the Hilbert scheme. The link between the coordinates is given by 
$$
\left \{ \begin{array}{cl}
\tau_n^2 = uv_0 \\
\mu_{2n}=\frac{1}{u} \\
\mu_{2k} = -\frac{u_{2k}}{u} &\text{ for } 1\leq k < n\\
t_{2k} = u_{2n-2k}+uv_{2n-2k} & \text{ for } 1\leq k \leq n-1.
\end{array}\right.
$$

A regular pair $[(A,B)]$ which is not cyclic has both $\tau_n$ and $\sigma_n$ equal to 0. In that case, we define the idealic map $I(A,B)$ to be the limit of $I(A,B+tS_t)$ for $t\in\C$ goes to 0. So we stay in a chart associated to the partition $[2n-1,1]$.

Notice that the map from $\Hilb^{reg}(\g)$ to the space of ideals $I_{\g}(\C^2)$ is not injective, since for $\tau_n$ and $-\tau_n$ we get the same ideal. Even in the zero-fiber the map is not injective, since $\sigma_n$ and $-\sigma_n$ give the same ideal. In addition, the map is not surjective either. Indeed the ideal $I=\langle x^5-y, xy,y^2 \rangle \in I_{\g}(\C^2)$ is not in the image since with the notation above we have $v_0=0$ and $u\neq 0$. Changing the chart, we can compute that $\tau_n^2=uv_0 = 0$. But for a matrix in $\Hilb^{reg}(\g)$ with $\tau_n=0$ we get $u=0$.

\begin{Remark}
In the usual Hilbert scheme, there is only one cell of maximal dimension. Comparing type $C_n$ and type $D_n$, we see that the zero-fiber of $$\{I \text{ ideal of }\C[x,y] \mid \codim I=2n, I \text{ invariant under } -\id\}$$ has two components of maximal dimension, $\Hilb^{reg}_0(\mf{sp}_{2n})$ and $\Hilb^{reg}_0(\mf{so}_{2n})$.
\hfill $\triangle$
\end{Remark}

\begin{Remark}
We notice the following analog to Higgs bundles: the pair $[(f,0)] \in \Hilb(\mf{so}_{2n})$ corresponds to the Higgs field given by $\Phi = f$ on the bundle $V=K^2\oplus K\oplus K^0 \oplus K^{-2} \oplus K^{-1}\oplus K^0$. This Higgs bundle $(V,\Phi)$ is not stable, only polystable. This could explain why the idealic map can not be continuously extended to $[(f,0)]$.
\hfill $\triangle$
\end{Remark}


\subsection{Topology of \texorpdfstring{$\g$}{g}-Hilbert schemes*}\label{topology}

It is clear that $\Hilb(\g)$ is a topological space, as a quotient of a subset of $\g^2$.
In this section, we explore this topology of $\Hilb(\g)$, especially for $\g=\mf{sl}_n$. We then formulate some conjectures on its general structure.

For $\g=\mf{sl}_2$, every non-zero element $A\in \g$ is regular and cyclic. Since the centralizer of the pair $(0,0)$ is all of $\mf{sl}_2$, this pair is not in $\Hilb(\mf{sl}_2)$. Thus we have $\Hilb(\mf{sl}_2) = \Hilb^{cycl}(\mf{sl}_2) = \Hilb^2_{red}(\C^2)$ which is a smooth projective variety.

For $\g=\mf{sl}_3$, a detailed analysis, putting $A$ into Jordan normal form, shows that $(A,B)$ has minimal centralizer and is not cyclic iff it is conjugated to a pair $P_1(b) = \left(\left(\begin{smallmatrix} 0&1&0 \\ 0&0&0 \\ 0&0&0\end{smallmatrix}\right), \left(\begin{smallmatrix} 0&b&1 \\ 0&0&0 \\ 0&0&0\end{smallmatrix}\right)\right)$. So 
$$\Hilb(\mf{sl}_3) = \Hilb^3_{red}(\C^2) \cup \{P_1(b) \mid b\in\C\}.$$
At first sight, the topology seems to be a smooth variety (the reduced Hilbert scheme) and a complex line. But a closer look shows that each point of the extra line is infinitesimally close to a point in the variety, meaning that these two points cannot be separated by open sets, infringing the Hausdorff property. The pair $P_1(b)$ is infinitesimally close to $P_2(b) = \left(\left(\begin{smallmatrix} 0&1&0 \\ 0&0&0 \\ 0&0&0\end{smallmatrix}\right),\left( \begin{smallmatrix} 0&b&0 \\ 0&0&0 \\ 0&1&0\end{smallmatrix}\right)\right)$. Indeed any neighborhood of the first pair $P_1(b)$ contains $\left(\left(\begin{smallmatrix} 0&1&0 \\ 0&0&0 \\ 0&0&0\end{smallmatrix}\right), \left(\begin{smallmatrix} 0&b&1 \\ 0&0&0 \\ 0&s&0\end{smallmatrix}\right)\right)$ for some small $s\in \C$ which is conjugated to $\left(\left(\begin{smallmatrix} 0&1&0 \\ 0&0&0 \\ 0&0&0\end{smallmatrix}\right),\left( \begin{smallmatrix} 0&b&s \\ 0&0&0 \\ 0&1&0\end{smallmatrix}\right)\right)$ which lies in a neighborhood of the second pair$P_2(b)$.
Since the idealic map is continuous and for $\mf{sl}_n$ injective on the cyclic part, there cannot be another point of the cyclic part which is infinitesimally close to the first pair $P_1(b)$. Finally, two elements of the extra line can be separated by open sets.
Hence, the space $\Hilb(\mf{sl}_3)$ is obtained from a smooth variety by adding ``double points'' (here in the sense of infinitesimally close points) along a complex line.

Since the idealic map is injective on the cyclic part $\Hilb^{cycl}(\mf{sl}_n)$, the same analysis holds for $\mf{sl}_n$, i.e. $\Hilb(\mf{sl}_n)$ is obtained from a smooth variety (the reduced Hilbert scheme) by adding double points.

There should exist a procedure, like a GIT quotient, giving a modified $\g$-Hilbert scheme which is a Hausdorff space. The GIT quotient does not apply here since $\{(A,B)\in \g^2 \mid [A,B]=0, \dim Z(A,B)=\rk \g\}$ is not a closed variety. In the language of GIT quotients, the pairs $P_1$ and $P_2$ above are both semistable, but there is no polystable element in their closure. 

To give a feeling about what happens, consider the action of $\R_{>0}$ on $\R^2 \backslash \{(0,0)\}$ given by $\lambda.(x_1, x_2)=(\lambda x_1, \lambda^{-1} x_2)$. The orbits are drawn in figure \ref{nonhaus}.
The quotient space is a set of two lines $L_1$ and $L_2$ with origins $O_1$ and $O_2$ together with two extra points $O_3$ and $O_4$ such that the pairs $(O_1, O_3), (O_1, O_4), (O_2, O_3)$ and $(O_2, O_4)$ are infinitesimally close points (the four points $O_i$ correspond to the four half-axis). In the figure, the dashed lines indicate infinitesimally close points.
From the GIT perspective, all points are semistable (take the constant function 1), the four half-axis are semistable and all other orbits are stable. The orbits of the half-axis are closed in $\R^2 \backslash \{(0,0)\}$ so they should be polystable, but in the quotient the points are still infinitesimally close.
\begin{figure}[h]
\centering
\includegraphics[height=4cm]{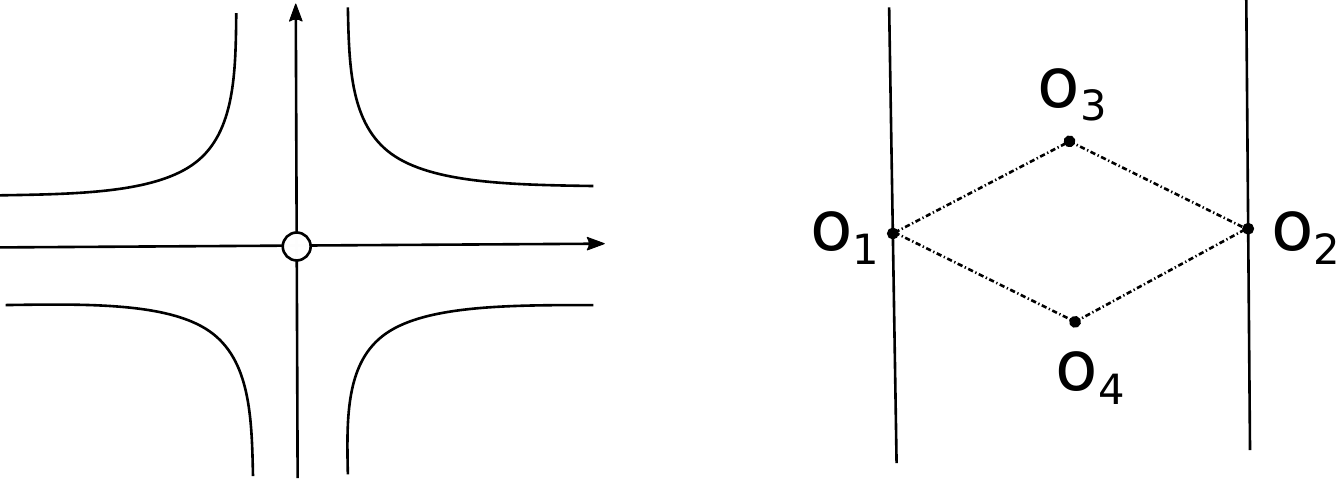}

\caption{Non-Hausdorff quotient}
\label{nonhaus}
\end{figure}

We conjecture the following:
\begin{conj}
There is a generalized GIT quotient procedure identifying infinitesimally close points in $\Hilb(\g)$, giving a modified $\g$-Hilbert scheme which is Hausdorff. 
\end{conj}
In particular one should find the reduced Hilbert scheme for $\g=\mf{sl}_n$. See also Conjecture \ref{conj1} for a modified $\g$-Hilbert scheme as a resolution of $\h^2/W$.

Assume a smooth version of the $\g$-Hilbert scheme exists. In the $\mf{sl}_n$-case the reduced Hilbert scheme is covered with charts parameterized by partitions of $n$, which also parameterizes nilpotent orbits of $\mf{sl}_n$. For $\g$ of classical type, the nilpotent orbits are parameterized by special partitions (see \cite{Coll}, chapter 5). In general, we conjecture the following for the zero-fiber of the $\g$-Hilbert scheme:

\begin{conj}
The smooth version of $\Hilb_0(\g)$ is covered with charts parameterized by nilpotent orbits and all these charts are necessary to cover $\Hilb_0(\g)$. 
In particular for classical $\g$, we conjecture that the modified version of $\Hilb_0(\g)$ is isomorphic to the space of ideals of $\C[x,y]$ which are of codimension $m$, $W$-invariant, supported at 0 and which lie in a chart associated to a partition of type $\g$.
\end{conj}

\section{\texorpdfstring{$\g$}{g}-complex structures}\label{g-complexstructure}

Using the $\g$-Hilbert scheme we are able to construct a new geometric structure on a smooth surface, generalizing both complex and higher complex structures.

\subsection{Definition}

Recall from section \ref{Highercomplexsection} that the higher complex structure is defined as a section $I$ of $\Hilb^n_0(T^{*\C}\Sigma)$ satisfying $I(z)\oplus \overline{I}(z) = \langle p,\bar{p}\rangle$ at every point $z\in \Sigma$ (see Definition \ref{highercomplexdef}). We call the condition on $I$ the \textit{reality constraint}.
In this definition we exclusively use the idealic viewpoint of the punctual Hilbert scheme. Since the $\g$-Hilbert scheme uses the matrix viewpoint, we have to think of $n$-complex structures as special 1-forms. More precisely, an $n$-complex structure is a gauge class of a $\mf{sl}_n$-valued 1-form which can be locally written as $\Phi_1dz+\Phi_2d\bar{z}$ where $[(\Phi_1,\Phi_2)]$ is in the zero-fiber of the Hilbert scheme.

We are now ready to give the definition of a $\g$-complex structure, but one difficulty stays: we have to incorporate the reality constraint in the matrix viewpoint. Recall from Proposition \ref{mu2prop} the map $\mu_2: \Hilb^{reg}_0(\g) \rightarrow \C$ associating to $[(A,B)]$ the unique $\mu_2\in \C$ such that $B-\mu_2A$ is irregular. 

\begin{definition}\label{def-g-complex-1}
A $\mathbf{\g}$\textbf{-complex structure} is a gauge class of elements locally of the form $$A(z) dz+ B(z) d\bar{z} \in \Omega^1(\Sigma, \g) = \Omega^1(\Sigma,\C)\otimes \g$$ such that $$[(A(z), B(z))] \in \Hilb^{reg}_0(\g)$$ and $\mu_2(z)\bar{\mu}_2(z) \neq 1$ for all $z\in \Sigma$.
\end{definition}

Notice that for complex structures, the map $\mu_2(z)$ is nothing but the Beltrami differential. So our reality constraint coincides with the one for complex structures. In particular, for $\g=\mf{sl}_2$, we get a usual complex structure.
In the general case, we have the following:

\begin{prop}\label{inducedcomplex}
A $\g$-complex structure induces a complex structure on $\Sigma$.
\end{prop}
\begin{proof}
Recall from \ref{mu2} equation \eqref{mu} the map $\mu:\Hilb^{reg}_0(\g) \rightarrow \Hilb_0(\mf{sl}_2)$ given by $\mu([(A,B)])=[(e,\mu_2 e)]$ or $[(\mu_2 e, e)]$ depending on whether $A$ or $B$ is regular. Since a $\mf{sl}_2$-complex structure is a complex structure, the map $\mu$ induces a map from $\g$-complex structures to complex structures.
\end{proof}
\begin{Remark}
To define the map $\mu$ in \ref{mu2}, we really need $\g$ to be simple. Thus, only for $\g$ simple we get a unique complex structure out of a $\g$-complex structure.
\hfill $\triangle$
\end{Remark}

In the definition of a higher complex structure from section \ref{Highercomplexsection}, we use the zero-fiber $\Hilb^n_0(\C^2)$, without imposing to be in the regular part. The fact that we actually are in the regular part follows from the reality constraint $I\oplus \overline{I}=\langle p,\bar{p} \rangle$.
The same can be obtained for $\g$ of classical type, where we can reformulate the definition of $\g$-complex structures in a nicer way using the idealic map.

\subsection{Idealic viewpoint}

Recall the space of ideals $I_{\g}(\C^2)$ constructed in \ref{idealic}. Denote by $I_{\g,0}(\C^2)$ the set of those ideals of $I_{\g}(\C^2)$ which are supported at the origin (the zero-fiber). We can rewrite the definition of a $\g$-complex structure in the following way:

\begin{definition}\label{def-g-complex}
For classical $\g$, a $\mathbf{\g}$\textbf{-complex structure} is a section $I$ of $I_{\g,0}(T^{*\C}\Sigma)$ such that for all $z\in\S$:
$$I(z) \oplus \overline{I}(z)= \left \{ \begin{array}{cl}
\langle p, \bar{p} \rangle & \text{ if  } \g \text{ of type } A_n, B_n, C_n \\
\langle p, \bar{p} \rangle^2 &\text{ if } \g \text{ of type } D_n.
\end{array} \right.$$
\end{definition}

Notice that the condition on the ideals does not depend on coordinates since $\langle p, \bar{p}\rangle$ is the maximal ideal associated to the origin.

We prove the equivalence of both definitions. For that recall that to an ideal $I$ one can associate a class of commuting matrices $[(A,B)]$ (see subsection \ref{hilbmatrixviewpoint}).
\begin{prop}
For classical $\g$, the condition on $I\oplus \overline{I}$ given in Definition \ref{def-g-complex} is equivalent to $[(A(z),B(z))]$ being in the regular part $\Hilb^{reg}_0(\g)$ and having $\mu_2\bar{\mu}_2\neq 1$, i.e. the condition in Definition \ref{def-g-complex-1}.
\end{prop}
\begin{proof}
The backwards direction is a direct computation using the preferred representatives for $\Hilb^{reg}_0(\g)$ from Proposition \ref{paramit}. So we concentrate on the direct implication.

\underline{Case $A_n$.} The case $\g$ of type $A_n$ has been treated in Proposition \ref{genericideal}. 

\underline{Case $B_n$.} For $\g$ of type $B_n$ the standard representation gives $\mf{so}_{2n+1} \hookrightarrow \mf{sl}_{2n+1}$. By virtue of the case $A_n$, we know that $I\oplus \overline{I} =\langle p, \bar{p} \rangle$ implies $\mu_2\bar{\mu}_2\neq 1$ and $(A,B)$ regular for $\mf{sl}_{2n+1}$, i.e. 
$$I(A,B)=\langle p^{2n+1}, -\bar{p}+\mu_2p+\mu_3p^2+...+\mu_{2n}p^{2n} \rangle.$$
Since we know that in case $B_n$, the ideal $I$ is invariant under the map $-\id$, we get $\mu_{2k+1}=0$ for all $k=1,...,n-1$. So $I$ corresponds to a pair $(f,Q(f))$ for $Q$ an odd polynomial of degree at most $2n-1$, which is precisely a representative of $\Hilb^{reg}_0(\mf{so}_{2n+1})$ (see subsection \ref{Bn}).

\underline{Case $C_n$.} This case is exactly analogous to $B_n$ via the injection $\mf{sp}_{2n}\hookrightarrow \mf{sl}_{2n}$.

\underline{Case $D_n$.} We imitate the strategy of the proof for case $A_n$ (see Proposition \ref{genericideal}) with only difference that we have to go further in the analysis, needing some computations. 
The main argument is an iteration process which always ends since $p^k\bar{p}^l = 0 \mod I$ for $k+l \geq 2n$.

Put $I_1 = (I \mod \langle p, \bar{p} \rangle^2)$, i.e. the set of all terms of degree at most 1 appearing in $I$.  If $I_1$ is of dimension 2, then $I=\langle p, \bar{p}\rangle$ since both $p$ and $\bar{p}$ can be expressed by higher terms which by iteration become 0. If $I_1$ is of dimension 1, then we have a relation of the form $\bar{p}=\mu_2p+p^2R(p,\bar{p})$ where $R$ is a polynomial, which gives $\bar{p}$ as a polynomial in $p$ by iteration. We can then explicitly check that $I\oplus \overline{I}$ is either $\langle p, \bar{p} \rangle$ or $\langle p=\bar{p}, p\bar{p}, p^2\rangle$. Since we suppose $I=\langle p, \bar{p} \rangle ^2$, we get $\dim I_1 = 0$, i.e. $I_1=\{0\}.$

Put $I_2 = (I \mod \langle p, \bar{p}\rangle^3)$. We have $I_2\oplus \overline{I}_2 = (I\oplus \overline{I})_2 = \langle p^2, p\bar{p}, \bar{p}^2\rangle$ by assumption on $I$. If $I_2$ is of dimension 3, then all of $p^2, p\bar{p}$ and $\bar{p}^2$ can be expressed by higher terms. By iteration, we get $I=\langle p^2, p\bar{p}, \bar{p}^2\rangle$ which is not of type $D_n$. If $\dim I_2 \leq 1$, then we also have $\dim \overline{I}_2 \leq 1$, so $2\geq \dim I_2+\dim \overline{I}_2 = \dim \langle p^2,p\bar{p}, \bar{p}^2 \rangle_2 = 3$, a contradiction. Hence $\dim I_2=2.$

There is a term containing $p\bar{p}$ in $I_2$ since if not, no such term would neither exist in $\overline{I}_2$, so neither in $I_2\oplus\overline{I}_2 = \langle p^2, p\bar{p}, \bar{p}^2 \rangle$, a contradiction. Without loss of generality, we can assume that there is another term containing $\bar{p}^2$ (if not change the role of $I$ and $\overline{I}$).

So there exist $\alpha, \beta, \gamma, \delta \in \C$ such that 
$$ \left \{ \begin{array}{cl}
\bar{p}^2 = \alpha p^2+\beta p\bar{p} &\mod I_2\\
p\bar{p} = \gamma p^2 +\delta \bar{p}^2  &\mod I_2.
\end{array}\right. $$

\noindent If $\beta\gamma \neq 1$, we can simplify by substitution one into the other to
$$\left \{ \begin{array}{cl}
\bar{p}^2 = \alpha' p^2 &\mod I_2 \\
p\bar{p} = \gamma' p^2 &\mod I_2.
\end{array}\right. $$
If $\beta\gamma = 1$, we have $p^2\in I_2$, so $p\bar{p}=\delta \bar{p}^2 \mod I_2$, so changing $I$ to $\overline{I}$ we are in the previous situation.

Iterating the substitution process we get that $\bar{p}^2$ and $p\bar{p}$ are polynomials in $p$. Using the invariance of $I$ under $-\id$, we see that these are polynomials in $p^2$, i.e. even polynomials. So the most generic ideal is given by 
$$I=\langle p^{2n-1}, p\bar{p}=\mu_2p^2+\mu_4p^4+...+\mu_{2n-2}p^{2n-2}, \bar{p}^2=\nu_2p^2+\nu_4p^4+...+\nu_{2n-2}p^{2n-2}\rangle$$
which corresponds to a regular element of $\Hilb^{reg}_0(\mf{so}_{2n})$.
One checks that $I\oplus \overline{I}$ with $I$ of the form above equals $\langle p,\bar{p}\rangle ^2$ iff $\mu_2\bar{\mu_2} \neq 1$.
\end{proof}

To end this section, we determine the geometric nature of the various higher Beltrami coefficients. Since $p$ and $\bar{p}$ are linear coordinates on $T^{*\C}\Sigma$, we can identify $p=\frac{\partial}{\partial z}= \partial$ and $\bar{p}=\frac{\partial}{\partial \bar{z}}=\bar{\partial}$. Denote by $K$ the canonical bundle, i.e. $K=T^{*(1,0)}\Sigma$, and by $\Gamma(B)$ the space of sections of a bundle $B$. 

Analyzing the behavior under a coordinate change $z \mapsto w(z,\bar{z})$ analogous to the computation in subsection \ref{basichighercomplex}, we get  \begin{equation}\label{naturemu}\mu_i\in \Gamma(K^{1-i}\otimes \bar{K}) \text{ and } \nu_{2i} \in \Gamma(K^{-2i}\otimes \bar{K}^2).\end{equation} 
Since $\sigma_n^2$ has the same nature as $\nu_{2n-2}$ (see \ref{Dn}), we get $\sigma_n\in \Gamma(K^{1-n}\otimes \bar{K})$.

\section{Moduli space}\label{g-modulispace}

In this section, we define the moduli space of $\g$-complex structures and explore its properties. In the whole section we suppose $\g$ of classical type. We first have to define an equivalence relation on $\g$-complex structures, which is accomplished by the notion of higher diffeomorphisms of type $\g$.

\subsection{Higher diffeomorphisms}

For higher complex structures, we defined in subsection \ref{higherdiffs} higher diffeomorphisms to be hamiltonian diffeomorphisms of $T^*\Sigma$ preserving the zero-section $\Sigma \subset T^*\Sigma$. This gives the higher diffeomorphisms for type $A_n$.
We generalize this idea to other classical $\g$. We use the standard representation of $\g$ on $\C^m$, i.e. $\mf{sl}_n \subset \mf{gl}_n, \mf{so}_n \subset \mf{gl}_n$ and $\mf{sp}_{2n} \subset \mf{gl}_{2n}$.

One way to think of a $\g$-complex structure is as a $m$-tuple of 1-forms with some symmetry, which collapses all to the zero-section. The space of higher diffeomorphisms which we are looking for has to preserve this symmetry. For example for $\mf{sl}_n$, we have $n$ sections whose barycenter at every fiber is the origin, i.e. their sum gives the zero-section. That is why we have to impose that the hamiltonian diffeomorphisms of $T^*\Sigma$  preserve the zero-section.

For $\g$ of type $B_n, C_n$ or $D_n$, the set of $m$ points is symmetric with respect to the origin. Thus we define:

\begin{definition}
A \textbf{higher diffeomorphism} of type $B_n, C_n$ or $D_n$ is a hamiltonian diffeomorphism of $T^*\Sigma$ invariant under the map 
$(z,p,\bar{p})\mapsto (z,-p,-\bar{p})$. We denote by $\Symp(\g,\Sigma)$ the space of higher diffeomorphisms of type $\g$.
\end{definition}

This group is the same for types $B_n, C_n$ and $D_n$, but differs from $\Symp_0(T^*\S)$.
In coordinates a hamiltonian diffeomorphism is generated by a function $H(z,\bar{z},p,\bar{p})$ which can be Taylor developed to 
$\sum_{k,l} w_{k,l}(z,\bar{z})p^k\bar{p}^l$. It is invariant under $-\id$ iff it has only odd terms, i.e. 
$w_{k,l}=0$ for all $k+l$ even. Notice that the associated flow automatically preserves the zero-section since $w_{0,0}=0$.

\subsection{Action on \texorpdfstring{$\g$}{g}-complex structures}\label{actiondiff}

We can now analyze how higher diffeomorphisms act on $\g$-complex structures.

Intuitively, hamiltonian diffeomorphisms of $T^*\Sigma$ act on the space of 1-forms, so also on $m$-tuples of them. The invariance condition implies that the symmetry of the collection of 1-forms is preserved. This action persists at the limit when the $m$-tuple of 1-forms is collapsed to the zero-section.

To compute the action, it is better to work in the idealic viewpoint. We imitate the steps from the higher complex structure (see section \ref{higherdiffs}).

Let $I$ be an ideal representing a $\g$-complex structure, with generators $\langle f_1, ..., f_r \rangle$. Each $f_k$ can be considered as a function on $T^{*\C}\Sigma$, so its variation under a Hamiltonian $H$ is given by the Poisson bracket $\{H, f_k\}$.
The tangent space at $I$ in the space of all ideals of codimension $m$ is the set of all ring homomorphisms from $I$ to $A/I$ (see Proposition \ref{idealvariation}). Thus a Hamiltonian $H$ changes $I$ to 
$\langle f_1+\varepsilon\{H,f_1\} \mod I, ..., f_r+\varepsilon \{H,f_r\} \mod I \rangle$.

For $n$-complex structures, the simplification lemma \ref{simplificationlemma} allowed us to reduce $H$ modulo $I$. This is still possible:

\begin{prop}
For classical $\g$, the generators $f_i$ of an ideal $I\in \Hilb^{reg}_0(\g)$ satisfy $\{f_i, f_j\} = 0 \mod I$.
\end{prop}
\begin{proof}
For $A_n$, we have $I=\langle p^n, \bar{p}=\mu_2p+...+\mu_np^{n-1}=Q(p) \rangle$. We compute $\{p^n, -\bar{p}+Q(p)\} = np^{n-1}\partial Q = 0 \mod I$ since there is no constant term in $Q$. 
The same argument holds for $B_n$ and $C_n$ since their ideals are special cases of the ideal of type $A_n$.

For $D_n$, the ideal $I$ is given by
\begin{align*}
\langle p^{2n-1}, p\bar{p} &= \mu_2p^2+\mu_4p^4+...+\mu_{2n-2}p^{2n-2}=Q(p)+\mu_{2n-2}p^{2n-2}, \\
 \bar{p}^2 &= \nu_2p^2+\nu_4p^4+...+\nu_{2n-2}p^{2n-2}=R(p)+\nu_{2n-2}p^{2n-2}\rangle.
\end{align*}
As before the Poisson brackets with the first generator $p^{2n-1}$ vanishes modulo $I$ since $Q$ and $R$ have no constant terms. To compute the last Poisson bracket, define $\tilde{Q}=Q/p$. By the relations in $I$, we have $R=\tilde{Q}^2+p^{2n-2}\tilde{R}$ for some polynomial $\tilde{R}$ (see subsection \ref{Dn}). Remark further that $\{a(z, \bar{z})p^k\bar{p}^l, b(z, \bar{z})p^{k'}\bar{p}^{l'}\}=0 \mod I$ whenever $k+l+k'+l' > n-1$ since any term of degree $n-1$ in $p$ and $\bar{p}$ is in $I$ and the Poisson bracket lowers this degree by 1.
With all this, we compute
\begin{align*}
&\{-p\bar{p}+Q+\mu_{2n-2}p^{2n-2}, -\bar{p}^2+R+\nu_{2n-2}p^{2n-2}\} \\
=& \; \{-p\bar{p}+p\tilde{Q}+\mu_{2n-2}p^{2n-2}, -\bar{p}^2+\tilde{Q}^2+p^{2n-2}(\tilde{R}+\nu_{2n-2})\}\\
=& \; \{-p\bar{p}+p\tilde{Q}, -\bar{p}^2+\tilde{Q}^2\} & \text{ by degree argument}  \\ 
=& \;  2(p\bar{p}-p\tilde{Q})\bar{\partial}\tilde{Q}-2(\bar{p}-\tilde{Q})\tilde{Q}\partial \tilde{Q} \\
=& \;  2(p\bar{p}-Q)(\bar{\partial}\tilde{Q}-\frac{\tilde{Q}}{p}\partial \tilde{Q}) \\
=& \;  2\mu_{2n-2}p^{2n-2}(\bar{\partial}\tilde{Q}-\frac{\tilde{Q}}{p}\partial \tilde{Q}) &\mod I \\
=& \;  0 &\mod I 
\end{align*}
where the last line comes from the fact that $p$ divides the polynomial $\bar{\partial}\tilde{Q}-\frac{\tilde{Q}}{p}\partial \tilde{Q}$.
\end{proof}

As a consequence, when computing the action of a Hamiltonian $H$ on a $\g$-complex structure, we can reduce it modulo $I$. In particular if $H \mod I = 0$, the higher diffeomorphism generated by $H$ does not act at all. For $\g$ of type $A_n, B_n$ or $C_n$ we can reduce $H$ to a polynomial in $p$, and for $D_n$ we can reduce it to $H=w_{-}\bar{p}+\sum_{k=0}^{n-2} w_{2k+1}p^{2k+1}$.

\subsection{Local theory}

Now, we can study the local theory of $\g$-complex structures. Let $z_0$ be a point on $\Sigma$ and take a small chart around it with image the unit disk $\Delta$ in the complex plane (with $z_0$ the origin). 

\begin{thm}[Local theory]\label{thm1}
For $\g$ of type $A_n$, $B_n$ or $C_n$, the $\g$-complex structure can be locally trivialized, i.e. there is a higher diffeomorphism of type $\g$ which sends all higher Beltrami differentials to 0 for all small $z \in \C$.

For $\g$ of type $D_n$, all $\g$-complex structures with non-vanishing $\sigma_n$ on the unit disk $\Delta$ are equivalent under higher diffeomorphisms. However, the zero locus of $\sigma_n$ on $\Delta$ is an invariant.
\end{thm}
\begin{proof}
The proof for $\g$ of type $A_n$ was done in Theorem \ref{loctrivial}, using a method in the spirit of the proof of Darboux's theorem in symplectic geometry. 

If $\g$ is of type $B_n$ or $C_n$, the standard representation realizes the $\g$-complex structure as a substructure of type $A_n$. Since the last is trivializable, so is the $\g$-complex structure in that case.

For $\g$ of type $D_n$, we use the same method as for type $A_n$ by a hamiltonian flow argument.
We start with an ideal $I$ determined by $(\mu_2, \mu_4, ..., \mu_{2n-2}, \nu_{2n-2})$, the higher Beltrami differentials. The action on $\mu_{2i}$ is the same as for $\g=\mf{sl}_{2n}$ so we can trivialize them using a Hamiltonian $H$ which is a polynomial in $p$. So we are left with 
\begin{equation}\label{idealdndn}
I=\langle p^{2n-1}, p\bar{p}, -\bar{p}^2+\nu_{2n-2}p^{2n-2}\rangle.
\end{equation}
We have seen at the end of subsection \ref{actiondiff} that in the case $D_n$, any Hamiltonian can be reduced to $H=w_{-}\bar{p}+\sum_{k=0}^{n-2}w_{2k+1}p^{2k+1}$.
The only part of this Hamiltonian acting on $\nu_{2n-2}$ is $H=w_{-}\bar{p}$, which also changes $\mu_{2n-2}$. So in order to assure that $\mu_{2n-2}$ stays zero, we use $$H=w_{-}\bar{p}+w_{2n-3}p^{2n-3}.$$

We compute the action of this Hamiltonian on the ideal $I$. For the second generator of $I$ we get:
$$\{w_{-}\bar{p}+w_{2n-3}p^{2n-3}, -p\bar{p}\} = p^{2n-2}(\bar{\partial}w_{2n-3}+\nu_{2n-2}\partial w_{-}) \mod I.$$
For the third generator of $I$ we get
$$\{w_{-}\bar{p}+w_{2n-3}p^{2n-3}, -\bar{p}^2+\nu_{2n-2}p^{2n-2}\} = p^{2n-2}(w_{-}\bar{\partial}\nu_{2n-2}+2\nu_{2n-2}\bar{\partial}w_{-}) \mod I.$$
Denote by $\mu_{2-2}^t$ and $\nu_{2n-2}^t$ the image of $\mu_{2n-2}$ and $\nu_{2n-2}$ under the flow generated by $H$ at time $t$.
From the above computation we get 
$$ \left \{\begin{array}{cl}
\frac{d}{dt} \mu_{2n-2}^t &= \bar{\partial}w_{2n-3}+\nu_{2n-2}^t\partial w_{-} \\
\frac{d}{dt} \nu_{2n-2}^t &= (w_{-}\bar{\partial}+2\bar{\partial}w_{-})\nu_{2n-2}^t.
\end{array}\right. $$

Instead of keeping $\nu_{2n-2}$, we work with the higher Beltrami differential $\sigma_n$. Since all the $\mu_{2i}$ are zero in $I$, we have $\nu_{2n-2}=\sigma_n^2$. Therefore we get from the second equation above $\frac{d}{dt}(\sigma_n^2)=(w_{-}\bar{\partial}+2\bar{\partial}w_{-})(\sigma_n^2)$ which gives
\begin{equation}\label{varsigma}
\frac{d}{dt}\sigma_n^t = \bar{\partial}(w_{-}^t\sigma_n^t).
\end{equation}

We wish to have $\frac{d}{dt} \mu_{2n-2}^t=0$ to keep $\mu_{2n-2}=0$. For $\sigma_n$, we show that we can deform it to the constant function 1 on the unit disk, assuming $\sigma_n$ vanishes nowhere on $\Delta$. We choose the path $\sigma_n^t=(1-t)\sigma_n^0+t$ from the initial $\sigma_n^0$ to the constant function 1. If $\sigma_n^t=0$ for some $t$, we have to slightly modify the path. We get $\frac{d}{dt}\sigma_n^t=1-\sigma_n^0$. 

Denote by $T$ the local inverse of the $\bar{\partial}$-operator, i.e. $\bar{\partial}(Tf)=f=T\bar{\partial}f$ for all $f\in L^2(\Delta)$. The operator $T$ is a pseudo-differential operator given by
$$Tf(z)=\frac{1}{2\pi i}\int_{\mathbb{C}}\frac{f(\zeta)}{\zeta-z}d\zeta \wedge d\bar{\zeta}.$$
We can solve equation \eqref{varsigma} using $T$: $$w_{-}^t=\frac{1}{\sigma_n^t}T(1-\sigma_n^0).$$
Notice that the Hamiltonian we need is time-dependent.
Putting this solution into the equation for $\frac{d}{dt}\mu_{2n-2}^t$, we can solve for $w_{2n-3}$:
$$w_{2n-3}^t=-T(\partial w_{-}^t\nu_{2n-2}^t).$$

Finally, we multiply $H$ by a bump function, a function on $\Delta$ which is 1 in a neighborhood of the origin and 0 outside a bigger neighborhood of the origin, which ensures that the hamiltonian vector field is compactly supported, so it can be integrated to all times. In particular for $t=1$ we get $\sigma_n(z)=1$ for all $z$ near the origin.

To show that the zero locus of $\sigma_n$ can not be changed by a higher diffeomorphism, consider the singularity defined by $$f(p, \bar{p})=-\frac{\nu_{2n-2}}{2n-1}p^{2n-1}+p\bar{p}^2=0$$ which is a Kleinian singularity of type $D_{2n}$ if $\nu_{2n-2}\neq 0$. Its deformation ideal $\left\langle \frac{\del f}{\del p}, \frac{\del f}{\del \bar{p}}  \right\rangle$ is directly linked to our ideal $I$ from equation \eqref{idealdndn} by $$\left\langle \frac{\del f}{\del p}, \frac{\del f}{\del \bar{p}}  \right\rangle\oplus \langle p, \bar{p}\rangle^{n-1} = \langle p^{2n-1}, p\bar{p}, -\bar{p}^2+\nu_{2n-2}p^{2n-2} \rangle.$$ 
Since the type of a singularity is invariant under diffeomorphisms, so is its deformation ideal. This is why we cannot change $\nu_{2n-2}=0$ to $\nu_{2n-2}\neq 0$ by higher diffeomorphisms. 
\end{proof}
\begin{Remark}
It is interesting to notice the appearance of Kleinian singularities, which have an $ADE$-classification. The fact that for $\g$ of type $D_n$ the singularity is of type $D_{2n}$ is linked to the representation of $\mf{so}_{2n}$ on $\C^{2n}$. There should be a more intrinsic way to link $\g$-complex structures to singularities of type $\g$.

An idea in this direction is the following: the singularity of type $\g$ appears inside the Lie algebra $\g$, more precisely inside the nilpotent variety along the subregular locus (see \cite{Steinberg}). A minimal resolution of this singularity is given by the Springer resolution. There should be a link between $\g$-Hilbert schemes and the Springer resolution.
\hfill $\triangle$
\end{Remark}

Since there are no local invariants for $\g$-complex structures, only their global geometry is non-trivial.

\subsection{Definition of the moduli space}

To define the moduli space of $\g$-complex structures, there is one more subtlety: in order to get one component, we have to fix an orientation on $\Sigma$. We then call a complex structure \textit{compatible} if the induced orientation coincides with the given orientation on $\Sigma$. We call a $\g$-complex structure \textit{compatible} if the induced complex structure is.
\begin{definition}
The \textbf{moduli space $\bm\hat{\mc{T}}_{\g}$} is the space of all compatible $\g$-complex structures modulo the action of higher diffeomorphisms of type $\g$.
\end{definition}

Notice that a $\g$-complex structure is compatible iff $\mu_2(z, \bar{z})\bar{\mu}_2(z, \bar{z}) < 1$. 
Reverting the orientation on $\Sigma$ we get another copy of $\bm\hat{\mc{T}}_{\g}$ corresponding to those $\g$-complex structures with $\mu_2(z, \bar{z})\bar{\mu}_2(z, \bar{z}) > 1$. 

For $\g=\mf{sl}_2$ we get Teichm\"uller space since we can reduce any Hamiltonian to $H=w(z,\bar{z})p$ which generates a linear diffeomorphism of $T^*\Sigma$, coming from a diffeomorphism on $\Sigma$ isotopic to the identity.

The moduli space has the following properties:
\begin{thm}[Global theory]\label{thm2}
For $\g$ of type $A_n, B_n$ or $C_n$, and a surface $\Sigma$ of genus $g\geq 2$, the moduli space $\bm\hat{\mathcal{T}}_{\g}$ is a contractible manifold of complex dimension $(g-1)\dim \g$. In addition, its cotangent space at any point $I$ is given by 
$$T^*_{I}\bm\hat{\mathcal{T}}_{\g} = \bigoplus_{m=1}^{r} H^0(K^{m_i+1})$$ where $(m_1,...,m_r)$ are the exponents of $\g$ and $r=\rk \g$ denotes the rank of $\g$.

\noindent For type $D_n$, the moduli space $\bm\hat{\mathcal{T}}_{\g}$ is a contractible topological space. The locus where the zero-set of the higher Beltrami differential $\sigma_n$ is a discrete set on $\S$ is a smooth manifold with the same properties as above (dimension and cotangent space).
\end{thm}
Notice that the differentials in $H^0(K^{m_i+1})$ are holomorphic with respect to the complex structure induced from the $\g$-complex structure (see Proposition \ref{inducedcomplex}).

For the case $D_n$, we conjecture that the moduli space $\bm\hat{\mathcal{T}}_{\g}$ is a topological manifold everywhere. The points where the zero-set of $\sigma_n$ is not discrete can have a cotangent space which is strictly bigger than the space of holomorphic differentials. One can think for example of the curve in $\R^2$ given by $t\mapsto (t^3,t^2)$, shown in figure \ref{curvett2}, which has a cusp at the origin, but is still a topological manifold.

\vspace*{0.3cm}
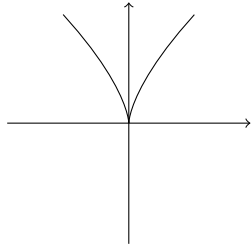
\begin{figure}
\begin{center}
\begin{tikzpicture}[scale=4]
	\draw[->] (0,-0.4)--(0,0.4);
	\draw[->] (-0.4,0)--(0.4,0);
\draw [domain=-0.6:0.6] plot (\x*\x*\x, \x*\x);
\end{tikzpicture}
\caption{Curve with cusp}
\label{curvett2}
\end{center}
\end{figure}

\begin{proof}
The case for $A_n$ has been treated in Theorem \ref{mainresultncomplex}. The cases $B_n$ and $C_n$ are exactly analogous:

One shows that at every point, the cotangent space exists and is of the form stated in the theorem. From this follows that $\bm\hat{\mc{T}}_{\g}$ is a manifold.
We have to check the appearance of the exponents of the Lie algebra. Since $\mu_{2i}$ is a section of $K^{1-2i}\otimes \bar{K}$ (see equation \eqref{naturemu}) its dual $t_{2i}$ is a section of $K^{2i}$. Since the exponents for $B_n$ and $C_n$ are the same and equal to $(1,3,...,2n-1)$, we get the desired form stated in the theorem.

For $\g$ of type $D_n$ we consider the subset on $\bm\hat{\mathcal{T}}_{\g}$ where the zero-locus of $\sigma_n$ is discrete in $\S$.
By the local theory, we know that there is a coordinate system in which $\mu_{2i}=0$ for all $i=1,..., n-1$. In that case, we know that the variation of $\mu_{2i}$ under a higher diffeomorphism generated by $H=w_{-}\bar{p}+\sum_{k=0}^{n-2}w_{2k+1}p^{2k+1}$ is given by $\delta \mu_{2i}=\bar{\partial}w_{2i-1}$ and equation \eqref{varsigma} gives $\delta \sigma_n=\bar{\partial}(w_{-}\sigma_n)$. The variation of $\mu_{2i}$ is the same as in the case of type $A_n$, so we know that these contribute to the cotangent bundle by a term $H^0(K^{2i})$. For the term $\sigma_n$ we use the pairing between differential of type $(1-n,1)$ and of type $(n,0)$ given by integration over the surface. We get 
\begin{align*}
(\{\delta\sigma_n\} / \bar{\partial}(w_{-}\sigma_n))^* &=  \{t_n \in \Gamma(K^n) \mid \textstyle\int t_n \bar{\partial}(w_{-}\sigma_n) = 0 \; \forall \, w_{-} \in \Gamma(\bar{K}) \} \\
&=  \{t_n \in \Gamma(K^n) \mid \textstyle\int \bar{\partial}t_n w_{-}\sigma_n = 0 \; \forall \, w_{-} \in \Gamma(\bar{K}) \} \\
&= \{t_n \in \Gamma(K^n) \mid \bar{\partial}t_n = 0 \} \\
&= H^0(K^n)
\end{align*}
where we used that $\sigma_n$ vanishes only on a discrete set.

Hence the cotangent bundle is given by $$T^*_I\bm\hat{\mathcal{T}}_{\g} = \bigoplus_{m=1}^{n-1}H^0(K^{2m})\oplus H^0(K^n).$$
The exponents of $\mf{so}_{2n}$ are precisely $(1,3,...,2n-3,n-1)$, so the cotangent bundle is of the form stated in the theorem.

For the dimension of $\bm\hat{\mathcal{T}}_{\g}$, we use $\dim H^0(K^{m_i+1})=(g-1)(2m_i+1)$ by Riemann-Roch (using $g \geq 2$). We get $$\dim \bm\hat{\mc{T}}_{\g}=(g-1)\sum_{i=1}^r (2m_i+1)=(g-1)\dim \g$$ using a well-known formula coming from the decomposition of $\g$ as $\mf{sl}_2$-module using the principal $\mf{sl}_2$-triple.

Contractibility for all types is analogous to the case $A_n$.
\end{proof}

The previous theorem gives lots of common features between our moduli space $\bm\hat{\mc{T}}_{\g}$ and the $G$-Hitchin component, in particular the contractibility and the dimension.
There is another common property to notice:
\begin{prop}\label{inclteich}
There is an injection of Teichm\"uller space into the moduli space $\bm\hat{\mc{T}}_{\g}$.
\end{prop}
\begin{proof}
The proposition follows from the map $\psi:\Hilb(\mf{sl}_2) \rightarrow \Hilb^{reg}(\g)$ constructed in equation \eqref{teichcopy}. This map restricts to a map between the zero-fibers and extends over the surface $\Sigma$. Finally the map descends to the quotient by higher diffeomorphisms since for $\mf{sl}_2$ we only quotient by diffeomorphisms of $\Sigma$.

For injectivity, it is possible to check for classical $\g$ that a Hamiltonian of degree at least 2 does not act on $\mu_2$. We then exactly reason as in the proof of Proposition \ref{copyteich}.
\end{proof}

The same property holds for $G$-Hitchin components which can be defined as the deformation space of representations of the form $\pi_1(\Sigma) \rightarrow \PSL_2(\R) \rightarrow G$ where the first map is a fuchsian representation and the second one is the principal map. So inside the $G$-Hitchin component sits a copy of Teichm\"uller space. 

We conjecture the equivalence of Hitchin's component and the moduli space of $\g$-complex structures:
\begin{conj}\label{conjmaj}
The moduli space $\bm\hat{\mathcal{T}}_{\g}$ is canonically homeomorphic to the Hitchin component in the character variety $\Rep(\pi_1(\Sigma),G)$ where $G$ is the real split Lie group associated to $\g$.
\end{conj}

\subsection{Spectral curve}\label{spectralcurve}

In this subsection, we construct a spectral curve in $T^{*\C}\Sigma$ associated to a cotangent vector to $\bm\hat{\mathcal{T}}_{\g}$, i.e. a $\g$-complex structure and a set of holomorphic differentials. The case for $\g$ of type $A_n$ was treated in section \ref{Spectralcurve1}.

In Proposition \ref{lagr} we proved that the zero-fiber $\Hilb^n_0(\C^2)$ is Lagrangian in the reduced Hilbert scheme $\Hilb^n_{red}(\C^2)$. This stays true for all classical $\g$:
\begin{prop}
For classical $\g$, the regular zero-fiber $\Hilb^{reg}_0(\g)$ is a Lagrangian subspace of $\Hilb^{reg}(\g)$.
\end{prop}
\begin{proof}
Since we are in the regular part, Proposition \ref{paramit} gives a parametrization. For classical $\g$, the standard representation allows to consider $\Hilb^{reg}(\g)$ as subset of $\Hilb^m_{red}(\C^2)$ which remains symplectic and we can explicitly check that the zero-fiber $\Hilb^{reg}_0(\g)$ is Lagrangian.
\end{proof}

For general $\g$ we conjecture the following:
\begin{conj}
The modified smooth version of the $\g$-Hilbert scheme is symplectic and the zero-fiber is a Lagrangian subspace.
\end{conj}
If we assume Conjecture \ref{conj1}, stating that the modified version of the $\g$-Hilbert scheme is a minimal resolution of $\h^2/W$, we get a symplectic structure. Indeed $\h^2=T^*\h$ has a canonical symplectic structure, which is invariant under the action of $W$. Hence it lifts to the minimal resolution.

\bigskip
\noindent Now we construct the spectral curve. First, we look at $\g$ of type $A_n$, $B_n$ or $C_n$.
We can write a point in $T^*\bm\hat{\mc{T}}_{\g}$ as an equivalence class of higher Beltrami differentials $\mu_i$ and holomorphic differentials $t_i$. To write in a uniform way, set $\mu_i$ or $t_i$ to 0 whenever it does not appear for $\g$. For example for type $B_n$ or $C_n$ all variables with odd index are 0.

Associate polynomials $P(p)=\sum_i t_ip^{m-i}$ and $Q(p)=\sum_i \mu_ip^{i-1}$ (where $m$ is the dimension of the standard representation of $\g$). Put $I=\langle -p^m+P, -\bar{p}+Q \rangle$. Define the \textbf{spectral curve} $\bm\tilde{\Sigma} \subset T^{*\C}\Sigma$ by the zero set of the two generators of $I$. It is a ramified cover over $\Sigma$ with $m$ sheets. 

For $\g$ of type $D_n$, a generic point in the cotangent bundle $T^*\bm\hat{\mc{T}}_{\g}$ corresponds to the ideal 
$$I=\langle p^{2n}+t_2p^{2n-2}+...+t_{2n-2}p^2+\tau_n^2, -\bar{p}+\mu_2p+...+\mu_{2n}p^{2n-1} \rangle$$ which can be seen as a special case of $A_n$. Thus we can proceed as above. In the case where $\tau_n = 0$ we have seen in \ref{Dn} that the ideal changes to an ideal with three generators. The zero-set of these generators still define a spectral curve in $T^{*\C}\Sigma$. It is the limit of the curve when $\tau_n \rightarrow 0$.

\begin{thm}
The spectral curve $\bm\tilde{\Sigma}$ is Lagrangian modulo $t^2$.
\end{thm}
This is the precise analog of Theorem \ref{spectralcurveprop}. 
\begin{proof}
In the case where the ideal has two generators $-p^m+P$ and $-\bar{p}+Q$, this is equivalent to $\{-p^m+P,-\bar{p}+Q\} = 0 \mod I \mod t^2$ for $I\in T^*\bm\hat{\mc{T}}_{\g}$.
For $A_n$, the proof is given in Theorem \ref{spectralcurveprop}. For $B_n$ and $C_n$ it is completely analogous since the $\g$-complex structure can be seen as a special case of $A_n$.

For $\g$ of type $D_n$, a generic ideal has still two generators, so we have a special case of $A_n$. If the ideal has three generators, the spectral curve is still Lagrangian since it can be obtained as a limit of Lagrangian curves, and the property of being Lagrangian is closed.
\end{proof}

Since the spectral curve is Lagrangian to order 1, the periods are well-defined up to this order. The ratios of these periods should give coordinates on $T^*\bm\hat{\mc{T}}_{\g}$ and also on $\bm\hat{\mc{T}}_{\g}$.
For the trivial $\g$-complex structure (where all higher Beltrami differentials are 0) we recover Hitchin's spectral curve.

Finally, we can recover the same spectral data as Hitchin in his paper on stable bundles \cite{Hit3}. 
From a $\g$-complex structure we get a bundle $V$ over the surface $\Sigma$ whose fiber at a point $z\in \Sigma$ is $\C[p,\bar{p}]/I(z)$ where we use the idealic viewpoint. We also get a line bundle $L$ on $\bm\tilde{\Sigma}$ whose fiber is the eigenspace of $M_p$, the multiplication operator by $p$ in the quotient $\C[p,\bar{p}]/I$. This gives the spectral data for type $A_n$.

For $\g$ of type $C_n$, we get in addition an involution $\sigma$ on the spectral curve $\bm\tilde{\Sigma}$ given by $(p,\bar{p})\mapsto (-p, -\bar{p})$.
For $\g$ of type $D_n$, the spectral curve is singular, having a double point. The spectral data is given by a desingularization of $\bm\tilde{\Sigma}$, the involution $\sigma$ as for $C_n$ and the line bundle $L$.
For $\g$ of type $B_n$, there is a canonical subbundle $V_0 \subset V=\C[p,\bar{p}]/I$ generated by the span of the image of $1\in \C[p, \bar{p}]$ in the quotient $\C[p,\bar{p}]/I$ (since for $B_n$, we have $I\subset \langle p, \bar{p}\rangle$). Thus the vector bundle $V$ is an extension $V_0 \rightarrow V\rightarrow V_1$. The spectral data is given by $(V_0, V_1, \sigma, L, \bm\tilde{\Sigma})$.


\cleardoublepage
\vspace*{3cm}
\part{Perspectives}\label{part4}
\vspace*{2cm}
\begin{flushright}
\textit{\glqq Ich habe bemerkt, sagte Herr K., das wir viele abschrecken von unserer Lehre dadurch, dass wir auf alles eine Antwort wissen. Könnten wir nicht im Interesse der Propaganda eine Liste der Fragen aufstellen, die uns ganz ungelöst erscheinen?\grqq{} }\footnote{\textit{``I have noticed,'' said Mr. K., ``that we put many people off our teaching because we have an answer to everything. Could we not, in the interests of propaganda, draw up a list of the questions that appear to us completely unsolved?''}, Translation by Martin Chalmers}

Bertolt Brecht, Geschichten vom Herrn Keuner
\end{flushright}
\vspace*{1cm}

\noindent In this last part of the thesis, we present open questions, conjectures and several links between higher complex structures and other domains.

In previous chapters we already pointed at several questions and conjectures which remain open. We present them in a uniform way and discuss possible approaches and answers. In particular we discuss the conjectural hyperkähler structure of $\cotang$ with its twistor space given in figure \ref{HK}, geometric interpretations of $\cotang$ as almost-Fuchsian representations, period coordinates on $\cotang$, the fiber of the forgetful map $\T^n\rightarrow \T^{n-1}$ which would partially answer a conjecture of Labourie from \cite{Lab}, and generalizations to Lie algebras other than $\mf{sl}_n$.

The field of higher Teichmüller theory is vast and studies numerous mathematical objects. We give an outlook on possible connections between higher complex structures and various fields.
In particular we discuss links to other components in the character variety, spectral networks, cluster varieties, opers, W-algebras and SYZ-mirror symmetry.

\vspace*{\fill}

\cleardoublepage

\section{Open questions and conjectures}

We have already encountered several open questions and conjectures in the previous chapters. We present them here in a uniform and organized way.

\paragraph{Hyperkähler structure on $\cotang$.}

The main conjecture to be solved is (see \ref{hkcotang}):
\begin{conj}
The space $\cotang$ is hyperkähler near the zero-section. Its twistor space is given by the space of flat parabolic $h$-connections.
\end{conj}

We have given three arguments for Conjecture \ref{hkcotang}.
In the twistor approach the main missing part is the existence and uniqueness of twistor lines, see Conjecture \ref{nahc}. In particular the generalized Toda system has to be studied.

In the Feix--Kaledin approach, it might be possible to directly prove that $\T^n$ is Kähler. For that, we have to find a symplectic structure compatible with its complex structure. Since on Hitchin's component there is the Goldman symplectic structure, we can ask the following:
\begin{oq}
Is it possible to recover the Goldman symplectic structure on $\T^n$? Does it combine with the complex structure to a Kähler structure?
\end{oq}

In the hyperkähler quotient approach, we use the HK structure of the Hilbert scheme to get a HK structure on $\Gamma(\Hilb^n_{red}(T^{*\C}\S))$. We conjecture: 
\begin{conj}
The spaces $\Gamma(\Hilb^n_{red}(T^{*\C}\S))$ and $\mc{A}\sslash\mc{P}$ are isomorphic as hyperkähler manifolds. Further there is a hyperkähler quotient of this space by $\Symp_0(T^*\S)$ (or a deformation) which gives a neighborhood of the zero-section of $\cotang$.
\end{conj}

\paragraph{Geometric interpretation of $\cotang$.}
The space $\T^n$ has a geometric origin, as it is the moduli space of higher complex structures. 
\begin{oq}
What is the geometric interpretation of $\cotang$?
\end{oq}
For $n=2$, the cotangent space to Teichmüller space $T^*\mc{T}^2$ is the space of half-translation surfaces (see subsection \ref{half-translation}). It seems difficult to generalize this notion to higher $n$.

For $n=2$, Donaldson explicitly constructs the Feix--Kaledin extension of Teichmüller space in \cite{Donald.2}. 
Trautwein deepens this study in \cite{Trautwein}. He shows that the Feix--Kaledin hyperkähler structure near the zero-section of $T^*\mc{T}^2$ describes the space of almost-Fuchsian representations. A representation $\pi_1(\S) \rightarrow \PSL_2(\C)$ is called \textbf{almost-Fuchsian} if there is a unique incompressible minimal surface in $\S\times \R$. Notice that an almost-Fuchsian representation is always quasi-Fuchsian, but the converse is not true.

We can generalize the notion of almost-Fuchsian representation as follows:
To a Lie group $G$ is associated its symmetric space $G/K$ (where $K$ is a maximal compact subgroup) whose isometry group is $G$. Thus, a representation $\pi_1(\S)\rightarrow G$ is equivalent to an action of $\pi_1(\S)$ on the symmetric space by isometries. In the seminal paper \cite{Corlette}, Corlette associates to every reductive representation of $\pi_1(\S)$ into some $G$ a harmonic map from the universal cover $\bm\tilde{\S}$ to $G/K$.
We call a representation $\pi_1(\S) \rightarrow G$ \textit{almost-Fuchsian} if the Corlette harmonic map embeds $\S$ as \textit{unique} minimal surface into the double quotient $\pi_1(\S)\backslash G/K$. 

For $n=3$ we have found in our setting a version of \c{T}i\c{t}eica's equation, see \eqref{Titeica}. In \cite{Loftin} it is shown that this equation describes almost-Fuchsian representations for $G=\SU(2,1)$ (with a minimal surface inside the complex hyperbolic plane).

\begin{conj}
The space $\cotang$ parameterizes near the zero-section the space of almost-Fuchsian representations. 
\end{conj}

In a generic fiber of the twistor space of $\cotang$, we have the space of flat parabolic connections. For $n=2$ these are precisely the developing maps of \textit{complex projective structures} on $\S$ (an atlas with charts in $\C P^1$ with coordinate changes which are Möbius transformations). 

\begin{conj}
There is a generalization if the complex projective structure, replacing $\C P^1$ by $\C P^{n-1}$, whose developing map is the monodromy of a flat parabolic connection.
\end{conj}

\paragraph{Period coordinates on $\T^n$ and $\cotang$.}
Another important open question is to construct coordinates on $\bm\hat{\mc{T}}^n$ and $\cotang$. 
\begin{conj}
There is a local isomorphism $$\cotang \cong H^1(\bm\tilde{\S}, \C)/H^1(\S,\C)$$ given by restricting the Liouville form $\alpha$ (locally given by $\alpha = pdz+\bar{p}\bar{z}$) to the spectral curve $\bm\tilde{\S}$. Furthermore some ratios of periods should give a coordinate system on $\bm\hat{\mc{T}}^n$.
\end{conj}
A simple computation gives $\dim \cotang = \dim H^1(\bm\tilde{\S}, \C)- \dim H^1(\S,\C)$ (see end of subsection \ref{sspectral}). In addition, for $n=2$, Fock and Goncharov in \cite{FockGoncharov} proved a local isomorphism of the above type for $T^*\mc{T}^2$.
The main difficulty is that the spectral curve $\bm\tilde{\S}$ is only Lagrangian modulo $t^2$. There might be a canonical deformation of $\bm\tilde{\S}$ which is Lagrangian.

\paragraph{Fiber of the forgetful map.}
Recall the forgetful map $\T^n\rightarrow \T^{n-1}$ from Theorem \ref{mainresultncomplex}. It would be interesting to compute its fiber.
\begin{conj}
The fiber of the forgetful map $\T^n\rightarrow \T^{n-1}$ is an affine space modeled on $H^0(K^n)$.
\end{conj}
The dimension count is right: $\dim \T^n = \dim \T^{n-1} + \dim H^0(K^n)$. For $n=3$, Labourie showed in \cite{Lab.2} that the fiber of $\mc{T}^3 \rightarrow \mc{T}^2$ is modeled on holomorphic cubic differentials using affine spheres and the Blaschke metric.

Composing the forgetful maps gives $\T^n \rightarrow \mc{T}^2$. Combined with the previous conjecture, this means that $\T^n$ is a bundle over Teichmüller space with fiber modeled on $\bigoplus_{i=3}^n H^0(K^i)$. This would partially confirm a conjecture of Labourie (see \cite{Lab}): there is  a description of Hitchin's component via a fiber bundle over Teichmüller space (with fiber $\bigoplus_{i=3}^n H^0(K^i)$) which allows to associate a preferred complex structure to a representation in the Hitchin component. In addition Labourie conjectures that the preferred complex structure is a minimum of the energy functional on Teichmüller space.


\paragraph{Surfaces with boundary.}
In the whole thesis, we assumed $\S$ to be closed. 
\begin{oq}
How to define higher complex structures for surfaces with marked points or with boundary components?
\end{oq}
For a surface with boundary, one should analyze if it is necessary to put an extra condition on the $n$-complex structure at the boundary of $\S$. The character variety $\Rep(\pi_1(\S), G)$ is not symplectic any more, but carries a Poisson structure, and its symplectic leaves are given by representations with prescribed monodromy around the boundary components (which are topological circles). 

For a surface with marked points, lots of extra structures have been introduced in the theory of Higgs bundles. There is the notion of parabolic Higgs bundles using flag structures on the marked points (see \cite{Konno}). Such flag structures are also used by Fock and Goncharov in \cite{FockGoncharov} to study higher Teichmüller spaces.

\paragraph{Various hyperkähler structures.}
In subsection \ref{hkhilbertscheme} we have seen that the punctual Hilbert scheme $\Hilb^n(\C^2)$ is hyperkähler. In another complex structure, we get the Calogero--Moser space $\Hilb^{n,q}(\C^2)$. 
\begin{oq}
Is there a link between the Calogero--Moser bundle $\Hilb^{n,q}(T^{*\C}\S)$ and the space of parabolic connections $\mc{A}\sslash\mc{P}$? 
\end{oq}
From the definition, an element of $\Hilb^{n,q}(\C^2)$ is a pair of matrices $(A,B)$ such that $$\rk ( [A,B]-q\id ) \leq 1.$$
A parabolic connection $d+A+\bar{A}$ has curvature of rank at most 1. So $$\rk (\del \bar A-\delbar A + [A,\bar{A}]) \leq 1.$$ These two conditions seem quite similar.

\paragraph{Generalizations to a simple Lie algebra $\g$.}
Finally there are several open questions for the generalization of punctual Hilbert schemes and higher complex structures to simple Lie algebras $\g$. 
\begin{conj}
There is a modified version of $\Hilb(\g)$, identifying some points, which is a smooth projective variety such that the Chow morphism is a resolution of singularities. Furthermore $\Hilb(\g)$ carries a complex symplectic structure for which the zero-fiber $\Hilb_0(\g)$ is Lagrangian.
\end{conj}
To define the moduli space of $\g$-complex structures, we quotient out by higher diffeomorphisms of type $\g$, which we only defined for classical $\g$ (see section \ref{g-modulispace}).
\begin{oq}
How to intrinsically define the space of higher diffeomorphisms of type $\g$?
\end{oq}
To get a link to $G$-Hitchin components, we have to generalize part \ref{part2}. In particular a generalization of parabolic reduction for general $\g$, has not been carried out for the moment.
\begin{conj}
There is a generalized parabolic reduction process, giving an isomorphism between the moduli space $\bm\hat{T}_\g$ and the $G$-Hitchin component.
\end{conj}
It seems not difficult to solve most of these open questions for classical Lie algebras by using their natural inclusion into some $\mf{sl}_n$. There is still one difficulty: for $\g$ other than $\mf{sl}_n$, there is no parabolic subalgebra $\mf{p}$ satisfying $\dim \mf{g}-\dim\mf{p}=\rk \g$.
What is also missing is an understanding for general $\g$ in intrinsic terms.

\section{Links to related topics}
We present here possible connections between punctual Hilbert schemes and higher complex structures to other mathematical objects or areas.

\subsection{Geometric approaches to higher Teichm\"uller theory}

There are various geometric approaches to Hitchin components. The exploration of links between them should enrich their understanding.
It would be interesting to compare the rigid geometric structures constructed by Goldman \cite{Goldman}, Guichard--Wienhard \cite{Guichard} and \cite{Lab} to the flexible higher complex structure. In particular, locally there should be a preferred higher complex structure associated to a rigid structure. This would generalize the local existence of a complex structure associated to a hyperbolic structure. To get a hyperbolic structure out of a complex structure, one has to solve Liouville's differential equation. There should be a \textit{generalized differential equation for higher complex structures}.

Punctual Hilbert schemes arise in various sorts and can potentially be used to give geometric approaches to other components of character varieties. 
In particular, there is the notion of \textbf{maximal components} in character varieties of Hermitian type Lie groups (see for example \cite{Bradlow} and \cite{Burger}). Giving a geometric approach to these components by flexible geometric structures is challenging. Maximal components can have non-trivial topology whereas moduli spaces for $\g$-complex structures are always contractible.

Another application of Hilbert schemes might be the following: both in Hitchin's approach with Higgs bundles and in the geometric approach with Hilbert schemes, one uses complex methods. Only with a reality constraint, one gets to the $G$-character variety for real $G$. 
It is possible to define a Hilbert scheme $\Hilb(\g)$ for a real Lie algebra, especially the punctual Hilbert scheme of the real plane $\Hilb^n(\R^2)$. A new geometric structure can be defined as a section of $\Hilb^n_0(T^*\Sigma)$ where we use the real cotangent bundle. The exploration of its local and global geometry could give a direct link to $\PSL_n(\R)$-Hitchin components, without passing through complex numbers.

\subsection{Spectral networks, cluster coordinates and positivity}\label{cluster}

The spectral curve in the higher complex structure setting potentially gives a generalized framework for spectral networks. Furthermore, the periods of this spectral curve should give a cluster coordinate system on the moduli space, which would give a notion of positivity.

\paragraph{Spectral networks.}
Gaiotto, Moore and Neitzke define in \cite{GMN} the notion of  \textit{spectral networks}. These are special networks on Riemann surfaces allowing an abelianization of connections over a spectral curve, i.e. a bijection between flat $G$-connections (for complex $G$) over $\Sigma$ and flat abelian connections over a branched cover $\bm\tilde{\Sigma}$.
A typical example is constructed out of a collection of holomorphic differentials, from which one defines a branched cover $\bm\tilde{\Sigma}$ and a set of foliations on the surface $\Sigma$. For $\mf{sl}_2$, the periods of $\bm\tilde{\Sigma}$ parameterize all spectral networks of this kind.
But in general the periods of the spectral curve have more degrees of freedom than the spectral networks. To our understanding, the GMN-spectral networks depend on a complex structure and holomorphic differentials whereas the periods of the spectral curve should parameterize higher complex structures and holomorphic differentials. 

Starting with a higher complex structure and holomorphic differentials, we get a spectral curve from which we can define a set of foliations, generalizing the typical example of spectral networks. This construction has to be carried out in detail and the properties of the resulting networks, in particular the abelianization of connections, to be studied.

\paragraph{Cluster coordinates and positivity.}
Higher Teichm\"uller spaces admit cluster coordinates using the periods of the spectral curve (see \cite{FockGoncharov}).
In our setting, the spectral curve associated to a cotangent vector to higher complex structures is Lagrangian (modulo $t^2$), so we can compute its periods (which only depend on the homotopy class). The ratios of periods should give coordinates on the cotangent bundle $T^*\mathcal{T}^n$ and also on $\mathcal{T}^n$. As for Fock--Goncharov coordinates, this coordinate system should be a cluster algebra, generated by flips. 
Since in cluster mutations, all coefficients are positive (see \cite{Fomin}), one can define cluster algebras over semi-fields, in particular $\R_{>0}$. This gives the notion of positivity in cluster varieties which should also apply for $\T^n$.



\subsection{Opers, Drinfeld--Sokolov reduction and W-geometry}

Higher complex structures and opers might be combined to a single uniformed object whose space is expected to be symplectic, which would give a two-dimensional analog of Drinfeld--Sokolov differential operators. This uniformed object should be the geometric structure needed in W-gravity theories.

\medskip
\noindent Over a Riemann surface $\Sigma$, there is the notion of a $G$-oper, generalizing the concept of a differential operator (see \cite{Drinfeld}). This space is in some sens transversal to Hitchin's section (using the non-abelian Hodge correspondence). The space of opers is parameterized by holomorphic differentials.

In dimension 1, on the circle $\mathbb{S}^1$,  the Drinfeld--Sokolov reduction of affine $\bm\hat{\g}$-connections gives a notion of $\g$-differential operators, see \cite{DS}. The space of these operators is symplectic, its deformation space is called \textit{the $W$-algebra of $\g$}.

It should be possible to combine both higher complex structures and opers to a single object on a smooth surface. The space of these objects would be parameterized by higher Beltrami differentials and holomorphic differentials, just like $T^*\mc{T}^n$. 
Roughly speaking, one expects to find Hitchin's component for $t=0$ and $\mu\neq 0$, whereas one should find opers for $t\neq 0$ and $\mu=0$.
For surfaces $\Sigma$ with boundary, one should be able to get Drinfeld--Sokolov differential operators on the boundary components out of generalized opers on $\Sigma$ by a restriction process.
In \cite{Kydonakis} and refined in \cite{Collier}, a procedure called \textit{conformal limit} is described which allows to identify Hitchin's component with the space of opers. It should be possible to understand the conformal limit in our setting.
Another question to explore is the following. It is known that $W$-transformations act on the space of opers. Can this action be extended to the space of generalized opers?

In conformal field theories for dimension 1+1, one considers classical fields on the circle $\mathbb{S}^1$ which evolve in time, sweeping out a Riemann surface with boundary. The diffeomorphism group of the circle acts as symmetry group. The infinitesimal symmetries are given by the vector fields on the circle, also called the \textit{Witt algebra}. On the quantum level, the infinitesimal symmetries form the \textit{Virasoro algebra} $\Vir$, the universal central extension of vector fields on $\mathbb{S}^1$. The Virasoro algebra is the W-algebra of $\mf{sl}_2$.
There is a generalization of this theory, called $W$-geometry or $W$-gravity (when applied to Einstein formalism), for which the symmetry algebra is the W-algebra of a simple Lie algebra $\g$ (see for instance \cite{Schoutens}). For this, the Riemann surface gets a stronger geometric structure. For the moment it is not clear what this structure should be. We conjecture that higher complex structures, or more generally $\g$-complex structures, should give a good framework for $W$-geometry.

\subsection{Quantization and mirror symmetry}

The moduli space of higher complex structures should admit a quantization. Moreover, Hitchin's integrable system is an example of mirror symmetry so one could investigate if an analog mirror symmetry can be observed within higher complex structures.

\paragraph{Quantization.}
Higher Teichm\"uller spaces are symplectic and have been quantized, using its cluster structure (see \cite{FockGoncharov}). Assuming the equivalence of Hitchin's component and the moduli space of higher complex structures, it would be interesting to describe this quantization in terms of higher complex structures. Alternatively, one could directly give a quantization once a cluster coordinate system by the periods of the spectral curve has been introduced (as explained in section \ref{cluster}).
In \cite{Teschner}, Teschner surveys the quantization of the moduli space of flat $\PSL(2,\C)$-connections. There should be generalizations of the various approaches to $\PSL(n,\C)$-connections and higher complex structures.

\paragraph{Mirror symmetry.}
In \cite{Hausel}, Hausel and Thaddeus show that the Hitchin integrable system is an example of SYZ-mirror symmetry using its hyperk\"ahler structure. In particular, complex Lagrangian torus fibers in the moduli space of $G$-Higgs bundles should correspond to hyperk\"ahler submanifolds in the space of ${}^LG$-Higgs bundles where ${}^LG$ denotes the Langlands dual group.

As explained in section \ref{bigpicture}, there should be a hyperk\"ahler manifold $\mc{M}$ associated to higher complex structures. In the complex structure $I$, the space $\mc{M}$ is the cotangent bundle to the moduli space of $\g$-complex structures $T^*\bm\hat{\mc{T}}_\g$. One should check if the SYZ-conditions are satisfied in this setting and explore the duality between $T^*\mc{T}_\g$ and $T^*\mc{T}_{{}^L\g}$.



\cleardoublepage

\appendix

\section{Regular elements in semisimple Lie algebras}\label{appendix:B}

In this appendix, we gather all properties of regular elements in semisimple Lie algebras we need in part \ref{part3} and we give precise references for these results. The main references are the books of Collingwood and McGovern \cite{Coll}, Steinberg \cite{Steinberg} and Humphreys \cite{Hum}, as well as the papers \cite{Kost} and \cite{Kost2} by Kostant.

\begin{definition}
An element $x \in \g$ is called \textbf{regular} if the dimension of its centralizer $Z(x)$ is equal to the rank of the Lie algebra $\rk(\g)$. A regular nilpotent element is called \textbf{principal nilpotent}.
\end{definition}

\begin{Remark} 
Notice that in older literature, regular elements are defined in another way, using the characteristic polynomial of the adjoint map. The ``old'' notion only includes semisimple regular elements (in the sense above). 
\hfill $\triangle$
\end{Remark}

The condition that the dimension of the centralizer has to be equal to the rank, does not come from nowhere: in fact it is the minimal possible dimension.

\begin{prop}
For any $x \in \g$, we have $\dim Z(x) \geq \rk(\g)$.
\end{prop}
See for example lemma 2.1.15. in \cite{Coll}.

For the Lie algebras $\mf{gl}_n$ and $\mf{sl}_n$, we have the following characterization of regular elements from Steinberg \cite{Steinberg}, proposition 2 in section 3.5:

\begin{prop}\label{regularsln}
For $\g=\mf{gl}_n$ or $\mf{sl}_n$ and $x\in \g$, we have the following equivalence:
$$x \text{ is regular} \; \Leftrightarrow \; \mu_x = \chi_x \; \Leftrightarrow \; x \text{ admits a cyclic vector}$$ where $\mu_x$ and $\chi_x$ respectively denote the minimal and the characteristic polynomial of $x$, seen as a matrix. 
\end{prop} 

Let us turn to the study of regular elements which are nilpotent. 

\begin{thm}
There is a unique open dense orbit (under the adjoint action) in the nilpotent variety, formed by principal nilpotent elements.
\end{thm}
The original proof is due to Kostant, see corollary 5.5. in \cite{Kost}.
See also theorem 4.1.6. in \cite{Coll}.

There is a useful characterization of principal nilpotent elements in coordinates. For this, fix a root system $R$, fix a direction giving positive roots $R_+$. Denote by $\mf{n}$ the positive nilpotent elements (upper triangular for $\mf{sl}_n$).
\begin{prop}\label{prinnilp}
Let $A \in \mf{n}$. Then $A$ is principal nilpotent iff $\alpha(A)\neq 0$ for all simple roots $\alpha$.
\end{prop}
The group version of this can be found in section 3.7. of \cite{Steinberg}.

For a principal nilpotent element $f$, its centralizer $Z(f)$ has properties quite analogous to a Cartan, the centralizer of a regular semisimple element:

\begin{thm}\label{thmKost}
For $f$ a principal nilpotent element, its centralizer $Z(f)$ is abelian and nilpotent.
\end{thm}
Kostant proves even more, using a limit argument: for any element $x\in \g$, there is an abelian subalgebra of $Z(x)$ of dimension $\rk \g$, see \cite{Kost}, theorem 5.7.
The nilpotency of $Z(f)$ can be found in \cite{Steinberg}, corollary in section 3.7. The more precise structure of $Z(x)$ for any nilpotent $x$ is described in \cite{Coll}, section 3.4.

A principal nilpotent element permits to give a preferred representative of a conjugacy class of regular elements. Given $f$ principal nilpotent, denote by $e$ the other nilpotent element in a principal $\mf{sl}_2$-triple constructed from $f$ (see Kostant \cite{Kost}). Then we get
\begin{prop}
Any regular orbit intersects $f+Z(e)$ in a unique point. So we have $\g^{reg}/G \cong f+Z(e)$.
\end{prop}
This follows from Lemma 10 of \cite{Kost2}. The set $f+Z(e)$ is called a \textit{principal slice} of $\g$ (also \textit{Kostant section}).

We are now going to ``double'' the previous setting. Define the commuting variety to be $\Comm(\g):=\{(A,B)\in \g^2 \mid [A,B]=0\}$. 

\begin{thm}[Richardson]\label{Richardson}
The set of commuting semisimple elements is dense in the commuting variety $\Comm(\g)$.
\end{thm}
See the paper of Richardson \cite{Richardson} for a proof.
As a consequence, $\Comm(\g)$ is an irreducible variety, but highly singular.

With this, we can explore the minimal dimension of a centralizer of a commuting pair:
\begin{prop}\label{doublecomm}
For $(A,B) \in \Comm(\g)$, we have $\dim Z(A,B) \geq \rk \g.$
\end{prop}
\begin{proof}
Consider the set $M$ of elements with centralizer of minimal dimension. Since $$M=\{(A,B) \in \Comm(\g) \mid \rk(ad_A, ad_B) \text{ maximal}\}$$ we see that $M$ is Zariski-open. By the theorem of Richardson it intersects the space of semisimple pairs for which the common centralizer is a Cartan $\h$, so of dimension $\rk \g$.
\end{proof}


\cleardoublepage


\end{document}